\documentclass[12pt,a4paper]{article}

\usepackage{amsmath,amscd,amssymb,amsthm,mathrsfs,amsfonts,layout,indentfirst,graphicx,caption,mathabx,  tikz, stmaryrd,appendix,calc,imakeidx}
\usepackage{palatino}  %template
\usepackage{slashed} % Dirac operator
\usepackage{mathrsfs} % Enable using \mathscr
\usepackage{extarrows} % long equal sign, \xlongequal{blablabla}
\makeindex

\usepackage[nottoc]{tocbibind}   % Add  reference to ToC

\usepackage{lipsum}
\let\OLDthebibliography\thebibliography
\renewcommand\thebibliography[1]{
	\OLDthebibliography{#1}
	\setlength{\parskip}{0pt}
	\setlength{\itemsep}{2pt} 
}

\allowdisplaybreaks  %allow aligns to break between pages
\usepackage{latexsym}
\usepackage{chngcntr}
\usepackage[colorlinks,linkcolor=blue,anchorcolor=blue,linktocpage,pagebackref]{hyperref}
\hypersetup{citecolor=[rgb]{0,0.5,0}}

\usepackage{fullpage}
\counterwithin{figure}{section}

\theoremstyle{definition}
\newtheorem{df}{Definition}[section]

\newtheorem{cv}[df]{Convention}

\theoremstyle{plain}
\newtheorem{thm}[df]{Theorem}

\newtheorem{pp}[df]{Proposition}
\newtheorem{co}[df]{Corollary}
\newtheorem{lm}[df]{Lemma}

\newcommand{\fk}{\mathfrak}
\newcommand{\mc}{\mathcal}
\newcommand{\wtd}{\widetilde}
\newcommand{\wht}{\widehat}

\newcommand{\ovl}{\overline}
\newcommand{\Tr}{\mathrm{Tr}}
\newcommand{\End}{\mathrm{End}} %endomorphism
\newcommand{\id}{\mathbf{1}}
\newcommand{\Hom}{\mathrm{Hom}}

\newcommand{\ev}{\mathrm{ev}}
\newcommand{\coev}{\mathrm{coev}}

\newcommand{\Rep}{\mathrm{Rep}}

\newcommand{\uni}{\mathrm{u}}
\newcommand{\ssp}{\mathrm{ss}}

\newcommand{\bk}[1]{\langle {#1}\rangle}

\newcommand{\im}{\mathbf{i}}
\newcommand{\Co}{\complement}
\newcommand{\RepV}{\mathrm{Rep}^\uni(V)}
\newcommand{\RepA}{\mathrm{Rep}^\uni(A)}
\newcommand{\RepAU}{\mathrm{Rep}^\uni(A_U)}
\newcommand{\RepU}{\mathrm{Rep}^\uni(U)}
\newcommand{\BIM}{\mathrm{BIM}^\uni}
\newcommand{\BIMA}{\mathrm{BIM}^\uni(A)}
\newcommand{\rfl}{\varepsilon}
\newcommand{\Pij}{\Psi_{i,j}}
\newcommand{\Pjk}{\Psi_{j,k}}

\newcommand{\Cij}{\chi_{i,j}}
\newcommand{\Cjk}{\chi_{j,k}}
\newcommand{\Cijk}{\chi_{i,j,k}}

\numberwithin{equation}{section}

\title{Q-systems and extensions of completely unitary vertex operator algebras}
\author{{\sc Bin Gui}%\\
}
\date{}
\begin{document}\sloppy % avoid stretch into margins
	\pagenumbering{arabic}
	%\pagenumbering{gobble}
	%\newpage
	%\setcounter{page}{1}
	\setcounter{section}{-1}

	\maketitle

\tableofcontents
	
\newpage

\begin{abstract}

	Complete unitarity is a natural condition on a CFT-type regular VOA  which ensures that its  modular tensor category is unitary. In this paper we show that any CFT-type unitary (conformal) extension $U$ of a completely unitary VOA $V$ is completely unitary. Our method is to relate  $U$ with a Q-system $A_U$  in the  $C^*$-tensor category $\RepV$ of unitary $V$-modules. We also update the main result of \cite{KO02} to the unitary cases by showing that the tensor category $\RepU$ of unitary $U$-modules is equivalent to the tensor category $\RepAU$ of unitary $A_U$-modules  as unitary modular tensor categories.

As an application, we obtain infinitely many new (regular and) completely unitary VOAs including all CFT-type $c<1$ unitary VOAs.   We also show that the latter are in one to one correspondence with the (irreducible) conformal nets of the same central charge $c$, the classification of which is given by \cite{KL04}.
\end{abstract}

\section{Introduction}

This is the first part in a series of papers to study the relations between  unitary vertex operator algebra (VOA) extensions and  conformal net extensions. We will always focus on rational conformal field theories, so our VOAs are assumed to be CFT-type, self-dual, and regular, so that the  categories of VOA modules are modular tensor categories (MTCs). 

Although both unitary VOAs and conformal nets are mathematical formulations of unitary chiral CFTs, they are defined and studied in rather different ways, with the former being more algebraic and geometric, and the latter mainly functional analytic. A systematic study to relate these two approaches was initiated by Carpi-Kawahigashi-Longo-Weiner \cite{CKLW18} and followed by \cite{Ten19a,Ten18b,Gui20,Ten19b,CW,CWX}, etc. In these works the methods of relating the two approaches are transcendental and have a lot of analytic subtleties. Due to these subtleties, certain models (such as unitary Virasoro VOAs, and unitary affine VOAs especially of type $A$) are easier to analyze than the others. On the other hand, when studying  the extensions and conformal inclusions of chiral CFTs, the main tools in the two approaches are quite similar: both are (commutative) associative algebras in a tensor category $\mc C$ (called $\mc C$-algebras); see \cite{KO02,HKL15,CKM17} for VOA extensions, and \cite{LR95,KL04,BKLR15} for conformal net extensions; see also \cite{FRS02,FS03} for the general notion of algebra objects inside a tensor category. In this and the following papers, we will see that $\mc C$-algebras are also powerful tools for relating  unitary VOA extensions and conformal net extensions in the above mentioned systematic and transcendental settings. 

There is, however, one important difference between the $\mc C$-algebras used in the two approaches:  for conformal net extensions the $\mc C$-algebras are \emph{unitary}. Unitarity is an essential property for conformal nets and operator algebras but not quite necessary for VOAs. However, it is impossible to relate VOAs and conformal nets without adding unitary structures on VOAs (and their representation categories). This point is already clear in \cite{Gui20}, where we have seen that  to relate the tensor categories of VOAs and conformal nets, one has to first make the VOA tensor categories unitary.

In this paper our main goal is to relate the $C^*$-tensor categories of unitary VOA extensions with those of unitary $\mc C$-algebras (also called $C^*$-Frobenius algebras or (under slightly stronger condition) $Q$-systems \cite{Lon94}).  As applications, we prove many important unitary properties of VOA extensions, the most important of which are the complete unitarity of VOAs as defined below.

\subsubsection*{Complete unitarity of unitary VOA extensions}

A CFT-type regular VOA $V$ is called \textbf{completely unitary} if the following conditions are satisfied.
\begin{itemize}
	\item $V$ is unitary \cite{DL14}, which means roughly that $V$ is equipped with an inner product and an antiunitary antiautomorphism $\Theta$ which relates the vertex operators of $V$ to their adjoints.
	\item Any $V$-module admits a unitary structure. Since $V$-modules are semisimple, it suffices to assume the unitarizability of irreducible $V$-modules.
	\item For any irreducible unitary $V$-modules $W_i,W_j,W_k$, the non-degenerate \emph{invariant sesquilinear form} $\Lambda$ introduced in \cite{Gui19b} and defined on  the dual vector space of type $k\choose i~j$  intertwining operators of $V$ is positive.
\end{itemize} 
The importance of complete unitarity lies in the following theorem.
\begin{thm}[\cite{Gui19b} theorem 7.9]
If $V$ is a CFT-type, regular, and completely unitary VOA, then the unitary $V$-modules form a unitary modular tensor category.
\end{thm}
However, compared to unitarity, complete unitarity is much harder to prove since  not only vertex operators but also  intertwining operators need to be taken care of.  In this paper, our main result as follows provides a powerful tool for proving the complete unitarity.
\begin{thm}\label{lb54}
Suppose that $V$ is a CFT-type, regular, and completely unitary VOA, 	and $U$ is a CFT-type unitary VOA extension of $V$. Then $U$ is also completely unitary.
\end{thm}

Roughly speaking, if we know that  a  unitary VOA $U$ is an extension of a completely unitary VOA $V$, then $U$ is also completely unitary.  As applications, since the complete unitarity has been established for  unitary affine VOAs and $c<1$ unitary Virasoro VOAs (minimal models) as well as their tensor products (proposition \ref{lb56}), we know that all their unitary extensions are completely unitary. In particular, these extensions have unitary modular tensor categories. We also show that these tensor categories are equivalent to the unitary modular tensor categories associated to the corresponding $Q$-system. To  be more precise, we prove

\begin{thm}\label{lb59}
Let $V$ be a CFT-type, regular, and completely unitary VOA, and let $U$ be a CFT-type unitary extension of $V$ whose Q-system is $A_U$. If $\RepU$ is the category of unitary $U$-modules, and  $\RepAU$ is the category of unitary $A_U$-modules, then $\RepU$ is natually equivalent to $\RepAU$ as unitary modular tensor categories.
\end{thm}

A non-unitary version of the above theorem has already been proved in \cite{CKM17}: $\RepU$ is known to be equivalent to $\RepAU$ as modular tensor categories. The above theorem says that the unitary structures of the two modular tensor categories are also equivalent in a natural way. We remark that the unitary tensor structures compatible with the $*$-structure of $\RepU$ are unique by \cite{Reu19,CCP}. Thus \cite{Reu19,CCP} provide a different method of proving the above theorem.

We would like to point out that our results on the relation between unitary VOA extensions and Q-systems also provide a new method of proving the unitarity of VOAs. For instance, the irreducible $c<1$ conformal nets are classified as (finite index) extensions of Virasoro nets  by Kawahigashi-Longo in \cite{KL04} (table 3), and their VOA counterparts are given by Dong-Lin in \cite{DL15}. However, Dong-Lin were not able to prove the unitarity of two exceptional cases: the types $(A_{10},E_6)$ and $(A_{28},E_8)$ (see remark 4.16 of \cite{DL15}). But since these two types are realized as commutative Q-systems in \cite{KL04}, we can now show that the corresponding VOAs are actually unitary. This proves that the  $c<1$ CFT-type   unitary VOAs are in one to one correspondence with the  irreducibile  conformal nets with the same central charge $c$.

%Let us  make some comments on the proof of theorem \ref{lb54}. We first consider the first two of the three conditions in the definition of complete unitarity. A VOA $V$ satisfying these two conditions is called \emph{strongly unitary}. In other words, strong unitarity means that one knows only the unitarity of $V$ and its representations but not of the tensor products of its representations. Interestingly, our knowledge on  strong unitarity is not much better than on complete unitarity. So far, the examples of strongly unitary VOAs are restricted to unitary Virasoro VOAs, unitary affine VOAs, even lattice VOAs, and unitary holomorphic VOAs (such as moonshine VOA). There are mainly two ways of proving the strong unitarity of a unitary VOA. The first one is Lie algebraic as in \cite{Kac94,Was10}, and can be applied only to affine VOAs. The second one is to embedd the VOA $V$ in some larger (super) VOA $\wtd V$. Conformal embedding is not required. Then a unitary $\wtd V$-module $\wtd W$ produces a lot of unitary $V$-modules by restrictions. If one can show that all irreducible $V$-modules arise in this way, then $V$ is proved to be strongly unitary. For example, if $V$ is a unitary affine type $A$ VOA then $\wtd V$ can be a tensor product of unitary free fermion VOAs whose only representations are $\wtd V$ itself, which is clearly unitary. (See )

We have left several important questions unanswered in this paper. We see that Q-systems can relate  unitary VOA extensions and conformal net extensions. But \cite{CKLW18} also provides a uniform way of relating unitary VOAs and conformal nets using smeared vertex operators. Are these two relations compatible? Moreover, do the VOA extensions and the corresponding conformal net extensions have the same tensor categories? Answers to these questions are out of scope of this paper, so we leave them to future works.

\subsubsection*{Outline of the paper}

In chapter 1 we review the construction and basis properties  of VOA tensor categories due to Huang-Lepowsky. We also review various methods of constructing new intertwining operators from  old ones, and translate them into tensor categorical language. The translation of adjoint and conjugate intertwining operators is the most important result of this chapter. Unitary VOAs, unitary representations, and the unitary structure on VOA tensor categories are also reviewed.

In chapter 2 we relate unitary VOA extensions and Q-systems as well as their (unitary) representations. The first two sections serve as background materials. In section 2.1 we review the relation between VOA extensions and commutative $\mc C$-algebras as in \cite{HKL15}. Their results are  adapted to our unitary setting. In section 2.2 we review various notions concerning dualizable objects in $C^*$-tensor categories. Most importantly, we review the construction of standard evaluations and coevaluations in $C^*$-tensor categories necessary for defining quantum traces and quantum dimensions. Standard reference for this topic is \cite{LR97}. We also explain why the naturally defined evaluations and coevaluations in the tensor categories associated to completely unitary VOAs are standard. In section 2.3 we define a notion of unitary $\mc C$-algebras which is a direct translation of unitary VOA extensions in categorical language. This notion is related to $C^*$-Frobenius algebras and Q-systems in section 2.4. The  equivalence of $c<1$ unitary VOAs and conformal nets is also proved in that section.     A VOA $U$ is called strongly unitary if it satisfies the first two of the three conditions defining complete unitarity.  Therefore, strong unitarity means the unitarity of $U$ and the unitarizability of all $U$-modules. In  section 2.5, we give two proofs that any unitary extension $U$ of a completely unitary VOA $V$ is strongly unitary. The first proof uses induced representations, and the second one uses a result of standard representations of Q-systems in \cite{BKLR15}.

In chapter 3 we use the $C^*$-tensor categories of the bimodules of Q-systems to prove the complete unitarity of unitary VOA extensions. We review the construction and basic properties of these $C^*$-tensor categories in the first four sections. Although these results are known to experts (cf. \cite{NY16} chapter 6 or \cite{NY18b} section 4.1), we provide detailed and self-contained proofs of all the relevant facts, which we hope are helpful to the readers who are not familiar with tensor categories. We present the theory in such a way that it can be directly compared with the (Hermitian) tensor categories of unitary VOA modules. So in some sense our approach is closer in spirit to \cite{KO02,CKM17}.  In section 3.5 we prove the main results of this paper:  theorems \ref{lb54} and \ref{lb59} (which are theorem \ref{lb55} of that section).  Finally, applications are given in section 3.6.

\subsubsection*{Acknowledgment}

I'm grateful to Robert McRae for a helpful discussion of the rationality of VOA extensions; Marcel Bischoff for informing me of certain literature on conformal inclusions; and Sergey Neshveyev and Makoto Yamashita for many discussions and helpful comments.

\section{Intertwining operators and tensor categories of unitary VOAs}\label{lb7}
\subsection{Braiding, fusion, and contragredient intertwining operators}\label{lb2}
Let $V$ be a self-dual VOA with vacuum vector $\Omega$ and conformal vector $\nu$. For any $v\in V$, its vertex operator is written as $Y(v,z)=\sum_{n\in\mathbb Z}Y(v)_nz^{-n-1}$. Then $\{L_n=Y(\nu)_{n+1}:n\in\mathbb Z \}$ are the Virasoro operators. Throughout this paper, we assume that the grading of $V$ satisfies $V=\mathbb C\Omega\oplus\big(\bigoplus_{n\in\mathbb Z_{>0}}V(n)\big)$ where $V(n)$ is the eigenspace of $L_0$ with eigenvalue $n$, i.e., $V$ is of CFT-type. We assume also that $V$ is  regular, which is equivalent to that $V$ is rational and $C_2$-cofinite. (See \cite{DLM97} for the definition of these terminologies as well as the equivalence theorem.) Such condition guarantees that the intertwining operators of $V$ satisfy the braiding and fusion relations \cite{Hua95,Hua05a} and the modular invariance \cite{Zhu96,Hua05b}, and that the  category $\Rep(V)$ \index{Rep@$\Rep(V)$} of (automatically semisimple) $V$-modules is indeed a modular tensor category \cite{Hua08}. 

%(In this article, unless otherwise stated, a $V$-module or equivalently a representation of $V$ is always understood to be semi-simple, i.e., a finite direct sum of irreducible ones.)

We refer the reader to \cite{BK01,EGNO,Tur} for the general theory of tensor categories, and \cite{Hua08} for the construction of the tensor category $\Rep(V)$ of $V$-modules. A brief review of this construction can also be found in \cite{Gui19a} section 2.4 or \cite{Gui20} section 4.1. Here we outline some of the key properties of $\Rep(V)$ which will be used in the future. 

Representations of $V$ are written as $W_i,W_j,W_k$,\index{Wi@$W_i$} etc. If a $V$-module $W_i$ is given, its contragredient module (cf. \cite{FHL93} section 5.2) is written as $W_{\ovl i}$ \index{Wibar@$W_{\ovl i}$}.  $W_{\ovl{\ovl i}}$, the contragredient module of $W_{\ovl i}$, can naturally be identified with $W_i$. So we write $\ovl{\ovl i}=i$.  Note that the symbol $i$ is now reserved for representations. So we write the imaginary unit $\sqrt{-1}$  as $\im$. Let $W_0$\index{Wi@$W_i$!$W_0=V$} be the vacuum $V$-module, which is also the identity object in $\Rep(V)$. The product of $V$-modules is constructed in such a way that for any $W_i,W_j,W_k$ there is a canonical isomorphism of linear spaces
\begin{gather}
\mc Y:\Hom (W_i\boxtimes W_j,W_k)\overset{\simeq}{\longrightarrow}\mc V{k\choose i~j},\qquad\alpha\mapsto\mc Y_\alpha\label{eq61}
\end{gather} 
where $\mc V{k\choose i~j}=\mc V{W_k\choose W_iW_j}$\index{Vijk@$\mc V{k\choose i~j}$}  is the (finite-dimensional) vector space of intertwining operators of $V$. For any $w^{(i)}\in W_i$, we write $\mc Y_\alpha(w^{(i)},z)=\sum_{s\in\mathbb C}\mc Y_\alpha(w^{(i)})_sz^{-s-1}$ \index{Ya@$\mc Y_\alpha$} where $z$ is a complex variable defined in $\mathbb C^\times:=\mathbb C\setminus\{0\}$ and $\mc Y_\alpha(w^{(i)})_s:W_j\rightarrow W_k$ is the $s$-th mode of the intertwining operator. We say that $W_i,W_j,W_k$ are respectively the charge space, the source space, and the target space of the intertwining operator $\mc Y_\alpha$. Tensor products of morphisms are defined such that the following condition is satisfied: If $F\in\Hom(W_{i'},W_i),G\in\Hom(W_{j'},W_j),K\in\Hom(W_k,W_{k'})$, then for any $w^{(i')}\in W_{i'}$,
\begin{align}
\mc Y_{K\alpha(F\otimes G)}(w^{(i')},z)=K\mc Y_\alpha(Fw^{(i')},z)G.\label{eq31}
\end{align}

One way to realize the above properties is as follows: Notice first of all that $V$ has finitely many equivalence classes of irreducible unitary $V$-modules. Fix, for each equivalence class, a representing element, and let them form a finite set $\mc E$.\index{E@$\mc E$} We assume that the vacuum unitary module $V=W_0$ is in $\mc E$. If $W_t\in\mc E$, we will use the notation $t\in\mc E$ to simplify formulas. We then define $W_i\boxtimes W_j$ to be $\bigoplus_{t\in\mc E}\mc V{t\choose i~j}^*\otimes W_t$ \cite{HL95} (here $\mc V{t\choose i~j}^*$ is the dual vector space of $\mc V{t\choose i~j}$). Then for each $t\in\mc E$ there is a natural identification between $\Hom(W_i\boxtimes W_j,W_t)$ and $\mc V{t\choose i~j}$, which can be extended to the general case $\Hom(W_i\boxtimes W_j,W_k)\simeq\mc V{k\choose i~j}$ via the canonical isomorphisms
\begin{gather}
\Hom(W_i\boxtimes W_j,W_k)\simeq\bigoplus_{t\in\mc E}\Hom(W_i\boxtimes W_j,W_t)\otimes \Hom(W_t,W_k),\label{eq62}\\
\mc V{k\choose i~j}\simeq\bigoplus_{t\in\mc E}\mc V{t\choose i~j}\otimes \Hom(W_t,W_k).
\end{gather}

The tensor structure of $\Rep(V)$ is defined in such a way that it is related to the fusion relations of the intertwining operators of $V$ as follows. Choose non-zero $z,\zeta$ with the same arguments (notation: $\arg z=\arg \zeta$) satisfying $0<|z-\zeta|<|\zeta|<|z|$. In particular, we assume that $z,\zeta$ are on a common ray stemming from the origin. We also choose $\arg(z-\zeta)=\arg\zeta=\arg z$. Suppose that we have $W_i,W_j,W_k,W_l,W_p,W_q$ in $\Rep(V)$, and intertwining operators $\mc Y_\alpha\in\mc V{l\choose i~p},\mc Y_\beta\in\mc V{p\choose j~k},\mc Y_\gamma\in\mc V{q\choose i~j},\mc Y_\delta\in\mc V{l\choose q~k}$, such that for any $w^{(i)}\in W_i,w^{(j)}\in W_j$, the following fusion relation holds when acting on $W_k$:
\begin{align}
\mc Y_\alpha(w^{(i)},z)\mc Y_\beta(w^{(j)},\zeta)=\mc Y_\delta\big(\mc Y_\gamma(w^{(i)},z-\zeta)w^{(j)},\zeta\big).\label{eq16}
\end{align}
Then, under the identification of $W_i\boxtimes(W_j\boxtimes W_k)$ and $(W_i\boxtimes W_j)\boxtimes W_k$ (which we denote by $W_i\boxtimes W_j\boxtimes W_k$) via the associativity isomorphism, we have the  identity  $\alpha(\id_i\otimes\beta)=\delta(\gamma\otimes\id_k)$, which can be expressed graphically as
\begin{align}
\vcenter{\hbox{{\def\svgscale{0.6}
%% Accompanies image file 'fusion-relation.pdf' (pdf, eps, ps)
\begingroup%
  \makeatletter%
  \providecommand\color[2][]{%
    \errmessage{(Inkscape) Color is used for the text in Inkscape, but the package 'color.sty' is not loaded}%
    \renewcommand\color[2][]{}%
  }%
  \providecommand\transparent[1]{%
    \errmessage{(Inkscape) Transparency is used (non-zero) for the text in Inkscape, but the package 'transparent.sty' is not loaded}%
    \renewcommand\transparent[1]{}%
  }%
  \providecommand\rotatebox[2]{#2}%
  \newcommand*\fsize{\dimexpr\f@size pt\relax}%
  \newcommand*\lineheight[1]{\fontsize{\fsize}{#1\fsize}\selectfont}%
  \ifx\svgwidth\undefined%
    \setlength{\unitlength}{135.55264824bp}%
    \ifx\svgscale\undefined%
      \relax%
    \else%
      \setlength{\unitlength}{\unitlength * \real{\svgscale}}%
    \fi%
  \else%
    \setlength{\unitlength}{\svgwidth}%
  \fi%
  \global\let\svgwidth\undefined%
  \global\let\svgscale\undefined%
  \makeatother%
  \begin{picture}(1,1.00166023)%
    \lineheight{1}%
    \setlength\tabcolsep{0pt}%
    \put(0,0){\includegraphics[width=\unitlength,page=1]{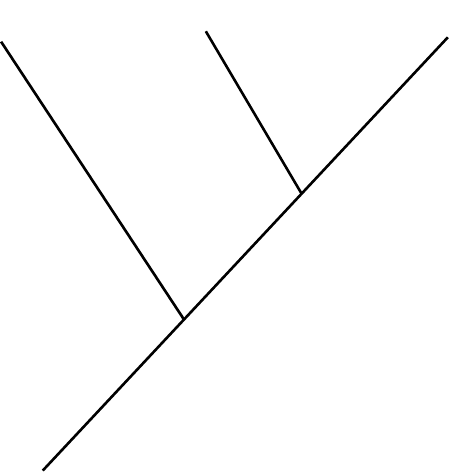}}%
    \put(0.073009,0.90778093){\color[rgb]{0,0,0}\makebox(0,0)[lt]{\lineheight{1.25}\smash{\begin{tabular}[t]{l}$i$\end{tabular}}}}%
    \put(0.44749362,0.94567521){\color[rgb]{0,0,0}\makebox(0,0)[lt]{\lineheight{1.25}\smash{\begin{tabular}[t]{l}$j$\end{tabular}}}}%
    \put(0.89553759,0.95682065){\color[rgb]{0,0,0}\makebox(0,0)[lt]{\lineheight{1.25}\smash{\begin{tabular}[t]{l}$k$\end{tabular}}}}%
    \put(0.11313235,0.17218609){\color[rgb]{0,0,0}\makebox(0,0)[lt]{\lineheight{1.25}\smash{\begin{tabular}[t]{l}$l$\end{tabular}}}}%
    \put(0.39894833,0.51269412){\color[rgb]{0,0,0}\makebox(0,0)[lt]{\lineheight{1.25}\smash{\begin{tabular}[t]{l}$p$\end{tabular}}}}%
    \put(0.64232305,0.52407988){\color[rgb]{0,0,0}\makebox(0,0)[lt]{\lineheight{1.25}\smash{\begin{tabular}[t]{l}$\beta$\end{tabular}}}}%
    \put(0.40959934,0.2747232){\color[rgb]{0,0,0}\makebox(0,0)[lt]{\lineheight{1.25}\smash{\begin{tabular}[t]{l}$\alpha$\end{tabular}}}}%
    \put(0,0){\includegraphics[width=\unitlength,page=2]{fusion-relation.pdf}}%
  \end{picture}%
\endgroup%
}}}~~=~~\vcenter{\hbox{{\def\svgscale{0.6}
%% Accompanies image file 'fusion-relation-2.pdf' (pdf, eps, ps)
\begingroup%
  \makeatletter%
  \providecommand\color[2][]{%
    \errmessage{(Inkscape) Color is used for the text in Inkscape, but the package 'color.sty' is not loaded}%
    \renewcommand\color[2][]{}%
  }%
  \providecommand\transparent[1]{%
    \errmessage{(Inkscape) Transparency is used (non-zero) for the text in Inkscape, but the package 'transparent.sty' is not loaded}%
    \renewcommand\transparent[1]{}%
  }%
  \providecommand\rotatebox[2]{#2}%
  \newcommand*\fsize{\dimexpr\f@size pt\relax}%
  \newcommand*\lineheight[1]{\fontsize{\fsize}{#1\fsize}\selectfont}%
  \ifx\svgwidth\undefined%
    \setlength{\unitlength}{124.99204469bp}%
    \ifx\svgscale\undefined%
      \relax%
    \else%
      \setlength{\unitlength}{\unitlength * \real{\svgscale}}%
    \fi%
  \else%
    \setlength{\unitlength}{\svgwidth}%
  \fi%
  \global\let\svgwidth\undefined%
  \global\let\svgscale\undefined%
  \makeatother%
  \begin{picture}(1,1.08629079)%
    \lineheight{1}%
    \setlength\tabcolsep{0pt}%
    \put(0,0){\includegraphics[width=\unitlength,page=1]{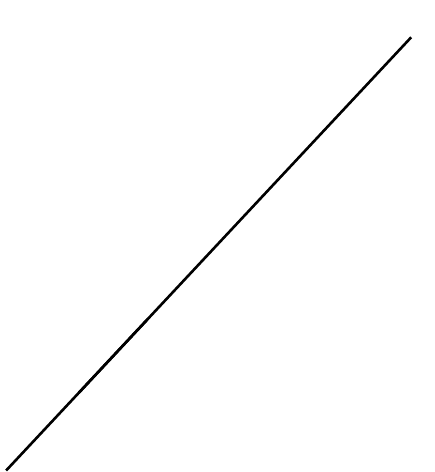}}%
    \put(-0.00531284,0.9844796){\color[rgb]{0,0,0}\makebox(0,0)[lt]{\lineheight{1.25}\smash{\begin{tabular}[t]{l}$i$\end{tabular}}}}%
    \put(0.40081207,1.02557541){\color[rgb]{0,0,0}\makebox(0,0)[lt]{\lineheight{1.25}\smash{\begin{tabular}[t]{l}$j$\end{tabular}}}}%
    \put(0.88671154,1.03766269){\color[rgb]{0,0,0}\makebox(0,0)[lt]{\lineheight{1.25}\smash{\begin{tabular}[t]{l}$k$\end{tabular}}}}%
    \put(0.03820055,0.18673404){\color[rgb]{0,0,0}\makebox(0,0)[lt]{\lineheight{1.25}\smash{\begin{tabular}[t]{l}$l$\end{tabular}}}}%
    \put(0,0){\includegraphics[width=\unitlength,page=2]{fusion-relation-2.pdf}}%
    \put(0.38147274,0.66538155){\color[rgb]{0,0,0}\makebox(0,0)[lt]{\lineheight{1.25}\smash{\begin{tabular}[t]{l}$q$\end{tabular}}}}%
    \put(0.14036893,0.69644643){\color[rgb]{0,0,0}\makebox(0,0)[lt]{\lineheight{1.25}\smash{\begin{tabular}[t]{l}$\gamma$\end{tabular}}}}%
    \put(0.54343935,0.45264917){\color[rgb]{0,0,0}\makebox(0,0)[lt]{\lineheight{1.25}\smash{\begin{tabular}[t]{l}$\delta$\end{tabular}}}}%
    \put(0,0){\includegraphics[width=\unitlength,page=3]{fusion-relation-2.pdf}}%
  \end{picture}%
\endgroup%
}}}.
\end{align}
Here we take the convention that morphisms go from top to bottom.

\begin{cv}\label{lb43}
When we consider fusion relations in the form \eqref{eq16}, we always assume $0<|z-\zeta|<|\zeta|<|z|$ and $\arg(z-\zeta)=\arg\zeta=\arg z$.
\end{cv}	

The braided and the contragredient intertwining operators are two major ways of constructing new intertwining operators from old ones \cite{FHL93}. As we shall see, they can all be translated into operations on morphisms. We first discuss braiding. Given $\mc Y_\alpha\in\mc V{k\choose i~j}$, we can define \textbf{braided intertwining operators} $B_+\mc Y_\alpha,B_-\mc Y_\alpha$\index{Ya@$\mc Y_\alpha$!$B_\pm\mc Y_\alpha=\mc Y_{B_\pm\alpha}$} of type $\mc V{k\choose j~i}$ in the following way: Choose any $w^{(i)}\in W_i,w^{(j)}\in W_j$. Then
\begin{align}
(B_\pm\mc Y_\alpha)(w^{(j)},z)w^{(i)}=e^{zL_{-1}}\mc Y_\alpha(w^{(i)},e^{\pm\im\pi}z)w^{(j)}.
\end{align}
Then the  braid isomorphism $\ss=\ss_{i,j}:W_i\boxtimes W_j\rightarrow W_j\boxtimes W_i$\index{zz@$\ss=\ss_{i,j}$} is constructed in such a way that $B_\pm\mc Y_\alpha=\mc Y_{\alpha\circ\ss^{\pm1}}$. Write $B_+\mc Y_\alpha=\mc Y_{B_+\alpha}$ and $B_-\mc Y_\alpha=\mc Y_{B_-\alpha}$. Then $B_\pm\alpha=\alpha\circ\ss^{\pm 1}$. Using $\vcenter{\hbox{{\def\svgscale{0.2}
%% Accompanies image file 'ss.pdf' (pdf, eps, ps)
\begingroup%
  \makeatletter%
  \providecommand\color[2][]{%
    \errmessage{(Inkscape) Color is used for the text in Inkscape, but the package 'color.sty' is not loaded}%
    \renewcommand\color[2][]{}%
  }%
  \providecommand\transparent[1]{%
    \errmessage{(Inkscape) Transparency is used (non-zero) for the text in Inkscape, but the package 'transparent.sty' is not loaded}%
    \renewcommand\transparent[1]{}%
  }%
  \providecommand\rotatebox[2]{#2}%
  \newcommand*\fsize{\dimexpr\f@size pt\relax}%
  \newcommand*\lineheight[1]{\fontsize{\fsize}{#1\fsize}\selectfont}%
  \ifx\svgwidth\undefined%
    \setlength{\unitlength}{86.32166943bp}%
    \ifx\svgscale\undefined%
      \relax%
    \else%
      \setlength{\unitlength}{\unitlength * \real{\svgscale}}%
    \fi%
  \else%
    \setlength{\unitlength}{\svgwidth}%
  \fi%
  \global\let\svgwidth\undefined%
  \global\let\svgscale\undefined%
  \makeatother%
  \begin{picture}(1,1.12223564)%
    \lineheight{1}%
    \setlength\tabcolsep{0pt}%
    \put(0,0){\includegraphics[width=\unitlength,page=1]{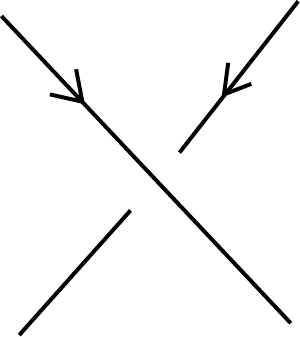}}%
  \end{picture}%
\endgroup%
}}}$ and $\vcenter{\hbox{{\def\svgscale{0.2}
%% Accompanies image file 'ss2.pdf' (pdf, eps, ps)
\begingroup%
  \makeatletter%
  \providecommand\color[2][]{%
    \errmessage{(Inkscape) Color is used for the text in Inkscape, but the package 'color.sty' is not loaded}%
    \renewcommand\color[2][]{}%
  }%
  \providecommand\transparent[1]{%
    \errmessage{(Inkscape) Transparency is used (non-zero) for the text in Inkscape, but the package 'transparent.sty' is not loaded}%
    \renewcommand\transparent[1]{}%
  }%
  \providecommand\rotatebox[2]{#2}%
  \newcommand*\fsize{\dimexpr\f@size pt\relax}%
  \newcommand*\lineheight[1]{\fontsize{\fsize}{#1\fsize}\selectfont}%
  \ifx\svgwidth\undefined%
    \setlength{\unitlength}{86.94598813bp}%
    \ifx\svgscale\undefined%
      \relax%
    \else%
      \setlength{\unitlength}{\unitlength * \real{\svgscale}}%
    \fi%
  \else%
    \setlength{\unitlength}{\svgwidth}%
  \fi%
  \global\let\svgwidth\undefined%
  \global\let\svgscale\undefined%
  \makeatother%
  \begin{picture}(1,1.14872907)%
    \lineheight{1}%
    \setlength\tabcolsep{0pt}%
    \put(0,0){\includegraphics[width=\unitlength,page=1]{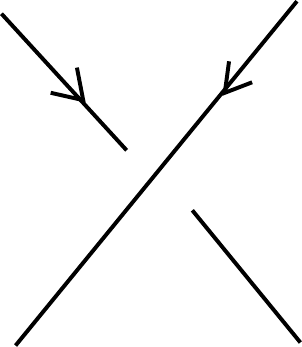}}%
  \end{picture}%
\endgroup%
}}}$ to denote $\ss$ and $\ss^{-1}$ respectively, this formula can be pictured as
\begin{gather}
\vcenter{\hbox{{\def\svgscale{0.6}
%% Accompanies image file 'braid-relation.pdf' (pdf, eps, ps)
\begingroup%
  \makeatletter%
  \providecommand\color[2][]{%
    \errmessage{(Inkscape) Color is used for the text in Inkscape, but the package 'color.sty' is not loaded}%
    \renewcommand\color[2][]{}%
  }%
  \providecommand\transparent[1]{%
    \errmessage{(Inkscape) Transparency is used (non-zero) for the text in Inkscape, but the package 'transparent.sty' is not loaded}%
    \renewcommand\transparent[1]{}%
  }%
  \providecommand\rotatebox[2]{#2}%
  \newcommand*\fsize{\dimexpr\f@size pt\relax}%
  \newcommand*\lineheight[1]{\fontsize{\fsize}{#1\fsize}\selectfont}%
  \ifx\svgwidth\undefined%
    \setlength{\unitlength}{83.7225393bp}%
    \ifx\svgscale\undefined%
      \relax%
    \else%
      \setlength{\unitlength}{\unitlength * \real{\svgscale}}%
    \fi%
  \else%
    \setlength{\unitlength}{\svgwidth}%
  \fi%
  \global\let\svgwidth\undefined%
  \global\let\svgscale\undefined%
  \makeatother%
  \begin{picture}(1,1.20701809)%
    \lineheight{1}%
    \setlength\tabcolsep{0pt}%
    \put(0,0){\includegraphics[width=\unitlength,page=1]{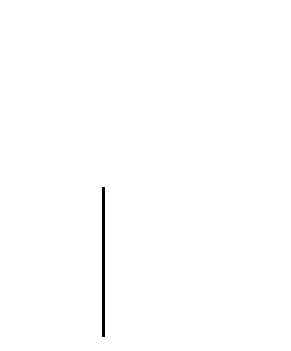}}%
    \put(-0.00396585,1.17071887){\color[rgb]{0,0,0}\makebox(0,0)[lt]{\lineheight{1.25}\smash{\begin{tabular}[t]{l}$j$\end{tabular}}}}%
    \put(0.62637033,1.16051137){\color[rgb]{0,0,0}\makebox(0,0)[lt]{\lineheight{1.25}\smash{\begin{tabular}[t]{l}$i$\end{tabular}}}}%
    \put(0.19508791,0.00702189){\color[rgb]{0,0,0}\makebox(0,0)[lt]{\lineheight{1.25}\smash{\begin{tabular}[t]{l}$k$\end{tabular}}}}%
    \put(0.38751657,0.46228151){\color[rgb]{0,0,0}\makebox(0,0)[lt]{\lineheight{1.25}\smash{\begin{tabular}[t]{l}$B_+\alpha$\end{tabular}}}}%
    \put(0,0){\includegraphics[width=\unitlength,page=2]{braid-relation.pdf}}%
  \end{picture}%
\endgroup%
}}}=\vcenter{\hbox{{\def\svgscale{0.6}
%% Accompanies image file 'braid-relation-2.pdf' (pdf, eps, ps)
\begingroup%
  \makeatletter%
  \providecommand\color[2][]{%
    \errmessage{(Inkscape) Color is used for the text in Inkscape, but the package 'color.sty' is not loaded}%
    \renewcommand\color[2][]{}%
  }%
  \providecommand\transparent[1]{%
    \errmessage{(Inkscape) Transparency is used (non-zero) for the text in Inkscape, but the package 'transparent.sty' is not loaded}%
    \renewcommand\transparent[1]{}%
  }%
  \providecommand\rotatebox[2]{#2}%
  \newcommand*\fsize{\dimexpr\f@size pt\relax}%
  \newcommand*\lineheight[1]{\fontsize{\fsize}{#1\fsize}\selectfont}%
  \ifx\svgwidth\undefined%
    \setlength{\unitlength}{58.3104557bp}%
    \ifx\svgscale\undefined%
      \relax%
    \else%
      \setlength{\unitlength}{\unitlength * \real{\svgscale}}%
    \fi%
  \else%
    \setlength{\unitlength}{\svgwidth}%
  \fi%
  \global\let\svgwidth\undefined%
  \global\let\svgscale\undefined%
  \makeatother%
  \begin{picture}(1,1.73304458)%
    \lineheight{1}%
    \setlength\tabcolsep{0pt}%
    \put(0,0){\includegraphics[width=\unitlength,page=1]{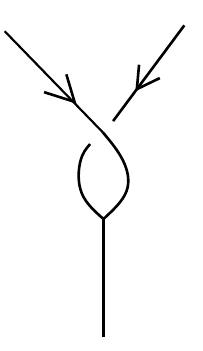}}%
    \put(-0.0056942,1.68092593){\color[rgb]{0,0,0}\makebox(0,0)[lt]{\lineheight{1.25}\smash{\begin{tabular}[t]{l}$j$\end{tabular}}}}%
    \put(0.89934668,1.66626993){\color[rgb]{0,0,0}\makebox(0,0)[lt]{\lineheight{1.25}\smash{\begin{tabular}[t]{l}$i$\end{tabular}}}}%
    \put(0.28010852,0.01008208){\color[rgb]{0,0,0}\makebox(0,0)[lt]{\lineheight{1.25}\smash{\begin{tabular}[t]{l}$k$\end{tabular}}}}%
    \put(0.58056723,0.53405218){\color[rgb]{0,0,0}\makebox(0,0)[lt]{\lineheight{1.25}\smash{\begin{tabular}[t]{l}$\alpha$\end{tabular}}}}%
    \put(0,0){\includegraphics[width=\unitlength,page=2]{braid-relation-2.pdf}}%
  \end{picture}%
\endgroup%
}}},\qquad \vcenter{\hbox{{\def\svgscale{0.6}
%% Accompanies image file 'braid-relation-3.pdf' (pdf, eps, ps)
\begingroup%
  \makeatletter%
  \providecommand\color[2][]{%
    \errmessage{(Inkscape) Color is used for the text in Inkscape, but the package 'color.sty' is not loaded}%
    \renewcommand\color[2][]{}%
  }%
  \providecommand\transparent[1]{%
    \errmessage{(Inkscape) Transparency is used (non-zero) for the text in Inkscape, but the package 'transparent.sty' is not loaded}%
    \renewcommand\transparent[1]{}%
  }%
  \providecommand\rotatebox[2]{#2}%
  \newcommand*\fsize{\dimexpr\f@size pt\relax}%
  \newcommand*\lineheight[1]{\fontsize{\fsize}{#1\fsize}\selectfont}%
  \ifx\svgwidth\undefined%
    \setlength{\unitlength}{79.85872854bp}%
    \ifx\svgscale\undefined%
      \relax%
    \else%
      \setlength{\unitlength}{\unitlength * \real{\svgscale}}%
    \fi%
  \else%
    \setlength{\unitlength}{\svgwidth}%
  \fi%
  \global\let\svgwidth\undefined%
  \global\let\svgscale\undefined%
  \makeatother%
  \begin{picture}(1,1.26541733)%
    \lineheight{1}%
    \setlength\tabcolsep{0pt}%
    \put(0,0){\includegraphics[width=\unitlength,page=1]{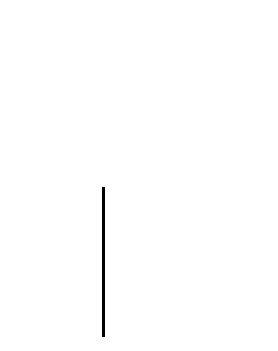}}%
    \put(-0.00415773,1.22736185){\color[rgb]{0,0,0}\makebox(0,0)[lt]{\lineheight{1.25}\smash{\begin{tabular}[t]{l}$j$\end{tabular}}}}%
    \put(0.65667606,1.21666048){\color[rgb]{0,0,0}\makebox(0,0)[lt]{\lineheight{1.25}\smash{\begin{tabular}[t]{l}$i$\end{tabular}}}}%
    \put(0.20452686,0.00736163){\color[rgb]{0,0,0}\makebox(0,0)[lt]{\lineheight{1.25}\smash{\begin{tabular}[t]{l}$k$\end{tabular}}}}%
    \put(0.40626581,0.48464811){\color[rgb]{0,0,0}\makebox(0,0)[lt]{\lineheight{1.25}\smash{\begin{tabular}[t]{l}$B_-\alpha$\end{tabular}}}}%
    \put(0,0){\includegraphics[width=\unitlength,page=2]{braid-relation-3.pdf}}%
  \end{picture}%
\endgroup%
}}}=\vcenter{\hbox{{\def\svgscale{0.6}
%% Accompanies image file 'braid-relation-4.pdf' (pdf, eps, ps)
\begingroup%
  \makeatletter%
  \providecommand\color[2][]{%
    \errmessage{(Inkscape) Color is used for the text in Inkscape, but the package 'color.sty' is not loaded}%
    \renewcommand\color[2][]{}%
  }%
  \providecommand\transparent[1]{%
    \errmessage{(Inkscape) Transparency is used (non-zero) for the text in Inkscape, but the package 'transparent.sty' is not loaded}%
    \renewcommand\transparent[1]{}%
  }%
  \providecommand\rotatebox[2]{#2}%
  \newcommand*\fsize{\dimexpr\f@size pt\relax}%
  \newcommand*\lineheight[1]{\fontsize{\fsize}{#1\fsize}\selectfont}%
  \ifx\svgwidth\undefined%
    \setlength{\unitlength}{60.66068648bp}%
    \ifx\svgscale\undefined%
      \relax%
    \else%
      \setlength{\unitlength}{\unitlength * \real{\svgscale}}%
    \fi%
  \else%
    \setlength{\unitlength}{\svgwidth}%
  \fi%
  \global\let\svgwidth\undefined%
  \global\let\svgscale\undefined%
  \makeatother%
  \begin{picture}(1,1.69055506)%
    \lineheight{1}%
    \setlength\tabcolsep{0pt}%
    \put(0,0){\includegraphics[width=\unitlength,page=1]{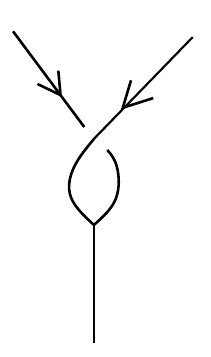}}%
    \put(-0.00547358,1.64045568){\color[rgb]{0,0,0}\makebox(0,0)[lt]{\lineheight{1.25}\smash{\begin{tabular}[t]{l}$j$\end{tabular}}}}%
    \put(0.90324638,1.59818976){\color[rgb]{0,0,0}\makebox(0,0)[lt]{\lineheight{1.25}\smash{\begin{tabular}[t]{l}$i$\end{tabular}}}}%
    \put(0.20585713,0.00969146){\color[rgb]{0,0,0}\makebox(0,0)[lt]{\lineheight{1.25}\smash{\begin{tabular}[t]{l}$k$\end{tabular}}}}%
    \put(0,0){\includegraphics[width=\unitlength,page=2]{braid-relation-4.pdf}}%
    \put(0.5862512,0.57323742){\color[rgb]{0,0,0}\makebox(0,0)[lt]{\lineheight{1.25}\smash{\begin{tabular}[t]{l}$\alpha$\end{tabular}}}}%
  \end{picture}%
\endgroup%
}}}.\label{eq65}
\end{gather}

Let $Y_i=Y_i(v,z)$\index{Yi@$Y_i$} be the vertex operator associated to the module $W_i$. Then $Y_i$ is also a type $i\choose 0~i$ intertwining operator. It's easy to verify that $B_+Y_i=B_-Y_i$ as type $\mc V{i\choose i~0}$ intertwining operators, which we denote by $\mc Y_{\kappa(i)}$\index{ki@$\kappa(i)$, $\mc Y_{\kappa(i)}$} and call the \textbf{creation operator} of $W_i$. Then the canonical isomorphism of the left  multiplication by identity $W_0\boxtimes W_i\overset{\simeq}{\rightarrow} W_i$ is defined to be the one corresponding to $Y_i$. Similarly the right multiplication by identity $W_i\boxtimes W_0\overset{\simeq}{\rightarrow} W_i$ is chosen to be $\kappa(i)$.

One can also construct \textbf{contragredient intertwining operators} $C_+\mc Y_\alpha\equiv\mc Y_{C_+\alpha},C_-\mc Y_\alpha\equiv\mc Y_{C_-\alpha}$ of $\mc Y_\alpha$, \index{Ya@$\mc Y_\alpha$!$C_\pm\mc Y_\alpha=\mc Y_{C_\pm\alpha}$} which are of type ${\ovl j\choose i~\ovl k}$, such that for any $w^{(i)}\in W_i,w^{(j)}\in W_j,w^{(\ovl k)}\in W_{\ovl k}$,
\begin{align}
\bk{\mc Y_{C_\pm\alpha}(w^{(i)},z)w^{(\ovl k)},w^{(j)} }=\bk{w^{(\ovl k)},\mc Y_\alpha(e^{zL_1}(e^{\mp\im\pi}z^{-2})^{L_0}w^{(i)},z^{-1})w^{(j)}}.
\end{align}
Here, and also throughout this paper, we follow the convention $\arg z^{r}=r\arg z$ ($r\in\mathbb R$) unless otherwise stated. To express contragredient intertwining operators graphically, we first introduce, for any $V$-module $W_i$ (together with its contragredient module $W_{\ovl i}$), two important intertwining operators $\mc Y_{\ev_{i,\ovl i}}\in\mc V{0\choose  i~\ovl i}$ and $\mc Y_{\ev_{\ovl i,i}}\in\mc V{0\choose \ovl i~i}$, called the \textbf{annihilation operators} of $W_{\ovl i}$ and $W_i$ respectively. Recall that $V$ is self dual. Fix an isomorphism $V=W_0\simeq W_{\ovl 0}$ and identify $W_0$ and $W_{\ovl 0}$ through this isomorphism.   We now define \index{evii@$\ev_{i,\ovl i}$, $\mc Y_{\ev_{i,\ovl i}}$}
\begin{align}
\mc Y_{\ev_{i,\ovl i}}=C_-\mc Y_{\kappa(i)}=C_-B_\pm Y_{i}.\label{eq73}
\end{align}
The type of $\mc Y_{\ev_{i,\ovl i}}$ shows that $\ev_{i,\ovl i}\in\Hom(W_i\boxtimes W_{\ovl i},V)$, which plays the role of the evaluation map of $W_{\ovl i}$. $\ev_{\ovl i,i}\in\Hom(W_{\ovl i}\boxtimes W_i,V)$ can be defined in a similar way. We write $\ev_{i,\ovl i}=\vcenter{\hbox{{\def\svgscale{0.4}
%% Accompanies image file 'ev.pdf' (pdf, eps, ps)
\begingroup%
  \makeatletter%
  \providecommand\color[2][]{%
    \errmessage{(Inkscape) Color is used for the text in Inkscape, but the package 'color.sty' is not loaded}%
    \renewcommand\color[2][]{}%
  }%
  \providecommand\transparent[1]{%
    \errmessage{(Inkscape) Transparency is used (non-zero) for the text in Inkscape, but the package 'transparent.sty' is not loaded}%
    \renewcommand\transparent[1]{}%
  }%
  \providecommand\rotatebox[2]{#2}%
  \newcommand*\fsize{\dimexpr\f@size pt\relax}%
  \newcommand*\lineheight[1]{\fontsize{\fsize}{#1\fsize}\selectfont}%
  \ifx\svgwidth\undefined%
    \setlength{\unitlength}{62.15604006bp}%
    \ifx\svgscale\undefined%
      \relax%
    \else%
      \setlength{\unitlength}{\unitlength * \real{\svgscale}}%
    \fi%
  \else%
    \setlength{\unitlength}{\svgwidth}%
  \fi%
  \global\let\svgwidth\undefined%
  \global\let\svgscale\undefined%
  \makeatother%
  \begin{picture}(1,0.48348698)%
    \lineheight{1}%
    \setlength\tabcolsep{0pt}%
    \put(0,0){\includegraphics[width=\unitlength,page=1]{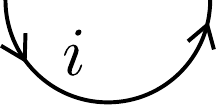}}%
  \end{picture}%
\endgroup%
}}}$ and $\ev_{\ovl i,i}=\vcenter{\hbox{{\def\svgscale{0.4}
%% Accompanies image file 'ev-2.pdf' (pdf, eps, ps)
\begingroup%
  \makeatletter%
  \providecommand\color[2][]{%
    \errmessage{(Inkscape) Color is used for the text in Inkscape, but the package 'color.sty' is not loaded}%
    \renewcommand\color[2][]{}%
  }%
  \providecommand\transparent[1]{%
    \errmessage{(Inkscape) Transparency is used (non-zero) for the text in Inkscape, but the package 'transparent.sty' is not loaded}%
    \renewcommand\transparent[1]{}%
  }%
  \providecommand\rotatebox[2]{#2}%
  \newcommand*\fsize{\dimexpr\f@size pt\relax}%
  \newcommand*\lineheight[1]{\fontsize{\fsize}{#1\fsize}\selectfont}%
  \ifx\svgwidth\undefined%
    \setlength{\unitlength}{62.15604006bp}%
    \ifx\svgscale\undefined%
      \relax%
    \else%
      \setlength{\unitlength}{\unitlength * \real{\svgscale}}%
    \fi%
  \else%
    \setlength{\unitlength}{\svgwidth}%
  \fi%
  \global\let\svgwidth\undefined%
  \global\let\svgscale\undefined%
  \makeatother%
  \begin{picture}(1,0.48348698)%
    \lineheight{1}%
    \setlength\tabcolsep{0pt}%
    \put(0,0){\includegraphics[width=\unitlength,page=1]{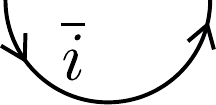}}%
  \end{picture}%
\endgroup%
}}}=\vcenter{\hbox{{\def\svgscale{0.4}
%% Accompanies image file 'ev-3.pdf' (pdf, eps, ps)
\begingroup%
  \makeatletter%
  \providecommand\color[2][]{%
    \errmessage{(Inkscape) Color is used for the text in Inkscape, but the package 'color.sty' is not loaded}%
    \renewcommand\color[2][]{}%
  }%
  \providecommand\transparent[1]{%
    \errmessage{(Inkscape) Transparency is used (non-zero) for the text in Inkscape, but the package 'transparent.sty' is not loaded}%
    \renewcommand\transparent[1]{}%
  }%
  \providecommand\rotatebox[2]{#2}%
  \newcommand*\fsize{\dimexpr\f@size pt\relax}%
  \newcommand*\lineheight[1]{\fontsize{\fsize}{#1\fsize}\selectfont}%
  \ifx\svgwidth\undefined%
    \setlength{\unitlength}{61.91265967bp}%
    \ifx\svgscale\undefined%
      \relax%
    \else%
      \setlength{\unitlength}{\unitlength * \real{\svgscale}}%
    \fi%
  \else%
    \setlength{\unitlength}{\svgwidth}%
  \fi%
  \global\let\svgwidth\undefined%
  \global\let\svgscale\undefined%
  \makeatother%
  \begin{picture}(1,0.48538758)%
    \lineheight{1}%
    \setlength\tabcolsep{0pt}%
    \put(0,0){\includegraphics[width=\unitlength,page=1]{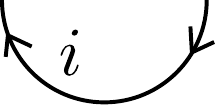}}%
  \end{picture}%
\endgroup%
}}}$,
following the convention that a vertical line with label $i$ but upward-pointing arrow means (the identity morphism of) $W_{\ovl i}$. We can now give a categorical description of $C_-\alpha$  with the help of the following fusion relation (cf. \cite{Gui19b} remark 5.4)
\begin{align}
\mc Y_{\ev_{j,\ovl j}}(w^{(j)},z)\mc Y_{C_-\alpha}(w^{(i)},\zeta)=\mc Y_{\ev_{k,\ovl k}}\big(\mc Y_{B_-\alpha}(w^{(j)},z-\zeta)w^{(i)},\zeta \big),
\end{align}
which can be translated to
\begin{align}\label{eq1}
\vcenter{\hbox{{\def\svgscale{0.6}
%% Accompanies image file 'contragredient.pdf' (pdf, eps, ps)
\begingroup%
  \makeatletter%
  \providecommand\color[2][]{%
    \errmessage{(Inkscape) Color is used for the text in Inkscape, but the package 'color.sty' is not loaded}%
    \renewcommand\color[2][]{}%
  }%
  \providecommand\transparent[1]{%
    \errmessage{(Inkscape) Transparency is used (non-zero) for the text in Inkscape, but the package 'transparent.sty' is not loaded}%
    \renewcommand\transparent[1]{}%
  }%
  \providecommand\rotatebox[2]{#2}%
  \newcommand*\fsize{\dimexpr\f@size pt\relax}%
  \newcommand*\lineheight[1]{\fontsize{\fsize}{#1\fsize}\selectfont}%
  \ifx\svgwidth\undefined%
    \setlength{\unitlength}{116.78392541bp}%
    \ifx\svgscale\undefined%
      \relax%
    \else%
      \setlength{\unitlength}{\unitlength * \real{\svgscale}}%
    \fi%
  \else%
    \setlength{\unitlength}{\svgwidth}%
  \fi%
  \global\let\svgwidth\undefined%
  \global\let\svgscale\undefined%
  \makeatother%
  \begin{picture}(1,0.81940997)%
    \lineheight{1}%
    \setlength\tabcolsep{0pt}%
    \put(0,0){\includegraphics[width=\unitlength,page=1]{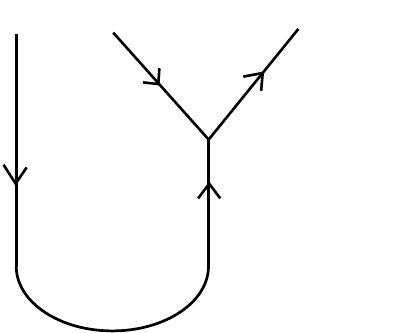}}%
    \put(-0.00284312,0.79338702){\color[rgb]{0,0,0}\makebox(0,0)[lt]{\lineheight{1.25}\smash{\begin{tabular}[t]{l}$j$\end{tabular}}}}%
    \put(0.2481268,0.79338702){\color[rgb]{0,0,0}\makebox(0,0)[lt]{\lineheight{1.25}\smash{\begin{tabular}[t]{l}$i$\end{tabular}}}}%
    \put(0.7151376,0.78821221){\color[rgb]{0,0,0}\makebox(0,0)[lt]{\lineheight{1.25}\smash{\begin{tabular}[t]{l}$k$\end{tabular}}}}%
    \put(0.58139322,0.39102647){\color[rgb]{0,0,0}\makebox(0,0)[lt]{\lineheight{1.25}\smash{\begin{tabular}[t]{l}$C_-\alpha$\end{tabular}}}}%
    \put(0.37810962,0.12158476){\color[rgb]{0,0,0}\makebox(0,0)[lt]{\lineheight{1.25}\smash{\begin{tabular}[t]{l}$j$\end{tabular}}}}%
  \end{picture}%
\endgroup%
}}}~~~=~~~\vcenter{\hbox{{\def\svgscale{0.6}
%% Accompanies image file 'contragredient-2.pdf' (pdf, eps, ps)
\begingroup%
  \makeatletter%
  \providecommand\color[2][]{%
    \errmessage{(Inkscape) Color is used for the text in Inkscape, but the package 'color.sty' is not loaded}%
    \renewcommand\color[2][]{}%
  }%
  \providecommand\transparent[1]{%
    \errmessage{(Inkscape) Transparency is used (non-zero) for the text in Inkscape, but the package 'transparent.sty' is not loaded}%
    \renewcommand\transparent[1]{}%
  }%
  \providecommand\rotatebox[2]{#2}%
  \newcommand*\fsize{\dimexpr\f@size pt\relax}%
  \newcommand*\lineheight[1]{\fontsize{\fsize}{#1\fsize}\selectfont}%
  \ifx\svgwidth\undefined%
    \setlength{\unitlength}{79.88899401bp}%
    \ifx\svgscale\undefined%
      \relax%
    \else%
      \setlength{\unitlength}{\unitlength * \real{\svgscale}}%
    \fi%
  \else%
    \setlength{\unitlength}{\svgwidth}%
  \fi%
  \global\let\svgwidth\undefined%
  \global\let\svgscale\undefined%
  \makeatother%
  \begin{picture}(1,1.25438664)%
    \lineheight{1}%
    \setlength\tabcolsep{0pt}%
    \put(0,0){\includegraphics[width=\unitlength,page=1]{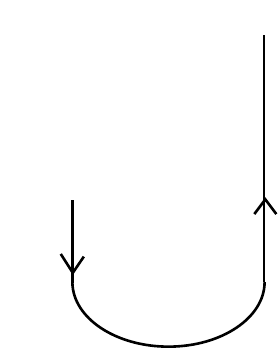}}%
    \put(-0.00415616,1.21634558){\color[rgb]{0,0,0}\makebox(0,0)[lt]{\lineheight{1.25}\smash{\begin{tabular}[t]{l}$j$\end{tabular}}}}%
    \put(0.44970926,1.21067207){\color[rgb]{0,0,0}\makebox(0,0)[lt]{\lineheight{1.25}\smash{\begin{tabular}[t]{l}$i$\end{tabular}}}}%
    \put(0.90546597,1.21634503){\color[rgb]{0,0,0}\makebox(0,0)[lt]{\lineheight{1.25}\smash{\begin{tabular}[t]{l}$k$\end{tabular}}}}%
    \put(0,0){\includegraphics[width=\unitlength,page=2]{contragredient-2.pdf}}%
    \put(0.34712951,0.38968502){\color[rgb]{0,0,0}\makebox(0,0)[lt]{\lineheight{1.25}\smash{\begin{tabular}[t]{l}$\alpha$\end{tabular}}}}%
    \put(0,0){\includegraphics[width=\unitlength,page=3]{contragredient-2.pdf}}%
  \end{picture}%
\endgroup%
}}}\quad.
\end{align}
By \cite{Hua08} section 3, there exist coevaluation maps $\coev_{i,\ovl i}=\vcenter{\hbox{{\def\svgscale{0.4}
%% Accompanies image file 'coev.pdf' (pdf, eps, ps)
\begingroup%
  \makeatletter%
  \providecommand\color[2][]{%
    \errmessage{(Inkscape) Color is used for the text in Inkscape, but the package 'color.sty' is not loaded}%
    \renewcommand\color[2][]{}%
  }%
  \providecommand\transparent[1]{%
    \errmessage{(Inkscape) Transparency is used (non-zero) for the text in Inkscape, but the package 'transparent.sty' is not loaded}%
    \renewcommand\transparent[1]{}%
  }%
  \providecommand\rotatebox[2]{#2}%
  \newcommand*\fsize{\dimexpr\f@size pt\relax}%
  \newcommand*\lineheight[1]{\fontsize{\fsize}{#1\fsize}\selectfont}%
  \ifx\svgwidth\undefined%
    \setlength{\unitlength}{60.84525407bp}%
    \ifx\svgscale\undefined%
      \relax%
    \else%
      \setlength{\unitlength}{\unitlength * \real{\svgscale}}%
    \fi%
  \else%
    \setlength{\unitlength}{\svgwidth}%
  \fi%
  \global\let\svgwidth\undefined%
  \global\let\svgscale\undefined%
  \makeatother%
  \begin{picture}(1,0.49390271)%
    \lineheight{1}%
    \setlength\tabcolsep{0pt}%
    \put(0,0){\includegraphics[width=\unitlength,page=1]{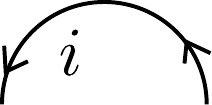}}%
  \end{picture}%
\endgroup%
}}}\in\Hom(V,W_i\boxtimes W_{\ovl i})$ \index{coev@$\coev_{i,\ovl i}$} and $\coev_{\ovl i,i}=\vcenter{\hbox{{\def\svgscale{0.4}
%% Accompanies image file 'coev-2.pdf' (pdf, eps, ps)
\begingroup%
  \makeatletter%
  \providecommand\color[2][]{%
    \errmessage{(Inkscape) Color is used for the text in Inkscape, but the package 'color.sty' is not loaded}%
    \renewcommand\color[2][]{}%
  }%
  \providecommand\transparent[1]{%
    \errmessage{(Inkscape) Transparency is used (non-zero) for the text in Inkscape, but the package 'transparent.sty' is not loaded}%
    \renewcommand\transparent[1]{}%
  }%
  \providecommand\rotatebox[2]{#2}%
  \newcommand*\fsize{\dimexpr\f@size pt\relax}%
  \newcommand*\lineheight[1]{\fontsize{\fsize}{#1\fsize}\selectfont}%
  \ifx\svgwidth\undefined%
    \setlength{\unitlength}{60.84525407bp}%
    \ifx\svgscale\undefined%
      \relax%
    \else%
      \setlength{\unitlength}{\unitlength * \real{\svgscale}}%
    \fi%
  \else%
    \setlength{\unitlength}{\svgwidth}%
  \fi%
  \global\let\svgwidth\undefined%
  \global\let\svgscale\undefined%
  \makeatother%
  \begin{picture}(1,0.49390271)%
    \lineheight{1}%
    \setlength\tabcolsep{0pt}%
    \put(0,0){\includegraphics[width=\unitlength,page=1]{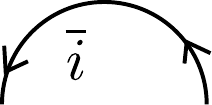}}%
  \end{picture}%
\endgroup%
}}}=\vcenter{\hbox{{\def\svgscale{0.4}
%% Accompanies image file 'coev-3.pdf' (pdf, eps, ps)
\begingroup%
  \makeatletter%
  \providecommand\color[2][]{%
    \errmessage{(Inkscape) Color is used for the text in Inkscape, but the package 'color.sty' is not loaded}%
    \renewcommand\color[2][]{}%
  }%
  \providecommand\transparent[1]{%
    \errmessage{(Inkscape) Transparency is used (non-zero) for the text in Inkscape, but the package 'transparent.sty' is not loaded}%
    \renewcommand\transparent[1]{}%
  }%
  \providecommand\rotatebox[2]{#2}%
  \newcommand*\fsize{\dimexpr\f@size pt\relax}%
  \newcommand*\lineheight[1]{\fontsize{\fsize}{#1\fsize}\selectfont}%
  \ifx\svgwidth\undefined%
    \setlength{\unitlength}{61.72387992bp}%
    \ifx\svgscale\undefined%
      \relax%
    \else%
      \setlength{\unitlength}{\unitlength * \real{\svgscale}}%
    \fi%
  \else%
    \setlength{\unitlength}{\svgwidth}%
  \fi%
  \global\let\svgwidth\undefined%
  \global\let\svgscale\undefined%
  \makeatother%
  \begin{picture}(1,0.48687211)%
    \lineheight{1}%
    \setlength\tabcolsep{0pt}%
    \put(0,0){\includegraphics[width=\unitlength,page=1]{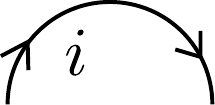}}%
  \end{picture}%
\endgroup%
}}}\in\Hom(V,W_{\ovl i}\boxtimes W_i)$ satisfying
\begin{gather}
(\ev_{i,\ovl i}\otimes\id_i)(\id_i\otimes\coev_{\ovl i,i})=\id_i=(\id_i\otimes\ev_{\ovl i,i})(\coev_{i,\ovl i}\otimes\id_i),\label{eq17}\\
(\ev_{\ovl i,i}\otimes\id_{\ovl i})(\id_{\ovl i}\otimes\coev_{i,\ovl i})=\id_{\ovl i}=(\id_{\ovl i}\otimes\ev_{i,\ovl i})(\coev_{\ovl i,i}\otimes\id_{\ovl i}).\label{eq18}
\end{gather}
Thus, let $W_{\ovl j}$ tensor both sides of  equation \eqref{eq1} from the left,  and then apply $\coev_{\ovl j,j}\otimes\id_i\otimes\id_k$ to the tops, we obtain

\begin{pp}\label{lb4}
For any $V$-modules $W_i,W_j,W_k$, and any $\mc Y_\alpha\in\mc V{k\choose i~j}$,
\end{pp}
\begin{align}\label{eq2}
\vcenter{\hbox{{\def\svgscale{0.6}
%% Accompanies image file 'contragredient-3.pdf' (pdf, eps, ps)
\begingroup%
  \makeatletter%
  \providecommand\color[2][]{%
    \errmessage{(Inkscape) Color is used for the text in Inkscape, but the package 'color.sty' is not loaded}%
    \renewcommand\color[2][]{}%
  }%
  \providecommand\transparent[1]{%
    \errmessage{(Inkscape) Transparency is used (non-zero) for the text in Inkscape, but the package 'transparent.sty' is not loaded}%
    \renewcommand\transparent[1]{}%
  }%
  \providecommand\rotatebox[2]{#2}%
  \newcommand*\fsize{\dimexpr\f@size pt\relax}%
  \newcommand*\lineheight[1]{\fontsize{\fsize}{#1\fsize}\selectfont}%
  \ifx\svgwidth\undefined%
    \setlength{\unitlength}{87.47467277bp}%
    \ifx\svgscale\undefined%
      \relax%
    \else%
      \setlength{\unitlength}{\unitlength * \real{\svgscale}}%
    \fi%
  \else%
    \setlength{\unitlength}{\svgwidth}%
  \fi%
  \global\let\svgwidth\undefined%
  \global\let\svgscale\undefined%
  \makeatother%
  \begin{picture}(1,1.08655325)%
    \lineheight{1}%
    \setlength\tabcolsep{0pt}%
    \put(0,0){\includegraphics[width=\unitlength,page=1]{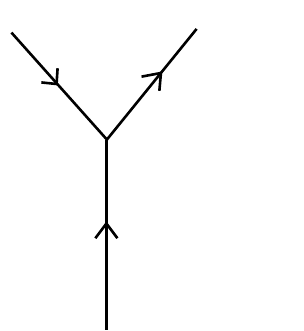}}%
    \put(-0.00379574,1.05181105){\color[rgb]{0,0,0}\makebox(0,0)[lt]{\lineheight{1.25}\smash{\begin{tabular}[t]{l}$i$\end{tabular}}}}%
    \put(0.61969165,1.04490237){\color[rgb]{0,0,0}\makebox(0,0)[lt]{\lineheight{1.25}\smash{\begin{tabular}[t]{l}$k$\end{tabular}}}}%
    \put(0.44113489,0.51463563){\color[rgb]{0,0,0}\makebox(0,0)[lt]{\lineheight{1.25}\smash{\begin{tabular}[t]{l}$C_-\alpha$\end{tabular}}}}%
    \put(0.16973911,0.15491482){\color[rgb]{0,0,0}\makebox(0,0)[lt]{\lineheight{1.25}\smash{\begin{tabular}[t]{l}$j$\end{tabular}}}}%
  \end{picture}%
\endgroup%
}}}~~~=~~~\vcenter{\hbox{{\def\svgscale{0.6}
%% Accompanies image file 'contragredient-4.pdf' (pdf, eps, ps)
\begingroup%
  \makeatletter%
  \providecommand\color[2][]{%
    \errmessage{(Inkscape) Color is used for the text in Inkscape, but the package 'color.sty' is not loaded}%
    \renewcommand\color[2][]{}%
  }%
  \providecommand\transparent[1]{%
    \errmessage{(Inkscape) Transparency is used (non-zero) for the text in Inkscape, but the package 'transparent.sty' is not loaded}%
    \renewcommand\transparent[1]{}%
  }%
  \providecommand\rotatebox[2]{#2}%
  \newcommand*\fsize{\dimexpr\f@size pt\relax}%
  \newcommand*\lineheight[1]{\fontsize{\fsize}{#1\fsize}\selectfont}%
  \ifx\svgwidth\undefined%
    \setlength{\unitlength}{99.77423893bp}%
    \ifx\svgscale\undefined%
      \relax%
    \else%
      \setlength{\unitlength}{\unitlength * \real{\svgscale}}%
    \fi%
  \else%
    \setlength{\unitlength}{\svgwidth}%
  \fi%
  \global\let\svgwidth\undefined%
  \global\let\svgscale\undefined%
  \makeatother%
  \begin{picture}(1,1.74437684)%
    \lineheight{1}%
    \setlength\tabcolsep{0pt}%
    \put(0,0){\includegraphics[width=\unitlength,page=1]{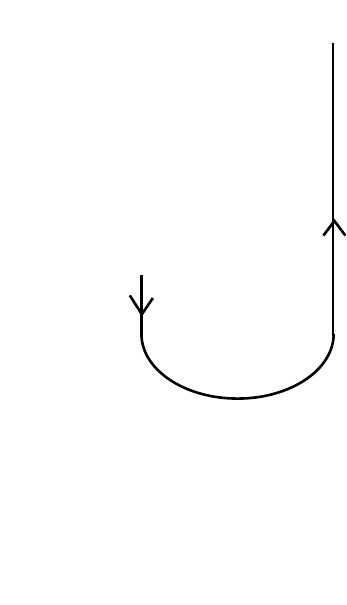}}%
    \put(0.53667056,1.71088937){\color[rgb]{0,0,0}\makebox(0,0)[lt]{\lineheight{1.25}\smash{\begin{tabular}[t]{l}$i$\end{tabular}}}}%
    \put(0.92279258,1.71391745){\color[rgb]{0,0,0}\makebox(0,0)[lt]{\lineheight{1.25}\smash{\begin{tabular}[t]{l}$k$\end{tabular}}}}%
    \put(0.48481919,0.87333682){\color[rgb]{0,0,0}\makebox(0,0)[lt]{\lineheight{1.25}\smash{\begin{tabular}[t]{l}$\alpha$\end{tabular}}}}%
    \put(0,0){\includegraphics[width=\unitlength,page=2]{contragredient-4.pdf}}%
    \put(0.58049964,0.12172797){\color[rgb]{0,0,0}\makebox(0,0)[lt]{\lineheight{1.25}\smash{\begin{tabular}[t]{l}$j$\end{tabular}}}}%
    \put(0,0){\includegraphics[width=\unitlength,page=3]{contragredient-4.pdf}}%
  \end{picture}%
\endgroup%
}}}\quad.
\end{align}

Finally, we remark that the ribbon structure on $\Rep(V)$ is defined by the twist $\vartheta=\vartheta_i:=e^{2\im\pi L_0}\in\End(W_i)$ for any $V$-module $W_i$. \index{zz@$\vartheta=\vartheta_i$}

\subsection{Unitary VOAs and unitary representations}

In this section, we only assume that $V$ is of CFT-type, and  discuss the unitary conditions on $V$ and its tensor category. We do not assume, at the beginning, that $V$ is self-dual. In particular, we do not identify $W_0$ with $W_{\ovl 0}$. As we shall see, the unitary structure on $V$ is closely related to certain $V$-module isomorphism $\epsilon:W_0\rightarrow W_{\ovl 0}$.

Suppose  $V$ is equipped with a normalized inner product $\bk{\cdot|\cdot}$ (``normalized" means $\bk{\Omega|\Omega}=1$). We say that $(V,\bk{\cdot|\cdot})$ (simply $V$ in the future) is \textbf{unitary} \cite{DL14,CKLW18}, if there exists an anti-unitary anti-automorphism $\Theta$\index{zz@$\Theta$} of $V$ (called \textbf{PCT operator}) satisfying that
\begin{align}
\bk{Y(v,z)v_1|v_2 }=\bk{v_1|Y(e^{\ovl zL_1}(-\ovl{z^{-2}} )^{L_0}\Theta v,\ovl{z^{-1}})v_2}.
\end{align}
We abbreviate the above relation to
\begin{align}
Y(v,z)=Y(e^{\ovl zL_1}(-\ovl{z^{-2}} )^{L_0}\Theta v,\ovl{z^{-1}})^\dagger,\label{eq5}
\end{align}
with $\dagger$ understood as formal adjoint\index{formal adjoint@Formal adjoint $\dagger$}.

By our definition, $\Theta:V\rightarrow V$ is an antilinear (i.e., conjugate linear) bijective  map satisfying
\begin{gather}
\bk{u|v}=\bk{\Theta v|\Theta u},\\
 \Theta Y(u,z)v=Y(\Theta u,\ovl z)\Theta v\label{eq6}
\end{gather}
for all $u,v\in V$. Such $\Theta$ is unique by \cite{CKLW18} proposition 5.1. Note that $\Theta v=\Theta Y(\Omega,z)v=Y(\Theta\Omega,\ovl z)\Theta v$ which implies $\Theta\Omega=\Omega$. Also $\Theta\nu=\nu$ by \cite{CKLW18} corollary 4.11. Indeed, by that corollary, under the assumption of anti-automorphism, anti-unitarity is equivalent to $\Theta\nu=\nu$, and also equivalent to that $\Theta$ preserves the grading of $V$. Note also that $\Theta$ is uniquely determined by $V$ and its inner product $\bk{\cdot|\cdot}$, and that $\Theta^2=\id_V$ (\cite{CKLW18} proposition 5.1).

Later we will discuss the unitarity of VOA extensions using tensor-categorical methods. For that purpose the map $\Theta$ is difficult to deal with due to its anti-linearity. So let us give an equivalent description of unitarity using linear maps. Let $V^*$ be the dual vector space of $V$.   The inner product $\bk{\cdot|\cdot}$ on $V$ induces naturally an antilinear injective map $\Co:V\rightarrow  V^*,v\mapsto\bk{\cdot|v}$. We set $\ovl V=\Co V$,\index{Conjugate@$\Co$} and define an inner product  on $\ovl V$, also denoted by $\bk{\cdot|\cdot}$, under which $\Co$ becomes anti-unitary. Equivalently, we require $\bk{u|v}=\bk{\Co v|\Co u}$ for all $u,v\in V$. The adjoint (which is also the inverse) of $\Co:V\rightarrow \ovl V$ is denoted by $\Co:\ovl V\rightarrow V$. The reason we use the same symbol for the two conjugation maps is to regard $\Co$ as an involution. We also adopt the notation $\ovl{v}=\Co v$ and $\ovl{\ovl v}=v$\index{Conjugate@$\Co$!vbar@$\ovl v=\Co v$}. We will use both $\ovl v$ and $\Co$ for conjugation very frequently, but the later is used more often when  no specific vectors are mentioned.

%Let $\ovl V$ be the vector space of the contragredient module of the vacuum module $V$. Thus $\ovl V$\index{Vbar@$\ovl V$} is the subspace  of $V^*$ (the dual vector space of $V$) spanned by the eigenspaces of $L_0^\tr$ (the transpose of $L_0$, acting on $V^*$). In other words $\ovl V$ is the finite-weight subspace of $V^*$.

\begin{pp}\label{lb9}
$(V,\bk{\cdot|\cdot})$ is unitary if and only if there exists a (unique) unitary map $\epsilon:V\rightarrow \ovl V$\index{zz@$\epsilon$}, called the \textbf{reflection operator} of $V$, such that the following two relations hold for all $v\in V$:
\begin{gather}
Y(v,z)=Y(e^{\ovl zL_1}(-\ovl{z^{-2}} )^{L_0}\ovl{\epsilon v},\ovl{z^{-1}})^\dagger,\label{eq3}\\
\epsilon Y(\epsilon^*\ovl v,z )\epsilon^*=\Co Y(v,\ovl z)\Co.\label{eq4}
\end{gather}
When $V$ is unitary, we have
\begin{align}
\Co\epsilon=\epsilon^*\Co.\label{eq7}
\end{align}
\end{pp}
Note that the first and the third equations are acting on $V$, while the second one on $\ovl V$.
\begin{proof}
$\epsilon$ and $\Theta$ are related by $\epsilon=\Co\Theta$. Then $\ovl{\epsilon v}=\Theta v,\epsilon^*\ovl v=\Theta^{-1}v$. It is easy to see that \eqref{eq3} and \eqref{eq4} are equivalent to \eqref{eq5} and \eqref{eq6} respectively. The uniqueness of $\epsilon$ follows from that of $\Theta$. When $V$ is unitary, $\Co\epsilon=\Theta=\Theta^{-1}=\epsilon^*\Co$.
\end{proof}

In the following we always assume $V$ to be unitary. Let $W_i$ be a $V$-module, and assume that the vector space $W_i$ is equipped with an inner product $\bk{\cdot|\cdot}$. We say that $(W_i,\bk{\cdot|\cdot})$ (or just $W_i$) is a \textbf{unitary} $V$-module, if for any $v\in V$ we have $Y_i(v,z)=Y_i(e^{\ovl zL_1}(-\ovl{z^{-2}} )^{L_0}\Theta v,\ovl{z^{-1}})^\dagger$ when acting on $W_i$. Equivalently,
\begin{align}
Y_i(v,z)=Y_i(e^{\ovl zL_1}(-\ovl{z^{-2}} )^{L_0}\ovl{\epsilon v},\ovl{z^{-1}})^\dagger.\label{eq8}
\end{align}
A $V$-module is called \textbf{unitarizable} if it can be equipped with an inner product under which it becomes a unitary $V$-module. An intertwining operator $\mc Y_\alpha$ of $V$ is called \textbf{unitary} if it is among unitary $V$-modules.

Note that just as the conjugations between $V$ and $\ovl V$, given a unitary $V$-module $W_i$, we have a natural conjugation $\Co:W_i\rightarrow W_{\ovl i}$\index{Conjugate@$\Co$} whose inverse is also written as $\Co:W_{\ovl i}\rightarrow W_i$. (Indeed, we first have an injective antilinear map $\Co:W_i\rightarrow W_i^*$. That $\Co W_i=W_{\ovl i}$, i.e., that vectors in $\Co W_i$ are precisely those with finite conformal weights, follows from the fact that $\bk{L_0\cdot|\cdot }=\bk{\cdot|L_0\cdot}$ which is a consequence of \eqref{eq8} and the fact that $\ovl{\epsilon\nu}=\Theta\nu=\nu$.) Fix an inner product on $W_{\ovl i}$ under which $\Co$ becomes anti-unitary. Then $W_{\ovl i}$ is also a unitary $V$-module. For any $w^{(i)}\in W_i$, write $\ovl{w^{(i)}}=\Co w^{(i)},\ovl{\ovl {w^{(i)}}}=\Co\ovl{ w^{(i)}}=w^{(i)}$.\index{Conjugate@$\Co$!wbar@$\ovl{w^{(i)}}=\Co w^{(i)}$}  Recall by our notation that $Y_i$ and $Y_{\ovl i}$ are the vertex operators of $W_i$ and $W_{\ovl i}$ respectively. The relation between these two operators is pretty simple: $Y_{\ovl i}(\Theta v,z)=\Co Y_i(v,\ovl z)\Co$, acting on $W_{\ovl i}$. (See \cite{Gui19a} formula (1.19).) Equivalently,
\begin{align}
Y_{\ovl i}(\ovl{\epsilon v},z)=\Co Y_i(v,\ovl z)\Co.
\end{align}
This equation (applied to $i=0$), together with \eqref{eq4} and \eqref{eq7}, shows for any $v\in V$ that $\epsilon Y_0(v,z)\epsilon^*=Y_{\ovl 0}(v,z)$, where $Y_0=Y$ is the vertex operator of the vacuum module $V$, and $Y_{\ovl 0}$ is that of its contragredient module. We conclude:
\begin{pp}\label{lb10}
	If $V$ is a unitary CFT-type VOA, then $V$ is self-dual. The vacuum module $V=W_0$ is unitarily equivalent to its contragredient module $\ovl V=W_{\ovl 0}$ via the reflection operator $\epsilon$.
\end{pp}

\begin{cv}\label{lb1}
	Unless otherwise stated, if $V$ is unitary and CFT-type, the isomorphism $W_0\overset{\simeq}{\rightarrow}W_{\ovl 0}$ is always chosen to be the reflection operator $\epsilon$, and the identification of $V=W_0$ and $\ovl V=W_{\ovl 0}$ is always assumed to be through $\epsilon$.
\end{cv}

\subsection{Adjoint and conjugate intertwining operators}\label{lb48}

We have seen in section \ref{lb2} two ways of producing new intertwining operators from old ones: the braided and the contragredient intertwining operators. In the unitary case there are two extra methods: the conjugate and the adjoint intertwining operators. In this section, our main goal is to derive tensor-categorical descriptions of these two constructions of intertwining operators.

First we review the definition of these two constructions; see \cite{Gui19a} section 1.3 for more details and basic properties. Let $V$ be unitary and CFT-type. Choose unitary $V$-modules $W_i,W_j,W_k$. For any $\mc Y_\alpha\in\mc V{k\choose i~j}$, its \textbf{conjugate intertwining operator} $\Co\mc Y_\alpha\equiv\mc Y_{\Co\alpha}\equiv\mc Y_{\ovl\alpha}$ \index{Ya@$\mc Y_\alpha$!$\Co\mc Y_\alpha=\mc Y_{\Co\alpha}=\mc Y_{\ovl\alpha}$} is of type $\mc V{\ovl k\choose \ovl i~\ovl j}$ defined by
\begin{align}
\mc Y_{\ovl\alpha}(\ovl{w^{(i)}},z)=\Co\mc Y_\alpha(w^{(i)},\ovl z)\Co
\end{align}
for any $w^{(i)}\in W_i$. Despite its simple form, one cannot directly translate this definition into  tensor categorical language, again due to the anti-linearity of $\Co$. Therefore we need to first consider the  \textbf{adjoint intertwining operator} $\mc Y^\dagger_\alpha\equiv\mc Y_{\alpha^\dagger}\in\mc V{j\choose\ovl i~k}$, \index{Ya@$\mc Y_\alpha$!$\mc Y^\dagger_\alpha=\mc Y_{\alpha^\dagger}$} defined by
\begin{align}
\alpha^\dagger=\ovl{C_+\alpha}=C_-\ovl\alpha,
\end{align}
which will be closely related to the $*$-structures of the $C^*$-tensor categories. Then for any $w^{(i)}\in W_i$, 
\begin{align}
\mc Y_{\alpha^\dagger}(\ovl{w^{(i)}},z)=\mc Y_\alpha(e^{\ovl zL_1}(e^{-\im\pi}\ovl{z^{-2}} )^{L_0}\ovl{w^{(i)}},\ovl{z^{-1}})^\dagger.
\end{align}
Note that $\dagger$ is an involution: $\alpha^{\dagger\dagger}=\alpha$. Moreover,  by unitarity, up to equivalence of the charge spaces $\epsilon:V\overset{\simeq}{\rightarrow}\ovl V$, $Y_i$ is equal to its adjoint intertwining operator, and obviously also equal to the conjugate intertwining operator of $Y_{\ovl i}$. Another important relation is
\begin{align}
\ev_{\ovl i,i}=\kappa(i)^\dagger\label{eq9}
\end{align}
\index{evii@$\ev_{i,\ovl i}$, $\mc Y_{\ev_{i,\ovl i}}$} by \cite{Gui19a} formula (1.44). Recall from  section \ref{lb2} that, up to the isomorphism $W_0\simeq W_{\ovl 0}$, $\ev_{\ovl i,i}$ is defined to be $C_-\kappa(\ovl i)$. Due to convention \ref{lb1}, the more precise definition is $\ev_{\ovl i,i}=\epsilon^{-1}\circ C_-\kappa(\ovl i)$. We will use \eqref{eq9} more often than this definition in our paper.

We shall now relate $\alpha^\dagger$ with the unitarity structure of the tensor category of unitary $V$-modules. Let $V$ be a CFT-type VOA. Assume that $V$ is \textbf{strongly unitary} \cite{Ten18a}, which means that $V$ is unitary, and any $V$-module is unitarizable. Then $\Rep^\uni(V)$,\index{RepV@$\Rep^\uni(V)$} the category of unitary $V$-modules, is a $C^*$-category, whose $*$ structure is defined as follows: if $W_i,W_j$ are unitary and $T\in\Hom(W_i,W_j)$, then $T^*\in\Hom(W_j,W_i)$ is simply the adjoint of $T$, defined with respect to the inner products of $W_i$ and $W_j$.

Assume also that $V$ is regular. To make $\Rep^\uni(V)$ a $C^*$-tensor category, we have to choose, for any unitary $V$-modules $W_i,W_j$, a suitable unitary structure on $W_i\boxtimes W_j$. Note that it is already  known that $W_i\boxtimes W_j$ is unitarizable by the strong unitary of $V$. But here the unitary structure has to be chosen such that the structural isomorphisms  become unitary. So, for instance, if $W_k$ is also unitary, the associativity isomorphism $(W_i\boxtimes W_j)\boxtimes W_k\overset\simeq\rightarrow W_i\boxtimes(W_j\boxtimes W_k)$ and the braid isomorphism $\ss:W_i\boxtimes W_j\overset\simeq\rightarrow W_j\boxtimes W_i$ have to be unitary. Then we can identify $(W_i\boxtimes W_j)\boxtimes W_k$ and $W_i\boxtimes(W_j\boxtimes W_k)$ as the same unitary $V$-module called $W_i\boxtimes W_j\boxtimes W_k$.

To fulfill these purposes, recall that $W_i\boxtimes W_j$ is defined to be  $\bigoplus_{t\in\mc E}\mc V{t\choose i~j}^*\otimes W_t$. Since every $W_t$ already has a unitary structure, it suffices to assume that the direct sum is orthogonal, and define a suitable inner product $\Lambda$ (called invariant inner product) on each $\mc V{t\choose i~j}^*$. 

Let us first assume that we can always find $\Lambda$ which make $\Rep^\uni(V)$ a braided $C^*$-tensor category. Then we can define an inner product on $\mc V{t\choose i~j}$, also denoted by $\Lambda$, under which the anti-linear map $\mc V{t\choose i~j}\rightarrow \mc V{t\choose i~j}^*$, $\mc Y_\alpha\mapsto \Lambda(\cdot|\mc Y_\alpha)$ becomes anti-unitary. Here we assume $\Lambda(\cdot|\cdot)$ to be linear on the first variable and anti-linear on the second one.  Let $\bk{\mc Y_\alpha:\alpha\in\Xi^t_{i,j}}$ \index{zz@$\Xi^t_{i,j}$} be a basis of $\mc V{t\choose i~j}$, and let $\bk{\widecheck{\mc Y}^\alpha:\alpha\in\Xi^t_{i,j} }$ be its dual basis in $\mc V{t\choose i~j}^*$. In other words we assume for any $\alpha,\beta\in\Xi^t_{i,j}$ that $\bk{\mc Y_\alpha,\widecheck{\mc Y}^\beta}=\delta_{\alpha,\beta}$. The following proposition relates $\Lambda$ with the categorical inner product.

\begin{pp}\label{lb3}
For any $\mc Y_\alpha,\mc Y_\beta\in\mc V{t\choose i~j}$ we have
\begin{align}
\alpha\beta^*=\Lambda(\mc Y_\alpha|\mc Y_\beta)\id_t.\label{eq10}
\end{align}
\end{pp}
\begin{proof}
Assume that $\Xi^t_{i,j}$ is orthonormal under $\Lambda$. Then so is $\bk{\widecheck{\mc Y}^\alpha:\alpha\in\Xi^t_{i,j} }$. For any $\alpha\in\Xi^t_{i,j}$, let $P_\alpha$ be the projection mapping $W_i\boxtimes W_j=\bigoplus^\perp_{t\in\mc E}\mc V{t\choose i~j}^*\otimes W_t$ onto $\widecheck{\mc Y}^\alpha\otimes W_t$, and let $U_\alpha$ be the  unitary map $\widecheck{\mc Y}^\alpha\otimes W_t\rightarrow W_t$, $\widecheck{\mc Y}^\alpha\otimes w^{(t)}\mapsto w^{(t)}$. Then we clearly have $\alpha=U_\alpha P_\alpha$. So $\alpha\alpha^*=\id_t$. If $\beta\in\Xi^t_{i,j}$ and $\beta\neq\alpha$, then clearly $P_\beta P_\alpha=0$, and hence $\alpha\beta^*=0$. Thus \eqref{eq9} holds for basis vectors, and hence holds in general.
\end{proof}

We warn the reader the difference of the two notations $\alpha^\dagger$ and $\alpha^*$. If $\mc Y_\alpha\in\mc V{t\choose i~j}$ then $\alpha^\dagger\in\Hom(W_{\ovl i}\boxtimes W_t,W_j)$, while $\alpha^*$, defined using the $*$-structure of $\Rep^\uni(V)$, is in $\Hom(W_t,W_i\boxtimes W_j)$. Thus, $\mc Y_{\alpha^\dagger}$ is the adjoint intertwining operator while $\mc Y_{\alpha^*}$ makes no sense. This is different from the notations used in \cite{Gui19a,Gui19b}, where $\alpha^\dagger$ is not defined but $\mc Y_{\alpha^*}$ (also written as $\mc Y_\alpha^\dagger$) denotes the adjoint intertwining operator. 
Despite such difference, $\alpha^\dagger$ and $\alpha^*$, the VOA adjoint and the categorical adjoint, should be related in a natural way. And it is time to review the construction of invariant $\Lambda$ introduced in \cite{Gui19b}.

\begin{df}\label{lb49}
Let $W_i,W_j$ be unitary $V$-modules. Then the \textbf{invariant sesquilinear form} $\Lambda$  on $\mc V{t\choose i~j}^*$ for any $t\in\mc E$ is defined such that the fusion relation holds for any $w^{(i)}_1,w^{(i)}_2\in W_i$:
\begin{align}
Y_j\big(\mc Y_{\ev_{\ovl i,i}}(\ovl{w^{(i)}_2},z-\zeta)w^{(i)}_1,\zeta \big)=\sum_{t\in\mc E}\sum_{\alpha,\beta\in\Xi^t_{i,j}}\Lambda(\widecheck{\mc Y}^\alpha|\widecheck{\mc Y}^\beta)\cdot\mc Y_{\beta^\dagger}(\ovl{w^{(i)}_2},z)\mc Y_\alpha(w^{(i)}_1,\zeta).\label{eq11}
\end{align}
\end{df}
Equation \eqref{eq11} can equivalently be presented as
\begin{align}
\vcenter{\hbox{{\def\svgscale{0.75}
%% Accompanies image file 'transport.pdf' (pdf, eps, ps)
\begingroup%
  \makeatletter%
  \providecommand\color[2][]{%
    \errmessage{(Inkscape) Color is used for the text in Inkscape, but the package 'color.sty' is not loaded}%
    \renewcommand\color[2][]{}%
  }%
  \providecommand\transparent[1]{%
    \errmessage{(Inkscape) Transparency is used (non-zero) for the text in Inkscape, but the package 'transparent.sty' is not loaded}%
    \renewcommand\transparent[1]{}%
  }%
  \providecommand\rotatebox[2]{#2}%
  \newcommand*\fsize{\dimexpr\f@size pt\relax}%
  \newcommand*\lineheight[1]{\fontsize{\fsize}{#1\fsize}\selectfont}%
  \ifx\svgwidth\undefined%
    \setlength{\unitlength}{79.6445857bp}%
    \ifx\svgscale\undefined%
      \relax%
    \else%
      \setlength{\unitlength}{\unitlength * \real{\svgscale}}%
    \fi%
  \else%
    \setlength{\unitlength}{\svgwidth}%
  \fi%
  \global\let\svgwidth\undefined%
  \global\let\svgscale\undefined%
  \makeatother%
  \begin{picture}(1,1.04620956)%
    \lineheight{1}%
    \setlength\tabcolsep{0pt}%
    \put(0.45966871,0.99667046){\color[rgb]{0,0,0}\makebox(0,0)[lt]{\lineheight{1.25}\smash{\begin{tabular}[t]{l}$i$\end{tabular}}}}%
    \put(0.92630835,1.00805175){\color[rgb]{0,0,0}\makebox(0,0)[lt]{\lineheight{1.25}\smash{\begin{tabular}[t]{l}$j$\end{tabular}}}}%
    \put(-0.00416891,0.98507628){\color[rgb]{0,0,0}\makebox(0,0)[lt]{\lineheight{1.25}\smash{\begin{tabular}[t]{l}$i$\end{tabular}}}}%
    \put(0,0){\includegraphics[width=\unitlength,page=1]{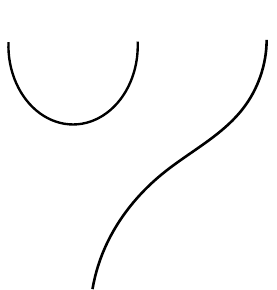}}%
    \put(0.19500648,0.055226){\color[rgb]{0,0,0}\makebox(0,0)[lt]{\lineheight{1.25}\smash{\begin{tabular}[t]{l}$j$\end{tabular}}}}%
    \put(0,0){\includegraphics[width=\unitlength,page=2]{transport.pdf}}%
  \end{picture}%
\endgroup%
}}}~~=~~\sum_{t\in\mc E}\sum_{\alpha,\beta\in\Xi^t_{i,j}}\Lambda(\widecheck{\mc Y}^\alpha|\widecheck{\mc Y}^\beta)\vcenter{\hbox{{\def\svgscale{0.5}
%% Accompanies image file 'transport-2.pdf' (pdf, eps, ps)
\begingroup%
  \makeatletter%
  \providecommand\color[2][]{%
    \errmessage{(Inkscape) Color is used for the text in Inkscape, but the package 'color.sty' is not loaded}%
    \renewcommand\color[2][]{}%
  }%
  \providecommand\transparent[1]{%
    \errmessage{(Inkscape) Transparency is used (non-zero) for the text in Inkscape, but the package 'transparent.sty' is not loaded}%
    \renewcommand\transparent[1]{}%
  }%
  \providecommand\rotatebox[2]{#2}%
  \newcommand*\fsize{\dimexpr\f@size pt\relax}%
  \newcommand*\lineheight[1]{\fontsize{\fsize}{#1\fsize}\selectfont}%
  \ifx\svgwidth\undefined%
    \setlength{\unitlength}{144.14952464bp}%
    \ifx\svgscale\undefined%
      \relax%
    \else%
      \setlength{\unitlength}{\unitlength * \real{\svgscale}}%
    \fi%
  \else%
    \setlength{\unitlength}{\svgwidth}%
  \fi%
  \global\let\svgwidth\undefined%
  \global\let\svgscale\undefined%
  \makeatother%
  \begin{picture}(1,0.99882065)%
    \lineheight{1}%
    \setlength\tabcolsep{0pt}%
    \put(0,0){\includegraphics[width=\unitlength,page=1]{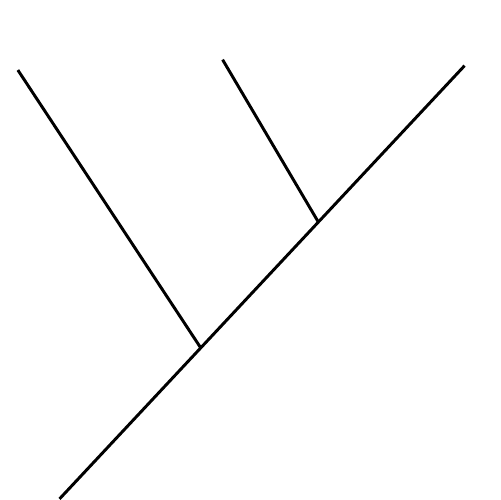}}%
    \put(-0.00460676,0.94577301){\color[rgb]{0,0,0}\makebox(0,0)[lt]{\lineheight{1.25}\smash{\begin{tabular}[t]{l}$i$\end{tabular}}}}%
    \put(0.91856871,0.95038671){\color[rgb]{0,0,0}\makebox(0,0)[lt]{\lineheight{1.25}\smash{\begin{tabular}[t]{l}$j$\end{tabular}}}}%
    \put(0.00792645,0.04950549){\color[rgb]{0,0,0}\makebox(0,0)[lt]{\lineheight{1.25}\smash{\begin{tabular}[t]{l}$j$\end{tabular}}}}%
    \put(0.40861184,0.48235278){\color[rgb]{0,0,0}\makebox(0,0)[lt]{\lineheight{1.25}\smash{\begin{tabular}[t]{l}$t$\end{tabular}}}}%
    \put(0.67897336,0.50343482){\color[rgb]{0,0,0}\makebox(0,0)[lt]{\lineheight{1.25}\smash{\begin{tabular}[t]{l}$\alpha$\end{tabular}}}}%
    \put(0.43214097,0.22616081){\color[rgb]{0,0,0}\makebox(0,0)[lt]{\lineheight{1.25}\smash{\begin{tabular}[t]{l}$\beta^\dagger$\end{tabular}}}}%
    \put(0,0){\includegraphics[width=\unitlength,page=2]{transport-2.pdf}}%
    \put(0.40577689,0.95665523){\color[rgb]{0,0,0}\makebox(0,0)[lt]{\lineheight{1.25}\smash{\begin{tabular}[t]{l}$i$\end{tabular}}}}%
    \put(0,0){\includegraphics[width=\unitlength,page=3]{transport-2.pdf}}%
  \end{picture}%
\endgroup%
}}}.\label{eq13}
\end{align}

It is not too hard to show that $\Lambda$ is Hermitian (i.e. $\Lambda(\widecheck{\mc Y}^\alpha|\widecheck{\mc Y}^\beta)=\ovl{\Lambda(\widecheck{\mc Y}^\beta|\widecheck{\mc Y}^\alpha)}$). Indeed, one can prove this by applying   \cite{Gui19b} formula (5.34) to the adjoint of \eqref{eq11}. (The intertwining operator $\mc Y_{\tilde\sigma}$ in that formula could be determined from the proofs of \cite{Gui19b} corollary 5.7 and theorem 5.5.) Moreover, from the rigidity of $\Rep(V)$ one can deduce that $\Lambda$ is non-degenerate; see \cite{HK07} theorem 3.4,\footnote{The relation between the bilinear form in \cite{HK07} theorem 3.4 and the sesquilinear form $\Lambda$ is explained in \cite{Gui19b} section 8.3.} or step 3 of the proof of \cite{Gui19b} theorem 6.7. However, to make $\Lambda$ an inner product one has to prove that $\Lambda$ is positive. In \cite{Gui19b,Gui19c} we have proved the positivity of $\Lambda$ for many examples of VOAs. As mentioned in the introduction, one of our main goal of this paper is to prove that all unitary extensions of these examples have positive $\Lambda$. We first introduce the following definition.

\begin{df}
Let $V$ be a regular and CFT-type VOA. We say that $V$ is \textbf{completely unitary}, if $V$ is strongly unitary (i.e., $V$ is unitary and any $V$-module is unitarizable), and if for any unitary $V$-modules $W_i,W_j$ and any $t\in\mc E$, the invariant  sesquilinear form $\Lambda$ defined on $\mc V{t\choose i~j}^*$ is positive. In this case we call $\Lambda$ the \textbf{invariant inner product} of $V$.
\end{df}

As we have said, since $\Lambda$ is non-degenerate, when $V$ is completely unitary $\Lambda$ becomes an inner product on $\mc V{t\choose i~j}^*$ (and hence also on $\mc V{t\choose i~j}$), which can be extended to an inner product on $W_i\boxtimes W_j$, also denoted by $\Lambda$. Then $W_i\boxtimes W_j$ becomes a unitary (but not just unitarizable) $V$-module. In other words it is an object not just in $\Rep(V)$ but also in $\Rep^\uni(V)$. Moreover, as shown in \cite{Gui19b}, $\Lambda$ is the right inner product which makes all structural isomorphisms unitary. More precisely:

\begin{thm}[cf. \cite{Gui19b} theorem 7.9]\label{lb51}
If $V$ is regular, CFT-type, and completely unitary, then $\Rep^\uni(V)$ is a unitary modular tensor category.
\end{thm}

In the remaining part of this paper, we assume, unless otherwise stated, that $V$ is regular, CFT-type, and completely unitary. The following relation is worth noting; see \cite{Gui19b} section 7.3.

\begin{pp}\label{lb5}
If $W_i$ is unitary then $\coev_{i,\ovl i}=\ev_{i,\ovl i}^*$.\index{coev@$\coev_{i,\ovl i}$}
\end{pp}

We are now ready to state the main results of this section.
\begin{thm}\label{lb6}
For any unitary $V$-modules $W_i,W_j,W_k$ and any $\mc Y_\alpha\in\mc V{k\choose i~j}$, we have $\alpha^\dagger=(\ev_{\ovl i,i}\otimes\id_j)(\id_{\ovl i}\otimes\alpha^*)$. In other words,
\begin{align}
\vcenter{\hbox{{\def\svgscale{0.6}
%% Accompanies image file 'adjoint.pdf' (pdf, eps, ps)
\begingroup%
  \makeatletter%
  \providecommand\color[2][]{%
    \errmessage{(Inkscape) Color is used for the text in Inkscape, but the package 'color.sty' is not loaded}%
    \renewcommand\color[2][]{}%
  }%
  \providecommand\transparent[1]{%
    \errmessage{(Inkscape) Transparency is used (non-zero) for the text in Inkscape, but the package 'transparent.sty' is not loaded}%
    \renewcommand\transparent[1]{}%
  }%
  \providecommand\rotatebox[2]{#2}%
  \newcommand*\fsize{\dimexpr\f@size pt\relax}%
  \newcommand*\lineheight[1]{\fontsize{\fsize}{#1\fsize}\selectfont}%
  \ifx\svgwidth\undefined%
    \setlength{\unitlength}{62.04422582bp}%
    \ifx\svgscale\undefined%
      \relax%
    \else%
      \setlength{\unitlength}{\unitlength * \real{\svgscale}}%
    \fi%
  \else%
    \setlength{\unitlength}{\svgwidth}%
  \fi%
  \global\let\svgwidth\undefined%
  \global\let\svgscale\undefined%
  \makeatother%
  \begin{picture}(1,1.62580073)%
    \lineheight{1}%
    \setlength\tabcolsep{0pt}%
    \put(0,0){\includegraphics[width=\unitlength,page=1]{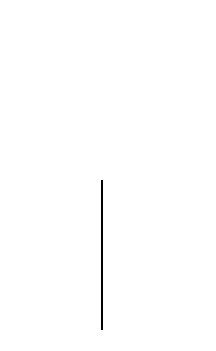}}%
    \put(0.55596038,0.01341029){\color[rgb]{0,0,0}\makebox(0,0)[lt]{\lineheight{1.25}\smash{\begin{tabular}[t]{l}$j$\end{tabular}}}}%
    \put(-0.00535153,1.57681853){\color[rgb]{0,0,0}\makebox(0,0)[lt]{\lineheight{1.25}\smash{\begin{tabular}[t]{l}$i$\end{tabular}}}}%
    \put(0.88588659,1.57647275){\color[rgb]{0,0,0}\makebox(0,0)[lt]{\lineheight{1.25}\smash{\begin{tabular}[t]{l}$k$\end{tabular}}}}%
    \put(0,0){\includegraphics[width=\unitlength,page=2]{adjoint.pdf}}%
    \put(0.58413489,0.6962492){\color[rgb]{0,0,0}\makebox(0,0)[lt]{\lineheight{1.25}\smash{\begin{tabular}[t]{l}$\alpha^\dagger$\end{tabular}}}}%
  \end{picture}%
\endgroup%
}}}~~~=~~\vcenter{\hbox{{\def\svgscale{0.6}
%% Accompanies image file 'adjoint-2.pdf' (pdf, eps, ps)
\begingroup%
  \makeatletter%
  \providecommand\color[2][]{%
    \errmessage{(Inkscape) Color is used for the text in Inkscape, but the package 'color.sty' is not loaded}%
    \renewcommand\color[2][]{}%
  }%
  \providecommand\transparent[1]{%
    \errmessage{(Inkscape) Transparency is used (non-zero) for the text in Inkscape, but the package 'transparent.sty' is not loaded}%
    \renewcommand\transparent[1]{}%
  }%
  \providecommand\rotatebox[2]{#2}%
  \newcommand*\fsize{\dimexpr\f@size pt\relax}%
  \newcommand*\lineheight[1]{\fontsize{\fsize}{#1\fsize}\selectfont}%
  \ifx\svgwidth\undefined%
    \setlength{\unitlength}{72.99625681bp}%
    \ifx\svgscale\undefined%
      \relax%
    \else%
      \setlength{\unitlength}{\unitlength * \real{\svgscale}}%
    \fi%
  \else%
    \setlength{\unitlength}{\svgwidth}%
  \fi%
  \global\let\svgwidth\undefined%
  \global\let\svgscale\undefined%
  \makeatother%
  \begin{picture}(1,1.30226525)%
    \lineheight{1}%
    \setlength\tabcolsep{0pt}%
    \put(0,0){\includegraphics[width=\unitlength,page=1]{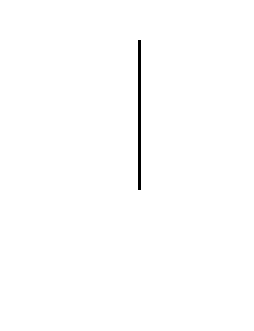}}%
    \put(0.91959669,0.01250773){\color[rgb]{0,0,0}\makebox(0,0)[lt]{\lineheight{1.25}\smash{\begin{tabular}[t]{l}$j$\end{tabular}}}}%
    \put(-0.00454861,1.26063212){\color[rgb]{0,0,0}\makebox(0,0)[lt]{\lineheight{1.25}\smash{\begin{tabular}[t]{l}$i$\end{tabular}}}}%
    \put(0.52323845,1.26033763){\color[rgb]{0,0,0}\makebox(0,0)[lt]{\lineheight{1.25}\smash{\begin{tabular}[t]{l}$k$\end{tabular}}}}%
    \put(0,0){\includegraphics[width=\unitlength,page=2]{adjoint-2.pdf}}%
    \put(0.63860896,0.58343767){\color[rgb]{0,0,0}\makebox(0,0)[lt]{\lineheight{1.25}\smash{\begin{tabular}[t]{l}$\alpha^*$\end{tabular}}}}%
  \end{picture}%
\endgroup%
}}}~~.\label{eq12}
\end{align}
\end{thm}
\begin{proof}
It suffices to assume $W_k$ to be irreducible. So let us prove \eqref{eq12} for all $k=t\in\mc E$ and any basis vector $\alpha\in\Xi^t_{i,j}$. Here we assume $\Xi^t_{i,j}$ to be orthonormal under $\Lambda$. Choose $\wtd\alpha\in \Hom(W_i\boxtimes W_j,W_t)$ such that $\wtd\alpha^\dagger$ equals $(\ev_{\ovl i,i}\otimes\id_j)(\id_{\ovl i}\otimes\alpha^*)$. We want to show $\wtd\alpha=\alpha$. 

By proposition \ref{lb3} we have $\alpha\beta^*=\delta_{\alpha,\beta}\id_t$ for any $\alpha,\beta\in\Xi^t_{i,j}$. This implies
\begin{align}
\vcenter{\hbox{{\def\svgscale{0.6}
%% Accompanies image file 'adjoint-3.pdf' (pdf, eps, ps)
\begingroup%
  \makeatletter%
  \providecommand\color[2][]{%
    \errmessage{(Inkscape) Color is used for the text in Inkscape, but the package 'color.sty' is not loaded}%
    \renewcommand\color[2][]{}%
  }%
  \providecommand\transparent[1]{%
    \errmessage{(Inkscape) Transparency is used (non-zero) for the text in Inkscape, but the package 'transparent.sty' is not loaded}%
    \renewcommand\transparent[1]{}%
  }%
  \providecommand\rotatebox[2]{#2}%
  \newcommand*\fsize{\dimexpr\f@size pt\relax}%
  \newcommand*\lineheight[1]{\fontsize{\fsize}{#1\fsize}\selectfont}%
  \ifx\svgwidth\undefined%
    \setlength{\unitlength}{58.37056847bp}%
    \ifx\svgscale\undefined%
      \relax%
    \else%
      \setlength{\unitlength}{\unitlength * \real{\svgscale}}%
    \fi%
  \else%
    \setlength{\unitlength}{\svgwidth}%
  \fi%
  \global\let\svgwidth\undefined%
  \global\let\svgscale\undefined%
  \makeatother%
  \begin{picture}(1,1.63807235)%
    \lineheight{1}%
    \setlength\tabcolsep{0pt}%
    \put(-0.00568833,1.58600737){\color[rgb]{0,0,0}\makebox(0,0)[lt]{\lineheight{1.25}\smash{\begin{tabular}[t]{l}$i$\end{tabular}}}}%
    \put(0.8821112,1.58563982){\color[rgb]{0,0,0}\makebox(0,0)[lt]{\lineheight{1.25}\smash{\begin{tabular}[t]{l}$j$\end{tabular}}}}%
    \put(0,0){\includegraphics[width=\unitlength,page=1]{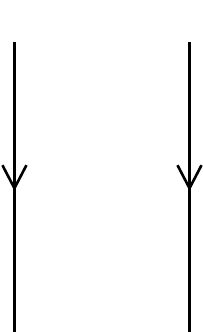}}%
  \end{picture}%
\endgroup%
}}}~~~=~~\sum_{t\in\mc E}\sum_{\alpha\in\Xi^t_{i,j}}\vcenter{\hbox{{\def\svgscale{0.6}
%% Accompanies image file 'adjoint-4.pdf' (pdf, eps, ps)
\begingroup%
  \makeatletter%
  \providecommand\color[2][]{%
    \errmessage{(Inkscape) Color is used for the text in Inkscape, but the package 'color.sty' is not loaded}%
    \renewcommand\color[2][]{}%
  }%
  \providecommand\transparent[1]{%
    \errmessage{(Inkscape) Transparency is used (non-zero) for the text in Inkscape, but the package 'transparent.sty' is not loaded}%
    \renewcommand\transparent[1]{}%
  }%
  \providecommand\rotatebox[2]{#2}%
  \newcommand*\fsize{\dimexpr\f@size pt\relax}%
  \newcommand*\lineheight[1]{\fontsize{\fsize}{#1\fsize}\selectfont}%
  \ifx\svgwidth\undefined%
    \setlength{\unitlength}{60.11783644bp}%
    \ifx\svgscale\undefined%
      \relax%
    \else%
      \setlength{\unitlength}{\unitlength * \real{\svgscale}}%
    \fi%
  \else%
    \setlength{\unitlength}{\svgwidth}%
  \fi%
  \global\let\svgwidth\undefined%
  \global\let\svgscale\undefined%
  \makeatother%
  \begin{picture}(1,1.87695894)%
    \lineheight{1}%
    \setlength\tabcolsep{0pt}%
    \put(-0.00552301,1.82640718){\color[rgb]{0,0,0}\makebox(0,0)[lt]{\lineheight{1.25}\smash{\begin{tabular}[t]{l}$i$\end{tabular}}}}%
    \put(0.87406499,1.82102407){\color[rgb]{0,0,0}\makebox(0,0)[lt]{\lineheight{1.25}\smash{\begin{tabular}[t]{l}$j$\end{tabular}}}}%
    \put(0,0){\includegraphics[width=\unitlength,page=1]{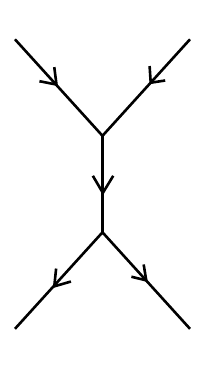}}%
    \put(0.90237272,0.0193487){\color[rgb]{0,0,0}\makebox(0,0)[lt]{\lineheight{1.25}\smash{\begin{tabular}[t]{l}$j$\end{tabular}}}}%
    \put(0.02532123,0.00977897){\color[rgb]{0,0,0}\makebox(0,0)[lt]{\lineheight{1.25}\smash{\begin{tabular}[t]{l}$i$\end{tabular}}}}%
    \put(0.22516575,0.93737248){\color[rgb]{0,0,0}\makebox(0,0)[lt]{\lineheight{1.25}\smash{\begin{tabular}[t]{l}$t$\end{tabular}}}}%
    \put(0.59960879,1.15600644){\color[rgb]{0,0,0}\makebox(0,0)[lt]{\lineheight{1.25}\smash{\begin{tabular}[t]{l}$\alpha$\end{tabular}}}}%
    \put(0.60212191,0.73632967){\color[rgb]{0,0,0}\makebox(0,0)[lt]{\lineheight{1.25}\smash{\begin{tabular}[t]{l}$\alpha^*$\end{tabular}}}}%
  \end{picture}%
\endgroup%
}}}~~.
\end{align}
Tensor $\id_{\ovl i}$ from the left and apply $\ev_{\ovl i,i}\otimes\id_j$ to the bottom, we obtain
\begin{align}
\vcenter{\hbox{{\def\svgscale{0.6}
%% Accompanies image file 'adjoint-5.pdf' (pdf, eps, ps)
\begingroup%
  \makeatletter%
  \providecommand\color[2][]{%
    \errmessage{(Inkscape) Color is used for the text in Inkscape, but the package 'color.sty' is not loaded}%
    \renewcommand\color[2][]{}%
  }%
  \providecommand\transparent[1]{%
    \errmessage{(Inkscape) Transparency is used (non-zero) for the text in Inkscape, but the package 'transparent.sty' is not loaded}%
    \renewcommand\transparent[1]{}%
  }%
  \providecommand\rotatebox[2]{#2}%
  \newcommand*\fsize{\dimexpr\f@size pt\relax}%
  \newcommand*\lineheight[1]{\fontsize{\fsize}{#1\fsize}\selectfont}%
  \ifx\svgwidth\undefined%
    \setlength{\unitlength}{59.32668501bp}%
    \ifx\svgscale\undefined%
      \relax%
    \else%
      \setlength{\unitlength}{\unitlength * \real{\svgscale}}%
    \fi%
  \else%
    \setlength{\unitlength}{\svgwidth}%
  \fi%
  \global\let\svgwidth\undefined%
  \global\let\svgscale\undefined%
  \makeatother%
  \begin{picture}(1,1.60443588)%
    \lineheight{1}%
    \setlength\tabcolsep{0pt}%
    \put(0,0){\includegraphics[width=\unitlength,page=1]{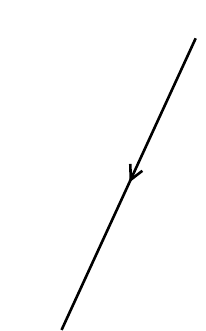}}%
    \put(-0.00162406,1.54076258){\color[rgb]{0,0,0}\makebox(0,0)[lt]{\lineheight{1.25}\smash{\begin{tabular}[t]{l}$i$\end{tabular}}}}%
    \put(0.41722511,1.5481167){\color[rgb]{0,0,0}\makebox(0,0)[lt]{\lineheight{1.25}\smash{\begin{tabular}[t]{l}$i$\end{tabular}}}}%
    \put(0.90107081,1.55320998){\color[rgb]{0,0,0}\makebox(0,0)[lt]{\lineheight{1.25}\smash{\begin{tabular}[t]{l}$j$\end{tabular}}}}%
    \put(0,0){\includegraphics[width=\unitlength,page=2]{adjoint-5.pdf}}%
  \end{picture}%
\endgroup%
}}}~~~=~~\sum_{t\in\mc E}\sum_{\alpha\in\Xi^t_{i,j}}\vcenter{\hbox{{\def\svgscale{0.6}
%% Accompanies image file 'adjoint-6.pdf' (pdf, eps, ps)
\begingroup%
  \makeatletter%
  \providecommand\color[2][]{%
    \errmessage{(Inkscape) Color is used for the text in Inkscape, but the package 'color.sty' is not loaded}%
    \renewcommand\color[2][]{}%
  }%
  \providecommand\transparent[1]{%
    \errmessage{(Inkscape) Transparency is used (non-zero) for the text in Inkscape, but the package 'transparent.sty' is not loaded}%
    \renewcommand\transparent[1]{}%
  }%
  \providecommand\rotatebox[2]{#2}%
  \newcommand*\fsize{\dimexpr\f@size pt\relax}%
  \newcommand*\lineheight[1]{\fontsize{\fsize}{#1\fsize}\selectfont}%
  \ifx\svgwidth\undefined%
    \setlength{\unitlength}{64.2530111bp}%
    \ifx\svgscale\undefined%
      \relax%
    \else%
      \setlength{\unitlength}{\unitlength * \real{\svgscale}}%
    \fi%
  \else%
    \setlength{\unitlength}{\svgwidth}%
  \fi%
  \global\let\svgwidth\undefined%
  \global\let\svgscale\undefined%
  \makeatother%
  \begin{picture}(1,1.85761453)%
    \lineheight{1}%
    \setlength\tabcolsep{0pt}%
    \put(0,0){\includegraphics[width=\unitlength,page=1]{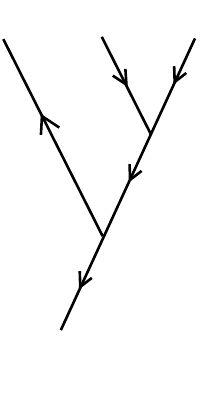}}%
    \put(-0.0047116,1.79821516){\color[rgb]{0,0,0}\makebox(0,0)[lt]{\lineheight{1.25}\smash{\begin{tabular}[t]{l}$i$\end{tabular}}}}%
    \put(0.3820241,1.80500544){\color[rgb]{0,0,0}\makebox(0,0)[lt]{\lineheight{1.25}\smash{\begin{tabular}[t]{l}$i$\end{tabular}}}}%
    \put(0.82877299,1.80970821){\color[rgb]{0,0,0}\makebox(0,0)[lt]{\lineheight{1.25}\smash{\begin{tabular}[t]{l}$j$\end{tabular}}}}%
    \put(0.45021228,1.09961316){\color[rgb]{0,0,0}\makebox(0,0)[lt]{\lineheight{1.25}\smash{\begin{tabular}[t]{l}$t$\end{tabular}}}}%
    \put(0.20332515,0.0128581){\color[rgb]{0,0,0}\makebox(0,0)[lt]{\lineheight{1.25}\smash{\begin{tabular}[t]{l}$j$\end{tabular}}}}%
    \put(0.72654491,1.17511938){\color[rgb]{0,0,0}\makebox(0,0)[lt]{\lineheight{1.25}\smash{\begin{tabular}[t]{l}$\alpha$\end{tabular}}}}%
    \put(0.54623509,0.63765951){\color[rgb]{0,0,0}\makebox(0,0)[lt]{\lineheight{1.25}\smash{\begin{tabular}[t]{l}$\widetilde\alpha^\dagger$\end{tabular}}}}%
  \end{picture}%
\endgroup%
}}}~~,
\end{align}
which, together with equation \eqref{eq13}, implies
\begin{align*}
\wtd\alpha^\dagger=\sum_{\beta\in\Xi^t_{i,j}}\Lambda(\widecheck{\mc Y}^\alpha|\widecheck{\mc Y}^\beta)\beta^\dagger=\sum_{\beta\in\Xi^t_{i,j}}\delta_{\alpha,\beta}\beta^\dagger=\alpha^\dagger.
\end{align*}
Thus $\wtd\alpha=\alpha$.
\end{proof}

%For any unitary $W_i,W_j$ and any $F\in\Hom(W_i,W_j)$, define a new morphism $F^\vee\in\Hom(W_{\ovl j},W_{\ovl i})$ by
%\begin{align}
%
%\end{align}
\begin{co}\label{lb13}
For any unitary $V$-modules $W_i,W_j,W_k$ and any $\mc Y_\alpha\in\mc V{k\choose i~j}$,
\begin{align}
\vcenter{\hbox{{\def\svgscale{0.6}
%% Accompanies image file 'conjugate.pdf' (pdf, eps, ps)
\begingroup%
  \makeatletter%
  \providecommand\color[2][]{%
    \errmessage{(Inkscape) Color is used for the text in Inkscape, but the package 'color.sty' is not loaded}%
    \renewcommand\color[2][]{}%
  }%
  \providecommand\transparent[1]{%
    \errmessage{(Inkscape) Transparency is used (non-zero) for the text in Inkscape, but the package 'transparent.sty' is not loaded}%
    \renewcommand\transparent[1]{}%
  }%
  \providecommand\rotatebox[2]{#2}%
  \newcommand*\fsize{\dimexpr\f@size pt\relax}%
  \newcommand*\lineheight[1]{\fontsize{\fsize}{#1\fsize}\selectfont}%
  \ifx\svgwidth\undefined%
    \setlength{\unitlength}{60.0764647bp}%
    \ifx\svgscale\undefined%
      \relax%
    \else%
      \setlength{\unitlength}{\unitlength * \real{\svgscale}}%
    \fi%
  \else%
    \setlength{\unitlength}{\svgwidth}%
  \fi%
  \global\let\svgwidth\undefined%
  \global\let\svgscale\undefined%
  \makeatother%
  \begin{picture}(1,1.60220007)%
    \lineheight{1}%
    \setlength\tabcolsep{0pt}%
    \put(0,0){\includegraphics[width=\unitlength,page=1]{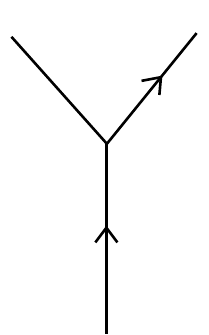}}%
    \put(-0.00552681,1.53149537){\color[rgb]{0,0,0}\makebox(0,0)[lt]{\lineheight{1.25}\smash{\begin{tabular}[t]{l}$i$\end{tabular}}}}%
    \put(0.90230549,1.5516135){\color[rgb]{0,0,0}\makebox(0,0)[lt]{\lineheight{1.25}\smash{\begin{tabular}[t]{l}$j$\end{tabular}}}}%
    \put(0.6516145,0.73530564){\color[rgb]{0,0,0}\makebox(0,0)[lt]{\lineheight{1.25}\smash{\begin{tabular}[t]{l}$\overline{\alpha}$\end{tabular}}}}%
    \put(0.24714959,0.22556459){\color[rgb]{0,0,0}\makebox(0,0)[lt]{\lineheight{1.25}\smash{\begin{tabular}[t]{l}$k$\end{tabular}}}}%
    \put(0,0){\includegraphics[width=\unitlength,page=2]{conjugate.pdf}}%
  \end{picture}%
\endgroup%
}}}~~~=~~~\vcenter{\hbox{{\def\svgscale{0.6}
%% Accompanies image file 'conjugate-2.pdf' (pdf, eps, ps)
\begingroup%
  \makeatletter%
  \providecommand\color[2][]{%
    \errmessage{(Inkscape) Color is used for the text in Inkscape, but the package 'color.sty' is not loaded}%
    \renewcommand\color[2][]{}%
  }%
  \providecommand\transparent[1]{%
    \errmessage{(Inkscape) Transparency is used (non-zero) for the text in Inkscape, but the package 'transparent.sty' is not loaded}%
    \renewcommand\transparent[1]{}%
  }%
  \providecommand\rotatebox[2]{#2}%
  \newcommand*\fsize{\dimexpr\f@size pt\relax}%
  \newcommand*\lineheight[1]{\fontsize{\fsize}{#1\fsize}\selectfont}%
  \ifx\svgwidth\undefined%
    \setlength{\unitlength}{133.11719607bp}%
    \ifx\svgscale\undefined%
      \relax%
    \else%
      \setlength{\unitlength}{\unitlength * \real{\svgscale}}%
    \fi%
  \else%
    \setlength{\unitlength}{\svgwidth}%
  \fi%
  \global\let\svgwidth\undefined%
  \global\let\svgscale\undefined%
  \makeatother%
  \begin{picture}(1,1.34250014)%
    \lineheight{1}%
    \setlength\tabcolsep{0pt}%
    \put(0,0){\includegraphics[width=\unitlength,page=1]{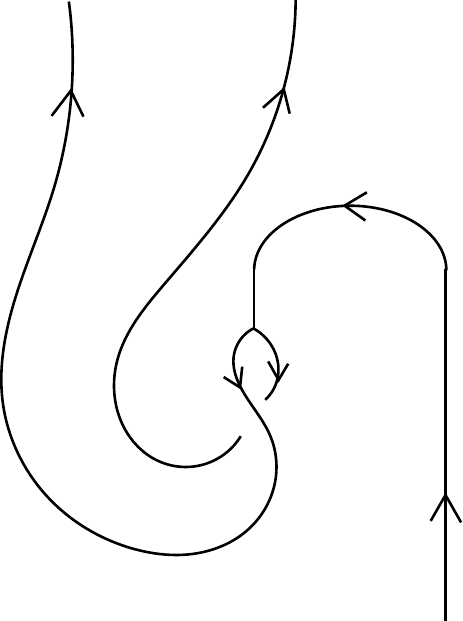}}%
    \put(0.01687577,1.2763304){\color[rgb]{0,0,0}\makebox(0,0)[lt]{\lineheight{1.25}\smash{\begin{tabular}[t]{l}$i$\end{tabular}}}}%
    \put(0.53209088,1.28917083){\color[rgb]{0,0,0}\makebox(0,0)[lt]{\lineheight{1.25}\smash{\begin{tabular}[t]{l}$j$\end{tabular}}}}%
    \put(0.85951736,0.05169115){\color[rgb]{0,0,0}\makebox(0,0)[lt]{\lineheight{1.25}\smash{\begin{tabular}[t]{l}$k$\end{tabular}}}}%
    \put(0.59629223,0.6487628){\color[rgb]{0,0,0}\makebox(0,0)[lt]{\lineheight{1.25}\smash{\begin{tabular}[t]{l}$\alpha^*$\end{tabular}}}}%
  \end{picture}%
\endgroup%
}}}\quad.\label{eq15}
\end{align}
\end{co}

\begin{proof}
We have $\ovl{\alpha}=\ovl{C_+C_-\alpha}=(C_-\alpha)^\dagger$. By propositions \ref{lb4}, \ref{lb5}, and the unitarity of $\ss$,  $(C_-\alpha)^*$ equals
\begin{align}
\vcenter{\hbox{{\def\svgscale{0.5}
%% Accompanies image file 'conjugate-3.pdf' (pdf, eps, ps)
\begingroup%
  \makeatletter%
  \providecommand\color[2][]{%
    \errmessage{(Inkscape) Color is used for the text in Inkscape, but the package 'color.sty' is not loaded}%
    \renewcommand\color[2][]{}%
  }%
  \providecommand\transparent[1]{%
    \errmessage{(Inkscape) Transparency is used (non-zero) for the text in Inkscape, but the package 'transparent.sty' is not loaded}%
    \renewcommand\transparent[1]{}%
  }%
  \providecommand\rotatebox[2]{#2}%
  \newcommand*\fsize{\dimexpr\f@size pt\relax}%
  \newcommand*\lineheight[1]{\fontsize{\fsize}{#1\fsize}\selectfont}%
  \ifx\svgwidth\undefined%
    \setlength{\unitlength}{98.95589782bp}%
    \ifx\svgscale\undefined%
      \relax%
    \else%
      \setlength{\unitlength}{\unitlength * \real{\svgscale}}%
    \fi%
  \else%
    \setlength{\unitlength}{\svgwidth}%
  \fi%
  \global\let\svgwidth\undefined%
  \global\let\svgscale\undefined%
  \makeatother%
  \begin{picture}(1,1.63496124)%
    \lineheight{1}%
    \setlength\tabcolsep{0pt}%
    \put(0,0){\includegraphics[width=\unitlength,page=1]{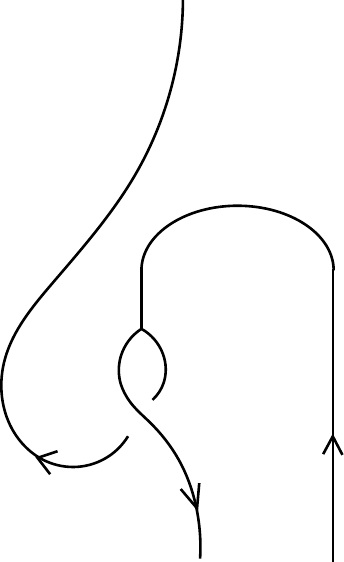}}%
    \put(0.45469956,0.69741609){\color[rgb]{0,0,0}\makebox(0,0)[lt]{\lineheight{1.25}\smash{\begin{tabular}[t]{l}$\alpha^*$\end{tabular}}}}%
    \put(0.81743117,0.05615818){\color[rgb]{0,0,0}\makebox(0,0)[lt]{\lineheight{1.25}\smash{\begin{tabular}[t]{l}$k$\end{tabular}}}}%
    \put(0.41799461,0.05831744){\color[rgb]{0,0,0}\makebox(0,0)[lt]{\lineheight{1.25}\smash{\begin{tabular}[t]{l}$i$\end{tabular}}}}%
    \put(0.36185767,1.48117485){\color[rgb]{0,0,0}\makebox(0,0)[lt]{\lineheight{1.25}\smash{\begin{tabular}[t]{l}$j$\end{tabular}}}}%
  \end{picture}%
\endgroup%
}}}~~.\label{eq14}
\end{align}
By theorem \ref{lb6}, one obtains $(C_-\alpha)^\dagger$ from $(C_-\alpha)^*$ by bending the leg $i$ to the top. Thus \eqref{eq14} becomes the right hand side of \eqref{eq15}.
\end{proof}

\section{Unitary VOA extensions, $C^*$-Frobenius algebras, and their representations}

\subsection{Preunitary VOA extensions and commutative $\mc C$-algebras}

In this chapter, we  fix a regular, CFT-type, and completely unitary VOA $V$. $W_{\ovl 0}=\ovl V$ is identified with $W_0=V$ via the reflection operator $\epsilon$. Thus $\epsilon=1$.   Let $U$ be a (VOA) extension of $V$,\index{Uext@$U$} whose vertex operator is denoted by $\mc Y_\mu$.\index{Uext@$U$!$\mc Y_\mu,\mu=\vcenter{\hbox{{\def\svgscale{0.4}
%% Accompanies image file 'C-algebra.pdf' (pdf, eps, ps)
\begingroup%
  \makeatletter%
  \providecommand\color[2][]{%
    \errmessage{(Inkscape) Color is used for the text in Inkscape, but the package 'color.sty' is not loaded}%
    \renewcommand\color[2][]{}%
  }%
  \providecommand\transparent[1]{%
    \errmessage{(Inkscape) Transparency is used (non-zero) for the text in Inkscape, but the package 'transparent.sty' is not loaded}%
    \renewcommand\transparent[1]{}%
  }%
  \providecommand\rotatebox[2]{#2}%
  \newcommand*\fsize{\dimexpr\f@size pt\relax}%
  \newcommand*\lineheight[1]{\fontsize{\fsize}{#1\fsize}\selectfont}%
  \ifx\svgwidth\undefined%
    \setlength{\unitlength}{54.19421239bp}%
    \ifx\svgscale\undefined%
      \relax%
    \else%
      \setlength{\unitlength}{\unitlength * \real{\svgscale}}%
    \fi%
  \else%
    \setlength{\unitlength}{\svgwidth}%
  \fi%
  \global\let\svgwidth\undefined%
  \global\let\svgscale\undefined%
  \makeatother%
  \begin{picture}(1,1.41535979)%
    \lineheight{1}%
    \setlength\tabcolsep{0pt}%
    \put(0,0){\includegraphics[width=\unitlength,page=1]{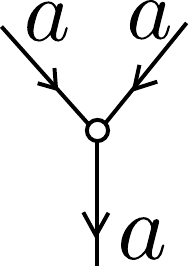}}%
  \end{picture}%
\endgroup%
}}}$} By definition, $U$ and $V$ share the same conformal vector $\nu$ and vacuum vector $\Omega$, and $V$ is a vertex operator subalgebra of $U$. As in the previous chapter, the symbol $Y$ is  reserved for the vertex operator of $V$. Then  $Y$  is the restriction of the action $\mc Y_\mu:U\curvearrowright U$ to $V\curvearrowright V$. More generally, let $Y_a$  be the restriction of $\mc Y_\mu$ to $V\curvearrowright U$. Then $(U,Y_a)$ becomes a representation of $V$. We write this $V$-module as $(W_a,Y_a)$, or  $W_a$ for short.\index{Uext@$U$!$(W_a,Y_a)$} (So by our notation, $U$ equals $W_a$ as a vector space.) Since all $V$-modules are unitarizable, we fix a unitary structure (i.e., an inner product) $\bk{\cdot|\cdot}$ on the $V$-module $W_a$ whose restriction to $V$ is the one $\bk{\cdot|\cdot}$ of the unitary VOA $V$. Such $(U,\bk{\cdot|\cdot})$, or $U$ for short, is called a \textbf{preunitary (VOA) extension} of $V$. It is clear that any extension of $V$ is preunitarizable.  Finally, we notice that $\mc Y_\mu$ is a type ${a\choose a~a}={W_a\choose W_a W_a}$ unitary intertwining operator of $V$.

The above discussion can be summarized as follows: The preunitary VOA extension $U$ is a unitary $V$ module $(W_a,Y_a)$; its vertex operator $\mc Y_\mu$ is in $\mc V{a\choose a~a}$. Thus $\mu\equiv\vcenter{\hbox{{\def\svgscale{0.4}
			}}}\in\Hom(W_a\boxtimes W_a,W_a)$ by our notation in the last chapter. Let $\iota:W_0\rightarrow W_a$ denote the embedding of $V$ into $U$. Then clearly $\iota\in\Hom(W_0,W_a)$. 
We write $\iota=\vcenter{\hbox{{\def\svgscale{0.4}
%% Accompanies image file 'C-algebra-2.pdf' (pdf, eps, ps)
\begingroup%
  \makeatletter%
  \providecommand\color[2][]{%
    \errmessage{(Inkscape) Color is used for the text in Inkscape, but the package 'color.sty' is not loaded}%
    \renewcommand\color[2][]{}%
  }%
  \providecommand\transparent[1]{%
    \errmessage{(Inkscape) Transparency is used (non-zero) for the text in Inkscape, but the package 'transparent.sty' is not loaded}%
    \renewcommand\transparent[1]{}%
  }%
  \providecommand\rotatebox[2]{#2}%
  \newcommand*\fsize{\dimexpr\f@size pt\relax}%
  \newcommand*\lineheight[1]{\fontsize{\fsize}{#1\fsize}\selectfont}%
  \ifx\svgwidth\undefined%
    \setlength{\unitlength}{21.95117885bp}%
    \ifx\svgscale\undefined%
      \relax%
    \else%
      \setlength{\unitlength}{\unitlength * \real{\svgscale}}%
    \fi%
  \else%
    \setlength{\unitlength}{\svgwidth}%
  \fi%
  \global\let\svgwidth\undefined%
  \global\let\svgscale\undefined%
  \makeatother%
  \begin{picture}(1,1.94741476)%
    \lineheight{1}%
    \setlength\tabcolsep{0pt}%
    \put(0,0){\includegraphics[width=\unitlength,page=1]{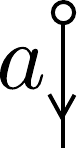}}%
  \end{picture}%
\endgroup%
}}}~~$.\index{Uext@$U$!$\iota=\vcenter{\hbox{{\def\svgscale{0.4}
				}}}$} Set $A_U=(W_a,\mu,\iota)$. \index{Uext@$U$!$A_U$}
Then $A_U$ is a \textbf{commutative associative algebra in $\Rep^\uni(V)$} (or \textbf{commutative $\RepV$-algebra} for short) \cite{Par95,KO02}, which means:
\begin{itemize}
	\item (Associativity) $\mu(\mu\otimes\id_a)=\mu(\id_a\otimes\mu)$.
	\item (Commutativity) $\mu\circ\ss=\mu$.
	\item (Unit) $\mu(\iota\otimes\id_a)=\id_a$.
\end{itemize}
Note that the associators and the unitors of $\Rep^\uni(V)$ have been suppressed to simplify discussions. We will also do so in the remaining part of this article.

Recall that $\ss$ is the braid isomorphism $\ss_{a,a}:W_a\boxtimes W_a\rightarrow W_a\boxtimes W_a$. These three conditions can respectively be pictured as
\begin{gather*}
\vcenter{\hbox{{\def\svgscale{0.6}
%% Accompanies image file 'C-algebra-3.pdf' (pdf, eps, ps)
\begingroup%
  \makeatletter%
  \providecommand\color[2][]{%
    \errmessage{(Inkscape) Color is used for the text in Inkscape, but the package 'color.sty' is not loaded}%
    \renewcommand\color[2][]{}%
  }%
  \providecommand\transparent[1]{%
    \errmessage{(Inkscape) Transparency is used (non-zero) for the text in Inkscape, but the package 'transparent.sty' is not loaded}%
    \renewcommand\transparent[1]{}%
  }%
  \providecommand\rotatebox[2]{#2}%
  \newcommand*\fsize{\dimexpr\f@size pt\relax}%
  \newcommand*\lineheight[1]{\fontsize{\fsize}{#1\fsize}\selectfont}%
  \ifx\svgwidth\undefined%
    \setlength{\unitlength}{59.00748267bp}%
    \ifx\svgscale\undefined%
      \relax%
    \else%
      \setlength{\unitlength}{\unitlength * \real{\svgscale}}%
    \fi%
  \else%
    \setlength{\unitlength}{\svgwidth}%
  \fi%
  \global\let\svgwidth\undefined%
  \global\let\svgscale\undefined%
  \makeatother%
  \begin{picture}(1,1.29553593)%
    \lineheight{1}%
    \setlength\tabcolsep{0pt}%
    \put(0,0){\includegraphics[width=\unitlength,page=1]{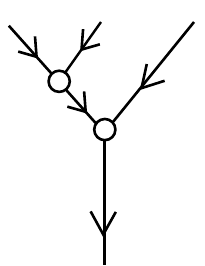}}%
    \put(-0.00562693,1.23407934){\color[rgb]{0,0,0}\makebox(0,0)[lt]{\lineheight{1.25}\smash{\begin{tabular}[t]{l}$a$\end{tabular}}}}%
    \put(0.41946425,1.24403219){\color[rgb]{0,0,0}\makebox(0,0)[lt]{\lineheight{1.25}\smash{\begin{tabular}[t]{l}$a$\end{tabular}}}}%
    \put(0.87776312,1.24403292){\color[rgb]{0,0,0}\makebox(0,0)[lt]{\lineheight{1.25}\smash{\begin{tabular}[t]{l}$a$\end{tabular}}}}%
    \put(0.55004101,0.04323865){\color[rgb]{0,0,0}\makebox(0,0)[lt]{\lineheight{1.25}\smash{\begin{tabular}[t]{l}$a$\end{tabular}}}}%
  \end{picture}%
\endgroup%
}}}~~=~\vcenter{\hbox{{\def\svgscale{0.6}
%% Accompanies image file 'C-algebra-4.pdf' (pdf, eps, ps)
\begingroup%
  \makeatletter%
  \providecommand\color[2][]{%
    \errmessage{(Inkscape) Color is used for the text in Inkscape, but the package 'color.sty' is not loaded}%
    \renewcommand\color[2][]{}%
  }%
  \providecommand\transparent[1]{%
    \errmessage{(Inkscape) Transparency is used (non-zero) for the text in Inkscape, but the package 'transparent.sty' is not loaded}%
    \renewcommand\transparent[1]{}%
  }%
  \providecommand\rotatebox[2]{#2}%
  \newcommand*\fsize{\dimexpr\f@size pt\relax}%
  \newcommand*\lineheight[1]{\fontsize{\fsize}{#1\fsize}\selectfont}%
  \ifx\svgwidth\undefined%
    \setlength{\unitlength}{59.00748267bp}%
    \ifx\svgscale\undefined%
      \relax%
    \else%
      \setlength{\unitlength}{\unitlength * \real{\svgscale}}%
    \fi%
  \else%
    \setlength{\unitlength}{\svgwidth}%
  \fi%
  \global\let\svgwidth\undefined%
  \global\let\svgscale\undefined%
  \makeatother%
  \begin{picture}(1,1.29553593)%
    \lineheight{1}%
    \setlength\tabcolsep{0pt}%
    \put(0,0){\includegraphics[width=\unitlength,page=1]{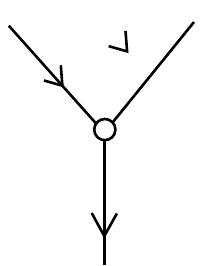}}%
    \put(-0.00562693,1.23407934){\color[rgb]{0,0,0}\makebox(0,0)[lt]{\lineheight{1.25}\smash{\begin{tabular}[t]{l}$a$\end{tabular}}}}%
    \put(0.41946425,1.24403219){\color[rgb]{0,0,0}\makebox(0,0)[lt]{\lineheight{1.25}\smash{\begin{tabular}[t]{l}$a$\end{tabular}}}}%
    \put(0.87776312,1.24403292){\color[rgb]{0,0,0}\makebox(0,0)[lt]{\lineheight{1.25}\smash{\begin{tabular}[t]{l}$a$\end{tabular}}}}%
    \put(0.55004101,0.04323865){\color[rgb]{0,0,0}\makebox(0,0)[lt]{\lineheight{1.25}\smash{\begin{tabular}[t]{l}$a$\end{tabular}}}}%
    \put(0,0){\includegraphics[width=\unitlength,page=2]{C-algebra-4.pdf}}%
  \end{picture}%
\endgroup%
}}}~~,\qquad \vcenter{\hbox{{\def\svgscale{0.6}
%% Accompanies image file 'C-algebra-5.pdf' (pdf, eps, ps)
\begingroup%
  \makeatletter%
  \providecommand\color[2][]{%
    \errmessage{(Inkscape) Color is used for the text in Inkscape, but the package 'color.sty' is not loaded}%
    \renewcommand\color[2][]{}%
  }%
  \providecommand\transparent[1]{%
    \errmessage{(Inkscape) Transparency is used (non-zero) for the text in Inkscape, but the package 'transparent.sty' is not loaded}%
    \renewcommand\transparent[1]{}%
  }%
  \providecommand\rotatebox[2]{#2}%
  \newcommand*\fsize{\dimexpr\f@size pt\relax}%
  \newcommand*\lineheight[1]{\fontsize{\fsize}{#1\fsize}\selectfont}%
  \ifx\svgwidth\undefined%
    \setlength{\unitlength}{59.00748267bp}%
    \ifx\svgscale\undefined%
      \relax%
    \else%
      \setlength{\unitlength}{\unitlength * \real{\svgscale}}%
    \fi%
  \else%
    \setlength{\unitlength}{\svgwidth}%
  \fi%
  \global\let\svgwidth\undefined%
  \global\let\svgscale\undefined%
  \makeatother%
  \begin{picture}(1,1.29553593)%
    \lineheight{1}%
    \setlength\tabcolsep{0pt}%
    \put(0,0){\includegraphics[width=\unitlength,page=1]{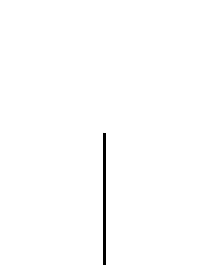}}%
    \put(-0.00562693,1.23407934){\color[rgb]{0,0,0}\makebox(0,0)[lt]{\lineheight{1.25}\smash{\begin{tabular}[t]{l}$a$\end{tabular}}}}%
    \put(0.87776312,1.24403292){\color[rgb]{0,0,0}\makebox(0,0)[lt]{\lineheight{1.25}\smash{\begin{tabular}[t]{l}$a$\end{tabular}}}}%
    \put(0.55004101,0.04323865){\color[rgb]{0,0,0}\makebox(0,0)[lt]{\lineheight{1.25}\smash{\begin{tabular}[t]{l}$a$\end{tabular}}}}%
    \put(0,0){\includegraphics[width=\unitlength,page=2]{C-algebra-5.pdf}}%
  \end{picture}%
\endgroup%
}}}~~=~\vcenter{\hbox{{\def\svgscale{0.6}
%% Accompanies image file 'C-algebra-6.pdf' (pdf, eps, ps)
\begingroup%
  \makeatletter%
  \providecommand\color[2][]{%
    \errmessage{(Inkscape) Color is used for the text in Inkscape, but the package 'color.sty' is not loaded}%
    \renewcommand\color[2][]{}%
  }%
  \providecommand\transparent[1]{%
    \errmessage{(Inkscape) Transparency is used (non-zero) for the text in Inkscape, but the package 'transparent.sty' is not loaded}%
    \renewcommand\transparent[1]{}%
  }%
  \providecommand\rotatebox[2]{#2}%
  \newcommand*\fsize{\dimexpr\f@size pt\relax}%
  \newcommand*\lineheight[1]{\fontsize{\fsize}{#1\fsize}\selectfont}%
  \ifx\svgwidth\undefined%
    \setlength{\unitlength}{60.60990783bp}%
    \ifx\svgscale\undefined%
      \relax%
    \else%
      \setlength{\unitlength}{\unitlength * \real{\svgscale}}%
    \fi%
  \else%
    \setlength{\unitlength}{\svgwidth}%
  \fi%
  \global\let\svgwidth\undefined%
  \global\let\svgscale\undefined%
  \makeatother%
  \begin{picture}(1,1.25423413)%
    \lineheight{1}%
    \setlength\tabcolsep{0pt}%
    \put(0,0){\includegraphics[width=\unitlength,page=1]{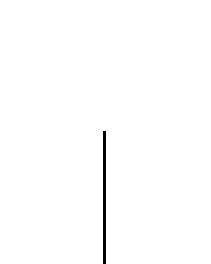}}%
    \put(-0.00547817,1.20145233){\color[rgb]{0,0,0}\makebox(0,0)[lt]{\lineheight{1.25}\smash{\begin{tabular}[t]{l}$a$\end{tabular}}}}%
    \put(0.88099486,1.20409278){\color[rgb]{0,0,0}\makebox(0,0)[lt]{\lineheight{1.25}\smash{\begin{tabular}[t]{l}$a$\end{tabular}}}}%
    \put(0.53549884,0.04209549){\color[rgb]{0,0,0}\makebox(0,0)[lt]{\lineheight{1.25}\smash{\begin{tabular}[t]{l}$a$\end{tabular}}}}%
    \put(0,0){\includegraphics[width=\unitlength,page=2]{C-algebra-6.pdf}}%
  \end{picture}%
\endgroup%
}}}~~,\qquad	\vcenter{\hbox{{\def\svgscale{0.6}%% Creator: Inkscape inkscape 0.92.4, www.inkscape.org
%% Accompanies image file 'C-algebra-7.pdf' (pdf, eps, ps)
\begingroup%
  \makeatletter%
  \providecommand\color[2][]{%
    \errmessage{(Inkscape) Color is used for the text in Inkscape, but the package 'color.sty' is not loaded}%
    \renewcommand\color[2][]{}%
  }%
  \providecommand\transparent[1]{%
    \errmessage{(Inkscape) Transparency is used (non-zero) for the text in Inkscape, but the package 'transparent.sty' is not loaded}%
    \renewcommand\transparent[1]{}%
  }%
  \providecommand\rotatebox[2]{#2}%
  \newcommand*\fsize{\dimexpr\f@size pt\relax}%
  \newcommand*\lineheight[1]{\fontsize{\fsize}{#1\fsize}\selectfont}%
  \ifx\svgwidth\undefined%
    \setlength{\unitlength}{48.94242412bp}%
    \ifx\svgscale\undefined%
      \relax%
    \else%
      \setlength{\unitlength}{\unitlength * \real{\svgscale}}%
    \fi%
  \else%
    \setlength{\unitlength}{\svgwidth}%
  \fi%
  \global\let\svgwidth\undefined%
  \global\let\svgscale\undefined%
  \makeatother%
  \begin{picture}(1,1.55323354)%
    \lineheight{1}%
    \setlength\tabcolsep{0pt}%
    \put(0,0){\includegraphics[width=\unitlength,page=1]{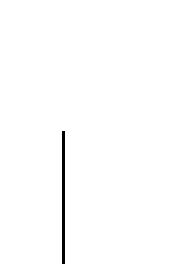}}%
    \put(0.03858,0.70208261){\color[rgb]{0,0,0}\makebox(0,0)[lt]{\lineheight{1.25}\smash{\begin{tabular}[t]{l}$a$\end{tabular}}}}%
    \put(0.85262499,1.4911389){\color[rgb]{0,0,0}\makebox(0,0)[lt]{\lineheight{1.25}\smash{\begin{tabular}[t]{l}$a$\end{tabular}}}}%
    \put(0.42476547,0.05213072){\color[rgb]{0,0,0}\makebox(0,0)[lt]{\lineheight{1.25}\smash{\begin{tabular}[t]{l}$a$\end{tabular}}}}%
    \put(0,0){\includegraphics[width=\unitlength,page=2]{C-algebra-7.pdf}}%
  \end{picture}%
\endgroup%
}}}~~=~\vcenter{\hbox{{\def\svgscale{0.6}
%% Accompanies image file 'C-algebra-8.pdf' (pdf, eps, ps)
\begingroup%
  \makeatletter%
  \providecommand\color[2][]{%
    \errmessage{(Inkscape) Color is used for the text in Inkscape, but the package 'color.sty' is not loaded}%
    \renewcommand\color[2][]{}%
  }%
  \providecommand\transparent[1]{%
    \errmessage{(Inkscape) Transparency is used (non-zero) for the text in Inkscape, but the package 'transparent.sty' is not loaded}%
    \renewcommand\transparent[1]{}%
  }%
  \providecommand\rotatebox[2]{#2}%
  \newcommand*\fsize{\dimexpr\f@size pt\relax}%
  \newcommand*\lineheight[1]{\fontsize{\fsize}{#1\fsize}\selectfont}%
  \ifx\svgwidth\undefined%
    \setlength{\unitlength}{12.01269797bp}%
    \ifx\svgscale\undefined%
      \relax%
    \else%
      \setlength{\unitlength}{\unitlength * \real{\svgscale}}%
    \fi%
  \else%
    \setlength{\unitlength}{\svgwidth}%
  \fi%
  \global\let\svgwidth\undefined%
  \global\let\svgscale\undefined%
  \makeatother%
  \begin{picture}(1,5.77862463)%
    \lineheight{1}%
    \setlength\tabcolsep{0pt}%
    \put(0,0){\includegraphics[width=\unitlength,page=1]{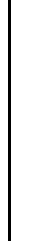}}%
    \put(0.39956116,4.84563392){\color[rgb]{0,0,0}\makebox(0,0)[lt]{\lineheight{1.25}\smash{\begin{tabular}[t]{l}$a$\end{tabular}}}}%
    \put(0,0){\includegraphics[width=\unitlength,page=2]{C-algebra-8.pdf}}%
  \end{picture}%
\endgroup%
}}}~~.
\end{gather*}
Indeed, associativity and commutativity are equivalent to the Jacobi identity for $\mc Y_\mu$. The unit property follows from the fact that $\mc Y_\mu$ restricts to $Y_a$. See \cite{HKL15} for more details. Note that we also have
\begin{gather}
\mu\circ\ss^n=\mu,\qquad\mu(\id_a\otimes\iota)=\id_a
\end{gather}
for any $n\in\mathbb Z$. The first equation follows from induction and that $\mu=\mu\ss\ss^{-1}=\mu\ss^{-1}$. The second equation holds because $\mu(\id_a\otimes\iota)=\mu\ss_{a,a}(\id_a\otimes\iota)=\mu(\iota\otimes\id_a)\ss_{a,0}=\id_a\ss_{a,0}=\id_a$.

Recall that the twist is defined by $e^{2\im\pi L_0}$. Since $L_0$ has only integral eigenvalues on $W_a$, the unitary $V$-module $W_a$ has trivial twist: $\vartheta_a=\id_a$. We also notice that $A$ is \textbf{normalized} (i.e., $\iota^*\iota=\id_0$) since the inner product on $U$ restricts to that of $V$. We thus conclude: \emph{If $U$ is a preunitary extension of $V$ then $A_U$ is a normalized commutative $\Rep^\uni(V)$-algebra with trivial twist. Conversely, any such commutative $\Rep^\uni(V)$-algebra arises from a preunitary CFT-type extension.} Moreover, if $U$ is of CFT type, then $A_U$ is \textbf{haploid}, which means that $\dim\Hom(W_0,W_a)=1$. Indeed, if $W_i$ is a unitary $V$-submodule of $W_a$ equivalent to $W_0$, then the lowest conformal weight of $W_i$ is $0$, which implies that $\Omega\in W_i$ and hence that $W_i=W_0$. Therefore $\dim\Hom(W_0,W_a)=\dim\Hom(W_0,W_0)=1$ by the simpleness of $V$.
Conversely, if $A_U$ is haploid, then $U$ is of CFT type provided that any irreducible $V$-module not equivalent to $W_0$ has no homogeneous vectors with conformal weight $0$. This converse statement will not be used in this paper (except in corollary \ref{lb58}). We thus content ourselves with the following result.

\begin{pp}[cf. \cite{HKL15} theorem 3.2]
If $U$ is a preunitary CFT-type VOA extension of $V$, then $A_U$ is a normalized haploid commutative $\RepV$-algebra with trivial twist.
\end{pp}

A detailed discussion of unitary VOA extensions will be given in the following sections. For now we first give the definition:

\begin{df}
Let $U$ be a CFT-type VOA extension of $V$, and assume that the vector space $U$ is equipped with a normalized inner product $\bk{\cdot|\cdot}$. We say that the extension $V\subset U$ is unitary (equivalently, that $U$ is a \textbf{unitary (VOA) extension} of $V$), if $\bk{\cdot|\cdot}$ restricts to the normalized inner product of $V$, and $(U,\bk{\cdot|\cdot})$ is a unitary VOA.
\end{df}

A unitary VOA extension is clearly preunitary. Another useful fact is the following:
\begin{pp}
If $U$ is a CFT-type unitary VOA extension of $V$, then the PCT operator $\Theta_U$ of $U$ restricts to the one $\Theta_V$ of $V$. In particular, $V$ is a $\Theta_U$-invariant subspace of $U$. 
\end{pp}
Thus we can let $\Theta$ denote unambiguously both the PCT operators of $U$ and of $V$.
\begin{proof}
By relation \eqref{eq5} and the fact that $\Theta_U^2=\id_U,\Theta_V^2=\id_V$, for any $v\in V\subset U$ we have
\begin{align*}
Y(\Theta_Uv,z)=Y(e^{\ovl zL_1}(-\ovl{z^{-2}} )^{L_0} v,\ovl{z^{-1}})^\dagger=Y(\Theta_Vv,z)
\end{align*}
when evaluating between vectors in $V$. It should be clear to the reader how the condition that the normalized inner product of $U$ restricts to that of $V$ is used in the above equations. Thus $\Theta_U|_V=\Theta_V$.
\end{proof}

\subsection{Duals and standard evaluations in $C^*$-tensor categories}\label{lb14}

\subsubsection*{Dualizable objects}

Let $\mc C$ be  a $C^*$-tensor category (cf. \cite{Yam04}) whose identity object $W_0$ is simple. We assume tacitly that $\mc C$ is closed under finite orthogonal direct sums and orthogonal subojects, which means that for a finite collection $\{W_s:s\in\mathfrak S\}$ of objects in $\mc C$ there exists an object $W_i$ and  partial isometries $\{u_s\in\Hom(W_i,W_s):s\in\mathfrak S \}$ satisfying  $u_tu_s^*=\delta_{s,t}\id_s$ ($\forall s,t\in\mathfrak S$) and $\sum_{s\in\mathfrak S}u_s^*u_s=\id_i$, and that for any object $W_i$ and a projection $p\in\End(W_i)$ there exists an object $W_j$ and a partial isometry $u\in\Hom(W_i,W_j)$ such that $uu^*=\id_j,u^*u=p$.\footnote{A morphism $u\in\Hom(W_i,W_j)$ is called a partial isometry if $u^*u$ and $uu^*$ are projections. A morphism $e\in\End(W_i)$ is called a projection if $e^2=e=e^*$.}

Assume that an object $W_i$ in $\mc C$ has a right dual $W_{\ovl i}$, which means that there exist evaluation $\ev_i\in\Hom(W_{\ovl i}\boxtimes W_i,W_0)$ and coevaluation $\coev_i\in\Hom(W_0, W_i\boxtimes W_{\ovl i})$ satisfying $(\id_i\otimes\ev_i)(\coev_i\otimes\id_i)=\id_i$ and $(\ev_i\otimes\id_{\ovl i})(\id_{\ovl i}\otimes\coev_i)=\id_{\ovl i}$. Set $\ev_{\ovl i,i}=\ev_i,\coev_{i,\ovl i}=\coev_i$, and  set also $\ev_{i,\ovl i}=(\coev_{i,\ovl i})^*,\coev_{\ovl i,i}=(\ev_{\ovl i,i})^*$. \index{evii@$\ev_{i,\ovl i}$, $\mc Y_{\ev_{i,\ovl i}}$} Then equations \eqref{eq17} and \eqref{eq18} are satisfied, which shows that $W_{\ovl i}$ is also a left dual of $W_i$. In this case we say that $W_i$ is \textbf{dualizable}. Note that $\ev_{i,\ovl i}$ determines the remaining three $\ev$ and $\coev$. In general, we say that $\ev_{i,\ovl i}\in\Hom(W_i\boxtimes W_{\ovl i},W_0),\ev_{\ovl i,i}\in\Hom(W_{\ovl i}\boxtimes W_i,W_0)$ are \textbf{evaluations} (or simply \textbf{ev}) of $W_i$ and $W_{\ovl i}$ if equations \eqref{eq17} and \eqref{eq18} are satisfied when setting  $\coev_{i,\ovl i}=\ev_{i,\ovl i}^*,\coev_{\ovl i,i}=\ev_{\ovl i,i}^*$.  In the case that $W_i$ is self-dual, we say that $\ev_{i,i}$ is an  \textbf{evaluation} (\textbf{ev}) of $W_i$, if, by setting $\coev_{i,i}=\ev_{i,i}^*$, we have $(\ev_{i,i}\otimes\id_i)(\id_i\otimes\coev_{i,i})=\id_i$. Taking adjoint, we also have $(\id_i\otimes\ev_{i,i})(\coev_{i,i}\otimes\id_i)=\id_i$.

Assume that $W_i,W_j$ are dualizable with duals $W_{\ovl i},W_{\ovl j}$ respectively. Choose evaluations $\ev_{i,\ovl i},\ev_{\ovl i,i},\ev_{j,\ovl j},\ev_{\ovl j,j}$. Then $W_{i\boxtimes j}:=W_i\boxtimes W_j$ is also dualizable with a dual $W_{\ovl j\boxtimes\ovl i}:=W_{\ovl j}\boxtimes W_{\ovl i}$ and evaluations
\begin{gather}
\ev_{i\boxtimes j,\ovl j\boxtimes\ovl i}=\ev_{i,\ovl i}(\id_i\otimes\ev_{j,\ovl j}\otimes\id_{\ovl i}),\qquad \ev_{\ovl j\boxtimes\ovl i,i\boxtimes j}=\ev_{\ovl j, j}(\id_{\ovl j}\otimes\ev_{\ovl i,i}\otimes\id_j).\label{eq22}
\end{gather}
\begin{cv}\label{lb35}
Unless otherwise stated, if the $\ev$ for $W_i,W_{\ovl i}$ and $W_j,W_{\ovl j}$ are chosen, then we always define the $\ev$ for $W_{i\boxtimes j},W_{\ovl j\boxtimes \ovl i}$ using equations \eqref{eq22}.
\end{cv}

Using $\ev$ and $\coev$ for $W_i,W_j$ and their duals $W_{\ovl i},W_{\ovl j}$, one can define for any $F\in\Hom(W_i,W_j)$ a pair of transposes $F^\vee,{^\vee}F$ by
\begin{gather*}
F^\vee=(\ev_{\ovl j,j}\otimes\id_{\ovl i})(\id_{\ovl j}\otimes F\otimes \id_{\ovl i})(\id_{\ovl j}\otimes\coev_{i,\ovl i}),\\
{}^\vee F=(\id_{\ovl i}\otimes\ev_{j,\ovl j})(\id_{\ovl i}\otimes F\otimes \id_{\ovl j})(\coev_{\ovl i,i}\otimes\id_{\ovl j}).
\end{gather*}
Pictorially,
\index{Fvee@$F^\vee,{}^\vee F$} 
\begin{align}
F^\vee=~\vcenter{\hbox{{\def\svgscale{0.8}
%% Accompanies image file 'transpose-2.pdf' (pdf, eps, ps)
\begingroup%
  \makeatletter%
  \providecommand\color[2][]{%
    \errmessage{(Inkscape) Color is used for the text in Inkscape, but the package 'color.sty' is not loaded}%
    \renewcommand\color[2][]{}%
  }%
  \providecommand\transparent[1]{%
    \errmessage{(Inkscape) Transparency is used (non-zero) for the text in Inkscape, but the package 'transparent.sty' is not loaded}%
    \renewcommand\transparent[1]{}%
  }%
  \providecommand\rotatebox[2]{#2}%
  \newcommand*\fsize{\dimexpr\f@size pt\relax}%
  \newcommand*\lineheight[1]{\fontsize{\fsize}{#1\fsize}\selectfont}%
  \ifx\svgwidth\undefined%
    \setlength{\unitlength}{65.06220872bp}%
    \ifx\svgscale\undefined%
      \relax%
    \else%
      \setlength{\unitlength}{\unitlength * \real{\svgscale}}%
    \fi%
  \else%
    \setlength{\unitlength}{\svgwidth}%
  \fi%
  \global\let\svgwidth\undefined%
  \global\let\svgscale\undefined%
  \makeatother%
  \begin{picture}(1,1.2003256)%
    \lineheight{1}%
    \setlength\tabcolsep{0pt}%
    \put(0,0){\includegraphics[width=\unitlength,page=1]{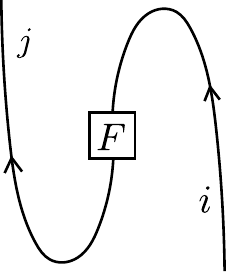}}%
  \end{picture}%
\endgroup%
}}}~,\qquad {}^\vee F=\vcenter{\hbox{{\def\svgscale{0.8}
%% Accompanies image file 'transpose-3.pdf' (pdf, eps, ps)
\begingroup%
  \makeatletter%
  \providecommand\color[2][]{%
    \errmessage{(Inkscape) Color is used for the text in Inkscape, but the package 'color.sty' is not loaded}%
    \renewcommand\color[2][]{}%
  }%
  \providecommand\transparent[1]{%
    \errmessage{(Inkscape) Transparency is used (non-zero) for the text in Inkscape, but the package 'transparent.sty' is not loaded}%
    \renewcommand\transparent[1]{}%
  }%
  \providecommand\rotatebox[2]{#2}%
  \newcommand*\fsize{\dimexpr\f@size pt\relax}%
  \newcommand*\lineheight[1]{\fontsize{\fsize}{#1\fsize}\selectfont}%
  \ifx\svgwidth\undefined%
    \setlength{\unitlength}{65.06220872bp}%
    \ifx\svgscale\undefined%
      \relax%
    \else%
      \setlength{\unitlength}{\unitlength * \real{\svgscale}}%
    \fi%
  \else%
    \setlength{\unitlength}{\svgwidth}%
  \fi%
  \global\let\svgwidth\undefined%
  \global\let\svgscale\undefined%
  \makeatother%
  \begin{picture}(1,1.2003256)%
    \lineheight{1}%
    \setlength\tabcolsep{0pt}%
    \put(0,0){\includegraphics[width=\unitlength,page=1]{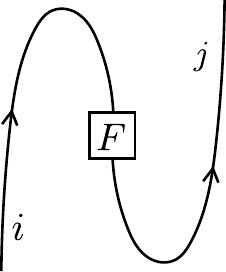}}%
  \end{picture}%
\endgroup%
}}}~.\label{eq25}
\end{align}
One easily checks that $^\vee(F^\vee)=F=({^\vee F})^\vee$, $(FG)^\vee=G^\vee F^\vee,{^\vee}(FG)=(^\vee G)(^\vee F),(F^\vee)^*={}^\vee(F^*)$.

\subsubsection*{Standard evaluations}

The evaluations defined above are not unique even up to unitaries. For any $\ev_{i,\ovl i},\ev_{\ovl i,i}$ of $W_i,W_{\ovl i}$, and any invertible $K\in\End(W_i)$, $\wtd\ev_{i,\ovl i}:=\ev_{i,\ovl i}(K\otimes\id_{\ovl i})$ and $\wtd\ev_{\ovl i,i}:=\ev_{\ovl i,i}(\id_{\ovl i}\otimes (K^*)^{-1})$ are also evaluations. Thus one can normalize evaluations to satisfy certain nice conditions. In $C^*$-tensor categories, the $\ev$ that attract most interest are the so called standard evaluations. It is known that for dualizable objects, standard $\ev$ always exist and are unique up to unitaries, and that the two transposes defined by standard $\ev$ are equal. We refer the reader to \cite{LR97,Yam04,BDH14} for these results. In the following, we review an explicit method of constructing standard evaluations following \cite{Yam04}. Since the evaluations for VOA tensor categories defined in the previous chapter can be realized by this construction, these evaluations are standard (see proposition \ref{lb60}).

%\footnote{We warn the reader that the result of \cite{Yam04} is slightly weaker than the other two since \cite{Yam04} assumes the rigidity of $\mc C$ at the beginning. Therefore, to construct standard $\ev$ and show its tracialness for a dualizable $W_i$, it is assumed without proof that all subobjects of $W_i$ are dualizable. In our treatment, we complement \cite{Yam04} by first proving the dualizability of subobjects by constructing a pair of tracial evaluations for $W_i$ using an idea from \cite{BDH14} theorem 4.12. Then, using \emph{tracial} evaluations of $W_i$, one can easily construct $\ev$ for subobjects. Note that \cite{LR97} theorem 2.4 constructs $\ev$ for subobjects without using tracialness, but it is assumed there that all idempotents (but not just projections) correspond to subobjects.}

Define scalars $\Tr_L(F)$ and $\Tr_R(F)$ for each $F\in\End(W_i)$ \index{Tr@$\Tr_L,\Tr_R$} such that $\ev_{i,\ovl i}(F\otimes\id_{\ovl i})\coev_{i,\ovl i}=\Tr_L(F)\id_0$ and $\ev_{\ovl i,i}(\id_{\ovl i}\otimes F)\coev_{\ovl i,i}=\Tr_R(F)\id_0$. Then $\Tr_L$ is a positive linear functional on $\End(W_i)$. Moreover, $\Tr_L$ is faithful: if $\Tr_L(F^*F)=0$, then $\ev_{i,\ovl i}\circ(F\otimes\id_{\ovl i})$ is zero (since its absolute value is zero). So $F=0$. Similar things can be said about $\Tr_R$. We say that $\ev_{i,\ovl i}$ is a \textbf{standard evaluation} if $\Tr_L(F)=\Tr_R(F)$  for all $F\in\End(W_i)$. Since $\Tr_L(F^\vee)=\Tr_R(F)=\Tr_L(F^\vee)$ by easy graphical calculus, it is easy to see that $\ev_{i,\ovl i}$ is standard if and only if $\ev_{\ovl i,i}$ is so. Standard $\ev$ are unique up to  unitaries: If $W_{i'}\simeq W_{\ovl i},T\in\Hom(W_{i'},W_{\ovl i})$ is unitary,  and $\wtd\ev_{i,i'},\wtd\ev_{i',i}$  are also standard, then we may find a unitary $K\in\End(W_i)$ satisfying $\wtd\ev_{i,i'}=\ev_{i,\ovl i}(K\otimes T),\wtd\ev_{i',i}=\ev_{\ovl i,i}(T\otimes K)$. (See \cite{Yam04} lemma 3.9-(iii).)

If $W_i$ is simple, a standard evaluation is easy to construct by multiplying $\ev_{i,\ovl i}$ by some nonzero constant $\lambda$ (and hence multiplying $\ev_{\ovl i,i}$ by $\ovl{\lambda^{-1}}$) such that $\Tr_L(\id_i)=\Tr_R(\id_i)$. In general, if $W_i$ is dualizable and hence semisimple, we have orthogonal irreducible decomposition $W_i\simeq \bigoplus^\perp_{s\in\mathfrak S} W_s$ where each irreducible subobject $W_s$ is dualizable (with a dual $W_{\ovl s}$). Choose  partial isometries  $\{u_s\in\Hom(W_i,W_s):s\in\mathfrak S \}$ and $\{v_{\ovl s}\in\Hom(W_{\ovl i},W_{\ovl s}):s\in\mathfrak S\}$ satisfying $u_t u_s^*=\delta_{s,t}\id_s,v_{\ovl t}v_{\ovl s}^*=\delta_{s,t}\id_{\ovl s}$ and $\sum_s u_s^*u_s=\id_i,\sum_s v_{\ovl s}^*v_{\ovl s}=\id_{\ovl i}$. Then we define
\begin{align}
\ev_{i,\ovl i}=\sum_{s\in\mathfrak S}\ev_{s,\ovl s}(u_s\otimes v_{\ovl s}),\qquad\ev_{\ovl i,i}=\sum_{s\in\mathfrak S}\ev_{\ovl s,s}(v_{\ovl s}\otimes u_s)\label{eq19}
\end{align}
(where $\ev_{s,\ovl s}$ and $\ev_{\ovl s,s}$ are standard for all $s$), define $\coev$ using adjoint. Then $\ev_{i,\ovl i}$ and $\ev_{\ovl i,i}$ are standard. (See \cite{Yam04} lemma 3.9 for details.)\footnote{In \cite{Yam04} the categories are assumed to be rigid. Thus any orthogonal subobject of $W_i$, i.e., any object $W_s$ which is a associated with a partial isometry $u:W_i\rightarrow W_s$ satisfying $uu^*=\id_s$, is dualizable. This fact is also true without assuming $\mc C$ to be rigid. (Cf. \cite{ABD04} lemma 4.20.) Here is one way to see this. Notice that we may assume $\Tr_L$ is tracial by multiplying $\ev_{i,\ovl i}$ by $K\otimes \id_{\ovl i}$ where $K$ is a positive invertible element of $\End(W_i)$. (Cf. the proof of \cite{ABD04} Thm. 4.12.) Thus, for any $F,G\in\End(W_i)$, we have $\Tr_L(GF)=\Tr_L(F^{\vee\vee}\cdot G)$ in general and $\Tr_L(FG)=\Tr_L(GF)$ by tracialness, which shows $F^{\vee\vee}=F$ and hence $F^\vee={}^\vee F$. Thus, $(F^\vee)^*=(F^*)^\vee$. So $F$ is a projection  iff $F^\vee$ is so. Let $p=u^*u$. Then $p^\vee\in\End(W_{\ovl i})$ is a projection. Thus, there exist an object $W_{\ovl s}$ and $v\in\Hom(W_{\ovl i},W_{\ovl s})$ satisfying $vv^*=\id_{\ovl s},v^*v=p^\vee$. Then $W_{\ovl s}$ is dual to $W_s$ since one can choose evaluations $\ev_{s,\ovl s}:=\ev_{i,\ovl i}(u^*\otimes v^*), \ev_{\ovl s,s}:=\ev_{\ovl i,i}(v^*\otimes u^*)$.}

%Beginning with this section, we always assume $U$ to be a preunitary CFT type VOA extension of $V$. We will introduce a notion of unitarity for commutative associative algebras in rigid braided $C^*$-tensor category, and then relate it to the unitarity of $U$. Since we will first work with abstract tensor categories, we need to add some constraints on the choice of evaluation and coevaluation maps. 

%The unitarity of $A_U$ can be studied without appeal to its VOA. So we want to work with general . But first we have to choose suitable evaluations and coevaluations. 

\begin{pp}\label{lb60}
If $V$ is a regular, CFT-type, and completely unitary VOA, then the $\ev$ and $\coev$ defined in chapter \ref{lb7} (same as those in \cite{Gui19a,Gui19b}) for any object $W_i$ in $\RepV$ and its contragredient module $W_{\ovl i}$ are standard.
\end{pp}

\begin{proof}
First of all, assume $W_i$ is irreducible. By \cite{Gui19b} proposition 7.7 and the paragraph before that, we have $\Tr_L(\id_i)=d_i=d_{\ovl i}=\Tr_R(\id_i)$. So $\ev_{i,\ovl i}$ is standard. Now, assume $W_i$ is semisimple with finite orthogonal irreducible decomposition $W_i=\bigoplus_s^\perp W_s$. For each irreducible summand $W_s$, define $u_s$ (resp. $v_{\ovl s}$) to be the projection of $W_i$ onto $W_s$ (resp. $W_{\ovl i}$ onto $W_{\ovl s}$). (Note that $W_{\ovl i}$ and $W_{\ovl s}$ are respectively the contragredient modules of $W_i,W_s$.) Using \eqref{eq73}, it is easy to see that $\mc Y_{\ev_{i,\ovl i}}(w,z)=\sum_s\mc Y_{\ev_{s,\ovl s}}(u_s w,z)v_{\ovl s}$ for any $w\in W_i$. So the first equation of \eqref{eq19} is satisfied. Similarly, the second one of \eqref{eq19} holds.
\end{proof}

\begin{cv}\label{lb11}
Unless otherwise stated, for any object $W_i$ in  $\RepV$, $W_{\ovl i}$ is always understood as the contragredient module of $W_i$, and the standard $\ev$ and $\coev$ for $W_i,W_{\ovl i}$ are always defined as in chapter \ref{lb7}.
\end{cv}

It is  worth noting that standardness is preserved by tensor products: If standard $\ev$ are chosen for $W_i,W_j$ and their dual $W_{\ovl i},W_{\ovl j}$, then  the $\ev$ of $W_{i\boxtimes j},W_{\ovl j\boxtimes\ovl i}$  defined by \eqref{eq22} are also standard.

Standard $\ev$ are also characterized by minimizing quantum dimensions.  Define constants $d_i,d_{\ovl i}$ satisfying $\ev_{i,\ovl i}\coev_{i,\ovl i}=d_i\id_0,\ev_{\ovl i,i}\coev_{\ovl i,i}=d_{\ovl i}\id_0$. Then standard $\ev$ are precisely those minimizing $d_id_{\ovl i}$ and satisfying $d_i=d_{\ovl i}$ (cf. \cite{LR97}). We will always assume $d_i,d_{\ovl i}$ \index{di@$d_i$} to be those defined by standard evaluations, and call them the \textbf{quantum dimensions} of $W_i,W_{\ovl i}$.

\subsection{Unitarity of $\mc C$-algebras and VOA extensions}

Let $W_a$ be an object in $\mc C$, choose $\mu\in\Hom(W_a\boxtimes W_a,W_a),\iota\in\Hom(W_0,W_a)$, and asume that  $A=(W_a,\mu,\iota)$ is  an \textbf{associative algebra in $\mc C$}\index{A@$A=(W_a,\mu,\iota)$} (also called \textbf{$\mc C$-algebra}\footnote{In \cite{KO02,HKL15,CKM17}, commutativity is required in the definition of $\mc C$-algebras when $\mc C$ is braided. This is not assumed in our paper.}), which means:
\begin{itemize}
	\item (Associativity) $\mu(\mu\otimes\id_a)=\mu(\id_a\otimes\mu)$.
	\item (Unit) $\mu(\iota\otimes\id_a)=\id_a=\mu(\id_a\otimes\iota)$.
\end{itemize}
Since $W_0$ is simple, we can choose $D_A>0$ (called the \textbf{quantum dimension} of $A$)   satisfying $\iota^*\mu\mu^*\iota=D_A\id_0$. \index{DA@$D_A$} We say that $A$ is
\begin{itemize}
	\item \textbf{haploid} if $\dim\Hom(W_0,W_a)=1$;
	\item \textbf{normalized} if $\iota^*\iota=\id_0$;
	\item \textbf{special} if $\mu\mu^*\in\mathbb C\id_a$; in this case we set scalar $d_A>0$ such that $\mu\mu^*=d_A\id_a$;\index{dA@$d_A$}
	\item \textbf{standard}  if $A$ is special, $W_a$ is dualizable (with quantum dimension $d_a$),   and $D_A=d_a$.
\end{itemize}

Note that any $\mc C$-algebra $A$ is clearly equivalent to a normalized one.

Assume that $W_a$ has a dual $W_{\ovl a}$, together with (not necessarily standard) $\ev_{a,\ovl a},\ev_{\ovl a,a}$. Define  $\ev$  for $W_a\boxtimes W_a$ and $W_{\ovl a}\boxtimes W_{\ovl a}$ using \eqref{eq22}. Assume also that $W_a$ is self-dual, i.e., $W_a\simeq W_{\ovl a}$. Choose a unitary morphism $\rfl=\vcenter{\hbox{{\def\svgscale{0.4}
%% Accompanies image file 'reflection.pdf' (pdf, eps, ps)
\begingroup%
  \makeatletter%
  \providecommand\color[2][]{%
    \errmessage{(Inkscape) Color is used for the text in Inkscape, but the package 'color.sty' is not loaded}%
    \renewcommand\color[2][]{}%
  }%
  \providecommand\transparent[1]{%
    \errmessage{(Inkscape) Transparency is used (non-zero) for the text in Inkscape, but the package 'transparent.sty' is not loaded}%
    \renewcommand\transparent[1]{}%
  }%
  \providecommand\rotatebox[2]{#2}%
  \newcommand*\fsize{\dimexpr\f@size pt\relax}%
  \newcommand*\lineheight[1]{\fontsize{\fsize}{#1\fsize}\selectfont}%
  \ifx\svgwidth\undefined%
    \setlength{\unitlength}{26.95796801bp}%
    \ifx\svgscale\undefined%
      \relax%
    \else%
      \setlength{\unitlength}{\unitlength * \real{\svgscale}}%
    \fi%
  \else%
    \setlength{\unitlength}{\svgwidth}%
  \fi%
  \global\let\svgwidth\undefined%
  \global\let\svgscale\undefined%
  \makeatother%
  \begin{picture}(1,2.35804907)%
    \lineheight{1}%
    \setlength\tabcolsep{0pt}%
    \put(0,0){\includegraphics[width=\unitlength,page=1]{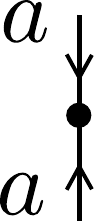}}%
  \end{picture}%
\endgroup%
}}}~~\in\Hom(W_a, W_{\ovl a})$, and write its adjoint as $\rfl^*=\vcenter{\hbox{{\def\svgscale{0.4}
%% Accompanies image file 'reflection-2.pdf' (pdf, eps, ps)
\begingroup%
  \makeatletter%
  \providecommand\color[2][]{%
    \errmessage{(Inkscape) Color is used for the text in Inkscape, but the package 'color.sty' is not loaded}%
    \renewcommand\color[2][]{}%
  }%
  \providecommand\transparent[1]{%
    \errmessage{(Inkscape) Transparency is used (non-zero) for the text in Inkscape, but the package 'transparent.sty' is not loaded}%
    \renewcommand\transparent[1]{}%
  }%
  \providecommand\rotatebox[2]{#2}%
  \newcommand*\fsize{\dimexpr\f@size pt\relax}%
  \newcommand*\lineheight[1]{\fontsize{\fsize}{#1\fsize}\selectfont}%
  \ifx\svgwidth\undefined%
    \setlength{\unitlength}{26.95796801bp}%
    \ifx\svgscale\undefined%
      \relax%
    \else%
      \setlength{\unitlength}{\unitlength * \real{\svgscale}}%
    \fi%
  \else%
    \setlength{\unitlength}{\svgwidth}%
  \fi%
  \global\let\svgwidth\undefined%
  \global\let\svgscale\undefined%
  \makeatother%
  \begin{picture}(1,2.35804907)%
    \lineheight{1}%
    \setlength\tabcolsep{0pt}%
    \put(0,0){\includegraphics[width=\unitlength,page=1]{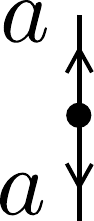}}%
  \end{picture}%
\endgroup%
}}}~~$. \index{zz@$\varepsilon=\vcenter{\hbox{{\def\svgscale{0.4}
				}}}~~,\varepsilon^*=\vcenter{\hbox{{\def\svgscale{0.4}
				}}}~~$}
Write also $\mu^*=\vcenter{\hbox{{\def\svgscale{0.4}
%% Accompanies image file 'C-algebra-9.pdf' (pdf, eps, ps)
\begingroup%
  \makeatletter%
  \providecommand\color[2][]{%
    \errmessage{(Inkscape) Color is used for the text in Inkscape, but the package 'color.sty' is not loaded}%
    \renewcommand\color[2][]{}%
  }%
  \providecommand\transparent[1]{%
    \errmessage{(Inkscape) Transparency is used (non-zero) for the text in Inkscape, but the package 'transparent.sty' is not loaded}%
    \renewcommand\transparent[1]{}%
  }%
  \providecommand\rotatebox[2]{#2}%
  \newcommand*\fsize{\dimexpr\f@size pt\relax}%
  \newcommand*\lineheight[1]{\fontsize{\fsize}{#1\fsize}\selectfont}%
  \ifx\svgwidth\undefined%
    \setlength{\unitlength}{69.38030774bp}%
    \ifx\svgscale\undefined%
      \relax%
    \else%
      \setlength{\unitlength}{\unitlength * \real{\svgscale}}%
    \fi%
  \else%
    \setlength{\unitlength}{\svgwidth}%
  \fi%
  \global\let\svgwidth\undefined%
  \global\let\svgscale\undefined%
  \makeatother%
  \begin{picture}(1,1.01846161)%
    \lineheight{1}%
    \setlength\tabcolsep{0pt}%
    \put(0,0){\includegraphics[width=\unitlength,page=1]{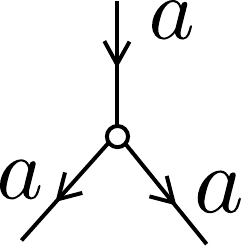}}%
  \end{picture}%
\endgroup%
}}}~~$, \index{Uext@$U$!$\mu^*=\vcenter{\hbox{{\def\svgscale{0.4}
				}}}~~$} $\iota^*=\vcenter{\hbox{{\def\svgscale{0.4}
%% Accompanies image file 'C-algebra-10.pdf' (pdf, eps, ps)
\begingroup%
  \makeatletter%
  \providecommand\color[2][]{%
    \errmessage{(Inkscape) Color is used for the text in Inkscape, but the package 'color.sty' is not loaded}%
    \renewcommand\color[2][]{}%
  }%
  \providecommand\transparent[1]{%
    \errmessage{(Inkscape) Transparency is used (non-zero) for the text in Inkscape, but the package 'transparent.sty' is not loaded}%
    \renewcommand\transparent[1]{}%
  }%
  \providecommand\rotatebox[2]{#2}%
  \newcommand*\fsize{\dimexpr\f@size pt\relax}%
  \newcommand*\lineheight[1]{\fontsize{\fsize}{#1\fsize}\selectfont}%
  \ifx\svgwidth\undefined%
    \setlength{\unitlength}{25.75592132bp}%
    \ifx\svgscale\undefined%
      \relax%
    \else%
      \setlength{\unitlength}{\unitlength * \real{\svgscale}}%
    \fi%
  \else%
    \setlength{\unitlength}{\svgwidth}%
  \fi%
  \global\let\svgwidth\undefined%
  \global\let\svgscale\undefined%
  \makeatother%
  \begin{picture}(1,1.59351988)%
    \lineheight{1}%
    \setlength\tabcolsep{0pt}%
    \put(0,0){\includegraphics[width=\unitlength,page=1]{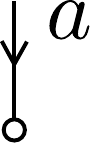}}%
  \end{picture}%
\endgroup%
}}}~~$. \index{Uext@$U$!$\iota^*=\vcenter{\hbox{{\def\svgscale{0.4}
				}}}~~$}
\begin{df}
A unitary $\rfl\in\Hom(W_a,W_{\ovl a})$ is called a \textbf{reflection operator}  of $A$ (with respect to the chosen $\ev$  of $W_a,W_{\ovl a}$), if the following two equations are satisfied:
\begin{gather}
\mu=(\ev_{\ovl a,a}\otimes\id_a)(\rfl\otimes\mu^*),\qquad
\rfl\mu(\rfl^*\otimes\rfl^*)=(\mu^*)^\vee.\label{eq30}
\end{gather}
Pictorially,
\begin{gather}
\vcenter{\hbox{{\def\svgscale{0.6}
%% Accompanies image file 'reflection-3.pdf' (pdf, eps, ps)
\begingroup%
  \makeatletter%
  \providecommand\color[2][]{%
    \errmessage{(Inkscape) Color is used for the text in Inkscape, but the package 'color.sty' is not loaded}%
    \renewcommand\color[2][]{}%
  }%
  \providecommand\transparent[1]{%
    \errmessage{(Inkscape) Transparency is used (non-zero) for the text in Inkscape, but the package 'transparent.sty' is not loaded}%
    \renewcommand\transparent[1]{}%
  }%
  \providecommand\rotatebox[2]{#2}%
  \newcommand*\fsize{\dimexpr\f@size pt\relax}%
  \newcommand*\lineheight[1]{\fontsize{\fsize}{#1\fsize}\selectfont}%
  \ifx\svgwidth\undefined%
    \setlength{\unitlength}{72.29037509bp}%
    \ifx\svgscale\undefined%
      \relax%
    \else%
      \setlength{\unitlength}{\unitlength * \real{\svgscale}}%
    \fi%
  \else%
    \setlength{\unitlength}{\svgwidth}%
  \fi%
  \global\let\svgwidth\undefined%
  \global\let\svgscale\undefined%
  \makeatother%
  \begin{picture}(1,1.27394707)%
    \lineheight{1}%
    \setlength\tabcolsep{0pt}%
    \put(0,0){\includegraphics[width=\unitlength,page=1]{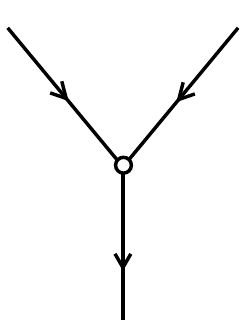}}%
    \put(-0.00459302,1.22181663){\color[rgb]{0,0,0}\makebox(0,0)[lt]{\lineheight{1.25}\smash{\begin{tabular}[t]{l}$a$\end{tabular}}}}%
    \put(0.90022336,1.23190741){\color[rgb]{0,0,0}\makebox(0,0)[lt]{\lineheight{1.25}\smash{\begin{tabular}[t]{l}$a$\end{tabular}}}}%
    \put(0.56084305,0.03546705){\color[rgb]{0,0,0}\makebox(0,0)[lt]{\lineheight{1.25}\smash{\begin{tabular}[t]{l}$a$\end{tabular}}}}%
  \end{picture}%
\endgroup%
}}}~~=~~\vcenter{\hbox{{\def\svgscale{0.6}
%% Accompanies image file 'reflection-4.pdf' (pdf, eps, ps)
\begingroup%
  \makeatletter%
  \providecommand\color[2][]{%
    \errmessage{(Inkscape) Color is used for the text in Inkscape, but the package 'color.sty' is not loaded}%
    \renewcommand\color[2][]{}%
  }%
  \providecommand\transparent[1]{%
    \errmessage{(Inkscape) Transparency is used (non-zero) for the text in Inkscape, but the package 'transparent.sty' is not loaded}%
    \renewcommand\transparent[1]{}%
  }%
  \providecommand\rotatebox[2]{#2}%
  \newcommand*\fsize{\dimexpr\f@size pt\relax}%
  \newcommand*\lineheight[1]{\fontsize{\fsize}{#1\fsize}\selectfont}%
  \ifx\svgwidth\undefined%
    \setlength{\unitlength}{86.58362042bp}%
    \ifx\svgscale\undefined%
      \relax%
    \else%
      \setlength{\unitlength}{\unitlength * \real{\svgscale}}%
    \fi%
  \else%
    \setlength{\unitlength}{\svgwidth}%
  \fi%
  \global\let\svgwidth\undefined%
  \global\let\svgscale\undefined%
  \makeatother%
  \begin{picture}(1,1.05666387)%
    \lineheight{1}%
    \setlength\tabcolsep{0pt}%
    \put(0,0){\includegraphics[width=\unitlength,page=1]{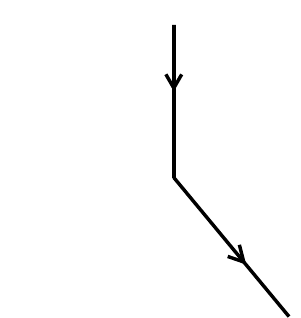}}%
    \put(-0.0038348,1.01488458){\color[rgb]{0,0,0}\makebox(0,0)[lt]{\lineheight{1.25}\smash{\begin{tabular}[t]{l}$a$\end{tabular}}}}%
    \put(0.53873812,1.02156413){\color[rgb]{0,0,0}\makebox(0,0)[lt]{\lineheight{1.25}\smash{\begin{tabular}[t]{l}$a$\end{tabular}}}}%
    \put(0.91669452,0.09417317){\color[rgb]{0,0,0}\makebox(0,0)[lt]{\lineheight{1.25}\smash{\begin{tabular}[t]{l}$a$\end{tabular}}}}%
    \put(0,0){\includegraphics[width=\unitlength,page=2]{reflection-4.pdf}}%
  \end{picture}%
\endgroup%
}}}~~,\label{eq23}\\[1.5ex]
\vcenter{\hbox{{\def\svgscale{0.6}
%% Accompanies image file 'reflection-5.pdf' (pdf, eps, ps)
\begingroup%
  \makeatletter%
  \providecommand\color[2][]{%
    \errmessage{(Inkscape) Color is used for the text in Inkscape, but the package 'color.sty' is not loaded}%
    \renewcommand\color[2][]{}%
  }%
  \providecommand\transparent[1]{%
    \errmessage{(Inkscape) Transparency is used (non-zero) for the text in Inkscape, but the package 'transparent.sty' is not loaded}%
    \renewcommand\transparent[1]{}%
  }%
  \providecommand\rotatebox[2]{#2}%
  \newcommand*\fsize{\dimexpr\f@size pt\relax}%
  \newcommand*\lineheight[1]{\fontsize{\fsize}{#1\fsize}\selectfont}%
  \ifx\svgwidth\undefined%
    \setlength{\unitlength}{67.12987308bp}%
    \ifx\svgscale\undefined%
      \relax%
    \else%
      \setlength{\unitlength}{\unitlength * \real{\svgscale}}%
    \fi%
  \else%
    \setlength{\unitlength}{\svgwidth}%
  \fi%
  \global\let\svgwidth\undefined%
  \global\let\svgscale\undefined%
  \makeatother%
  \begin{picture}(1,1.28594584)%
    \lineheight{1}%
    \setlength\tabcolsep{0pt}%
    \put(0,0){\includegraphics[width=\unitlength,page=1]{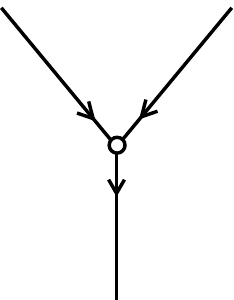}}%
    \put(0.05682584,1.23458174){\color[rgb]{0,0,0}\makebox(0,0)[lt]{\lineheight{1.25}\smash{\begin{tabular}[t]{l}$a$\end{tabular}}}}%
    \put(0.78453577,1.24067445){\color[rgb]{0,0,0}\makebox(0,0)[lt]{\lineheight{1.25}\smash{\begin{tabular}[t]{l}$a$\end{tabular}}}}%
    \put(0.57661207,0.03819352){\color[rgb]{0,0,0}\makebox(0,0)[lt]{\lineheight{1.25}\smash{\begin{tabular}[t]{l}$a$\end{tabular}}}}%
    \put(0,0){\includegraphics[width=\unitlength,page=2]{reflection-5.pdf}}%
  \end{picture}%
\endgroup%
}}}~~=~~\vcenter{\hbox{{\def\svgscale{0.6}
%% Accompanies image file 'reflection-6.pdf' (pdf, eps, ps)
\begingroup%
  \makeatletter%
  \providecommand\color[2][]{%
    \errmessage{(Inkscape) Color is used for the text in Inkscape, but the package 'color.sty' is not loaded}%
    \renewcommand\color[2][]{}%
  }%
  \providecommand\transparent[1]{%
    \errmessage{(Inkscape) Transparency is used (non-zero) for the text in Inkscape, but the package 'transparent.sty' is not loaded}%
    \renewcommand\transparent[1]{}%
  }%
  \providecommand\rotatebox[2]{#2}%
  \newcommand*\fsize{\dimexpr\f@size pt\relax}%
  \newcommand*\lineheight[1]{\fontsize{\fsize}{#1\fsize}\selectfont}%
  \ifx\svgwidth\undefined%
    \setlength{\unitlength}{93.0799952bp}%
    \ifx\svgscale\undefined%
      \relax%
    \else%
      \setlength{\unitlength}{\unitlength * \real{\svgscale}}%
    \fi%
  \else%
    \setlength{\unitlength}{\svgwidth}%
  \fi%
  \global\let\svgwidth\undefined%
  \global\let\svgscale\undefined%
  \makeatother%
  \begin{picture}(1,0.95649317)%
    \lineheight{1}%
    \setlength\tabcolsep{0pt}%
    \put(-0.00356716,0.92384316){\color[rgb]{0,0,0}\makebox(0,0)[lt]{\lineheight{1.25}\smash{\begin{tabular}[t]{l}$a$\end{tabular}}}}%
    \put(0.28526483,0.92356435){\color[rgb]{0,0,0}\makebox(0,0)[lt]{\lineheight{1.25}\smash{\begin{tabular}[t]{l}$a$\end{tabular}}}}%
    \put(0.92250869,0.10959213){\color[rgb]{0,0,0}\makebox(0,0)[lt]{\lineheight{1.25}\smash{\begin{tabular}[t]{l}$a$\end{tabular}}}}%
    \put(0,0){\includegraphics[width=\unitlength,page=1]{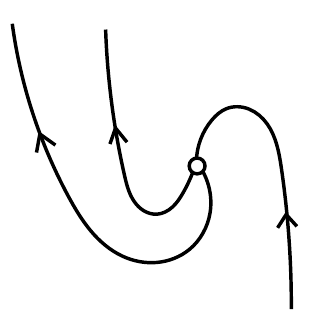}}%
  \end{picture}%
\endgroup%
}}}~~.\label{eq24}
%=~~\vcenter{\hbox{{\def\svgscale{0.6}
%\input{reflection-7.pdf_tex}}}}~~
\end{gather}
\end{df}

\begin{pp}
The reflection operator $\rfl$ is uniquely determined by the dual $W_{\ovl a}$ and the $\ev$  of $W_a$ and $W_{\ovl a}$.
\end{pp}
\begin{proof}
Apply $\iota^*$ to the bottom of \eqref{eq23}, and then apply the unit property,  we have $\ev_{\ovl a,a}(\rfl\otimes\id_a)=\iota^*\mu$, which, by rigidity, implies
\begin{align}
\rfl=(\iota^*\mu\otimes\id_{\ovl a})(\id_a\otimes\coev_{a,\ovl a}).\label{eq26}
\end{align}
\end{proof}

\begin{df}
Let $\mc C$ be a $C^*$-tensor category with simple $W_0$. A $\mc C$-algebra $A=(W_a,\mu,\iota)$ in $\mc C$ is called \textbf{unitary}, if $W_a$ is dualizable, and for a choice of $W_{\ovl a}$ dual to $W_a$ and $\ev$ of $W_a,W_{\ovl a}$, there exists a reflection operator $\rfl\in\Hom(W_a,W_{\ovl a})$. If, moreover, the $\ev$ of $W_a,W_{\ovl a}$ are standard, we say that $A$ is  \textbf{s-unitary}\footnote{The letter ``s" stands for several closely related notions: standard evaluations, spherical tensor categories, symmetric Frobenius algebras \cite{FRS02}.}.
\end{df}

The definition of s-unitary $\mc C$-algebras is independent of the choice of duals and standard $\ev$, as shown below:
\begin{pp}\label{lb8}
If $A$ is s-unitary,   then for any $W_{a'}$ dual to $W_a$, and any standard $\wtd\ev$ of $W_a,W_{a'}$, there exists a reflection operator $\wtd\rfl:W_a\rightarrow W_{a'}$ 
\end{pp}
\begin{proof}
Since $A$ is s-unitary, we can choose $W_{\ovl a}$ dual to $W_a$ and standard $\ev$ of $W_a,W_{\ovl a}$ such that there exists a reflection operator $\rfl:W_a\rightarrow W_{\ovl a}$.	Let us define $\wtd\rfl=(\iota^*\mu\otimes\id_{\ovl a})(\id_a\otimes\wtd\coev_{a,a'})$ and show that $\wtd\rfl$ is a reflection operator. We first show that $\wtd\rfl$ is unitary. Choose a unitary $T\in\Hom(W_{a'},W_{\ovl a})$. Then by the up to unitary uniqueness of standard $\ev$, there exists a unitary $K\in\Hom(W_a,W_a)$ such that $\wtd\ev_{a,a'}=\ev_{a,\ovl a}(K\otimes T),\wtd\ev_{a',a}=\ev_{\ovl a,a}(T\otimes K)$. By \eqref{eq25}, $\ev_{a,\ovl a}(K\otimes\id_{\ovl a})=\ev_{a,\ovl a}(\id_a\otimes K^\vee)$. Therefore $\wtd\ev_{a,a'}=\ev_{a,\ovl a}(\id_a\otimes (K^\vee) T)$. Thus $\wtd\rfl=(\iota^*\mu\otimes\id_{\ovl a})(\id_a\otimes T^*(K^\vee)^*)(\id_a\otimes\coev_{a,\ovl a})=(\id_0\otimes T^*(K^\vee)^*)(\iota^*\mu\otimes\id_{\ovl a})(\id_a\otimes\coev_{a,\ovl a})$, which, together with \eqref{eq26}, implies $\wtd\rfl=T^*(K^\vee)^*\rfl$. Thus $\wtd\rfl$ is unitary since $\rfl,K^\vee,T$ are unitary.
	
	Now, from the definition of $\wtd\rfl$, we see that $\wtd\ev_{a',a}(\wtd\rfl\otimes\id_a)=\iota^*\mu$. Therefore $\wtd\ev_{a',a}(\wtd\rfl\otimes\id_a)=\ev_{\ovl a,a}(\rfl\otimes\id_a)$. Using this fact, one can now easily check that $\eqref{eq23}$ and $\eqref{eq24}$ hold for $\wtd\rfl$ and the standard $\wtd\ev$ of $W_a,W_{a'}$.
\end{proof}

We now relate  s-unitary $\mc C$-algebras with unitary VOA extensions. First we need a lemma.

\begin{lm}\label{lb12}
Let $U$ be a preunitary CFT-type extension  of $V$, $A_U=(W_a,\mu,\iota)$ the associated $\RepV$-algebra, and $W_{\ovl a}$ the contragredient module of $W_a$. Then $U$ is a unitary VOA if and only if there exists a unitary $\rfl\in\Hom(W_a,W_{\ovl a})$ satisfying for all $w^{(a)}\in W_a$ that
\begin{gather}
\mc Y_\mu(w^{(a)},z)=\mc Y_{\mu^\dagger}(\rfl w^{(a)},z),\label{eq27}\\
\rfl\mc Y_\mu(\rfl^*\ovl{w^{(a)}},z)\rfl^*=\mc Y_{\ovl\mu}(\ovl{w^{(a)}},z).\label{eq28}
\end{gather}
\end{lm}
\begin{proof}
Suppose that such $\rfl$ exists, then by proposition \ref{lb9} and the definition of adjoint and conjugate intertwining operators, $U$ is unitary. Conversely, if $U$ is unitary, then by proposition \ref{lb9} there exists a unitary map $\rfl:W_a\rightarrow W_{\ovl a}$ such that \eqref{eq27} and \eqref{eq28} are true. Moreover, by proposition \ref{lb10}, $\rfl$ is a homomorphism of $U$-modules. So it is also a homomorphism of $V$-modules. This proves $\rfl\in\Hom(W_a,W_{\ovl a})$.
\end{proof}

The following is the main result of this section.

\begin{thm}\label{lb17}
Let $V$ be a regular, CFT-type, and completely unitary VOA. Let $U$ be a CFT-type preunitary extension of $V$, and let $A_U=(W_a,\mu,\iota)$ be the haploid commutative $\RepV$-algebra associated to $U$. Then $U$ is unitary if and only if   $A_U$ is  s-unitary.
\end{thm}

\begin{proof}
Following convention \ref{lb11}, we let $W_{\ovl a}$ be the contragredient $V$-module of $W_a$, and choose standard $\ev$ for $W_a,W_{\ovl a}$ as in chapter \ref{lb7}. By lemma \ref{lb12} and relation \eqref{eq31}, the unitarity of $U$ is equivalent to the existence of a unitary $\rfl\in\Hom(W_a,W_{\ovl a})$ satisfying
\begin{align}
\mu=\mu^\dagger(\rfl\otimes\id_a),\qquad \rfl\mu(\rfl^*\otimes\rfl^*)=\ovl\mu.\label{eq29}
\end{align}
Now, by theorem \ref{lb6}, $\mu^\dagger=(\ev_{\ovl a,a}\otimes\id_a)(\id_{\ovl a}\otimes\mu^*)$. Thus $\mu^\dagger(\rfl\otimes\id_a)=(\ev_{\ovl a,a}\otimes\id_a)(\rfl\otimes\mu^*)$. Therefore the first equation of \eqref{eq29} is equivalent to the first one of \eqref{eq30}. By corollary \ref{lb13}, $\ovl\mu=(\ss_{a,a}\mu^*)^\vee=((\mu\ss_{a,a}^{-1})^*)^\vee$, which equals $(\mu^*)^\vee$ by the commutativity of $A_U$. Therefore the second equation of \eqref{eq29} is also equivalent to that of \eqref{eq30}. We conclude that $\rfl$ satisfies \eqref{eq29} if and only if $\rfl$ is a reflection operator. Thus the unitarity of $U$ is equivalent to the existence of a reflection operator under the standard $\ev$, which is precisely the s-unitarity of $A_U$.
\end{proof}

\subsection{Unitary $\mc C$-algebras and $C^*$-Frobenius algebras}

Let $A=(W_a,\mu,\iota)$ be a $\mc C$-algebra. $A$ is called a \textbf{$C^*$-Frobenius algebra} in $\mc C$ if $(\id_a\otimes\mu)(\mu^*\otimes\id_a)=\mu^*\mu$. By taking adjoint we have the equivalent condition $\mu^*\mu=(\mu\otimes\id_a)(\id_a\otimes\mu^*)$. Assume in this section that all line segments in the pictures are labeled by $a$.   Then these two equations read
\begin{align}
\vcenter{\hbox{{\def\svgscale{0.6}
%% Accompanies image file 'Frobenius.pdf' (pdf, eps, ps)
\begingroup%
  \makeatletter%
  \providecommand\color[2][]{%
    \errmessage{(Inkscape) Color is used for the text in Inkscape, but the package 'color.sty' is not loaded}%
    \renewcommand\color[2][]{}%
  }%
  \providecommand\transparent[1]{%
    \errmessage{(Inkscape) Transparency is used (non-zero) for the text in Inkscape, but the package 'transparent.sty' is not loaded}%
    \renewcommand\transparent[1]{}%
  }%
  \providecommand\rotatebox[2]{#2}%
  \newcommand*\fsize{\dimexpr\f@size pt\relax}%
  \newcommand*\lineheight[1]{\fontsize{\fsize}{#1\fsize}\selectfont}%
  \ifx\svgwidth\undefined%
    \setlength{\unitlength}{51.37902103bp}%
    \ifx\svgscale\undefined%
      \relax%
    \else%
      \setlength{\unitlength}{\unitlength * \real{\svgscale}}%
    \fi%
  \else%
    \setlength{\unitlength}{\svgwidth}%
  \fi%
  \global\let\svgwidth\undefined%
  \global\let\svgscale\undefined%
  \makeatother%
  \begin{picture}(1,1.63029008)%
    \lineheight{1}%
    \setlength\tabcolsep{0pt}%
    \put(0,0){\includegraphics[width=\unitlength,page=1]{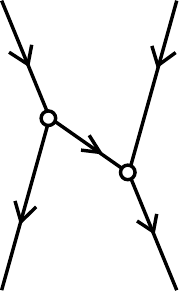}}%
  \end{picture}%
\endgroup%
}}}~~~=~~\vcenter{\hbox{{\def\svgscale{0.6}
%% Accompanies image file 'Frobenius-2.pdf' (pdf, eps, ps)
\begingroup%
  \makeatletter%
  \providecommand\color[2][]{%
    \errmessage{(Inkscape) Color is used for the text in Inkscape, but the package 'color.sty' is not loaded}%
    \renewcommand\color[2][]{}%
  }%
  \providecommand\transparent[1]{%
    \errmessage{(Inkscape) Transparency is used (non-zero) for the text in Inkscape, but the package 'transparent.sty' is not loaded}%
    \renewcommand\transparent[1]{}%
  }%
  \providecommand\rotatebox[2]{#2}%
  \newcommand*\fsize{\dimexpr\f@size pt\relax}%
  \newcommand*\lineheight[1]{\fontsize{\fsize}{#1\fsize}\selectfont}%
  \ifx\svgwidth\undefined%
    \setlength{\unitlength}{51.15810042bp}%
    \ifx\svgscale\undefined%
      \relax%
    \else%
      \setlength{\unitlength}{\unitlength * \real{\svgscale}}%
    \fi%
  \else%
    \setlength{\unitlength}{\svgwidth}%
  \fi%
  \global\let\svgwidth\undefined%
  \global\let\svgscale\undefined%
  \makeatother%
  \begin{picture}(1,1.6430023)%
    \lineheight{1}%
    \setlength\tabcolsep{0pt}%
    \put(0,0){\includegraphics[width=\unitlength,page=1]{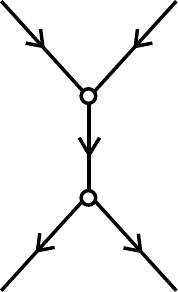}}%
  \end{picture}%
\endgroup%
}}}~~=~~\vcenter{\hbox{{\def\svgscale{0.6}
%% Accompanies image file 'Frobenius-3.pdf' (pdf, eps, ps)
\begingroup%
  \makeatletter%
  \providecommand\color[2][]{%
    \errmessage{(Inkscape) Color is used for the text in Inkscape, but the package 'color.sty' is not loaded}%
    \renewcommand\color[2][]{}%
  }%
  \providecommand\transparent[1]{%
    \errmessage{(Inkscape) Transparency is used (non-zero) for the text in Inkscape, but the package 'transparent.sty' is not loaded}%
    \renewcommand\transparent[1]{}%
  }%
  \providecommand\rotatebox[2]{#2}%
  \newcommand*\fsize{\dimexpr\f@size pt\relax}%
  \newcommand*\lineheight[1]{\fontsize{\fsize}{#1\fsize}\selectfont}%
  \ifx\svgwidth\undefined%
    \setlength{\unitlength}{51.37902103bp}%
    \ifx\svgscale\undefined%
      \relax%
    \else%
      \setlength{\unitlength}{\unitlength * \real{\svgscale}}%
    \fi%
  \else%
    \setlength{\unitlength}{\svgwidth}%
  \fi%
  \global\let\svgwidth\undefined%
  \global\let\svgscale\undefined%
  \makeatother%
  \begin{picture}(1,1.63029008)%
    \lineheight{1}%
    \setlength\tabcolsep{0pt}%
    \put(0,0){\includegraphics[width=\unitlength,page=1]{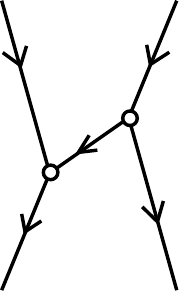}}%
  \end{picture}%
\endgroup%
}}}~~.\label{eq32}
\end{align}
A special (i.e. $\mu\mu^*\in\mathbb C\id_a$) $C^*$-Frobenius algebra is called a \textbf{Q-system}.\footnote{We warn the reader that in the literature there is no agreement  on whether standardness is required in the definition of Q-systems. For example, the Q-systems in \cite{BKLR15} are in fact standard Q-systems in our paper.} We remark that the $\mc C$-algebra $A$ is a Q-system if and only if it is special. In other words, the Frobenius relations \eqref{eq32} are consequences of the unit property, the associativity, and the specialness of $A$; see \cite{LR97} or \cite{BKLR15} lemma 3.7. 

The main goal of this section is to relate  (s-)unitarity to Frobenius property. More precisely, we shall show 
\begin{align}
\text{Unitary $\mc C$-algebras}~&=~\text{$C^*$-Frobenius algebras}\nonumber\\
&~\cup\nonumber\\
\text{Special unitary $\mc C$-algebras}~&=~\text{Q-systems}\nonumber\\
&~\cup\nonumber\\
\text{Special s-unitary $\mc C$-algebras}~&=~\text{standard Q-systems}.\label{eq38}
\end{align}
Moreover, under the assumption of haploid condition, all these notions are equivalent. In the process of the proof we shall also see that \eqref{eq24} is a consequence of \eqref{eq23}. Thus the definition of reflection operator can be simplified to assume only \eqref{eq23}.

To begin with, let us fix a dual $W_{\ovl a}$ of $W_a$ together with evaluations $\ev_{a,\ovl a},\ev_{\ovl a,a}$ of $W_a,W_{\ovl a}$. Choose a unitary $\rfl\in\Hom(W_a,W_{\ovl a})$.
\begin{pp}
If $\rfl$  satisfies \eqref{eq23}, then $\ev_{\ovl a,a}(\rfl\otimes\id_a)=\ev_{a,\ovl a}(\id_a\otimes\rfl)$; equivalently,
\begin{align}
\vcenter{\hbox{{\def\svgscale{0.8}
%% Accompanies image file 'unitary-properties.pdf' (pdf, eps, ps)
\begingroup%
  \makeatletter%
  \providecommand\color[2][]{%
    \errmessage{(Inkscape) Color is used for the text in Inkscape, but the package 'color.sty' is not loaded}%
    \renewcommand\color[2][]{}%
  }%
  \providecommand\transparent[1]{%
    \errmessage{(Inkscape) Transparency is used (non-zero) for the text in Inkscape, but the package 'transparent.sty' is not loaded}%
    \renewcommand\transparent[1]{}%
  }%
  \providecommand\rotatebox[2]{#2}%
  \newcommand*\fsize{\dimexpr\f@size pt\relax}%
  \newcommand*\lineheight[1]{\fontsize{\fsize}{#1\fsize}\selectfont}%
  \ifx\svgwidth\undefined%
    \setlength{\unitlength}{38.49520826bp}%
    \ifx\svgscale\undefined%
      \relax%
    \else%
      \setlength{\unitlength}{\unitlength * \real{\svgscale}}%
    \fi%
  \else%
    \setlength{\unitlength}{\svgwidth}%
  \fi%
  \global\let\svgwidth\undefined%
  \global\let\svgscale\undefined%
  \makeatother%
  \begin{picture}(1,0.52580736)%
    \lineheight{1}%
    \setlength\tabcolsep{0pt}%
    \put(0,0){\includegraphics[width=\unitlength,page=1]{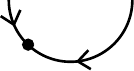}}%
  \end{picture}%
\endgroup%
}}}~~~=~~\vcenter{\hbox{{\def\svgscale{0.8}
%% Accompanies image file 'unitary-properties-2.pdf' (pdf, eps, ps)
\begingroup%
  \makeatletter%
  \providecommand\color[2][]{%
    \errmessage{(Inkscape) Color is used for the text in Inkscape, but the package 'color.sty' is not loaded}%
    \renewcommand\color[2][]{}%
  }%
  \providecommand\transparent[1]{%
    \errmessage{(Inkscape) Transparency is used (non-zero) for the text in Inkscape, but the package 'transparent.sty' is not loaded}%
    \renewcommand\transparent[1]{}%
  }%
  \providecommand\rotatebox[2]{#2}%
  \newcommand*\fsize{\dimexpr\f@size pt\relax}%
  \newcommand*\lineheight[1]{\fontsize{\fsize}{#1\fsize}\selectfont}%
  \ifx\svgwidth\undefined%
    \setlength{\unitlength}{38.49520815bp}%
    \ifx\svgscale\undefined%
      \relax%
    \else%
      \setlength{\unitlength}{\unitlength * \real{\svgscale}}%
    \fi%
  \else%
    \setlength{\unitlength}{\svgwidth}%
  \fi%
  \global\let\svgwidth\undefined%
  \global\let\svgscale\undefined%
  \makeatother%
  \begin{picture}(1,0.52580736)%
    \lineheight{1}%
    \setlength\tabcolsep{0pt}%
    \put(0,0){\includegraphics[width=\unitlength,page=1]{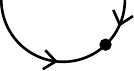}}%
  \end{picture}%
\endgroup%
}}}~~.\label{eq34}
\end{align}
\end{pp}
\begin{proof}
Take the adjoint of \eqref{eq23} and apply $(\rfl\otimes\id_a)$ to the bottom, we have
\begin{align*}
\vcenter{\hbox{{\def\svgscale{0.6}
%% Accompanies image file 'unitary-properties-proof.pdf' (pdf, eps, ps)
\begingroup%
  \makeatletter%
  \providecommand\color[2][]{%
    \errmessage{(Inkscape) Color is used for the text in Inkscape, but the package 'color.sty' is not loaded}%
    \renewcommand\color[2][]{}%
  }%
  \providecommand\transparent[1]{%
    \errmessage{(Inkscape) Transparency is used (non-zero) for the text in Inkscape, but the package 'transparent.sty' is not loaded}%
    \renewcommand\transparent[1]{}%
  }%
  \providecommand\rotatebox[2]{#2}%
  \newcommand*\fsize{\dimexpr\f@size pt\relax}%
  \newcommand*\lineheight[1]{\fontsize{\fsize}{#1\fsize}\selectfont}%
  \ifx\svgwidth\undefined%
    \setlength{\unitlength}{67.12987308bp}%
    \ifx\svgscale\undefined%
      \relax%
    \else%
      \setlength{\unitlength}{\unitlength * \real{\svgscale}}%
    \fi%
  \else%
    \setlength{\unitlength}{\svgwidth}%
  \fi%
  \global\let\svgwidth\undefined%
  \global\let\svgscale\undefined%
  \makeatother%
  \begin{picture}(1,1.25712581)%
    \lineheight{1}%
    \setlength\tabcolsep{0pt}%
    \put(0,0){\includegraphics[width=\unitlength,page=1]{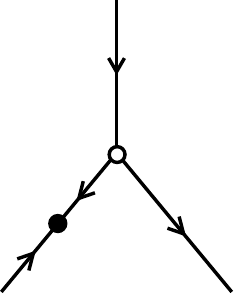}}%
  \end{picture}%
\endgroup%
}}}~~=~~\vcenter{\hbox{{\def\svgscale{0.6}
%% Accompanies image file 'unitary-properties-proof-2.pdf' (pdf, eps, ps)
\begingroup%
  \makeatletter%
  \providecommand\color[2][]{%
    \errmessage{(Inkscape) Color is used for the text in Inkscape, but the package 'color.sty' is not loaded}%
    \renewcommand\color[2][]{}%
  }%
  \providecommand\transparent[1]{%
    \errmessage{(Inkscape) Transparency is used (non-zero) for the text in Inkscape, but the package 'transparent.sty' is not loaded}%
    \renewcommand\transparent[1]{}%
  }%
  \providecommand\rotatebox[2]{#2}%
  \newcommand*\fsize{\dimexpr\f@size pt\relax}%
  \newcommand*\lineheight[1]{\fontsize{\fsize}{#1\fsize}\selectfont}%
  \ifx\svgwidth\undefined%
    \setlength{\unitlength}{67.89149164bp}%
    \ifx\svgscale\undefined%
      \relax%
    \else%
      \setlength{\unitlength}{\unitlength * \real{\svgscale}}%
    \fi%
  \else%
    \setlength{\unitlength}{\svgwidth}%
  \fi%
  \global\let\svgwidth\undefined%
  \global\let\svgscale\undefined%
  \makeatother%
  \begin{picture}(1,1.24302315)%
    \lineheight{1}%
    \setlength\tabcolsep{0pt}%
    \put(0,0){\includegraphics[width=\unitlength,page=1]{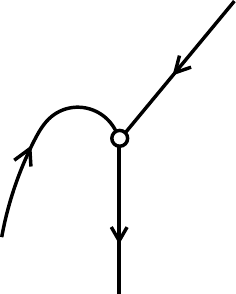}}%
  \end{picture}%
\endgroup%
}}}~~.
\end{align*}
Bending the left legs to the top proves that $\mu$ equals
\begin{align}
\vcenter{\hbox{{\def\svgscale{0.6}
%% Accompanies image file 'unitary-properties-proof-3.pdf' (pdf, eps, ps)
\begingroup%
  \makeatletter%
  \providecommand\color[2][]{%
    \errmessage{(Inkscape) Color is used for the text in Inkscape, but the package 'color.sty' is not loaded}%
    \renewcommand\color[2][]{}%
  }%
  \providecommand\transparent[1]{%
    \errmessage{(Inkscape) Transparency is used (non-zero) for the text in Inkscape, but the package 'transparent.sty' is not loaded}%
    \renewcommand\transparent[1]{}%
  }%
  \providecommand\rotatebox[2]{#2}%
  \newcommand*\fsize{\dimexpr\f@size pt\relax}%
  \newcommand*\lineheight[1]{\fontsize{\fsize}{#1\fsize}\selectfont}%
  \ifx\svgwidth\undefined%
    \setlength{\unitlength}{81.32373538bp}%
    \ifx\svgscale\undefined%
      \relax%
    \else%
      \setlength{\unitlength}{\unitlength * \real{\svgscale}}%
    \fi%
  \else%
    \setlength{\unitlength}{\svgwidth}%
  \fi%
  \global\let\svgwidth\undefined%
  \global\let\svgscale\undefined%
  \makeatother%
  \begin{picture}(1,1.03771298)%
    \lineheight{1}%
    \setlength\tabcolsep{0pt}%
    \put(0,0){\includegraphics[width=\unitlength,page=1]{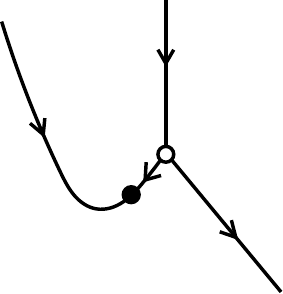}}%
  \end{picture}%
\endgroup%
}}}~~.\label{eq33}
\end{align}
We thus see that the right hand side of \eqref{eq23} equals \eqref{eq33}. Finally, we apply $\iota^*$ to their bottoms and use the unit property. This proves equation \eqref{eq34}.
\end{proof}

\begin{co}\label{lb15}
If $\rfl$ and the evaluations $\ev_{a,\ovl a},\ev_{\ovl a,a}$ of $W_a,W_{\ovl a}$ satisfy \eqref{eq23}, then there exists an evaluation $\wtd\ev_{a,a}$ for the self-dual object $W_a$ such that $\wtd\rfl:=\id_a\in\End(W_a)$ and $\wtd\ev_{a,a}$ also satisfy \eqref{eq23}. Moreover, if $\ev_{a,\ovl a},\ev_{\ovl a,a}$ are standard, then one can also choose $\wtd\ev_{a,a}$ to be standard.
\end{co}

\begin{proof}
We remind the reader that the definition of the evaluations of a self-dual object is given at the beginning of section \ref{lb14}. Assuming that $\rfl:W_a\rightarrow W_{\ovl a}$ satisfies \eqref{eq23}, we simply define $\wtd\ev_{a,a}\in\Hom(W_a\boxtimes W_a,W_0)$ to be the left and also the right hand side of \eqref{eq34}, i.e.,
\begin{align}
\ev_{a,a}:=\ev_{\ovl a,a}(\rfl\otimes\id_a)=\ev_{a,\ovl a}(\id_a\otimes\rfl).\label{eq39}
\end{align}
Then one easily checks that  $(\wtd\ev_{a,a}\otimes\id_a)(\id_a\otimes\wtd\coev_{a,a})=\id_a$ where $\wtd\coev_{a,a}:=(\wtd\ev_{a,a})^*$, and that $\id_a$ and $\wtd\ev_{a,a}$ also satisfy \eqref{eq23}. If the $\ev$ for $W_a,W_{\ovl a}$ are standard, then $\wtd\ev_{a,a}$ is also standard by the unitarity of $\rfl$.
\end{proof}

We shall write $\wtd\ev_{a,a}$ as $\ev_{a,a}$ instead. Then the above corollary says that when \eqref{eq23} holds, we may well assume that $\ovl a=a$, $\ev_{a,\ovl a}=\ev_{\ovl a,a}$ (which is written as $\ev_{a,a}$), and $\rfl=\id_a$. Pictorially, we may remove the arrows and the $\bullet$ on the strings to simplify calculations. Moreover, by \eqref{eq23} and the unit property one has $\ev_{a,a}=\iota^*\mu$:
\begin{align}
\vcenter{\hbox{{\def\svgscale{0.6}
%% Accompanies image file 'unitary-properties-3.pdf' (pdf, eps, ps)
\begingroup%
  \makeatletter%
  \providecommand\color[2][]{%
    \errmessage{(Inkscape) Color is used for the text in Inkscape, but the package 'color.sty' is not loaded}%
    \renewcommand\color[2][]{}%
  }%
  \providecommand\transparent[1]{%
    \errmessage{(Inkscape) Transparency is used (non-zero) for the text in Inkscape, but the package 'transparent.sty' is not loaded}%
    \renewcommand\transparent[1]{}%
  }%
  \providecommand\rotatebox[2]{#2}%
  \newcommand*\fsize{\dimexpr\f@size pt\relax}%
  \newcommand*\lineheight[1]{\fontsize{\fsize}{#1\fsize}\selectfont}%
  \ifx\svgwidth\undefined%
    \setlength{\unitlength}{36.58982956bp}%
    \ifx\svgscale\undefined%
      \relax%
    \else%
      \setlength{\unitlength}{\unitlength * \real{\svgscale}}%
    \fi%
  \else%
    \setlength{\unitlength}{\svgwidth}%
  \fi%
  \global\let\svgwidth\undefined%
  \global\let\svgscale\undefined%
  \makeatother%
  \begin{picture}(1,0.49999181)%
    \lineheight{1}%
    \setlength\tabcolsep{0pt}%
    \put(0,0){\includegraphics[width=\unitlength,page=1]{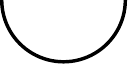}}%
  \end{picture}%
\endgroup%
}}}~~~=~~\vcenter{\hbox{{\def\svgscale{0.6}
%% Accompanies image file 'unitary-properties-4.pdf' (pdf, eps, ps)
\begingroup%
  \makeatletter%
  \providecommand\color[2][]{%
    \errmessage{(Inkscape) Color is used for the text in Inkscape, but the package 'color.sty' is not loaded}%
    \renewcommand\color[2][]{}%
  }%
  \providecommand\transparent[1]{%
    \errmessage{(Inkscape) Transparency is used (non-zero) for the text in Inkscape, but the package 'transparent.sty' is not loaded}%
    \renewcommand\transparent[1]{}%
  }%
  \providecommand\rotatebox[2]{#2}%
  \newcommand*\fsize{\dimexpr\f@size pt\relax}%
  \newcommand*\lineheight[1]{\fontsize{\fsize}{#1\fsize}\selectfont}%
  \ifx\svgwidth\undefined%
    \setlength{\unitlength}{32.94943613bp}%
    \ifx\svgscale\undefined%
      \relax%
    \else%
      \setlength{\unitlength}{\unitlength * \real{\svgscale}}%
    \fi%
  \else%
    \setlength{\unitlength}{\svgwidth}%
  \fi%
  \global\let\svgwidth\undefined%
  \global\let\svgscale\undefined%
  \makeatother%
  \begin{picture}(1,0.83322139)%
    \lineheight{1}%
    \setlength\tabcolsep{0pt}%
    \put(0,0){\includegraphics[width=\unitlength,page=1]{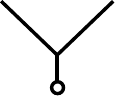}}%
  \end{picture}%
\endgroup%
}}}~~~.\label{eq35}
\end{align}
We now prove the main results of this section. Recall that $\mc C$ is a $C^*$-tensor category with simple $W_0$ and $A=(W_a,\mu,\iota)$ is a $\mc C$-algebra.

\begin{thm}\label{lb16}
$A$ is a unitary $\mc C$-algebra if and only if $A$ is a $C^*$-Frobenius algebra in $\mc C$.
\end{thm}
\begin{thm}
If there exists $W_{\ovl a}$ dual to $W_a$, evaluations $\ev_{a,\ovl a},\ev_{\ovl a,a}$ of $W_a,W_{\ovl a}$, and a unitary $\rfl\in\Hom(W_a,W_{\ovl a})$ satisfying \eqref{eq23}, then $\rfl$ also satisfies \eqref{eq24}. Consequently, $\rfl$ is a reflection operator, and $A$ is unitary. 
\end{thm}
We prove the two theorems simultaneously.
\begin{proof}
Step 1. Suppose there exists a dual object $W_{\ovl a}$, evaluations of $W_a,W_{\ovl a}$, and a unitary morphism $\rfl:W_a\rightarrow W_{\ovl a}$ satisfying \eqref{eq23}. We
assume that $\ovl a=a$, $\ev_{a,a}:=\ev_{a,\ovl a}=\ev_{\ovl a,a}$ is an evaluation of $W_a$, and $\rfl=\id_a$. Then
\begin{align*}
\vcenter{\hbox{{\def\svgscale{0.6}
%% Accompanies image file 'unitary-Frobenius-proof.pdf' (pdf, eps, ps)
\begingroup%
  \makeatletter%
  \providecommand\color[2][]{%
    \errmessage{(Inkscape) Color is used for the text in Inkscape, but the package 'color.sty' is not loaded}%
    \renewcommand\color[2][]{}%
  }%
  \providecommand\transparent[1]{%
    \errmessage{(Inkscape) Transparency is used (non-zero) for the text in Inkscape, but the package 'transparent.sty' is not loaded}%
    \renewcommand\transparent[1]{}%
  }%
  \providecommand\rotatebox[2]{#2}%
  \newcommand*\fsize{\dimexpr\f@size pt\relax}%
  \newcommand*\lineheight[1]{\fontsize{\fsize}{#1\fsize}\selectfont}%
  \ifx\svgwidth\undefined%
    \setlength{\unitlength}{51.37902103bp}%
    \ifx\svgscale\undefined%
      \relax%
    \else%
      \setlength{\unitlength}{\unitlength * \real{\svgscale}}%
    \fi%
  \else%
    \setlength{\unitlength}{\svgwidth}%
  \fi%
  \global\let\svgwidth\undefined%
  \global\let\svgscale\undefined%
  \makeatother%
  \begin{picture}(1,1.63029008)%
    \lineheight{1}%
    \setlength\tabcolsep{0pt}%
    \put(0,0){\includegraphics[width=\unitlength,page=1]{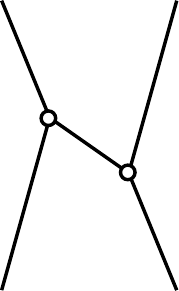}}%
  \end{picture}%
\endgroup%
}}}~~~=~~\vcenter{\hbox{{\def\svgscale{0.6}
%% Accompanies image file 'unitary-Frobenius-proof-2.pdf' (pdf, eps, ps)
\begingroup%
  \makeatletter%
  \providecommand\color[2][]{%
    \errmessage{(Inkscape) Color is used for the text in Inkscape, but the package 'color.sty' is not loaded}%
    \renewcommand\color[2][]{}%
  }%
  \providecommand\transparent[1]{%
    \errmessage{(Inkscape) Transparency is used (non-zero) for the text in Inkscape, but the package 'transparent.sty' is not loaded}%
    \renewcommand\transparent[1]{}%
  }%
  \providecommand\rotatebox[2]{#2}%
  \newcommand*\fsize{\dimexpr\f@size pt\relax}%
  \newcommand*\lineheight[1]{\fontsize{\fsize}{#1\fsize}\selectfont}%
  \ifx\svgwidth\undefined%
    \setlength{\unitlength}{63.46604988bp}%
    \ifx\svgscale\undefined%
      \relax%
    \else%
      \setlength{\unitlength}{\unitlength * \real{\svgscale}}%
    \fi%
  \else%
    \setlength{\unitlength}{\svgwidth}%
  \fi%
  \global\let\svgwidth\undefined%
  \global\let\svgscale\undefined%
  \makeatother%
  \begin{picture}(1,1.31893444)%
    \lineheight{1}%
    \setlength\tabcolsep{0pt}%
    \put(0,0){\includegraphics[width=\unitlength,page=1]{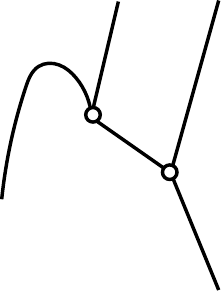}}%
  \end{picture}%
\endgroup%
}}}~~~=~~\vcenter{\hbox{{\def\svgscale{0.6}
%% Accompanies image file 'unitary-Frobenius-proof-3.pdf' (pdf, eps, ps)
\begingroup%
  \makeatletter%
  \providecommand\color[2][]{%
    \errmessage{(Inkscape) Color is used for the text in Inkscape, but the package 'color.sty' is not loaded}%
    \renewcommand\color[2][]{}%
  }%
  \providecommand\transparent[1]{%
    \errmessage{(Inkscape) Transparency is used (non-zero) for the text in Inkscape, but the package 'transparent.sty' is not loaded}%
    \renewcommand\transparent[1]{}%
  }%
  \providecommand\rotatebox[2]{#2}%
  \newcommand*\fsize{\dimexpr\f@size pt\relax}%
  \newcommand*\lineheight[1]{\fontsize{\fsize}{#1\fsize}\selectfont}%
  \ifx\svgwidth\undefined%
    \setlength{\unitlength}{63.46847858bp}%
    \ifx\svgscale\undefined%
      \relax%
    \else%
      \setlength{\unitlength}{\unitlength * \real{\svgscale}}%
    \fi%
  \else%
    \setlength{\unitlength}{\svgwidth}%
  \fi%
  \global\let\svgwidth\undefined%
  \global\let\svgscale\undefined%
  \makeatother%
  \begin{picture}(1,1.32964521)%
    \lineheight{1}%
    \setlength\tabcolsep{0pt}%
    \put(0,0){\includegraphics[width=\unitlength,page=1]{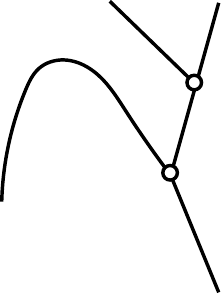}}%
  \end{picture}%
\endgroup%
}}}~~~=~~\vcenter{\hbox{{\def\svgscale{0.6}
%% Accompanies image file 'unitary-Frobenius-proof-4.pdf' (pdf, eps, ps)
\begingroup%
  \makeatletter%
  \providecommand\color[2][]{%
    \errmessage{(Inkscape) Color is used for the text in Inkscape, but the package 'color.sty' is not loaded}%
    \renewcommand\color[2][]{}%
  }%
  \providecommand\transparent[1]{%
    \errmessage{(Inkscape) Transparency is used (non-zero) for the text in Inkscape, but the package 'transparent.sty' is not loaded}%
    \renewcommand\transparent[1]{}%
  }%
  \providecommand\rotatebox[2]{#2}%
  \newcommand*\fsize{\dimexpr\f@size pt\relax}%
  \newcommand*\lineheight[1]{\fontsize{\fsize}{#1\fsize}\selectfont}%
  \ifx\svgwidth\undefined%
    \setlength{\unitlength}{34.71508591bp}%
    \ifx\svgscale\undefined%
      \relax%
    \else%
      \setlength{\unitlength}{\unitlength * \real{\svgscale}}%
    \fi%
  \else%
    \setlength{\unitlength}{\svgwidth}%
  \fi%
  \global\let\svgwidth\undefined%
  \global\let\svgscale\undefined%
  \makeatother%
  \begin{picture}(1,2.45819983)%
    \lineheight{1}%
    \setlength\tabcolsep{0pt}%
    \put(0,0){\includegraphics[width=\unitlength,page=1]{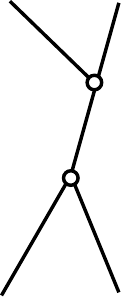}}%
  \end{picture}%
\endgroup%
}}}~~~,
\end{align*}
where we have used successively equation \eqref{eq23}, associativity, and again equation \eqref{eq23} in the above equations. This proves the first and hence also the second equation of \eqref{eq32}. We conclude that $A$ is a $C^*$-Frobenius algebra.

Step 2. Assume that $A$ is a $C^*$-Frobenius algebra in $\mc C$. We shall show that $W_a$ is self-dual, and construct a reflection operator. Define $\ev_{a,a}\in\Hom(W_a\boxtimes W_a,W_0)$ to be $\ev_{a,a}=\iota^*\mu$ (see  \eqref{eq35}). One then easily verifies $(\ev_{a,a}\otimes\id_a)(\id_a\otimes(\ev_{a,a})^*)=\id_a$ by applying respectively $\id_a\otimes\iota$ and $\iota^*\otimes\id_a$ to the top and the bottom of the second equation of \eqref{eq32}, and then applying the unit property. This shows that $W_a$ is self-dual and $\ev_{a,a}=\iota^*\mu$  is an evaluation of $W_a$. Therefore we can also omit arrows. Apply $\iota^*\otimes\id_a$ and $\id_a\otimes\iota^*$ to the bottoms of the second and the first equation of \eqref{eq32} respectively, and then use the unit property and equation \eqref{eq35}, we obtain
\begin{align}
\vcenter{\hbox{{\def\svgscale{0.6}
%% Accompanies image file 'unitary-properties-5.pdf' (pdf, eps, ps)
\begingroup%
  \makeatletter%
  \providecommand\color[2][]{%
    \errmessage{(Inkscape) Color is used for the text in Inkscape, but the package 'color.sty' is not loaded}%
    \renewcommand\color[2][]{}%
  }%
  \providecommand\transparent[1]{%
    \errmessage{(Inkscape) Transparency is used (non-zero) for the text in Inkscape, but the package 'transparent.sty' is not loaded}%
    \renewcommand\transparent[1]{}%
  }%
  \providecommand\rotatebox[2]{#2}%
  \newcommand*\fsize{\dimexpr\f@size pt\relax}%
  \newcommand*\lineheight[1]{\fontsize{\fsize}{#1\fsize}\selectfont}%
  \ifx\svgwidth\undefined%
    \setlength{\unitlength}{66.23684144bp}%
    \ifx\svgscale\undefined%
      \relax%
    \else%
      \setlength{\unitlength}{\unitlength * \real{\svgscale}}%
    \fi%
  \else%
    \setlength{\unitlength}{\svgwidth}%
  \fi%
  \global\let\svgwidth\undefined%
  \global\let\svgscale\undefined%
  \makeatother%
  \begin{picture}(1,1.27407488)%
    \lineheight{1}%
    \setlength\tabcolsep{0pt}%
    \put(0,0){\includegraphics[width=\unitlength,page=1]{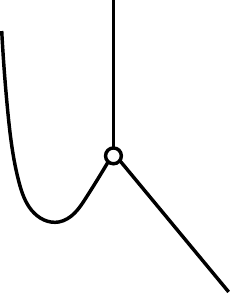}}%
  \end{picture}%
\endgroup%
}}}~~=~~\vcenter{\hbox{{\def\svgscale{0.6}
%% Accompanies image file 'unitary-properties-6.pdf' (pdf, eps, ps)
\begingroup%
  \makeatletter%
  \providecommand\color[2][]{%
    \errmessage{(Inkscape) Color is used for the text in Inkscape, but the package 'color.sty' is not loaded}%
    \renewcommand\color[2][]{}%
  }%
  \providecommand\transparent[1]{%
    \errmessage{(Inkscape) Transparency is used (non-zero) for the text in Inkscape, but the package 'transparent.sty' is not loaded}%
    \renewcommand\transparent[1]{}%
  }%
  \providecommand\rotatebox[2]{#2}%
  \newcommand*\fsize{\dimexpr\f@size pt\relax}%
  \newcommand*\lineheight[1]{\fontsize{\fsize}{#1\fsize}\selectfont}%
  \ifx\svgwidth\undefined%
    \setlength{\unitlength}{67.12987308bp}%
    \ifx\svgscale\undefined%
      \relax%
    \else%
      \setlength{\unitlength}{\unitlength * \real{\svgscale}}%
    \fi%
  \else%
    \setlength{\unitlength}{\svgwidth}%
  \fi%
  \global\let\svgwidth\undefined%
  \global\let\svgscale\undefined%
  \makeatother%
  \begin{picture}(1,1.25712581)%
    \lineheight{1}%
    \setlength\tabcolsep{0pt}%
    \put(0,0){\includegraphics[width=\unitlength,page=1]{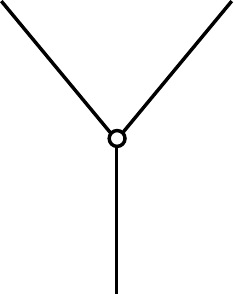}}%
  \end{picture}%
\endgroup%
}}}~~~=~~\vcenter{\hbox{{\def\svgscale{0.6}
%% Accompanies image file 'unitary-properties-7.pdf' (pdf, eps, ps)
\begingroup%
  \makeatletter%
  \providecommand\color[2][]{%
    \errmessage{(Inkscape) Color is used for the text in Inkscape, but the package 'color.sty' is not loaded}%
    \renewcommand\color[2][]{}%
  }%
  \providecommand\transparent[1]{%
    \errmessage{(Inkscape) Transparency is used (non-zero) for the text in Inkscape, but the package 'transparent.sty' is not loaded}%
    \renewcommand\transparent[1]{}%
  }%
  \providecommand\rotatebox[2]{#2}%
  \newcommand*\fsize{\dimexpr\f@size pt\relax}%
  \newcommand*\lineheight[1]{\fontsize{\fsize}{#1\fsize}\selectfont}%
  \ifx\svgwidth\undefined%
    \setlength{\unitlength}{66.23684144bp}%
    \ifx\svgscale\undefined%
      \relax%
    \else%
      \setlength{\unitlength}{\unitlength * \real{\svgscale}}%
    \fi%
  \else%
    \setlength{\unitlength}{\svgwidth}%
  \fi%
  \global\let\svgwidth\undefined%
  \global\let\svgscale\undefined%
  \makeatother%
  \begin{picture}(1,1.27407488)%
    \lineheight{1}%
    \setlength\tabcolsep{0pt}%
    \put(0,0){\includegraphics[width=\unitlength,page=1]{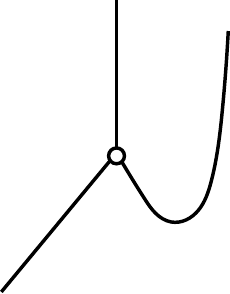}}%
  \end{picture}%
\endgroup%
}}}~~~,\label{eq36}
\end{align}
which proves that $\id_a$ and $\ev_{a,a}$ satisfy equation \eqref{eq23}. Since the first and the third items of \eqref{eq36} are equal, we take the adjoint of them and bend their left or right legs to the top to obtain 
\begin{align}
\vcenter{\hbox{{\def\svgscale{0.6}
%% Accompanies image file 'unitary-properties-8.pdf' (pdf, eps, ps)
\begingroup%
  \makeatletter%
  \providecommand\color[2][]{%
    \errmessage{(Inkscape) Color is used for the text in Inkscape, but the package 'color.sty' is not loaded}%
    \renewcommand\color[2][]{}%
  }%
  \providecommand\transparent[1]{%
    \errmessage{(Inkscape) Transparency is used (non-zero) for the text in Inkscape, but the package 'transparent.sty' is not loaded}%
    \renewcommand\transparent[1]{}%
  }%
  \providecommand\rotatebox[2]{#2}%
  \newcommand*\fsize{\dimexpr\f@size pt\relax}%
  \newcommand*\lineheight[1]{\fontsize{\fsize}{#1\fsize}\selectfont}%
  \ifx\svgwidth\undefined%
    \setlength{\unitlength}{75.1675434bp}%
    \ifx\svgscale\undefined%
      \relax%
    \else%
      \setlength{\unitlength}{\unitlength * \real{\svgscale}}%
    \fi%
  \else%
    \setlength{\unitlength}{\svgwidth}%
  \fi%
  \global\let\svgwidth\undefined%
  \global\let\svgscale\undefined%
  \makeatother%
  \begin{picture}(1,1.06929635)%
    \lineheight{1}%
    \setlength\tabcolsep{0pt}%
    \put(0,0){\includegraphics[width=\unitlength,page=1]{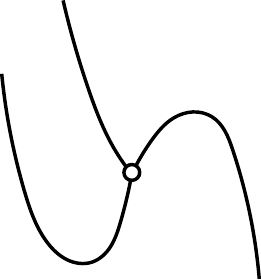}}%
  \end{picture}%
\endgroup%
}}}~~~=~~\vcenter{\hbox{{\def\svgscale{0.6}
			}}}~~~=~~\vcenter{\hbox{{\def\svgscale{0.6}
%% Accompanies image file 'unitary-properties-9.pdf' (pdf, eps, ps)
\begingroup%
  \makeatletter%
  \providecommand\color[2][]{%
    \errmessage{(Inkscape) Color is used for the text in Inkscape, but the package 'color.sty' is not loaded}%
    \renewcommand\color[2][]{}%
  }%
  \providecommand\transparent[1]{%
    \errmessage{(Inkscape) Transparency is used (non-zero) for the text in Inkscape, but the package 'transparent.sty' is not loaded}%
    \renewcommand\transparent[1]{}%
  }%
  \providecommand\rotatebox[2]{#2}%
  \newcommand*\fsize{\dimexpr\f@size pt\relax}%
  \newcommand*\lineheight[1]{\fontsize{\fsize}{#1\fsize}\selectfont}%
  \ifx\svgwidth\undefined%
    \setlength{\unitlength}{75.1675434bp}%
    \ifx\svgscale\undefined%
      \relax%
    \else%
      \setlength{\unitlength}{\unitlength * \real{\svgscale}}%
    \fi%
  \else%
    \setlength{\unitlength}{\svgwidth}%
  \fi%
  \global\let\svgwidth\undefined%
  \global\let\svgscale\undefined%
  \makeatother%
  \begin{picture}(1,1.06929635)%
    \lineheight{1}%
    \setlength\tabcolsep{0pt}%
    \put(0,0){\includegraphics[width=\unitlength,page=1]{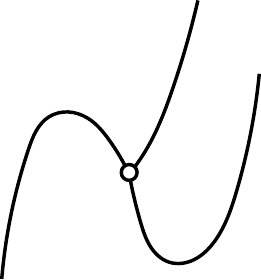}}%
  \end{picture}%
\endgroup%
}}}~~~.\label{eq37}
\end{align}
In other words, $\mu$ is invariant under clockwise and anticlockwise ``1-click rotations". Thus $\mu$ equals the clockwise 1-click rotation of the left hand side of \eqref{eq36}, which proves \eqref{eq24} for $\id_a$ and $\ev_{a,a}$. By \eqref{eq39}, $\rfl$ and the original evaluations $\ev_{a,\ovl a},\ev_{\ovl a,a}$ also satisfy equation \eqref{eq24}. Therefore $\rfl$ is a reflection operator of $A$ with respect to $W_{\ovl a}$ and the given evaluations of $W_a,W_{\ovl a}$. This finishes the proof of the two theorems.
\end{proof}

\begin{co}
$A$ is a special unitary $\mc C$-algebra if and only if $A$ is a Q-system.
\end{co}
\begin{proof}
Q-systems are by definition special $C^*$-Frobenius algebras.
\end{proof}

We now relate s-unitarity and standardness. We have seen that if $A$ is unitary then $\ev_{a,a}=\iota^*\mu$ is an evaluation of $W_a$. Therefore, setting $\coev_{a,a}=\ev_{a,a}^*$, we have $\ev_{a,a}\coev_{a,a}=D_A\id_0$ by the definition of $D_A$. By the minimizing property of standard evaluations, we have $D_A\geq d_a$, with equality holds if and only if $\ev_{a,a}$ is standard. Note that $\id_a$ is a reflection operator with respect to $\ev_{a,a}$. Therefore we have the following:

\begin{pp}
Let $A$ be unitary. Then $\ev_{a,a}:=\iota^*\mu$ is an evaluation of the self-dual object $W_a$, and $D_A\geq d_a$. Moreover, we have $D_A=d_a$ if and only if $A$ is s-unitary.
\end{pp}

\begin{co}
$A$ is a special s-unitary $\mc C$-algebra if and only if $A$ is a standard Q-system.
\end{co}
\begin{proof}
By theorem \ref{lb16} and the definition of Q-systems, $A$ is a standard Q-system if and only if $A$ is a special unitary $\mc C$-algebra satisfying $D_A=d_a$. Thus the corollary follows immediately from the above proposition.
\end{proof}

Thus we've finished proving the relations \eqref{eq38} given at the beginning of this section.

\begin{pp}\label{lb36}
Assuming haploid condition, the six notions in \eqref{eq38} are equivalent.
\end{pp}
\begin{proof}
Let $A$ be a haploid $C^*$-Frobenius algebra in $\mc C$. Then by \cite{BKLR15} lemma 3.3, $A$ is special. The standardness follows from \cite{LR97} section 6 (see also \cite{Mueg03} remark 5.6-3, or \cite{NY18a} theorem 2.9). Therefore $A$ is a standard Q-system.
\end{proof}

We can now restate the main result of the last section (theorem \ref{lb17}) in the following way:

\begin{thm}\label{lb18}
Let $V$ be a regular, CFT-type, and completely unitary VOA. Let $U$ be a CFT-type preunitary extension of $V$, and let $A_U=(W_a,\mu,\iota)$ be the haploid commutative $\RepV$-algebra  associated to $U$. Then $U$ is a unitary extension of $V$ if and only if   $A_U$ is   a $C^*$-Frobenius algebra. If this is true then $A_U$ is also a standard Q-system.
\end{thm}

Let us give an application of this theorem.

\begin{co}\label{lb58}
The $c<1$ CFT-type   unitary VOAs are in one to one correspondence with the  irreducibile  conformal nets with the same central charge $c$. Their classifications are given by \cite{KL04} table 3.
\end{co}

\begin{proof}
As shown in \cite{KL04} proposition 3.5, $c<1$ irreducible conformal nets are precisely irreducible finite-index extensions of the Virasoro net $\mc A_c$ with central charge $c$. Thus they are in 1-1 correspondence with the haploid  commutative Q-systems in $\Rep^\ssp(\mc A_c)$, where $\Rep^\ssp(\mc A_c)$ is the unitary modular tensor category of the semisimple representations of $\mc A_c$. By \cite{Gui20} theorem 5.1, $\Rep^\ssp(\mc A_c)$ is unitarily equivalent to $\Rep^\uni(V)$, where $V$ is the unitary Virasoro VOA with central charge $c$. By \cite{HKL15} theorem 3.2, haploid commutative $\Rep^\uni(V)$-algebras with trivial twist are in 1-1 correspondence with CFT-type extensions of $V$. Thus, by theorem \ref{lb18} and the equivalence of unitary modular tensor categories, CFT-type unitary extensions of $V$ $\Leftrightarrow$ haploid commutative Q-systems in $\Rep^\ssp(\mc A_c)$. (The trivial twist condition is redundant; see theorem \ref{lb44}.) But also CFT-type unitary extensions of $V$ $\Leftrightarrow$ unitary VOAs with central charge $c$ by \cite{DL14} theorem 5.1. This proves the desired result.
\end{proof}

\subsection{Strong unitarity of unitary VOA extensions}\label{lb38}

Starting from this section, $A=(W_a,\mu,\iota)$ is assumed to be a unitary $\mc C$-algebra, or equivalently, a $C^*$-Frobenius algebra in $\mc C$. We say that $(W_i,\mu_L)$ (resp. $(W_i,\mu_R)$)\footnote{Later we will write $\mu_L$ and $\mu_R$ as $\mu^i_L$ and $\mu^i_R$ to emphasize the dependence of $\mu_L,\mu_R$ on the $W_i$.}\index{Wi@$(W_i,\mu_L),(W_i,\mu_R)$}  is a \textbf{left $A$-module} (resp. \textbf{right $A$-module}), if $W_i$ is an object in $\mc C$, and $\mu_L\in\Hom(W_a\boxtimes W_i,W_i)$ (resp. $\mu_R\in\Hom(W_i\boxtimes W_a,W_i)$) satisfies the \emph{unit property}
\begin{align}
\mu_L(\iota\otimes\id_a)\qquad\text{resp.}\qquad \mu_R(\id_a\otimes\iota)=\id_a
\end{align}
and the \emph{associativity}:
\begin{align}
\mu_L(\id_a\otimes\mu_L)=\mu_L(\mu\otimes\id_i)\qquad\text{resp.}\qquad \mu_R(\mu_R\otimes\id_a)=\mu_R(\id_i\otimes\mu).
\end{align}
We write $\mu_L=\vcenter{\hbox{{\def\svgscale{0.4}
%% Accompanies image file 'module.pdf' (pdf, eps, ps)
\begingroup%
  \makeatletter%
  \providecommand\color[2][]{%
    \errmessage{(Inkscape) Color is used for the text in Inkscape, but the package 'color.sty' is not loaded}%
    \renewcommand\color[2][]{}%
  }%
  \providecommand\transparent[1]{%
    \errmessage{(Inkscape) Transparency is used (non-zero) for the text in Inkscape, but the package 'transparent.sty' is not loaded}%
    \renewcommand\transparent[1]{}%
  }%
  \providecommand\rotatebox[2]{#2}%
  \newcommand*\fsize{\dimexpr\f@size pt\relax}%
  \newcommand*\lineheight[1]{\fontsize{\fsize}{#1\fsize}\selectfont}%
  \ifx\svgwidth\undefined%
    \setlength{\unitlength}{54.19421239bp}%
    \ifx\svgscale\undefined%
      \relax%
    \else%
      \setlength{\unitlength}{\unitlength * \real{\svgscale}}%
    \fi%
  \else%
    \setlength{\unitlength}{\svgwidth}%
  \fi%
  \global\let\svgwidth\undefined%
  \global\let\svgscale\undefined%
  \makeatother%
  \begin{picture}(1,1.50252776)%
    \lineheight{1}%
    \setlength\tabcolsep{0pt}%
    \put(0,0){\includegraphics[width=\unitlength,page=1]{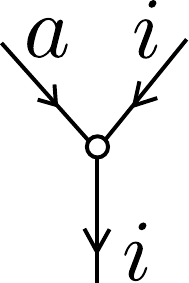}}%
  \end{picture}%
\endgroup%
}}}~~,~\mu_L^*=\vcenter{\hbox{{\def\svgscale{0.4}
%% Accompanies image file 'module-2.pdf' (pdf, eps, ps)
\begingroup%
  \makeatletter%
  \providecommand\color[2][]{%
    \errmessage{(Inkscape) Color is used for the text in Inkscape, but the package 'color.sty' is not loaded}%
    \renewcommand\color[2][]{}%
  }%
  \providecommand\transparent[1]{%
    \errmessage{(Inkscape) Transparency is used (non-zero) for the text in Inkscape, but the package 'transparent.sty' is not loaded}%
    \renewcommand\transparent[1]{}%
  }%
  \providecommand\rotatebox[2]{#2}%
  \newcommand*\fsize{\dimexpr\f@size pt\relax}%
  \newcommand*\lineheight[1]{\fontsize{\fsize}{#1\fsize}\selectfont}%
  \ifx\svgwidth\undefined%
    \setlength{\unitlength}{65.8609381bp}%
    \ifx\svgscale\undefined%
      \relax%
    \else%
      \setlength{\unitlength}{\unitlength * \real{\svgscale}}%
    \fi%
  \else%
    \setlength{\unitlength}{\svgwidth}%
  \fi%
  \global\let\svgwidth\undefined%
  \global\let\svgscale\undefined%
  \makeatother%
  \begin{picture}(1,1.15378713)%
    \lineheight{1}%
    \setlength\tabcolsep{0pt}%
    \put(0,0){\includegraphics[width=\unitlength,page=1]{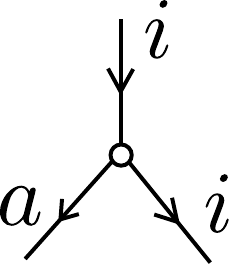}}%
  \end{picture}%
\endgroup%
}}}~~,~\mu_R=\vcenter{\hbox{{\def\svgscale{0.4}
%% Accompanies image file 'module-3.pdf' (pdf, eps, ps)
\begingroup%
  \makeatletter%
  \providecommand\color[2][]{%
    \errmessage{(Inkscape) Color is used for the text in Inkscape, but the package 'color.sty' is not loaded}%
    \renewcommand\color[2][]{}%
  }%
  \providecommand\transparent[1]{%
    \errmessage{(Inkscape) Transparency is used (non-zero) for the text in Inkscape, but the package 'transparent.sty' is not loaded}%
    \renewcommand\transparent[1]{}%
  }%
  \providecommand\rotatebox[2]{#2}%
  \newcommand*\fsize{\dimexpr\f@size pt\relax}%
  \newcommand*\lineheight[1]{\fontsize{\fsize}{#1\fsize}\selectfont}%
  \ifx\svgwidth\undefined%
    \setlength{\unitlength}{54.19421239bp}%
    \ifx\svgscale\undefined%
      \relax%
    \else%
      \setlength{\unitlength}{\unitlength * \real{\svgscale}}%
    \fi%
  \else%
    \setlength{\unitlength}{\svgwidth}%
  \fi%
  \global\let\svgwidth\undefined%
  \global\let\svgscale\undefined%
  \makeatother%
  \begin{picture}(1,1.48858912)%
    \lineheight{1}%
    \setlength\tabcolsep{0pt}%
    \put(0,0){\includegraphics[width=\unitlength,page=1]{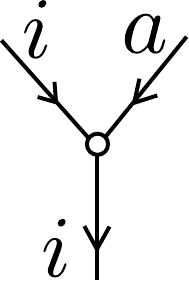}}%
  \end{picture}%
\endgroup%
}}}~~,~\mu_R^*=\vcenter{\hbox{{\def\svgscale{0.4}
%% Accompanies image file 'module-4.pdf' (pdf, eps, ps)
\begingroup%
  \makeatletter%
  \providecommand\color[2][]{%
    \errmessage{(Inkscape) Color is used for the text in Inkscape, but the package 'color.sty' is not loaded}%
    \renewcommand\color[2][]{}%
  }%
  \providecommand\transparent[1]{%
    \errmessage{(Inkscape) Transparency is used (non-zero) for the text in Inkscape, but the package 'transparent.sty' is not loaded}%
    \renewcommand\transparent[1]{}%
  }%
  \providecommand\rotatebox[2]{#2}%
  \newcommand*\fsize{\dimexpr\f@size pt\relax}%
  \newcommand*\lineheight[1]{\fontsize{\fsize}{#1\fsize}\selectfont}%
  \ifx\svgwidth\undefined%
    \setlength{\unitlength}{67.39948185bp}%
    \ifx\svgscale\undefined%
      \relax%
    \else%
      \setlength{\unitlength}{\unitlength * \real{\svgscale}}%
    \fi%
  \else%
    \setlength{\unitlength}{\svgwidth}%
  \fi%
  \global\let\svgwidth\undefined%
  \global\let\svgscale\undefined%
  \makeatother%
  \begin{picture}(1,1.14538166)%
    \lineheight{1}%
    \setlength\tabcolsep{0pt}%
    \put(0,0){\includegraphics[width=\unitlength,page=1]{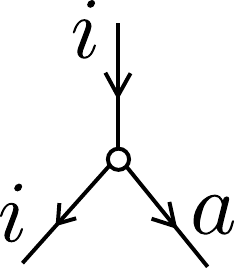}}%
  \end{picture}%
\endgroup%
}}}~~$.
If $(W_i,\mu_L)$ is a left $A$-module and $(W_i,\mu_R)$ is a right $A$-module, we say that $(W_i,\mu_L,\mu_R)$ is an \textbf{A-bimodule} if the following bimodule associativity holds:
\begin{align}
\mu_R(\mu_L\otimes\id_a)=\mu_L(\id_a\otimes\mu_R).
\end{align}
We leave it to the reader to draw the pictures of associativity and unit property. We abbreviate $(W_i,\mu_L)$, $(W_i,\mu_R)$, or $(W_i,\mu_L,\mu_R)$ to $W_i$ when no confusion arises.

Set $\ev_{a,a}=\iota^*\mu$ as in the last section. A left (resp. right) $A$-module $(W_i,\mu_L)$ (resp. $(W_i,\mu_R)$) is called \textbf{unitary}, if
\begin{align}
\mu_L=(\ev_{a,a}\otimes\id_i)(\id_a\otimes\mu_L^*)\qquad\text{resp.}\qquad \mu_R=(\id_i\otimes\ev_{a,a})(\mu_R^*\otimes\id_a).
\end{align}
An $A$-bimodule $(W_a,\mu_L,\mu_R)$ is called \textbf{unitary} if $(W_a,\mu_L)$ is a unitary left $A$-module and $(W_a,\mu_R)$ is a unitary right $A$-module. Unitarity can be stated for any evaluations and reflection operators:
\begin{pp}\label{lb20}
	Let $W_{\ovl a}$ be dual to $W_a$, $\ev_{a,\ovl a},\ev_{\ovl a,a}$ evaluations of $W_{a,\ovl a},W_{\ovl a,a}$, and $\rfl:W_a\rightarrow W_{\ovl a}$ a reflection operator. Then a left (resp. right) $A$-module $(W_a,\mu_L)$ (resp. $(W_a,\mu_R)$) is unitary if and only if
	\begin{align}
	\mu_L=(\ev_{\ovl a,a}\otimes\id_i)(\rfl\otimes\mu_L^*)\qquad\text{resp.}\qquad \mu_R=(\id_i\otimes\ev_{a,\ovl a})(\mu_R^*\otimes\rfl).
	\end{align}
	Graphically,
\begin{align}
\vcenter{\hbox{{\def\svgscale{0.6}
%% Accompanies image file 'unitary-module.pdf' (pdf, eps, ps)
\begingroup%
  \makeatletter%
  \providecommand\color[2][]{%
    \errmessage{(Inkscape) Color is used for the text in Inkscape, but the package 'color.sty' is not loaded}%
    \renewcommand\color[2][]{}%
  }%
  \providecommand\transparent[1]{%
    \errmessage{(Inkscape) Transparency is used (non-zero) for the text in Inkscape, but the package 'transparent.sty' is not loaded}%
    \renewcommand\transparent[1]{}%
  }%
  \providecommand\rotatebox[2]{#2}%
  \newcommand*\fsize{\dimexpr\f@size pt\relax}%
  \newcommand*\lineheight[1]{\fontsize{\fsize}{#1\fsize}\selectfont}%
  \ifx\svgwidth\undefined%
    \setlength{\unitlength}{71.02084401bp}%
    \ifx\svgscale\undefined%
      \relax%
    \else%
      \setlength{\unitlength}{\unitlength * \real{\svgscale}}%
    \fi%
  \else%
    \setlength{\unitlength}{\svgwidth}%
  \fi%
  \global\let\svgwidth\undefined%
  \global\let\svgscale\undefined%
  \makeatother%
  \begin{picture}(1,1.29726949)%
    \lineheight{1}%
    \setlength\tabcolsep{0pt}%
    \put(0,0){\includegraphics[width=\unitlength,page=1]{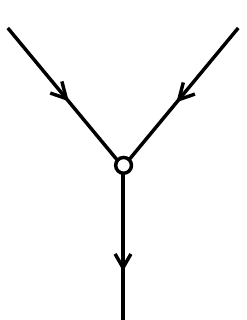}}%
    \put(-0.00426261,1.24365718){\color[rgb]{0,0,0}\makebox(0,0)[lt]{\lineheight{1.25}\smash{\begin{tabular}[t]{l}$a$\end{tabular}}}}%
    \put(0.91672779,1.25392834){\color[rgb]{0,0,0}\makebox(0,0)[lt]{\lineheight{1.25}\smash{\begin{tabular}[t]{l}$i$\end{tabular}}}}%
    \put(0.57128089,0.03610104){\color[rgb]{0,0,0}\makebox(0,0)[lt]{\lineheight{1.25}\smash{\begin{tabular}[t]{l}$i$\end{tabular}}}}%
  \end{picture}%
\endgroup%
}}}~~=~~\vcenter{\hbox{{\def\svgscale{0.6}
%% Accompanies image file 'unitary-module-2.pdf' (pdf, eps, ps)
\begingroup%
  \makeatletter%
  \providecommand\color[2][]{%
    \errmessage{(Inkscape) Color is used for the text in Inkscape, but the package 'color.sty' is not loaded}%
    \renewcommand\color[2][]{}%
  }%
  \providecommand\transparent[1]{%
    \errmessage{(Inkscape) Transparency is used (non-zero) for the text in Inkscape, but the package 'transparent.sty' is not loaded}%
    \renewcommand\transparent[1]{}%
  }%
  \providecommand\rotatebox[2]{#2}%
  \newcommand*\fsize{\dimexpr\f@size pt\relax}%
  \newcommand*\lineheight[1]{\fontsize{\fsize}{#1\fsize}\selectfont}%
  \ifx\svgwidth\undefined%
    \setlength{\unitlength}{85.31408934bp}%
    \ifx\svgscale\undefined%
      \relax%
    \else%
      \setlength{\unitlength}{\unitlength * \real{\svgscale}}%
    \fi%
  \else%
    \setlength{\unitlength}{\svgwidth}%
  \fi%
  \global\let\svgwidth\undefined%
  \global\let\svgscale\undefined%
  \makeatother%
  \begin{picture}(1,1.0728456)%
    \lineheight{1}%
    \setlength\tabcolsep{0pt}%
    \put(0,0){\includegraphics[width=\unitlength,page=1]{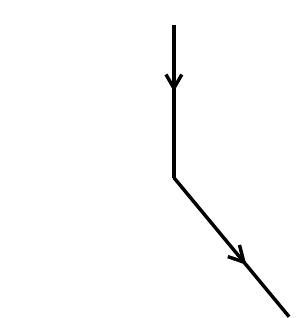}}%
    \put(-0.00354847,1.02998675){\color[rgb]{0,0,0}\makebox(0,0)[lt]{\lineheight{1.25}\smash{\begin{tabular}[t]{l}$a$\end{tabular}}}}%
    \put(0.54709831,1.03676569){\color[rgb]{0,0,0}\makebox(0,0)[lt]{\lineheight{1.25}\smash{\begin{tabular}[t]{l}$i$\end{tabular}}}}%
    \put(0.93067895,0.09557453){\color[rgb]{0,0,0}\makebox(0,0)[lt]{\lineheight{1.25}\smash{\begin{tabular}[t]{l}$i$\end{tabular}}}}%
    \put(0,0){\includegraphics[width=\unitlength,page=2]{unitary-module-2.pdf}}%
  \end{picture}%
\endgroup%
}}}~~\qquad\text{resp.}\qquad~~\vcenter{\hbox{{\def\svgscale{0.6}
%% Accompanies image file 'unitary-module-3.pdf' (pdf, eps, ps)
\begingroup%
  \makeatletter%
  \providecommand\color[2][]{%
    \errmessage{(Inkscape) Color is used for the text in Inkscape, but the package 'color.sty' is not loaded}%
    \renewcommand\color[2][]{}%
  }%
  \providecommand\transparent[1]{%
    \errmessage{(Inkscape) Transparency is used (non-zero) for the text in Inkscape, but the package 'transparent.sty' is not loaded}%
    \renewcommand\transparent[1]{}%
  }%
  \providecommand\rotatebox[2]{#2}%
  \newcommand*\fsize{\dimexpr\f@size pt\relax}%
  \newcommand*\lineheight[1]{\fontsize{\fsize}{#1\fsize}\selectfont}%
  \ifx\svgwidth\undefined%
    \setlength{\unitlength}{72.79374784bp}%
    \ifx\svgscale\undefined%
      \relax%
    \else%
      \setlength{\unitlength}{\unitlength * \real{\svgscale}}%
    \fi%
  \else%
    \setlength{\unitlength}{\svgwidth}%
  \fi%
  \global\let\svgwidth\undefined%
  \global\let\svgscale\undefined%
  \makeatother%
  \begin{picture}(1,1.26603079)%
    \lineheight{1}%
    \setlength\tabcolsep{0pt}%
    \put(0,0){\includegraphics[width=\unitlength,page=1]{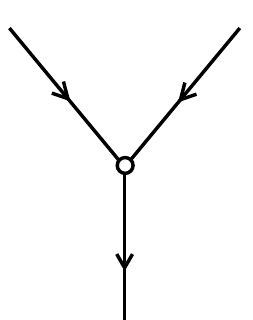}}%
    \put(0.90083283,1.22374522){\color[rgb]{0,0,0}\makebox(0,0)[lt]{\lineheight{1.25}\smash{\begin{tabular}[t]{l}$a$\end{tabular}}}}%
    \put(-0.0041588,1.21093628){\color[rgb]{0,0,0}\makebox(0,0)[lt]{\lineheight{1.25}\smash{\begin{tabular}[t]{l}$i$\end{tabular}}}}%
    \put(0.34162516,0.03314629){\color[rgb]{0,0,0}\makebox(0,0)[lt]{\lineheight{1.25}\smash{\begin{tabular}[t]{l}$i$\end{tabular}}}}%
  \end{picture}%
\endgroup%
}}}~~=~~\vcenter{\hbox{{\def\svgscale{0.6}
%% Accompanies image file 'unitary-module-4.pdf' (pdf, eps, ps)
\begingroup%
  \makeatletter%
  \providecommand\color[2][]{%
    \errmessage{(Inkscape) Color is used for the text in Inkscape, but the package 'color.sty' is not loaded}%
    \renewcommand\color[2][]{}%
  }%
  \providecommand\transparent[1]{%
    \errmessage{(Inkscape) Transparency is used (non-zero) for the text in Inkscape, but the package 'transparent.sty' is not loaded}%
    \renewcommand\transparent[1]{}%
  }%
  \providecommand\rotatebox[2]{#2}%
  \newcommand*\fsize{\dimexpr\f@size pt\relax}%
  \newcommand*\lineheight[1]{\fontsize{\fsize}{#1\fsize}\selectfont}%
  \ifx\svgwidth\undefined%
    \setlength{\unitlength}{94.759343bp}%
    \ifx\svgscale\undefined%
      \relax%
    \else%
      \setlength{\unitlength}{\unitlength * \real{\svgscale}}%
    \fi%
  \else%
    \setlength{\unitlength}{\svgwidth}%
  \fi%
  \global\let\svgwidth\undefined%
  \global\let\svgscale\undefined%
  \makeatother%
  \begin{picture}(1,0.97228552)%
    \lineheight{1}%
    \setlength\tabcolsep{0pt}%
    \put(0,0){\includegraphics[width=\unitlength,page=1]{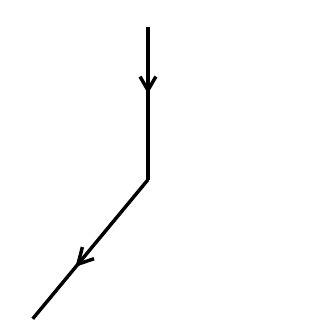}}%
    \put(0.92382018,0.93210428){\color[rgb]{0,0,0}\makebox(0,0)[lt]{\lineheight{1.25}\smash{\begin{tabular}[t]{l}$a$\end{tabular}}}}%
    \put(0.4199299,0.93980191){\color[rgb]{0,0,0}\makebox(0,0)[lt]{\lineheight{1.25}\smash{\begin{tabular}[t]{l}$i$\end{tabular}}}}%
    \put(-0.00319477,0.04778396){\color[rgb]{0,0,0}\makebox(0,0)[lt]{\lineheight{1.25}\smash{\begin{tabular}[t]{l}$i$\end{tabular}}}}%
    \put(0,0){\includegraphics[width=\unitlength,page=2]{unitary-module-4.pdf}}%
  \end{picture}%
\endgroup%
}}}~~.\label{eq40}
\end{align}	
\end{pp}
\begin{proof}
This is obvious since we have equations \eqref{eq39}.
\end{proof}

If $W_i$ and $W_j$ are left (resp. right) $A$-modules, then a morphism $F\in\Hom(W_i,W_j)$ of $\mc C$ is called a \textbf{left (resp. right) $A$-module morphism}, if
\begin{align}
\mu_L(\id_a\otimes F)=F\mu_L\qquad \text{resp.}\qquad \mu_R(F\otimes\id_a)=F\mu_R,
\end{align} 
pictorially,
\begin{align}
\vcenter{\hbox{{\def\svgscale{0.6}
%% Accompanies image file 'morphism.pdf' (pdf, eps, ps)
\begingroup%
  \makeatletter%
  \providecommand\color[2][]{%
    \errmessage{(Inkscape) Color is used for the text in Inkscape, but the package 'color.sty' is not loaded}%
    \renewcommand\color[2][]{}%
  }%
  \providecommand\transparent[1]{%
    \errmessage{(Inkscape) Transparency is used (non-zero) for the text in Inkscape, but the package 'transparent.sty' is not loaded}%
    \renewcommand\transparent[1]{}%
  }%
  \providecommand\rotatebox[2]{#2}%
  \newcommand*\fsize{\dimexpr\f@size pt\relax}%
  \newcommand*\lineheight[1]{\fontsize{\fsize}{#1\fsize}\selectfont}%
  \ifx\svgwidth\undefined%
    \setlength{\unitlength}{71.02084401bp}%
    \ifx\svgscale\undefined%
      \relax%
    \else%
      \setlength{\unitlength}{\unitlength * \real{\svgscale}}%
    \fi%
  \else%
    \setlength{\unitlength}{\svgwidth}%
  \fi%
  \global\let\svgwidth\undefined%
  \global\let\svgscale\undefined%
  \makeatother%
  \begin{picture}(1,1.29726949)%
    \lineheight{1}%
    \setlength\tabcolsep{0pt}%
    \put(0,0){\includegraphics[width=\unitlength,page=1]{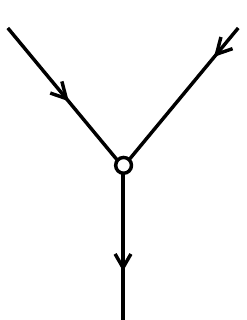}}%
    \put(-0.00426261,1.24365718){\color[rgb]{0,0,0}\makebox(0,0)[lt]{\lineheight{1.25}\smash{\begin{tabular}[t]{l}$a$\end{tabular}}}}%
    \put(0.91672779,1.25392834){\color[rgb]{0,0,0}\makebox(0,0)[lt]{\lineheight{1.25}\smash{\begin{tabular}[t]{l}$i$\end{tabular}}}}%
    \put(0.57128089,0.03610104){\color[rgb]{0,0,0}\makebox(0,0)[lt]{\lineheight{1.25}\smash{\begin{tabular}[t]{l}$j$\end{tabular}}}}%
    \put(0,0){\includegraphics[width=\unitlength,page=2]{morphism.pdf}}%
    \put(0.65105182,0.55523327){\color[rgb]{0,0,0}\makebox(0,0)[lt]{\lineheight{1.25}\smash{\begin{tabular}[t]{l}$j$\end{tabular}}}}%
  \end{picture}%
\endgroup%
}}}~~=~~\vcenter{\hbox{{\def\svgscale{0.6}
%% Accompanies image file 'morphism-2.pdf' (pdf, eps, ps)
\begingroup%
  \makeatletter%
  \providecommand\color[2][]{%
    \errmessage{(Inkscape) Color is used for the text in Inkscape, but the package 'color.sty' is not loaded}%
    \renewcommand\color[2][]{}%
  }%
  \providecommand\transparent[1]{%
    \errmessage{(Inkscape) Transparency is used (non-zero) for the text in Inkscape, but the package 'transparent.sty' is not loaded}%
    \renewcommand\transparent[1]{}%
  }%
  \providecommand\rotatebox[2]{#2}%
  \newcommand*\fsize{\dimexpr\f@size pt\relax}%
  \newcommand*\lineheight[1]{\fontsize{\fsize}{#1\fsize}\selectfont}%
  \ifx\svgwidth\undefined%
    \setlength{\unitlength}{71.02084401bp}%
    \ifx\svgscale\undefined%
      \relax%
    \else%
      \setlength{\unitlength}{\unitlength * \real{\svgscale}}%
    \fi%
  \else%
    \setlength{\unitlength}{\svgwidth}%
  \fi%
  \global\let\svgwidth\undefined%
  \global\let\svgscale\undefined%
  \makeatother%
  \begin{picture}(1,1.29726949)%
    \lineheight{1}%
    \setlength\tabcolsep{0pt}%
    \put(0,0){\includegraphics[width=\unitlength,page=1]{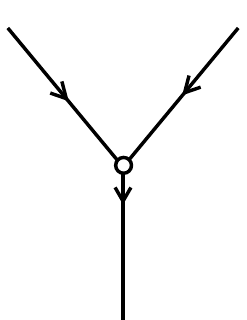}}%
    \put(-0.00426261,1.24365718){\color[rgb]{0,0,0}\makebox(0,0)[lt]{\lineheight{1.25}\smash{\begin{tabular}[t]{l}$a$\end{tabular}}}}%
    \put(0.91672779,1.25392834){\color[rgb]{0,0,0}\makebox(0,0)[lt]{\lineheight{1.25}\smash{\begin{tabular}[t]{l}$i$\end{tabular}}}}%
    \put(0.57128089,0.03610104){\color[rgb]{0,0,0}\makebox(0,0)[lt]{\lineheight{1.25}\smash{\begin{tabular}[t]{l}$j$\end{tabular}}}}%
    \put(0,0){\includegraphics[width=\unitlength,page=2]{morphism-2.pdf}}%
    \put(0.60893459,0.51010778){\color[rgb]{0,0,0}\makebox(0,0)[lt]{\lineheight{1.25}\smash{\begin{tabular}[t]{l}$i$\end{tabular}}}}%
    \put(0,0){\includegraphics[width=\unitlength,page=3]{morphism-2.pdf}}%
  \end{picture}%
\endgroup%
}}}~~\qquad\text{resp.}\qquad~~\vcenter{\hbox{{\def\svgscale{0.6}
%% Accompanies image file 'morphism-3.pdf' (pdf, eps, ps)
\begingroup%
  \makeatletter%
  \providecommand\color[2][]{%
    \errmessage{(Inkscape) Color is used for the text in Inkscape, but the package 'color.sty' is not loaded}%
    \renewcommand\color[2][]{}%
  }%
  \providecommand\transparent[1]{%
    \errmessage{(Inkscape) Transparency is used (non-zero) for the text in Inkscape, but the package 'transparent.sty' is not loaded}%
    \renewcommand\transparent[1]{}%
  }%
  \providecommand\rotatebox[2]{#2}%
  \newcommand*\fsize{\dimexpr\f@size pt\relax}%
  \newcommand*\lineheight[1]{\fontsize{\fsize}{#1\fsize}\selectfont}%
  \ifx\svgwidth\undefined%
    \setlength{\unitlength}{71.73619481bp}%
    \ifx\svgscale\undefined%
      \relax%
    \else%
      \setlength{\unitlength}{\unitlength * \real{\svgscale}}%
    \fi%
  \else%
    \setlength{\unitlength}{\svgwidth}%
  \fi%
  \global\let\svgwidth\undefined%
  \global\let\svgscale\undefined%
  \makeatother%
  \begin{picture}(1,1.30575533)%
    \lineheight{1}%
    \setlength\tabcolsep{0pt}%
    \put(0,0){\includegraphics[width=\unitlength,page=1]{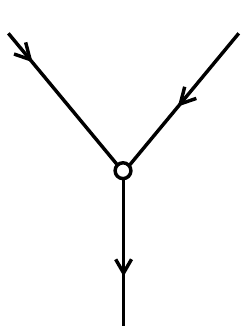}}%
    \put(0.89937088,1.26284638){\color[rgb]{0,0,0}\makebox(0,0)[lt]{\lineheight{1.25}\smash{\begin{tabular}[t]{l}$a$\end{tabular}}}}%
    \put(-0.00422011,1.24984861){\color[rgb]{0,0,0}\makebox(0,0)[lt]{\lineheight{1.25}\smash{\begin{tabular}[t]{l}$i$\end{tabular}}}}%
    \put(0.3360416,0.04205874){\color[rgb]{0,0,0}\makebox(0,0)[lt]{\lineheight{1.25}\smash{\begin{tabular}[t]{l}$j$\end{tabular}}}}%
    \put(0,0){\includegraphics[width=\unitlength,page=2]{morphism-3.pdf}}%
    \put(0.22767012,0.5707569){\color[rgb]{0,0,0}\makebox(0,0)[lt]{\lineheight{1.25}\smash{\begin{tabular}[t]{l}$j$\end{tabular}}}}%
  \end{picture}%
\endgroup%
}}}~~=~~\vcenter{\hbox{{\def\svgscale{0.6}
%% Accompanies image file 'morphism-4.pdf' (pdf, eps, ps)
\begingroup%
  \makeatletter%
  \providecommand\color[2][]{%
    \errmessage{(Inkscape) Color is used for the text in Inkscape, but the package 'color.sty' is not loaded}%
    \renewcommand\color[2][]{}%
  }%
  \providecommand\transparent[1]{%
    \errmessage{(Inkscape) Transparency is used (non-zero) for the text in Inkscape, but the package 'transparent.sty' is not loaded}%
    \renewcommand\transparent[1]{}%
  }%
  \providecommand\rotatebox[2]{#2}%
  \newcommand*\fsize{\dimexpr\f@size pt\relax}%
  \newcommand*\lineheight[1]{\fontsize{\fsize}{#1\fsize}\selectfont}%
  \ifx\svgwidth\undefined%
    \setlength{\unitlength}{72.05111099bp}%
    \ifx\svgscale\undefined%
      \relax%
    \else%
      \setlength{\unitlength}{\unitlength * \real{\svgscale}}%
    \fi%
  \else%
    \setlength{\unitlength}{\svgwidth}%
  \fi%
  \global\let\svgwidth\undefined%
  \global\let\svgscale\undefined%
  \makeatother%
  \begin{picture}(1,1.30378442)%
    \lineheight{1}%
    \setlength\tabcolsep{0pt}%
    \put(0,0){\includegraphics[width=\unitlength,page=1]{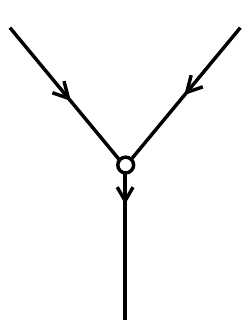}}%
    \put(0.89981071,1.26106301){\color[rgb]{0,0,0}\makebox(0,0)[lt]{\lineheight{1.25}\smash{\begin{tabular}[t]{l}$a$\end{tabular}}}}%
    \put(-0.00420166,1.25636075){\color[rgb]{0,0,0}\makebox(0,0)[lt]{\lineheight{1.25}\smash{\begin{tabular}[t]{l}$i$\end{tabular}}}}%
    \put(0.27208546,0.01146647){\color[rgb]{0,0,0}\makebox(0,0)[lt]{\lineheight{1.25}\smash{\begin{tabular}[t]{l}$j$\end{tabular}}}}%
    \put(0,0){\includegraphics[width=\unitlength,page=2]{morphism-4.pdf}}%
    \put(0.26768576,0.48759198){\color[rgb]{0,0,0}\makebox(0,0)[lt]{\lineheight{1.25}\smash{\begin{tabular}[t]{l}$i$\end{tabular}}}}%
    \put(0,0){\includegraphics[width=\unitlength,page=3]{morphism-4.pdf}}%
  \end{picture}%
\endgroup%
}}}~~.\label{eq41}
\end{align}
We let $\Hom_{A,-}(W_i,W_j)$ (resp. $\Hom_{-,A}(W_i,W_j)$) be the vector space of left (resp. right) $A$-module morphisms. If $W_i,W_j$ are $A$-bimodules, then we set $\Hom_A(W_i,W_j)=\Hom_{A,-}(W_i,W_j)\cap\Hom_{-,A}(W_i,W_j)$ to be the vector space of \textbf{A-bimodule morphisms}.\index{Hom@$\Hom_{A,-}(W_i,W_j),\Hom_{-,A}(W_i,W_j)$}\index{Hom@$\Hom_A(W_i,W_j)$} In the case $W_i=W_j$, we write these spaces of morphisms as $\End_{A,-}(W_i),\End_{-,A}(W_i),\End_A(W_i)$ respectively. $W_i$ is called a \textbf{simple} or \textbf{irreducible} left $A$-module (resp. right $A$-module, $A$-bimodule) if $\End_{A,-}(W_i)$ (resp. $\End_{-,A}(W_i)$, $\End_A(W_i)$) is spanned by $\id_i$. The following proposition is worth noting:

\begin{pp}[cf. \cite{NY16} section 6.1]\label{lb24}
The category of unitary left $A$-modules (resp. right $A$-modules, $A$-bimodules) is a $C^*$-category whose $*$-structure inherits from that of $\mc C$. In particular, this category is closed under finite orthogonal direct sums and subobjects.
\end{pp}
\begin{proof}
If $W_i,W_j$ are unitary left $A$-modules and $F\in\Hom_{A,-}(W_i,W_j)$, one can easily check that $F^*\in\Hom_{A,-}(W_j,W_i)$ using  figures \eqref{eq40} and \eqref{eq41}. Hence, the $C^*$-ness of the category of left $A$-modules follows from that of $\mc C$. Existence of finite orthogonal direct sums follow from that of $\mc C$. If $p\in\End_{A,-}(W_i)$ is a projection of the unitary left $A$-module $(W_i,\mu^i_L)$, we choose an object $W_k$  in $\mc C$ and a partial isometry $u\in\Hom(W_i,W_k)$ such that $uu^*=\id_k,u^*u=p$. Then $(W_k,u\mu^i_Lu^*)$ is easily verified to be a unitary left $A$-module. Note that the fact that $p$ intertwines the left action of $A$ is used to verify the associativity.

The cases of right modules and bimodules are proved in a similar way. 
\end{proof}

A left $A$-module (resp. right $A$-module, $A$-bimodule) $W_i$  is called \textbf{$\mc C$-dualizable} if $W_i$ is a dualizable  object in $\mc C$. $W_i$ is called \textbf{unitarizable} if there exists a unitary left $A$-module (resp. right $A$-module, $A$-bimodule) $W_j$ and an invertible $F\in\Hom_{A,-}(W_i,W_j)$ (resp. $F\in\Hom_{-,A}(W_i,W_j)$, $F\in\Hom_A(W_i,W_j)$).
\begin{co}
The category of unitary $\mc C$-dualizable left $A$-modules (resp. right $A$-modules, $A$-bimodules) is a semisimple $C^*$-category whose $*$-structure inherits from that of $\mc C$.
\end{co}

\begin{proof}
Suppose $W_i$ is a $\mc C$-dualizable left $A$-module. Then $\End_A(W_i)$ is a $C^*$-subalgebra of the finite dimensional  $C^*$-algebra $E(W_i)$. Thus $\End_A(W_i)$ is a direct sum of matrix algebras which implies that $W_i$ is a finite orthogonal direct sum  of irreducible left $A$-modules. The other types of modules are treated in a similar way.
\end{proof}

We are now going to prove the first main result of this section, that any $\mc C$-dualizable module is unitarizable. First we need a lemma.

\begin{lm}\label{lb19}
Let $W_i,W_k$ be $\mc C$-dualizable left $A$-modules (resp. right $A$-modules, $A$-bimodules). If $W_k$ is unitary, and there exists a surjective $F\in\Hom_{A,-}(W_k,W_i)$ (resp. $F\in\Hom_{-,A}(W_k,W_i)$, $F\in\Hom_A(W_k,W_i)$), then $W_i$ is unitarizable. In particular, $W_i$ is semisimple as a left $A$-module (resp. right $A$-module, $A$-bimodule).
\end{lm}

We remark that this lemma is obvious when $W_i$ is already known to be semisimple as a left, right or bi $A$-module, which is enough for our application to representations of VOA extensions.  (Indeed, the extension $U$ of $V$ considered in this paper is always regular, hence its modules are semisimple.) Those who are only interested in the application to VOAs can skip the following proof.

\begin{proof}
We only prove this for left modules, since the other cases can be proved similarly. Write the two modules as $(W_i,\mu_L^i),(W_k,\mu_L^k)$. Note that $F^*F$ is a positive element in the finite dimensional $C^*$-algebra $\End(W_k)$. So $\lim_{n\rightarrow \infty}(F^*F)^{1/n}$ converges under the $C^*$-norm to a projection $P\in\End(W_k)$  which is the range projection of $F^*$.  Set $G=P\mu_L^k$ and $H=P\mu_L^k(\id_a\otimes P)$  which are morphisms in $\Hom(W_a\boxtimes W_k,W_k)$. Then using \eqref{eq41} and the fact that $F=FP$, we obtain $FG=FH$, since both equal $F\mu_L^k$. Therefore $(F^*F)^nG=(F^*F)^nH$ for any integer $n>0$, and hence $(F^*F)^{1/n}G=(F^*F)^{1/n}H$ by polynomial interpolation. Thus $G=PG=PH=H$. We conclude
\begin{align}
P\mu_L^k=P\mu_L^k(\id_a\otimes P).\label{eq42}
\end{align}
This equation, together with  \eqref{eq40}, shows $(\id_a\otimes P)(\mu_L^k)^*=(\id_a\otimes P)(\mu_L^k)^*P$, whose adjoint is
\begin{align}
\mu_L^k(\id_a\otimes P)=P\mu_L^k(\id_a\otimes P).\label{eq43}
\end{align}
We can therefore combine \eqref{eq42} and \eqref{eq43} to get $P\mu_L^k=\mu_L^k(\id_a\otimes P)$. In other words, $P$ is a projection in $\End_{A,-}(W_k)$. Thus, by proposition \ref{lb24}, one can find a unitary left $A$-module $W_j$ and a partial isometry $K\in\Hom_{A,-}(W_j,W_k)$ satisfying $K^*K=\id_j$ and $KK^*=P$. Therefore the left $A$-module $W_i$ is equivalent to $W_j$ since $FK\in\Hom_{A,-}(W_j,W_i)$ is invertible. 
\end{proof}

\begin{thm}\label{lb21}
$\mc C$-dualizable left $A$-modules, right $A$-modules, and $A$-bimodules are unitarizable.
\end{thm}
In particular, when $\mc C$ is rigid,  any left $A$-module, right $A$-module, or $A$-bimodule is unitarizable.
\begin{proof}
For any $\mc C$-dualizable object $W_i$ in $\mc C$ the induced left $A$-module $(W_a\boxtimes W_i, \mu\otimes\id_i)$, abbreviated to $W_a\boxtimes W_i$, is clearly $\mc C$-dualizable. By the unitarity of $A$, one easily checks that $W_a\boxtimes W_i$ is a unitary left $A$-module. Now assume that $(W_i,\mu_L)$ is  a left $A$-module. Then $\mu_L\in\Hom_{A,-}(W_a\boxtimes W_i,W_i)$. Moreover, $\mu_L$ is surjective since $(\iota\otimes\id_i)\mu_L=\id_i$ is surjective. Therefore $W_i$ is unitarizable by lemma \ref{lb19}. The case of right modules is proved in a similar way. In the case that $(W_i,\mu_L,\mu_R)$ is a $\mc C$-dualizable $A$-bimodule, we notice that $(W_a\boxtimes W_i\boxtimes W_a,\mu\otimes\id_i\otimes\id_a,\id_a\otimes\id_i\otimes\mu)$ is a unitary $A$-bimodule, and $\mu_R(\mu_L\otimes\id_a)=\mu_L(\id_a\otimes\mu_R)\in\Hom_A(W_a\boxtimes W_i\boxtimes W_a,W_i)$ whose surjectivity follows again from the unit property. Thus, again, $W_i$ is a unitarizable $A$-bimodule.
\end{proof}

In the case that $A$ is special, there is another proof of unitarizability due to \cite{BKLR15} which does not require dualizability. To begin with, we say that a left  $A$-module $(W_i,\mu_L)$  is \textbf{standard} if $\mu_L\mu_L^*\in\mathbb C\id_i$. We now follow the argument of \cite{BKLR15} lemma 3.22. For any left $A$-module $(W_i,\mu_L)$,   $\Delta:=(\mu_L\mu_L^*)^{1/2}$ is invertible. Therefore $(W_i,\mu_L)$ is equivalent to $(W_i,\wtd\mu_L)$ where $\wtd\mu_L=\Delta^{-1}\mu_L(\id_a\otimes\Delta)$. Using associativity and the fact that $\mu\mu^*=d_A\id_a$, one can check that $(W_i,\wtd\mu_L)$ is standard and $\wtd\mu_L\wtd\mu_L^*=d_A\id_i$. In particular, if $(W_i,\mu_L)$ is standard then $\Delta$ is a constant and hence $\wtd\mu_L=\mu_L$. Thus we must have $\mu_L\mu_L^*=d_A\id_i$. This proves that any left $A$-module is equivalent to a standard left $A$-module, that any standard left $A$-module must satisfy $\mu_L\mu_L^*=d_A\id_a$. Moreover, by \cite{BKLR15} formula (3.4.5), any standard left $A$-module is unitary. Conversely, if $(W_i,\mu_L)$ is unitary, one can check that $\mu_L\mu_L^*\in\End_{A,-}(W_i)$ and hence $\Delta\in\End_{A,-}(W_i)$. This proves that $\wtd\mu_L=\mu_L$ and hence that $(W_i,\mu_L)$ is standard. Right $A$-modules can be proved in a similar way. When $(W_i,\mu_L,\mu_R)$ is an $A$-bimodule, we define $\mu_{LR}=\mu_R(\mu_L\otimes\id_a)=\mu_L(\id_a\otimes\mu_R)$, and say that the $A$-bimodule $W_i$ is \textbf{standard} if $\mu_{LR}\mu_{LR}^*\in\mathbb C\id_i$ (cf. \cite{BKLR15} section 3.6). With the help of $\Delta:=(\mu_{LR}\mu_{LR}^*)^{1/2}$ one can prove similar results as of left $A$-modules, with the only exception being that $\mu_{LR}\mu_{LR}^*=d_A^2\id_i$. We summarize the discussion in the following theorems:

\begin{thm}\label{lb22}
If $A$ is a Q-system in $\mc C$, then all left $A$-modules, right $A$-modules, and $A$-bimodules are unitarizable.
\end{thm}
\begin{thm}\label{lb25}
Let $A$ be a $Q$-system in $\mc C$, and $(W_i,\mu_L)$ (resp. $(W_i,\mu_R),(W_i,\mu_L,\mu_R)$) a left $A$-module (resp. right $A$-module, $A$-bimodule). Then the following statements are equivalent.
\begin{itemize}
	\item $W_i$ is unitary.
	\item $W_i$ is standard.
	\item $\mu_L\mu_L^*=d_A\id_i$ (left $A$-module case), or $\mu_R\mu_R^*=d_A\id_i$ (right $A$-module case), or $\mu_{LR}\mu_{LR}^*=d_A^2\id_i$ where $\mu_{LR}=\mu_R(\mu_L\otimes\id_a)=\mu_L(\id_a\otimes\mu_R)$ ($A$-bimodule case).
\end{itemize}
\end{thm}

If $\mc C$ is braided with braid operator $\ss$, and if $A$ is commutative, a left $A$-module $(W_i,\mu_L)$ is called \textbf{single-valued} if $\mu_L=\mu_L\ss^2$ (more precisely, $\mu_L=\mu_L\ss_{i,a}\ss_{a,i}$). If $(W_i,\mu_L)$ is a single-valued left $A$-module, then $(W_i,\mu_R)$ is a right $A$-module where $\mu_R=\mu_L\ss_{i,a}=\mu_L\ss_{a,i}^{-1}$. Moreover, by the associativity of $(W_i,\mu_L)$, $(W_i,\mu_L,\mu_R)$ is an $A$-bimodule. We summarize that when $\mc C$ is braided and $A$ is commutative, any single-valued left $A$-module is an $A$-bimodule. Moreover, if $W_i$ is a unitary single-valued left $A$-module, the it is also a unitary $A$-bimodule. The category of (unitary) single-valued left $A$-module is naturally a full subcategory of the ($C^*$-)categories of (unitary) left $A$-modules, (unitary) right $A$-modules, and (unitary) $A$-bimodules.\\

With the results obtained so far, we prove that any CFT-type unitary extension $U$ of $V$ is strongly unitary. Let $A_U=(W_a,\mu,\iota)$ be the Q-system associated to $U$. Suppose that $W_i$ is a $U$-module with vertex operator $\mc Y_{\mu_L}$. Then $W_i$ is also a $V$-module, thus me may fix an inner product $\bk{\cdot|\cdot}$ on the vector space $W_i$ under which $W_i$ becomes a unitary $V$-module. We call $W_i$, together with $\bk{\cdot|\cdot}$, a \textbf{preunitary $U$-module}.  Moreover, $\mc Y_{\mu_L}$ can be regarded as a unitary intertwining operator of $V$ of type $i\choose a~i$. Thus $\mu_L\in\Hom(W_a\boxtimes W_i,W_i)$. One can check that $(W_i,\mu_L)$ is a single-valued left $A_U$-module. Indeed, unit property is obvious, associativity follows from the Jacobi identity for $\mc Y_{\mu_L}$, and single-valued property follows from the fact that $\mc Y_{\mu_L}(\cdot,z)$ has only integer powers of $z$. Conversely, any preunitary $U$-module arises from a single-valued $A_U$-module. Thus the category of preunitary $U$-modules is naturally equivalent to the category of single-valued left $A_U$-modules. See \cite{HKL15} for more details. 

To discuss the unitarizability of $U$-modules, the following is needed:
\begin{thm}\label{lb23}
Assume that $V$ is a CFT-type, regular, and completely unitary VOA,  $U$ is a CFT-type unitary extension of $V$, and $W_i$ is a preunitary $U$-module. Then $W_i$ is a unitary $U$-module if and only if $W_i$ is a unitary left $A_U$-module. If this is true then $W_i$ is also a unitary $A_U$-bimodule.
\end{thm}
\begin{proof}
Let $W_{\ovl a}$ be the contragredient module of $W_a$, $\ev_{a,\ovl a},\ev_{\ovl a,a}$ the evaluations of $W_a,W_{\ovl a}$ defined in chapter \ref{lb7}, and $\rfl:U=W_a\rightarrow \ovl U=W_{\ovl a}$ the reflection operator with respect to the chosen dual and evaluations. By equation \eqref{eq8} and the definition of adjoint intertwining operators, $W_i$ is unitary if and only if $\mc Y_{\mu_L}(w^{(a)},z)=\mc Y_{\mu_L^\dagger}(\rfl w^{(a)},z)$ for any $w^{(a)}\in W_a=U$. From equation \eqref{eq31} we know that $W_i$ is unitary if and only if $\mu_L=\mu_L^\dagger(\rfl\otimes\id_i)$. With the help of theorem \ref{lb6}, this equation is equivalent to $\mu_L=(\ev_{\ovl a,a}\otimes\id_i)(\rfl\otimes\mu_L^*)$, which by proposition \ref{lb20} means precisely the unitarity of the left $A_U$-module $W_i$. If we already have that $W_i$ is a unitary left $A_U$-module, then, since $W_i$ is single-valued, it is also a unitary $A$-bimodule. 
\end{proof}

We now prove the strong unitarity of $U$.

\begin{thm}\label{lb53}
If $V$ is a CFT-type, regular, and completely unitary VOA, and $U$ is a CFT-type unitary extension of $V$, then $U$ is strongly unitary, i.e., any $U$-module is unitarizable.
\end{thm}

\begin{proof}
Since $U$-modules are clearly preunitarizable, we choose a preunitary $U$-module $W_i$. Then by either theorem \ref{lb21} or theorem \ref{lb22}, $W_i$ is unitarizable as a left $A_U$-module. Thus, by equation \eqref{eq31} and theorem \ref{lb23}, $W_i$ is unitarizable as a $U$-module.
\end{proof}

\section{$C^*$-tensor categories associated to Q-systems and unitary VOA extensions}

\subsection{Unitary tensor products of unitary bimodules of Q-systems}

In this chapter, $\mc C$ is a $C^*$-tensor category with simple $W_0$ as before, and $A=(W_a,\mu,\iota)$ is a  Q-system (i.e., special $C^*$-Frobenius algebra) in $\mc C$. Set evaluation $\ev_{a,a}=\iota^*\mu$ as usual. We suppress the label $a$ in diagram calculus. Let $\BIMA$ be the $C^*$-category of unitary $A$-bimodules whose morphisms are $A$-bimodule morphisms. \index{BIMA@$\BIMA$} In this and the next sections, we review the construction of a $C^*$-tensor structure on $\BIMA$. See \cite{NY16,KO02,CKM17} for reference. Note that our setting is slightly more general than that of \cite{NY16}, since we do not assume $\mc C$ is rigid or $A$ is standard. Nevertheless, many ideas in \cite{NY16} still work in our setting. To make  our article self-contained, we include detailed proofs for all the relevant results. 

Choose unitary $A$-bimodules $(W_i,\mu^i_L,\mu^i_R),(W_j,\mu^j_L,\mu^j_R)$. Then $W_i\boxtimes W_j$ is a unitary $A$-bimodule with left action $\mu^i_L\otimes\id_j$ and right action $\id_i\otimes\mu^j_R$. Define  $\Pij\in\Hom_A(W_i\boxtimes W_a\boxtimes W_j,W_i\boxtimes W_j)$ \index{zz@$\Pij$} and $\Cij\in\End_A(W_i\boxtimes W_j)$ \index{zz@$\Cij$} to be
\begin{gather}
\Pij=\mu^i_R\otimes\id_j-\id_i\otimes\mu^j_L,\\
\Cij=(\id_i\otimes\mu^j_L)((\mu^i_R)^*\otimes\id_j)=(\mu^i_R\otimes\id_j)(\id_i\otimes(\mu^j_L)^*).
\end{gather}
 
\begin{df}\label{lb47}
Let $W_i,W_j$ be unitary $A$-bimodules. We say that $(W_{ij},\mu_{i,j})$\index{Wij@$(W_{ij},\mu_{i,j})$} (abbreviated to $W_{ij}$ when no confusion arises) is a \textbf{tensor product} of $W_i, W_j$ over $A$ (cf.  \cite{CKM17}), if
\begin{itemize}
\item  $W_{ij}=(W_{ij},\mu^{ij}_L,\mu^{ij}_R)$ is an $A$-bimodule,  $\mu_{i,j}\in\Hom_A(W_i\boxtimes W_j,W_{ij})$,\footnote{One should not confuse $i\boxtimes j$ and $ij$. By our notation,  $W_{i\boxtimes j}=W_i\boxtimes W_j$ is different from $W_{ij}$.} and $\mu_{i,j}\Pij=0$.
\item (Universal property) If $(W_k,\mu^k_L,\mu^k_R)$ is a unitary $A$-bimodule, $\alpha\in\Hom_A(W_i\boxtimes W_j,W_k)$, and $\alpha\Pij=0$,\footnote{Such $\alpha$ is called a categorical intertwining operator in \cite{CKM17}.} then there exists a unique $\wtd\alpha\in\Hom_A(W_{ij},W_k)$ satisfying $\alpha=\wtd \alpha\mu_{i,j}$. In this case, we say that $\wtd\alpha$ is \textbf{induced by} $\alpha$ via the tensor product $W_{ij}$.
\end{itemize}
The tensor product $(W_{ij},\mu_{i,j})$ is called \textbf{unitary} if $W_{ij}$ is a unitary $A$-bimodule and $\Cij=\mu_{i,j}^*\mu_{i,j}$. 
\end{df}
We write $\mu_{i,j}=\vcenter{\hbox{{\def\svgscale{0.4}
%% Accompanies image file 'tensor-product.pdf' (pdf, eps, ps)
\begingroup%
  \makeatletter%
  \providecommand\color[2][]{%
    \errmessage{(Inkscape) Color is used for the text in Inkscape, but the package 'color.sty' is not loaded}%
    \renewcommand\color[2][]{}%
  }%
  \providecommand\transparent[1]{%
    \errmessage{(Inkscape) Transparency is used (non-zero) for the text in Inkscape, but the package 'transparent.sty' is not loaded}%
    \renewcommand\transparent[1]{}%
  }%
  \providecommand\rotatebox[2]{#2}%
  \newcommand*\fsize{\dimexpr\f@size pt\relax}%
  \newcommand*\lineheight[1]{\fontsize{\fsize}{#1\fsize}\selectfont}%
  \ifx\svgwidth\undefined%
    \setlength{\unitlength}{54.70314335bp}%
    \ifx\svgscale\undefined%
      \relax%
    \else%
      \setlength{\unitlength}{\unitlength * \real{\svgscale}}%
    \fi%
  \else%
    \setlength{\unitlength}{\svgwidth}%
  \fi%
  \global\let\svgwidth\undefined%
  \global\let\svgscale\undefined%
  \makeatother%
  \begin{picture}(1,1.56885026)%
    \lineheight{1}%
    \setlength\tabcolsep{0pt}%
    \put(0,0){\includegraphics[width=\unitlength,page=1]{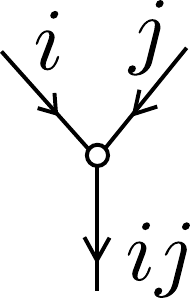}}%
  \end{picture}%
\endgroup%
}}}~~,~\mu_{i,j}^*=\vcenter{\hbox{{\def\svgscale{0.4}
%% Accompanies image file 'tensor-product-2.pdf' (pdf, eps, ps)
\begingroup%
  \makeatletter%
  \providecommand\color[2][]{%
    \errmessage{(Inkscape) Color is used for the text in Inkscape, but the package 'color.sty' is not loaded}%
    \renewcommand\color[2][]{}%
  }%
  \providecommand\transparent[1]{%
    \errmessage{(Inkscape) Transparency is used (non-zero) for the text in Inkscape, but the package 'transparent.sty' is not loaded}%
    \renewcommand\transparent[1]{}%
  }%
  \providecommand\rotatebox[2]{#2}%
  \newcommand*\fsize{\dimexpr\f@size pt\relax}%
  \newcommand*\lineheight[1]{\fontsize{\fsize}{#1\fsize}\selectfont}%
  \ifx\svgwidth\undefined%
    \setlength{\unitlength}{65.70358555bp}%
    \ifx\svgscale\undefined%
      \relax%
    \else%
      \setlength{\unitlength}{\unitlength * \real{\svgscale}}%
    \fi%
  \else%
    \setlength{\unitlength}{\svgwidth}%
  \fi%
  \global\let\svgwidth\undefined%
  \global\let\svgscale\undefined%
  \makeatother%
  \begin{picture}(1,1.14148049)%
    \lineheight{1}%
    \setlength\tabcolsep{0pt}%
    \put(0,0){\includegraphics[width=\unitlength,page=1]{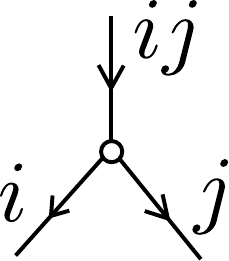}}%
  \end{picture}%
\endgroup%
}}}~~$. \index{Wij@$(W_{ij},\mu_{i,j})$!$\mu_{i,j}=\vcenter{\hbox{{\def\svgscale{0.4}
	}}}~~,~\mu_{i,j}^*=\vcenter{\hbox{{\def\svgscale{0.4}
	}}}~~$} Then the equation $\Cij=\mu_{i,j}^*\mu_{i,j}$ reads
\begin{align}
\vcenter{\hbox{{\def\svgscale{0.6}
%% Accompanies image file 'Frobenius-4.pdf' (pdf, eps, ps)
\begingroup%
  \makeatletter%
  \providecommand\color[2][]{%
    \errmessage{(Inkscape) Color is used for the text in Inkscape, but the package 'color.sty' is not loaded}%
    \renewcommand\color[2][]{}%
  }%
  \providecommand\transparent[1]{%
    \errmessage{(Inkscape) Transparency is used (non-zero) for the text in Inkscape, but the package 'transparent.sty' is not loaded}%
    \renewcommand\transparent[1]{}%
  }%
  \providecommand\rotatebox[2]{#2}%
  \newcommand*\fsize{\dimexpr\f@size pt\relax}%
  \newcommand*\lineheight[1]{\fontsize{\fsize}{#1\fsize}\selectfont}%
  \ifx\svgwidth\undefined%
    \setlength{\unitlength}{57.98234608bp}%
    \ifx\svgscale\undefined%
      \relax%
    \else%
      \setlength{\unitlength}{\unitlength * \real{\svgscale}}%
    \fi%
  \else%
    \setlength{\unitlength}{\svgwidth}%
  \fi%
  \global\let\svgwidth\undefined%
  \global\let\svgscale\undefined%
  \makeatother%
  \begin{picture}(1,2.00826707)%
    \lineheight{1}%
    \setlength\tabcolsep{0pt}%
    \put(0,0){\includegraphics[width=\unitlength,page=1]{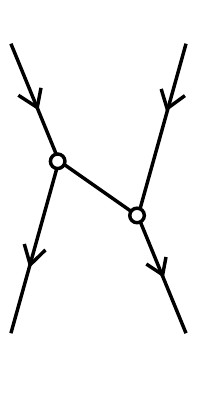}}%
    \put(0.02479523,1.9281425){\color[rgb]{0,0,0}\makebox(0,0)[lt]{\lineheight{1.25}\smash{\begin{tabular}[t]{l}$i$\end{tabular}}}}%
    \put(0.85636755,1.95123275){\color[rgb]{0,0,0}\makebox(0,0)[lt]{\lineheight{1.25}\smash{\begin{tabular}[t]{l}$j$\end{tabular}}}}%
    \put(0.84950541,0.02933658){\color[rgb]{0,0,0}\makebox(0,0)[lt]{\lineheight{1.25}\smash{\begin{tabular}[t]{l}$j$\end{tabular}}}}%
    \put(-0.00544197,0.01532669){\color[rgb]{0,0,0}\makebox(0,0)[lt]{\lineheight{1.25}\smash{\begin{tabular}[t]{l}$i$\end{tabular}}}}%
  \end{picture}%
\endgroup%
}}}~~~=~~\vcenter{\hbox{{\def\svgscale{0.6}
%% Accompanies image file 'Frobenius-5.pdf' (pdf, eps, ps)
\begingroup%
  \makeatletter%
  \providecommand\color[2][]{%
    \errmessage{(Inkscape) Color is used for the text in Inkscape, but the package 'color.sty' is not loaded}%
    \renewcommand\color[2][]{}%
  }%
  \providecommand\transparent[1]{%
    \errmessage{(Inkscape) Transparency is used (non-zero) for the text in Inkscape, but the package 'transparent.sty' is not loaded}%
    \renewcommand\transparent[1]{}%
  }%
  \providecommand\rotatebox[2]{#2}%
  \newcommand*\fsize{\dimexpr\f@size pt\relax}%
  \newcommand*\lineheight[1]{\fontsize{\fsize}{#1\fsize}\selectfont}%
  \ifx\svgwidth\undefined%
    \setlength{\unitlength}{56.75315313bp}%
    \ifx\svgscale\undefined%
      \relax%
    \else%
      \setlength{\unitlength}{\unitlength * \real{\svgscale}}%
    \fi%
  \else%
    \setlength{\unitlength}{\svgwidth}%
  \fi%
  \global\let\svgwidth\undefined%
  \global\let\svgscale\undefined%
  \makeatother%
  \begin{picture}(1,2.05138488)%
    \lineheight{1}%
    \setlength\tabcolsep{0pt}%
    \put(0,0){\includegraphics[width=\unitlength,page=1]{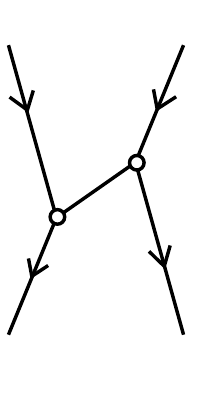}}%
    \put(0.00490101,1.98672735){\color[rgb]{0,0,0}\makebox(0,0)[lt]{\lineheight{1.25}\smash{\begin{tabular}[t]{l}$i$\end{tabular}}}}%
    \put(0.88404538,2.0015102){\color[rgb]{0,0,0}\makebox(0,0)[lt]{\lineheight{1.25}\smash{\begin{tabular}[t]{l}$j$\end{tabular}}}}%
    \put(-0.00540252,0.01226656){\color[rgb]{0,0,0}\makebox(0,0)[lt]{\lineheight{1.25}\smash{\begin{tabular}[t]{l}$i$\end{tabular}}}}%
    \put(0.8803968,0.02838214){\color[rgb]{0,0,0}\makebox(0,0)[lt]{\lineheight{1.25}\smash{\begin{tabular}[t]{l}$j$\end{tabular}}}}%
  \end{picture}%
\endgroup%
}}}~~~=~~\vcenter{\hbox{{\def\svgscale{0.6}
%% Accompanies image file 'Frobenius-6.pdf' (pdf, eps, ps)
\begingroup%
  \makeatletter%
  \providecommand\color[2][]{%
    \errmessage{(Inkscape) Color is used for the text in Inkscape, but the package 'color.sty' is not loaded}%
    \renewcommand\color[2][]{}%
  }%
  \providecommand\transparent[1]{%
    \errmessage{(Inkscape) Transparency is used (non-zero) for the text in Inkscape, but the package 'transparent.sty' is not loaded}%
    \renewcommand\transparent[1]{}%
  }%
  \providecommand\rotatebox[2]{#2}%
  \newcommand*\fsize{\dimexpr\f@size pt\relax}%
  \newcommand*\lineheight[1]{\fontsize{\fsize}{#1\fsize}\selectfont}%
  \ifx\svgwidth\undefined%
    \setlength{\unitlength}{57.60221093bp}%
    \ifx\svgscale\undefined%
      \relax%
    \else%
      \setlength{\unitlength}{\unitlength * \real{\svgscale}}%
    \fi%
  \else%
    \setlength{\unitlength}{\svgwidth}%
  \fi%
  \global\let\svgwidth\undefined%
  \global\let\svgscale\undefined%
  \makeatother%
  \begin{picture}(1,2.00623005)%
    \lineheight{1}%
    \setlength\tabcolsep{0pt}%
    \put(0,0){\includegraphics[width=\unitlength,page=1]{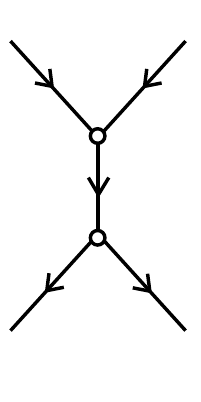}}%
    \put(0.00645279,1.91551817){\color[rgb]{0,0,0}\makebox(0,0)[lt]{\lineheight{1.25}\smash{\begin{tabular}[t]{l}$i$\end{tabular}}}}%
    \put(0.84873084,1.94101148){\color[rgb]{0,0,0}\makebox(0,0)[lt]{\lineheight{1.25}\smash{\begin{tabular}[t]{l}$j$\end{tabular}}}}%
    \put(-0.0070344,0.01493903){\color[rgb]{0,0,0}\makebox(0,0)[lt]{\lineheight{1.25}\smash{\begin{tabular}[t]{l}$i$\end{tabular}}}}%
    \put(0.85622598,0.02207274){\color[rgb]{0,0,0}\makebox(0,0)[lt]{\lineheight{1.25}\smash{\begin{tabular}[t]{l}$j$\end{tabular}}}}%
    \put(0.58454978,1.02386227){\color[rgb]{0,0,0}\makebox(0,0)[lt]{\lineheight{1.25}\smash{\begin{tabular}[t]{l}$ij$\end{tabular}}}}%
  \end{picture}%
\endgroup%
}}}~~~,\label{eq44}
\end{align}
which is a special case of the Frobenius relations for unitary tensor products to be proved later (theorem \ref{lb32}). (Note that the first equality of \eqref{eq44} follows from the unitarity of $W_i$ and $W_j$.)

The existence of a tensor product is clear if one assumes moreover that $\mc C$ is abelian: Let $W_{ij}$ be a cokernel of $\Pij$, and define the bimodule structure on $W_{ij}$ using that of $W_i\boxtimes W_j$. Then $W_{ij}$ becomes a tensor product over $A$ of $W_i$ and $W_j$; see \cite{KO02} for more details. However, to make the tensor product unitary one has to be more careful when choosing the cokernel. In the following, we proceed in a slightly different way motivated by  \cite{BKLR15} section 3.7, and we do not require the abelianess of $\mc C$. To begin with,  using $\mu\mu^*=d_A\id_a$ one verifies easily that $\Cij^2=d_A\Cij$. Therefore,
\begin{lm}[\cite{BKLR15} lemma 3.36]\label{lb26}
$d_A^{-1}\Cij\in\End_A(W_i\boxtimes W_j)$ is a projection.
\end{lm}
By proposition \ref{lb24}, there exists a unitary $A$-bimodule $W_{ij}$ and a partial isometry $u_{i,j}\in\Hom_A(W_i\boxtimes W_j,W_{ij})$ satisfying $u_{i,j}u_{i,j}^*=\id_{ij}$ and $u_{i,j}^*u_{i,j}=d_A^{-1}\Cij$. Setting $\mu_{i,j}=\sqrt{d_A}u_{i,j}$, one obtains $\mu_{i,j}^*\mu_{i,j}=\Cij$ (equations \eqref{eq44}) and  $\mu_{i,j}\mu_{i,j}^*=d_A\id_{ij}$. We now show that $(W_{ij},\mu_{i,j})$ is a unitary tensor product.

\begin{pp}\label{lb27}
Let $W_{ij}$ be a unitary $A$-bimodule, and $\mu_{i,j}\in\Hom_A(W_i\boxtimes W_j,W_{ij})$. Then $(W_{ij},\mu_{i,j})$ is a unitary tensor product of $W_i,W_j$ over $A$ if and only if  $\mu_{i,j}^*\mu_{i,j}=\Cij$ and $\mu_{i,j}\mu_{i,j}^*=d_A\id_{ij}$.
\end{pp}
\begin{proof}
``If'': Since $\mu_{i,j}^*\mu_{i,j}=\Cij$ and the unitarity of the $A$-bimodule $W_{ij}$ are assumed, it suffices to show that $W_{ij}$ is a tensor product. Since $\Cij\Pij$ clearly equals $0$, we compute
\begin{align}
\mu_{i,j}\Pij=d_A^{-1}\mu_{i,j}\mu_{i,j}^*\mu_{i,j}\Pij=d_A^{-1}\mu_{i,j}\Cij\Pij=0.
\end{align}
If $W_k$ is a unitary $A$-bimodule, $\alpha\in\Hom_A(W_i\boxtimes W_j,W_k)$, and $\alpha\Pij=0$, then one can set $\wtd\alpha=d_A^{-1}\alpha\mu_{i,j}^*$ and compute
\begin{align*}
\wtd\alpha\mu_{i,j}=d_A^{-1}\alpha\mu_{i,j}^*\mu_{i,j}=d_A^{-1}\alpha\Cij=d_A^{-1}\alpha(\mu^i_R(\mu^i_R)^*\otimes\id_j)=\alpha,
\end{align*}
where we have used $\alpha\Pij=0$ and
$\mu^i_R(\mu^i_R)^*=d_A\id_i$ (theorem \ref{lb25}) respectively to prove the third and the fourth equalities. If there is another $\wht\alpha$ satisfying also $\alpha=\wht\alpha\mu_{i,j}$, then $\wht\alpha=d_A^{-1}\wht\alpha\mu_{i,j}\mu_{i,j}^*=d_A^{-1}\alpha\mu_{i,j}^*=\wtd\alpha$. Thus the universal property is checked.

``Only if": This will be proved after the next theorem.
\end{proof}

\begin{thm}\label{lb28}
Let $W_i,W_j$ be unitary $A$-bimodules. Then unitary tensor products over $A$ of $W_i,W_j$ exist and are unique up to unitaries. More precisely,  uniqueness means that if $(W_{ij},\mu_{i,j})$ and $(W_{i\bullet j},\eta_{i,j})$ are unitary tensor products of $W_i,W_j$ over $A$, then there exists a (unique) unitary $u\in\Hom_A(W_{ij},W_{i\bullet j})$ such that $\eta_{i,j}=u\mu_{i,j}$.
\end{thm}
\begin{proof}
Existence has already been proved. We now prove the uniqueness. Since $\eta_{i,j}\in\Hom_A(W_i\boxtimes W_j,W_{i\bullet j})$ is annihilated by $\Pij$, by the universal property for $(W_{ij},\mu_{i,j})$ there exists a unique $u\in\Hom_A(W_{ij},W_{i\bullet j})$ satisfying $\eta_{i,j}=u\mu_{i,j}$. In other words $u$ is the $A$-bimodule morphism induced by $\eta_{i,j}$. It remains to prove that $u$ is unitary.

We first show that $u$ is invertible. By the universal property for $(W_{i\bullet j},\eta_{i,j})$, there exists $v\in\Hom_A(W_{i\bullet j},W_{ij})$ such that $\mu_{i,j}=v\eta_{i,j}$. Thus $\eta_{i,j}=uv\eta_{i,j}$. Therefore,  $uv$ is induced by $\eta_{i,j}$ via the tensor product $W_{i\bullet j}$. But $\id_{i\bullet j}$ is clearly also induced by $\eta_{i,j}$ via $W_{i\bullet j}$. Therefore $uv=\id_{i\bullet j}$. Similarly $vu=\id_{ij}$. This proves that $u$ is invertible. 

We now calculate
\begin{align*}
\mu_{i,j}^*u^*u\mu_{i,j}=\eta_{i,j}^*\eta_{i,j}=\Cij=\mu_{i,j}^*\mu_{i,j}.
\end{align*}
By the universal property, $\mu_{i,j}^*u^*u$ and $\mu_{i,j}^*$ are equal since they are induced by the same morphism via $W_{ij}$. So $u^*u\mu_{i,j}=\mu_{i,j}$. By the universal property again, we have $u^*u=\id_{ij}$. Therefore $u$ is unitary.
\end{proof}

\begin{proof}[Proof of the ``only if" part of Proposition \ref{lb27}]
By the ``if" part of proposition \ref{lb27} and the paragraph before that, there exists a unitary tensor product $(W_{ij},\mu_{i,j})$ satisfying $\mu_{i,j}\mu_{i,j}^*=d_A\id_{ij}$. By uniqueness up to unitaries, this equation holds for any unitary tensor product.
\end{proof}

In the remaining part of this section, we generalize the notion of unitary tensor product to more than two unitary $A$-bimodules. For simplicity we only discuss the case of three bimodules. The more general cases can be treated in a similar fashion and are thus left to the reader. 

Choose unitary $A$-bimodules $W_i,W_j,W_k$ with left actions $\mu^i_L,\mu^j_L,\mu^k_L$ and right actions $\mu^i_R,\mu^j_R,\mu^k_R$ respectively. Then $(W_i\boxtimes W_j\boxtimes W_k,\mu^i_L\otimes\id_j\otimes\id_k,\id_i\otimes\id_j\otimes\mu^k_R)$ is a unitary $A$-bimodule.
\begin{lm}
$\Cij\otimes\id_k$ and $\id_i\otimes\Cjk$ commute. Define $\Cijk\in\End_A(W_i\boxtimes W_j\boxtimes W_k)$ \index{zz@$\Cijk$} to be their product. Then $d_A^{-2}\Cijk$ is a projection. 
\end{lm}
\begin{proof}
The commutativity of these two morphisms is verified using the commutativity of the left and right actions of $W_j$. Thus $d_A^{-1}\Cij\otimes\id_k$ and $d_A^{-1}\id_i\otimes\Cjk$ are commuting projections, whose product is therefore also a projection.
\end{proof}

\begin{df}
$(W_{ijk},\mu_{i,j,k})$  (or $W_{ijk}$ for short) is called a \textbf{unitary tensor product} of $W_i, W_j,W_k$ over $A$, if
\begin{itemize}
\item  $W_{ijk}=(W_{ijk},\mu^{ijk}_L,\mu^{ijk}_R)$ is a unitary $A$-bimodule,  $\mu_{i,j,k}\in\Hom_A(W_i\boxtimes W_j\boxtimes W_k,W_{ijk})$, and $\mu_{i,j,k}(\Pij\otimes\id_k)=\mu_{i,j,k}(\id_i\otimes \Pjk)=0$.
\item (Universal property) If $(W_l,\mu^l_L,\mu^l_R)$ is a unitary $A$-bimodule, $\alpha\in\Hom_A(W_i\boxtimes W_j\boxtimes W_k,W_l)$, and $\alpha(\Pij\otimes\id_k)=\alpha(\id_i\otimes\Pjk)=0$, then there exists a unique $\wtd\alpha\in\Hom_A(W_{ijk},W_l)$ satisfying $\alpha=\wtd \alpha\mu_{i,j,k}$. In this case, we say that $\wtd\alpha$ is \textbf{induced by} $\alpha$ via the tensor product $W_{ijk}$.
\item (Unitarity) $\Cijk=\mu_{i,j,k}\mu_{i,j,k}^*$.
\end{itemize}
\end{df}

\begin{pp}\label{lb29}
Let $W_{ijk}$ be a unitary $A$-bimodules, and $\mu_{i,j,k}\in\Hom_A(W_i\boxtimes W_j\boxtimes W_k,W_{ijk})$. Then $(W_{ijk},\mu_{i,j,k})$ is a unitary tensor product of $W_i,W_j,W_k$ over $A$ if and only if  $\mu_{i,j,k}^*\mu_{i,j,k}=\Cijk$ and $\mu_{i,j,k}\mu_{i,j,k}^*=d_A^2\id_{ijk}$.
\end{pp}
\begin{thm}\label{lb30}
Unitary tensor products of $W_i,W_j,W_k$ exist and are unique up to unitaries.
\end{thm}
We omit the proofs of these two results since they can be proved in a similar way as proposition \ref{lb27} and theorem \ref{lb28}.

\subsection{$C^*$-tensor categories associated to Q-systems}\label{lb52}

We are now ready to define the unitary tensor structure on the $C^*$-category $\BIMA$ of unitary $A$-bimodules. The tensor bifunctor $\boxtimes_A$ is defined as follows. For any unitary $A$-bimodules $W_i,W_j$, we choose a unitary tensor product $(W_{ij},\mu_{i,j})$. Then $W_i\boxtimes_A W_j$ is just the unitary $A$-bimodule $W_{ij}$. To define tensor product of morphisms, we choose another pair of unitary $A$-bimodules $W_{i'},W_{j'}$, and choose any $F\in\Hom_A(W_i,W_{i'})$ and $G\in\Hom_A(W_j,W_{j'})$. Of course, there is also a chosen unitary tensor product $(W_{i'j'},\mu_{i',j'})$ of $W_{i'},W_{j'}$ over $A$. Since $F\otimes G:W_i\boxtimes W_j\rightarrow W_{i'}\boxtimes W_{j'}$ is clearly an $A$-bimodule morphism, we have $\mu_{i',j'}(F\otimes G)\in\Hom_A(W_i\boxtimes W_j,W_{i'j'})$, and one can easily show that  $\mu_{i',j'}(F\otimes G)\Pij=0$. Therefore, by universal property, there exists a unique morphism in $\Hom_A(W_{ij},W_{i'j'})$, denoted by $F\otimes_A G$, such that
\begin{align}
\mu_{i',j'}(F\otimes G)=(F\otimes_A G)\mu_{i,j}.\label{eq45}
\end{align}
This defines the tensor product of $F$ and $G$ in $\BIMA$. We now show that $\boxtimes_A$ is a $*$-bifunctor. Notice that $F^*\otimes_A G^*$ is defined by $\mu_{i,j}(F^*\otimes G^*)=(F^*\otimes_A G^*)\mu_{i',j'}$. Therefore, using $(F\otimes G)^*=F^*\otimes G^*$  we compute
\begin{align*}
&(F\otimes_A G)^*=d_A^{-1}\mu_{ij}\mu_{ij}^*(F\otimes_A G)^*=d_A^{-1}\mu_{ij}(F\otimes G)^*\mu_{i',j'}^*=d_A^{-1}\mu_{ij}(F^*\otimes G^*)\mu_{i',j'}^*\\
=&d_A^{-1}(F^*\otimes_A G^*)\mu_{i',j'}\mu_{i',j'}^*=F^*\otimes_A G^*.
\end{align*}

To construct associativity isomorphisms we need the following:

\begin{pp}\label{lb31}
Let $W_i,W_j,W_k$ be unitary $A$-bimodules. Then $(W_{(ij)k},\mu_{ij,k}(\mu_{i,j}\otimes\id_k))$ and $(W_{i(jk)},\mu_{i,jk}(\id_i\otimes\mu_{j,k}))$ are unitary tensor products of $W_i,W_j,W_k$ over $A$.
\end{pp}
Note that here $W_{(ij)k}$ is understood as the unitary tensor product of $W_{ij}$ and $W_k$ over $A$, and $W_{i(jk)}$ is understood similarly.
\begin{proof}
The two cases can be treated in a similar way. So we only prove the first one. Set $\mu_{i,j,k}=\mu_{ij,k}(\mu_{i,j}\otimes\id_k)$. By proposition \ref{lb29}, it suffices to prove $\mu_{i,j,k}^*\mu_{i,j,k}=\Cijk$ and $\mu_{i,j,k}\mu_{i,j,k}^*=d_A^2\id_{ijk}$. The second equation follows directly from that $\mu_{i,j}\mu_{i,j}^*=d_A\id_{ij}$ and $\mu_{ij,k}\mu_{ij,k}^*=d_A\id_{(ij)k}$. To prove the first one, we compute (recalling that we have suppressed the label $a$)
\begin{align*}
&\mu_{i,j,k}^*\mu_{i,j,k}=~~~\vcenter{\hbox{{\def\svgscale{0.9}
%% Accompanies image file 'Associativity-proof.pdf' (pdf, eps, ps)
\begingroup%
  \makeatletter%
  \providecommand\color[2][]{%
    \errmessage{(Inkscape) Color is used for the text in Inkscape, but the package 'color.sty' is not loaded}%
    \renewcommand\color[2][]{}%
  }%
  \providecommand\transparent[1]{%
    \errmessage{(Inkscape) Transparency is used (non-zero) for the text in Inkscape, but the package 'transparent.sty' is not loaded}%
    \renewcommand\transparent[1]{}%
  }%
  \providecommand\rotatebox[2]{#2}%
  \newcommand*\fsize{\dimexpr\f@size pt\relax}%
  \newcommand*\lineheight[1]{\fontsize{\fsize}{#1\fsize}\selectfont}%
  \ifx\svgwidth\undefined%
    \setlength{\unitlength}{55.26681141bp}%
    \ifx\svgscale\undefined%
      \relax%
    \else%
      \setlength{\unitlength}{\unitlength * \real{\svgscale}}%
    \fi%
  \else%
    \setlength{\unitlength}{\svgwidth}%
  \fi%
  \global\let\svgwidth\undefined%
  \global\let\svgscale\undefined%
  \makeatother%
  \begin{picture}(1,1.76175478)%
    \lineheight{1}%
    \setlength\tabcolsep{0pt}%
    \put(0,0){\includegraphics[width=\unitlength,page=1]{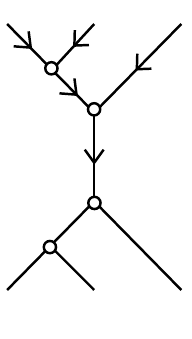}}%
    \put(0.00848576,1.71998299){\color[rgb]{0,0,0}\makebox(0,0)[lt]{\lineheight{1.25}\smash{\begin{tabular}[t]{l}$i$\end{tabular}}}}%
    \put(0.43785294,1.71995795){\color[rgb]{0,0,0}\makebox(0,0)[lt]{\lineheight{1.25}\smash{\begin{tabular}[t]{l}$j$\end{tabular}}}}%
    \put(0.90256783,1.71995795){\color[rgb]{0,0,0}\makebox(0,0)[lt]{\lineheight{1.25}\smash{\begin{tabular}[t]{l}$k$\end{tabular}}}}%
    \put(-0.00410827,0.01123663){\color[rgb]{0,0,0}\makebox(0,0)[lt]{\lineheight{1.25}\smash{\begin{tabular}[t]{l}$i$\end{tabular}}}}%
    \put(0.4252597,0.01121159){\color[rgb]{0,0,0}\makebox(0,0)[lt]{\lineheight{1.25}\smash{\begin{tabular}[t]{l}$j$\end{tabular}}}}%
    \put(0.88997459,0.01121159){\color[rgb]{0,0,0}\makebox(0,0)[lt]{\lineheight{1.25}\smash{\begin{tabular}[t]{l}$k$\end{tabular}}}}%
    \put(0,0){\includegraphics[width=\unitlength,page=2]{Associativity-proof.pdf}}%
    \put(0.58664624,0.90673017){\color[rgb]{0,0,0}\makebox(0,0)[lt]{\lineheight{1.25}\smash{\begin{tabular}[t]{l}$(ij)k$\end{tabular}}}}%
    \put(0.16334872,1.07596059){\color[rgb]{0,0,0}\makebox(0,0)[lt]{\lineheight{1.25}\smash{\begin{tabular}[t]{l}$ij$\end{tabular}}}}%
    \put(0.16336163,0.70556391){\color[rgb]{0,0,0}\makebox(0,0)[lt]{\lineheight{1.25}\smash{\begin{tabular}[t]{l}$ij$\end{tabular}}}}%
  \end{picture}%
\endgroup%
}}}~~~~~=~~~\vcenter{\hbox{{\def\svgscale{0.9}
%% Accompanies image file 'Associativity-proof-1.pdf' (pdf, eps, ps)
\begingroup%
  \makeatletter%
  \providecommand\color[2][]{%
    \errmessage{(Inkscape) Color is used for the text in Inkscape, but the package 'color.sty' is not loaded}%
    \renewcommand\color[2][]{}%
  }%
  \providecommand\transparent[1]{%
    \errmessage{(Inkscape) Transparency is used (non-zero) for the text in Inkscape, but the package 'transparent.sty' is not loaded}%
    \renewcommand\transparent[1]{}%
  }%
  \providecommand\rotatebox[2]{#2}%
  \newcommand*\fsize{\dimexpr\f@size pt\relax}%
  \newcommand*\lineheight[1]{\fontsize{\fsize}{#1\fsize}\selectfont}%
  \ifx\svgwidth\undefined%
    \setlength{\unitlength}{55.26681141bp}%
    \ifx\svgscale\undefined%
      \relax%
    \else%
      \setlength{\unitlength}{\unitlength * \real{\svgscale}}%
    \fi%
  \else%
    \setlength{\unitlength}{\svgwidth}%
  \fi%
  \global\let\svgwidth\undefined%
  \global\let\svgscale\undefined%
  \makeatother%
  \begin{picture}(1,1.76175478)%
    \lineheight{1}%
    \setlength\tabcolsep{0pt}%
    \put(0,0){\includegraphics[width=\unitlength,page=1]{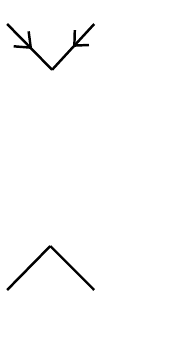}}%
    \put(0.00848576,1.71998299){\color[rgb]{0,0,0}\makebox(0,0)[lt]{\lineheight{1.25}\smash{\begin{tabular}[t]{l}$i$\end{tabular}}}}%
    \put(0.43785294,1.71995795){\color[rgb]{0,0,0}\makebox(0,0)[lt]{\lineheight{1.25}\smash{\begin{tabular}[t]{l}$j$\end{tabular}}}}%
    \put(0.90256783,1.71995795){\color[rgb]{0,0,0}\makebox(0,0)[lt]{\lineheight{1.25}\smash{\begin{tabular}[t]{l}$k$\end{tabular}}}}%
    \put(-0.00410827,0.01123663){\color[rgb]{0,0,0}\makebox(0,0)[lt]{\lineheight{1.25}\smash{\begin{tabular}[t]{l}$i$\end{tabular}}}}%
    \put(0.4252597,0.01121159){\color[rgb]{0,0,0}\makebox(0,0)[lt]{\lineheight{1.25}\smash{\begin{tabular}[t]{l}$j$\end{tabular}}}}%
    \put(0.88997459,0.01121159){\color[rgb]{0,0,0}\makebox(0,0)[lt]{\lineheight{1.25}\smash{\begin{tabular}[t]{l}$k$\end{tabular}}}}%
    \put(0,0){\includegraphics[width=\unitlength,page=2]{Associativity-proof-1.pdf}}%
    \put(0.01188822,1.17840716){\color[rgb]{0,0,0}\makebox(0,0)[lt]{\lineheight{1.25}\smash{\begin{tabular}[t]{l}$ij$\end{tabular}}}}%
    \put(0.01190114,0.70169773){\color[rgb]{0,0,0}\makebox(0,0)[lt]{\lineheight{1.25}\smash{\begin{tabular}[t]{l}$ij$\end{tabular}}}}%
    \put(0,0){\includegraphics[width=\unitlength,page=3]{Associativity-proof-1.pdf}}%
  \end{picture}%
\endgroup%
}}}~~~~=~~~\vcenter{\hbox{{\def\svgscale{0.9}
%% Accompanies image file 'Associativity-proof-2.pdf' (pdf, eps, ps)
\begingroup%
  \makeatletter%
  \providecommand\color[2][]{%
    \errmessage{(Inkscape) Color is used for the text in Inkscape, but the package 'color.sty' is not loaded}%
    \renewcommand\color[2][]{}%
  }%
  \providecommand\transparent[1]{%
    \errmessage{(Inkscape) Transparency is used (non-zero) for the text in Inkscape, but the package 'transparent.sty' is not loaded}%
    \renewcommand\transparent[1]{}%
  }%
  \providecommand\rotatebox[2]{#2}%
  \newcommand*\fsize{\dimexpr\f@size pt\relax}%
  \newcommand*\lineheight[1]{\fontsize{\fsize}{#1\fsize}\selectfont}%
  \ifx\svgwidth\undefined%
    \setlength{\unitlength}{55.26681141bp}%
    \ifx\svgscale\undefined%
      \relax%
    \else%
      \setlength{\unitlength}{\unitlength * \real{\svgscale}}%
    \fi%
  \else%
    \setlength{\unitlength}{\svgwidth}%
  \fi%
  \global\let\svgwidth\undefined%
  \global\let\svgscale\undefined%
  \makeatother%
  \begin{picture}(1,1.76175478)%
    \lineheight{1}%
    \setlength\tabcolsep{0pt}%
    \put(0.00848576,1.71998299){\color[rgb]{0,0,0}\makebox(0,0)[lt]{\lineheight{1.25}\smash{\begin{tabular}[t]{l}$i$\end{tabular}}}}%
    \put(0.43785294,1.71995795){\color[rgb]{0,0,0}\makebox(0,0)[lt]{\lineheight{1.25}\smash{\begin{tabular}[t]{l}$j$\end{tabular}}}}%
    \put(0.90256783,1.71995795){\color[rgb]{0,0,0}\makebox(0,0)[lt]{\lineheight{1.25}\smash{\begin{tabular}[t]{l}$k$\end{tabular}}}}%
    \put(-0.00410827,0.01123663){\color[rgb]{0,0,0}\makebox(0,0)[lt]{\lineheight{1.25}\smash{\begin{tabular}[t]{l}$i$\end{tabular}}}}%
    \put(0.4252597,0.01121159){\color[rgb]{0,0,0}\makebox(0,0)[lt]{\lineheight{1.25}\smash{\begin{tabular}[t]{l}$j$\end{tabular}}}}%
    \put(0.88997459,0.01121159){\color[rgb]{0,0,0}\makebox(0,0)[lt]{\lineheight{1.25}\smash{\begin{tabular}[t]{l}$k$\end{tabular}}}}%
    \put(0.02282269,0.94331561){\color[rgb]{0,0,0}\makebox(0,0)[lt]{\lineheight{1.25}\smash{\begin{tabular}[t]{l}$ij$\end{tabular}}}}%
    \put(0,0){\includegraphics[width=\unitlength,page=1]{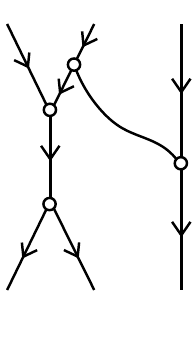}}%
  \end{picture}%
\endgroup%
}}}\\[3ex]
=&~~~\vcenter{\hbox{{\def\svgscale{0.9}
%% Accompanies image file 'Associativity-proof-3.pdf' (pdf, eps, ps)
\begingroup%
  \makeatletter%
  \providecommand\color[2][]{%
    \errmessage{(Inkscape) Color is used for the text in Inkscape, but the package 'color.sty' is not loaded}%
    \renewcommand\color[2][]{}%
  }%
  \providecommand\transparent[1]{%
    \errmessage{(Inkscape) Transparency is used (non-zero) for the text in Inkscape, but the package 'transparent.sty' is not loaded}%
    \renewcommand\transparent[1]{}%
  }%
  \providecommand\rotatebox[2]{#2}%
  \newcommand*\fsize{\dimexpr\f@size pt\relax}%
  \newcommand*\lineheight[1]{\fontsize{\fsize}{#1\fsize}\selectfont}%
  \ifx\svgwidth\undefined%
    \setlength{\unitlength}{56.15678862bp}%
    \ifx\svgscale\undefined%
      \relax%
    \else%
      \setlength{\unitlength}{\unitlength * \real{\svgscale}}%
    \fi%
  \else%
    \setlength{\unitlength}{\svgwidth}%
  \fi%
  \global\let\svgwidth\undefined%
  \global\let\svgscale\undefined%
  \makeatother%
  \begin{picture}(1,1.73382126)%
    \lineheight{1}%
    \setlength\tabcolsep{0pt}%
    \put(0,0){\includegraphics[width=\unitlength,page=1]{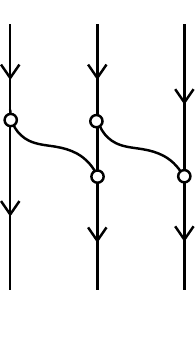}}%
    \put(0.02419935,1.69271147){\color[rgb]{0,0,0}\makebox(0,0)[lt]{\lineheight{1.25}\smash{\begin{tabular}[t]{l}$i$\end{tabular}}}}%
    \put(0.44676227,1.69268683){\color[rgb]{0,0,0}\makebox(0,0)[lt]{\lineheight{1.25}\smash{\begin{tabular}[t]{l}$j$\end{tabular}}}}%
    \put(0.90411194,1.69268683){\color[rgb]{0,0,0}\makebox(0,0)[lt]{\lineheight{1.25}\smash{\begin{tabular}[t]{l}$k$\end{tabular}}}}%
    \put(0.01178181,0.01105856){\color[rgb]{0,0,0}\makebox(0,0)[lt]{\lineheight{1.25}\smash{\begin{tabular}[t]{l}$i$\end{tabular}}}}%
    \put(0.43434474,0.01103391){\color[rgb]{0,0,0}\makebox(0,0)[lt]{\lineheight{1.25}\smash{\begin{tabular}[t]{l}$j$\end{tabular}}}}%
    \put(0.89169556,0.01103391){\color[rgb]{0,0,0}\makebox(0,0)[lt]{\lineheight{1.25}\smash{\begin{tabular}[t]{l}$k$\end{tabular}}}}%
  \end{picture}%
\endgroup%
}}}~~~=\Cijk.
\end{align*}
\end{proof}

\begin{co}
For any unitary $A$-bimodules $W_i,W_j,W_k$ there exists a (unique) unitary $\mathfrak A_{i,j,k}\in\Hom_A(W_{(ij)k},W_{i(jk)})$ \index{Aijk@$\mathfrak A_{i,j,k}$} satisfying
\begin{align}
\mu_{i,jk}(\id_i\otimes\mu_{j,k})=\mathfrak A_{i,j,k}\cdot\mu_{ij,k}(\mu_{i,j}\otimes\id_k).\label{eq46}
\end{align}
We define the unitary \textbf{associativity isomorphism} $W_{(ij)k}\rightarrow W_{i(jk)}$ to be $\mathfrak A_{i,j,k}$.  
\end{co}
\begin{proof}
This follows immediately from the above proposition and theorem \ref{lb30}.
\end{proof}

\begin{pp}[Pentagon axiom]\label{lb40}
Let $W_i,W_j,W_k,W_l$ be unitary $A$-bimodules. Then
\begin{align}
(\id_i\otimes_A\fk A_{j,k,l})\fk A_{i,jk,l}(\fk A_{i,j,k}\otimes_A\id_l)=\fk A_{i,j,kl}\fk A_{ij,k,l}\label{eq47}
\end{align}
\end{pp}
\begin{proof}
One can define unitary tensor products of $W_i,W_j,W_k,W_l$ over $A$ in a similar way as those of three unitary $A$-bimodules. Moreover, using the argument of proposition \ref{lb31} one shows that $(W_{((ij)k)l},\mu_{(ij)k,l}(\mu_{ij,k}(\mu_{i,j}\otimes\id_k)\otimes\id_l))$ is a unitary tensor product of $W_i,W_j,W_k,W_l$ over $A$. We now compute
\begin{align*}
&~~~~~(\id_i\otimes_A\fk A_{j,k,l})\fk A_{i,jk,l}(\fk A_{i,j,k}\otimes_A\id_l)\cdot\mu_{(ij)k,l}(\mu_{ij,k}(\mu_{i,j}\otimes\id_k)\otimes\id_l)\\
&\xlongequal{\eqref{eq45}}(\id_i\otimes_A\fk A_{j,k,l})\fk A_{i,jk,l}\cdot\mu_{i(jk),l}(\fk A_{i,j,k}\otimes\id_l)(\mu_{ij,k}(\mu_{i,j}\otimes\id_k)\otimes\id_l)\\
&=(\id_i\otimes_A\fk A_{j,k,l})\fk A_{i,jk,l}\cdot\mu_{i(jk),l}((\fk A_{i,j,k}\cdot\mu_{ij,k}(\mu_{i,j}\otimes\id_k))\otimes\id_l)\\
&\xlongequal{\eqref{eq46}}(\id_i\otimes_A\fk A_{j,k,l})\fk A_{i,jk,l}\cdot\mu_{i(jk),l}(\mu_{i,jk}(\id_i\otimes\mu_{j,k})\otimes\id_l)\\
&=(\id_i\otimes_A\fk A_{j,k,l})\fk A_{i,jk,l}\cdot\mu_{i(jk),l}(\mu_{i,jk}\otimes\id_l)(\id_i\otimes\mu_{j,k}\otimes\id_l)\\
&\xlongequal{\eqref{eq46}}(\id_i\otimes_A\fk A_{j,k,l})\mu_{i,(jk)l}(\id_i\otimes\mu_{jk,l})(\id_i\otimes\mu_{j,k}\otimes\id_l)\\
&\xlongequal{\eqref{eq45}}\mu_{i,j(kl)}(\id_i\otimes\fk A_{j,k,l})(\id_i\otimes\mu_{jk,l})(\id_i\otimes\mu_{j,k}\otimes\id_l)\\
&=\mu_{i,j(kl)}(\id_i\otimes(\fk A_{j,k,l}\cdot\mu_{jk,l}(\mu_{j,k}\otimes\id_l)))\\
&\xlongequal{\eqref{eq46}}\mu_{i,j(kl)}(\id_i\otimes\mu_{j,kl}(\id_j\otimes\mu_{k,l})),
\end{align*}
and
\begin{align*}
&~~~~~\fk A_{i,j,kl}\fk A_{ij,k,l}\cdot\mu_{(ij)k,l}(\mu_{ij,k}(\mu_{i,j}\otimes\id_k)\otimes\id_l)\\
&=\fk A_{i,j,kl}\fk A_{ij,k,l}\cdot\mu_{(ij)k,l}(\mu_{ij,k}\otimes\id_l)(\mu_{i,j}\otimes\id_k\otimes\id_l)\\
&\xlongequal{\eqref{eq46}}\fk A_{i,j,kl}\cdot\mu_{ij,kl}(\id_{ij}\otimes\mu_{k,l})(\mu_{i,j}\otimes\id_k\otimes\id_l)\\
&=\fk A_{i,j,kl}\cdot\mu_{ij,kl}(\mu_{i,j}\otimes\id_{kl})(\id_i\otimes\id_j\otimes\mu_{k,l})\\
&\xlongequal{\eqref{eq46}}\mu_{i,j(kl)}(\id_i\otimes\mu_{j,kl})(\id_i\otimes\id_j\otimes\mu_{k,l})=\mu_{i,j(kl)}(\id_i\otimes\mu_{j,kl}(\id_j\otimes\mu_{k,l})).
\end{align*}
Thus equation \eqref{eq47} holds when both sides are multiplied by $\mu_{(ij)k,l}(\mu_{ij,k}(\mu_{i,j}\otimes\id_k)\otimes\id_l)$. Hence equation \eqref{eq47} is true by the universal property for the unitary tensor products of $W_i,W_j,W_k,W_l$.
\end{proof}

We choose the vacuum bimodule $(W_a,\mu,\mu)$ to be the identity object of $\BIMA$. Then by proposition \ref{lb27}, for any unitary $A$-bimodule $(W_i,\mu^i_L,\mu^i_R)$, we have that $(W_i,\mu^i_L)$ is a unitary tensor product of $W_a\boxtimes W_i$ and $(W_i,\mu^i_R)$ is a unitary tensor product of $W_i\boxtimes W_a$. By uniqueness up to unitaries, there exist unique unitary $\fk l_i\in\Hom_A(W_{ai},W_i)$ and $\fk r_i\in\Hom_A(W_{ia},W_i)$ satisfying
\begin{align}
\mu^i_L=\fk l_i\mu_{a,i},\qquad \mu^i_R=\fk r_i\mu_{i,a}.\label{eq52}
\end{align}
\begin{pp}[Triangle axiom]
For any unitary $A$-bimodules $W_i,W_j$ we have
\begin{align}
(\id_i\otimes_A\fk l_j)\fk A_{i,a,j}=\fk r_i\otimes_A\id_j.
\end{align}	
\end{pp}
\begin{proof}
Similar to (and simpler than) the proof of pentagon axiom, we show that
\begin{align*}
(\id_i\otimes_A\fk l_j)\fk A_{i,a,j}\cdot\mu_{ia,j}(\mu_{i,a}\otimes\id_j)=\mu_{i,j}(\id_i\otimes\mu^L_j)=\mu_{i,j}(\mu^R_i\otimes\id_j)=(\fk r_i\otimes_A\id_j)\cdot\mu_{ia,j}(\mu_{i,a}\otimes\id_j),
\end{align*}
which proves triangle axiom by universal property.
\end{proof}

We conclude:
\begin{thm}
With the $*$-bifunctor $\boxtimes_A$, the associativity isomorphisms, the unit object, and the left and right multiplications by unit defined above, $\BIMA$ is a $C^*$-tensor category.
\end{thm}

In the following, we identify different ways of unitary tensor products via associativity isomorphisms, and identify $W_{ai}$ with $W_i$ and $W_{ia}$ with $W_i$ via $\fk l_i$ and $\fk r_i$ respectively. Then $\BIMA$ can be treated as if it is a strict $C^*$-tensor category. We have $(ij)k=i(jk)$, both denoted by $ijk$, and also $ai=i=ia$. Thus the $\fk A_{i,j,k},\fk l_i,\fk r_i$ are all identity morphisms. Therefore $\mu^i_L=\mu_{a,i},\mu^i_R=\mu_{i,a}$, and in particular $\mu=\mu_{a,a}$. Moreover, equation \eqref{eq46} now reads
\begin{align}
\vcenter{\hbox{{\def\svgscale{0.4}
%% Accompanies image file 'associativity-2.pdf' (pdf, eps, ps)
\begingroup%
  \makeatletter%
  \providecommand\color[2][]{%
    \errmessage{(Inkscape) Color is used for the text in Inkscape, but the package 'color.sty' is not loaded}%
    \renewcommand\color[2][]{}%
  }%
  \providecommand\transparent[1]{%
    \errmessage{(Inkscape) Transparency is used (non-zero) for the text in Inkscape, but the package 'transparent.sty' is not loaded}%
    \renewcommand\transparent[1]{}%
  }%
  \providecommand\rotatebox[2]{#2}%
  \newcommand*\fsize{\dimexpr\f@size pt\relax}%
  \newcommand*\lineheight[1]{\fontsize{\fsize}{#1\fsize}\selectfont}%
  \ifx\svgwidth\undefined%
    \setlength{\unitlength}{138.62082524bp}%
    \ifx\svgscale\undefined%
      \relax%
    \else%
      \setlength{\unitlength}{\unitlength * \real{\svgscale}}%
    \fi%
  \else%
    \setlength{\unitlength}{\svgwidth}%
  \fi%
  \global\let\svgwidth\undefined%
  \global\let\svgscale\undefined%
  \makeatother%
  \begin{picture}(1,1.25994442)%
    \lineheight{1}%
    \setlength\tabcolsep{0pt}%
    \put(0,0){\includegraphics[width=\unitlength,page=1]{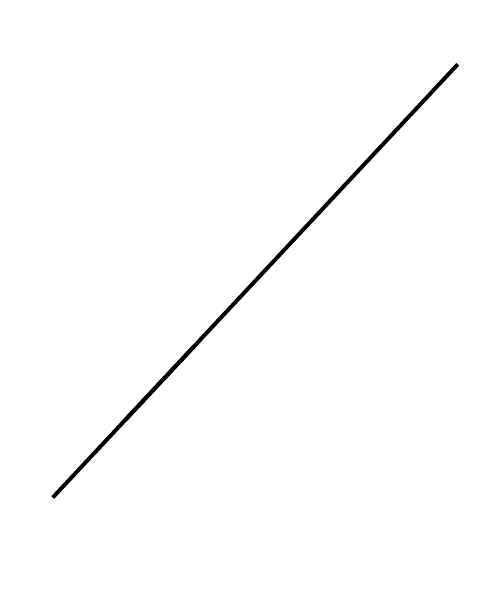}}%
    \put(0.09426907,1.19591602){\color[rgb]{0,0,0}\makebox(0,0)[lt]{\lineheight{1.25}\smash{\begin{tabular}[t]{l}$i$\end{tabular}}}}%
    \put(0.45828527,1.20463496){\color[rgb]{0,0,0}\makebox(0,0)[lt]{\lineheight{1.25}\smash{\begin{tabular}[t]{l}$j$\end{tabular}}}}%
    \put(0.89641257,1.2155337){\color[rgb]{0,0,0}\makebox(0,0)[lt]{\lineheight{1.25}\smash{\begin{tabular}[t]{l}$k$\end{tabular}}}}%
    \put(0,0){\includegraphics[width=\unitlength,page=2]{associativity-2.pdf}}%
    \put(-0.00310104,0.0096499){\color[rgb]{0,0,0}\makebox(0,0)[lt]{\lineheight{1.25}\smash{\begin{tabular}[t]{l}$ijk$\end{tabular}}}}%
    \put(0.17569066,0.68782502){\color[rgb]{0,0,0}\makebox(0,0)[lt]{\lineheight{1.25}\smash{\begin{tabular}[t]{l}$ij$\end{tabular}}}}%
  \end{picture}%
\endgroup%
}}}~~=~~\vcenter{\hbox{{\def\svgscale{0.4}
%% Accompanies image file 'associativity.pdf' (pdf, eps, ps)
\begingroup%
  \makeatletter%
  \providecommand\color[2][]{%
    \errmessage{(Inkscape) Color is used for the text in Inkscape, but the package 'color.sty' is not loaded}%
    \renewcommand\color[2][]{}%
  }%
  \providecommand\transparent[1]{%
    \errmessage{(Inkscape) Transparency is used (non-zero) for the text in Inkscape, but the package 'transparent.sty' is not loaded}%
    \renewcommand\transparent[1]{}%
  }%
  \providecommand\rotatebox[2]{#2}%
  \newcommand*\fsize{\dimexpr\f@size pt\relax}%
  \newcommand*\lineheight[1]{\fontsize{\fsize}{#1\fsize}\selectfont}%
  \ifx\svgwidth\undefined%
    \setlength{\unitlength}{141.90014094bp}%
    \ifx\svgscale\undefined%
      \relax%
    \else%
      \setlength{\unitlength}{\unitlength * \real{\svgscale}}%
    \fi%
  \else%
    \setlength{\unitlength}{\svgwidth}%
  \fi%
  \global\let\svgwidth\undefined%
  \global\let\svgscale\undefined%
  \makeatother%
  \begin{picture}(1,1.25093346)%
    \lineheight{1}%
    \setlength\tabcolsep{0pt}%
    \put(0,0){\includegraphics[width=\unitlength,page=1]{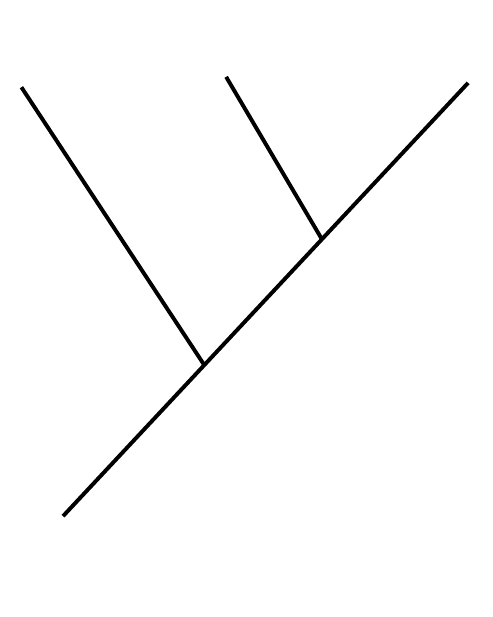}}%
    \put(0.00873238,1.19264333){\color[rgb]{0,0,0}\makebox(0,0)[lt]{\lineheight{1.25}\smash{\begin{tabular}[t]{l}$i$\end{tabular}}}}%
    \put(0.40479408,1.20329021){\color[rgb]{0,0,0}\makebox(0,0)[lt]{\lineheight{1.25}\smash{\begin{tabular}[t]{l}$j$\end{tabular}}}}%
    \put(0.89880648,1.20754908){\color[rgb]{0,0,0}\makebox(0,0)[lt]{\lineheight{1.25}\smash{\begin{tabular}[t]{l}$k$\end{tabular}}}}%
    \put(0,0){\includegraphics[width=\unitlength,page=2]{associativity.pdf}}%
    \put(0.55020175,0.54722674){\color[rgb]{0,0,0}\makebox(0,0)[lt]{\lineheight{1.25}\smash{\begin{tabular}[t]{l}$jk$\end{tabular}}}}%
    \put(-0.00302937,0.00942689){\color[rgb]{0,0,0}\makebox(0,0)[lt]{\lineheight{1.25}\smash{\begin{tabular}[t]{l}$ijk$\end{tabular}}}}%
    \put(0,0){\includegraphics[width=\unitlength,page=3]{associativity.pdf}}%
  \end{picture}%
\endgroup%
}}}~~,\label{eq48}
\end{align}
which means that the left action $i\curvearrowright j$ and the right action $j\curvearrowleft k$ commute. These two actions indeed commute \emph{adjointly}, as indicated below.

\begin{thm}[Frobenius relations]\label{lb32}
\begin{align}
\vcenter{\hbox{{\def\svgscale{0.6}
%% Accompanies image file 'Frobenius-7.pdf' (pdf, eps, ps)
\begingroup%
  \makeatletter%
  \providecommand\color[2][]{%
    \errmessage{(Inkscape) Color is used for the text in Inkscape, but the package 'color.sty' is not loaded}%
    \renewcommand\color[2][]{}%
  }%
  \providecommand\transparent[1]{%
    \errmessage{(Inkscape) Transparency is used (non-zero) for the text in Inkscape, but the package 'transparent.sty' is not loaded}%
    \renewcommand\transparent[1]{}%
  }%
  \providecommand\rotatebox[2]{#2}%
  \newcommand*\fsize{\dimexpr\f@size pt\relax}%
  \newcommand*\lineheight[1]{\fontsize{\fsize}{#1\fsize}\selectfont}%
  \ifx\svgwidth\undefined%
    \setlength{\unitlength}{60.55577061bp}%
    \ifx\svgscale\undefined%
      \relax%
    \else%
      \setlength{\unitlength}{\unitlength * \real{\svgscale}}%
    \fi%
  \else%
    \setlength{\unitlength}{\svgwidth}%
  \fi%
  \global\let\svgwidth\undefined%
  \global\let\svgscale\undefined%
  \makeatother%
  \begin{picture}(1,1.93143974)%
    \lineheight{1}%
    \setlength\tabcolsep{0pt}%
    \put(0,0){\includegraphics[width=\unitlength,page=1]{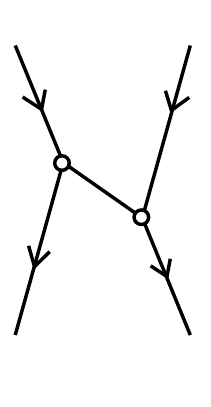}}%
    \put(-0.00521071,1.86737226){\color[rgb]{0,0,0}\makebox(0,0)[lt]{\lineheight{1.25}\smash{\begin{tabular}[t]{l}$ij$\end{tabular}}}}%
    \put(0.84041828,1.86831166){\color[rgb]{0,0,0}\makebox(0,0)[lt]{\lineheight{1.25}\smash{\begin{tabular}[t]{l}$k$\end{tabular}}}}%
    \put(0.77033877,0.03161815){\color[rgb]{0,0,0}\makebox(0,0)[lt]{\lineheight{1.25}\smash{\begin{tabular}[t]{l}$jk$\end{tabular}}}}%
    \put(0.01523287,0.01467536){\color[rgb]{0,0,0}\makebox(0,0)[lt]{\lineheight{1.25}\smash{\begin{tabular}[t]{l}$i$\end{tabular}}}}%
    \put(0.44291535,1.1929069){\color[rgb]{0,0,0}\makebox(0,0)[lt]{\lineheight{1.25}\smash{\begin{tabular}[t]{l}$j$\end{tabular}}}}%
    \put(0,0){\includegraphics[width=\unitlength,page=2]{Frobenius-7.pdf}}%
  \end{picture}%
\endgroup%
}}}~~~=~~\vcenter{\hbox{{\def\svgscale{0.6}
%% Accompanies image file 'Frobenius-8.pdf' (pdf, eps, ps)
\begingroup%
  \makeatletter%
  \providecommand\color[2][]{%
    \errmessage{(Inkscape) Color is used for the text in Inkscape, but the package 'color.sty' is not loaded}%
    \renewcommand\color[2][]{}%
  }%
  \providecommand\transparent[1]{%
    \errmessage{(Inkscape) Transparency is used (non-zero) for the text in Inkscape, but the package 'transparent.sty' is not loaded}%
    \renewcommand\transparent[1]{}%
  }%
  \providecommand\rotatebox[2]{#2}%
  \newcommand*\fsize{\dimexpr\f@size pt\relax}%
  \newcommand*\lineheight[1]{\fontsize{\fsize}{#1\fsize}\selectfont}%
  \ifx\svgwidth\undefined%
    \setlength{\unitlength}{60.89424291bp}%
    \ifx\svgscale\undefined%
      \relax%
    \else%
      \setlength{\unitlength}{\unitlength * \real{\svgscale}}%
    \fi%
  \else%
    \setlength{\unitlength}{\svgwidth}%
  \fi%
  \global\let\svgwidth\undefined%
  \global\let\svgscale\undefined%
  \makeatother%
  \begin{picture}(1,1.88539327)%
    \lineheight{1}%
    \setlength\tabcolsep{0pt}%
    \put(0,0){\includegraphics[width=\unitlength,page=1]{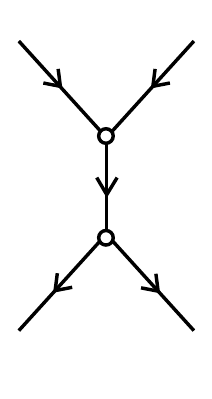}}%
    \put(-0.00665411,1.7995853){\color[rgb]{0,0,0}\makebox(0,0)[lt]{\lineheight{1.25}\smash{\begin{tabular}[t]{l}$ij$\end{tabular}}}}%
    \put(0.84219031,1.82370051){\color[rgb]{0,0,0}\makebox(0,0)[lt]{\lineheight{1.25}\smash{\begin{tabular}[t]{l}$k$\end{tabular}}}}%
    \put(0.59229128,0.95613371){\color[rgb]{0,0,0}\makebox(0,0)[lt]{\lineheight{1.25}\smash{\begin{tabular}[t]{l}$ijk$\end{tabular}}}}%
    \put(-0.00550634,0.0141314){\color[rgb]{0,0,0}\makebox(0,0)[lt]{\lineheight{1.25}\smash{\begin{tabular}[t]{l}$i$\end{tabular}}}}%
    \put(0.79123729,0.02832284){\color[rgb]{0,0,0}\makebox(0,0)[lt]{\lineheight{1.25}\smash{\begin{tabular}[t]{l}$jk$\end{tabular}}}}%
  \end{picture}%
\endgroup%
}}}~~~,\qquad~~\vcenter{\hbox{{\def\svgscale{0.6}
%% Accompanies image file 'Frobenius-10.pdf' (pdf, eps, ps)
\begingroup%
  \makeatletter%
  \providecommand\color[2][]{%
    \errmessage{(Inkscape) Color is used for the text in Inkscape, but the package 'color.sty' is not loaded}%
    \renewcommand\color[2][]{}%
  }%
  \providecommand\transparent[1]{%
    \errmessage{(Inkscape) Transparency is used (non-zero) for the text in Inkscape, but the package 'transparent.sty' is not loaded}%
    \renewcommand\transparent[1]{}%
  }%
  \providecommand\rotatebox[2]{#2}%
  \newcommand*\fsize{\dimexpr\f@size pt\relax}%
  \newcommand*\lineheight[1]{\fontsize{\fsize}{#1\fsize}\selectfont}%
  \ifx\svgwidth\undefined%
    \setlength{\unitlength}{62.81581428bp}%
    \ifx\svgscale\undefined%
      \relax%
    \else%
      \setlength{\unitlength}{\unitlength * \real{\svgscale}}%
    \fi%
  \else%
    \setlength{\unitlength}{\svgwidth}%
  \fi%
  \global\let\svgwidth\undefined%
  \global\let\svgscale\undefined%
  \makeatother%
  \begin{picture}(1,1.83007075)%
    \lineheight{1}%
    \setlength\tabcolsep{0pt}%
    \put(0,0){\includegraphics[width=\unitlength,page=1]{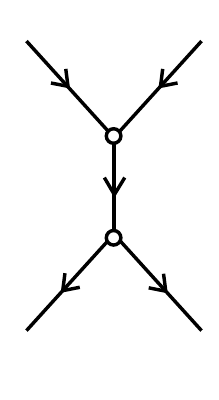}}%
    \put(0.06218467,1.7468877){\color[rgb]{0,0,0}\makebox(0,0)[lt]{\lineheight{1.25}\smash{\begin{tabular}[t]{l}$i$\end{tabular}}}}%
    \put(0.81771979,1.77026521){\color[rgb]{0,0,0}\makebox(0,0)[lt]{\lineheight{1.25}\smash{\begin{tabular}[t]{l}$jk$\end{tabular}}}}%
    \put(0.60913663,0.92923779){\color[rgb]{0,0,0}\makebox(0,0)[lt]{\lineheight{1.25}\smash{\begin{tabular}[t]{l}$ijk$\end{tabular}}}}%
    \put(-0.00645056,0.01605187){\color[rgb]{0,0,0}\makebox(0,0)[lt]{\lineheight{1.25}\smash{\begin{tabular}[t]{l}$ij$\end{tabular}}}}%
    \put(0.84288333,0.02980919){\color[rgb]{0,0,0}\makebox(0,0)[lt]{\lineheight{1.25}\smash{\begin{tabular}[t]{l}$k$\end{tabular}}}}%
  \end{picture}%
\endgroup%
}}}~~~=~~\vcenter{\hbox{{\def\svgscale{0.6}
%% Accompanies image file 'Frobenius-9.pdf' (pdf, eps, ps)
\begingroup%
  \makeatletter%
  \providecommand\color[2][]{%
    \errmessage{(Inkscape) Color is used for the text in Inkscape, but the package 'color.sty' is not loaded}%
    \renewcommand\color[2][]{}%
  }%
  \providecommand\transparent[1]{%
    \errmessage{(Inkscape) Transparency is used (non-zero) for the text in Inkscape, but the package 'transparent.sty' is not loaded}%
    \renewcommand\transparent[1]{}%
  }%
  \providecommand\rotatebox[2]{#2}%
  \newcommand*\fsize{\dimexpr\f@size pt\relax}%
  \newcommand*\lineheight[1]{\fontsize{\fsize}{#1\fsize}\selectfont}%
  \ifx\svgwidth\undefined%
    \setlength{\unitlength}{61.6804674bp}%
    \ifx\svgscale\undefined%
      \relax%
    \else%
      \setlength{\unitlength}{\unitlength * \real{\svgscale}}%
    \fi%
  \else%
    \setlength{\unitlength}{\svgwidth}%
  \fi%
  \global\let\svgwidth\undefined%
  \global\let\svgscale\undefined%
  \makeatother%
  \begin{picture}(1,1.88210121)%
    \lineheight{1}%
    \setlength\tabcolsep{0pt}%
    \put(0,0){\includegraphics[width=\unitlength,page=1]{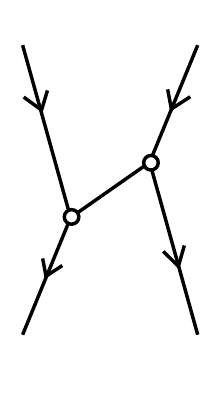}}%
    \put(0.03635133,1.82260872){\color[rgb]{0,0,0}\makebox(0,0)[lt]{\lineheight{1.25}\smash{\begin{tabular}[t]{l}$i$\end{tabular}}}}%
    \put(0.81832272,1.83621074){\color[rgb]{0,0,0}\makebox(0,0)[lt]{\lineheight{1.25}\smash{\begin{tabular}[t]{l}$jk$\end{tabular}}}}%
    \put(-0.00497094,0.01322508){\color[rgb]{0,0,0}\makebox(0,0)[lt]{\lineheight{1.25}\smash{\begin{tabular}[t]{l}$ij$\end{tabular}}}}%
    \put(0.87619998,0.02070504){\color[rgb]{0,0,0}\makebox(0,0)[lt]{\lineheight{1.25}\smash{\begin{tabular}[t]{l}$k$\end{tabular}}}}%
    \put(0,0){\includegraphics[width=\unitlength,page=2]{Frobenius-9.pdf}}%
    \put(0.40713451,1.11039789){\color[rgb]{0,0,0}\makebox(0,0)[lt]{\lineheight{1.25}\smash{\begin{tabular}[t]{l}$j$\end{tabular}}}}%
  \end{picture}%
\endgroup%
}}}~~~.\label{eq49}
\end{align}
\end{thm}
\begin{proof}
Let $F$ and $G$ be the first and the second item of \eqref{eq49}. One computes
\begin{align*}
&F(\mu_{i,j}\otimes\id_k)=~~\vcenter{\hbox{{\def\svgscale{0.6}
%% Accompanies image file 'Frobenius-proof.pdf' (pdf, eps, ps)
\begingroup%
  \makeatletter%
  \providecommand\color[2][]{%
    \errmessage{(Inkscape) Color is used for the text in Inkscape, but the package 'color.sty' is not loaded}%
    \renewcommand\color[2][]{}%
  }%
  \providecommand\transparent[1]{%
    \errmessage{(Inkscape) Transparency is used (non-zero) for the text in Inkscape, but the package 'transparent.sty' is not loaded}%
    \renewcommand\transparent[1]{}%
  }%
  \providecommand\rotatebox[2]{#2}%
  \newcommand*\fsize{\dimexpr\f@size pt\relax}%
  \newcommand*\lineheight[1]{\fontsize{\fsize}{#1\fsize}\selectfont}%
  \ifx\svgwidth\undefined%
    \setlength{\unitlength}{74.83884579bp}%
    \ifx\svgscale\undefined%
      \relax%
    \else%
      \setlength{\unitlength}{\unitlength * \real{\svgscale}}%
    \fi%
  \else%
    \setlength{\unitlength}{\svgwidth}%
  \fi%
  \global\let\svgwidth\undefined%
  \global\let\svgscale\undefined%
  \makeatother%
  \begin{picture}(1,1.94563152)%
    \lineheight{1}%
    \setlength\tabcolsep{0pt}%
    \put(0,0){\includegraphics[width=\unitlength,page=1]{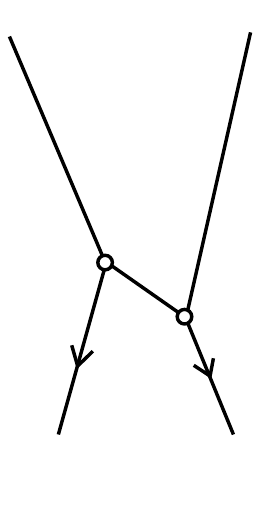}}%
    \put(0.78929754,0.02558381){\color[rgb]{0,0,0}\makebox(0,0)[lt]{\lineheight{1.25}\smash{\begin{tabular}[t]{l}$jk$\end{tabular}}}}%
    \put(0.17830443,0.01187455){\color[rgb]{0,0,0}\makebox(0,0)[lt]{\lineheight{1.25}\smash{\begin{tabular}[t]{l}$i$\end{tabular}}}}%
    \put(0.52436324,0.96523932){\color[rgb]{0,0,0}\makebox(0,0)[lt]{\lineheight{1.25}\smash{\begin{tabular}[t]{l}$j$\end{tabular}}}}%
    \put(0,0){\includegraphics[width=\unitlength,page=2]{Frobenius-proof.pdf}}%
    \put(0.03109438,1.07536911){\color[rgb]{0,0,0}\makebox(0,0)[lt]{\lineheight{1.25}\smash{\begin{tabular}[t]{l}$ij$\end{tabular}}}}%
    \put(-0.00421624,1.87456728){\color[rgb]{0,0,0}\makebox(0,0)[lt]{\lineheight{1.25}\smash{\begin{tabular}[t]{l}$i$\end{tabular}}}}%
    \put(0.42585941,1.89315346){\color[rgb]{0,0,0}\makebox(0,0)[lt]{\lineheight{1.25}\smash{\begin{tabular}[t]{l}$j$\end{tabular}}}}%
    \put(0,0){\includegraphics[width=\unitlength,page=3]{Frobenius-proof.pdf}}%
    \put(0.9040647,1.90450148){\color[rgb]{0,0,0}\makebox(0,0)[lt]{\lineheight{1.25}\smash{\begin{tabular}[t]{l}$k$\end{tabular}}}}%
  \end{picture}%
\endgroup%
}}}~~~\xlongequal{\eqref{eq44}}~~\vcenter{\hbox{{\def\svgscale{0.6}
%% Accompanies image file 'Frobenius-proof-2.pdf' (pdf, eps, ps)
\begingroup%
  \makeatletter%
  \providecommand\color[2][]{%
    \errmessage{(Inkscape) Color is used for the text in Inkscape, but the package 'color.sty' is not loaded}%
    \renewcommand\color[2][]{}%
  }%
  \providecommand\transparent[1]{%
    \errmessage{(Inkscape) Transparency is used (non-zero) for the text in Inkscape, but the package 'transparent.sty' is not loaded}%
    \renewcommand\transparent[1]{}%
  }%
  \providecommand\rotatebox[2]{#2}%
  \newcommand*\fsize{\dimexpr\f@size pt\relax}%
  \newcommand*\lineheight[1]{\fontsize{\fsize}{#1\fsize}\selectfont}%
  \ifx\svgwidth\undefined%
    \setlength{\unitlength}{74.83884594bp}%
    \ifx\svgscale\undefined%
      \relax%
    \else%
      \setlength{\unitlength}{\unitlength * \real{\svgscale}}%
    \fi%
  \else%
    \setlength{\unitlength}{\svgwidth}%
  \fi%
  \global\let\svgwidth\undefined%
  \global\let\svgscale\undefined%
  \makeatother%
  \begin{picture}(1,1.94563152)%
    \lineheight{1}%
    \setlength\tabcolsep{0pt}%
    \put(0,0){\includegraphics[width=\unitlength,page=1]{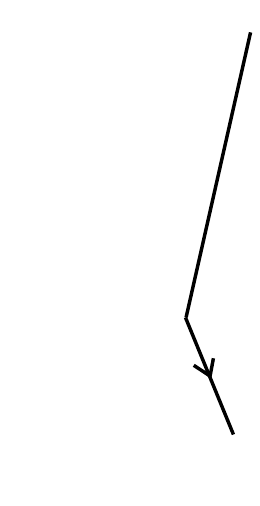}}%
    \put(0.78929754,0.02558381){\color[rgb]{0,0,0}\makebox(0,0)[lt]{\lineheight{1.25}\smash{\begin{tabular}[t]{l}$jk$\end{tabular}}}}%
    \put(0.17830443,0.01187455){\color[rgb]{0,0,0}\makebox(0,0)[lt]{\lineheight{1.25}\smash{\begin{tabular}[t]{l}$i$\end{tabular}}}}%
    \put(-0.00421624,1.87456727){\color[rgb]{0,0,0}\makebox(0,0)[lt]{\lineheight{1.25}\smash{\begin{tabular}[t]{l}$i$\end{tabular}}}}%
    \put(0.42585941,1.89315345){\color[rgb]{0,0,0}\makebox(0,0)[lt]{\lineheight{1.25}\smash{\begin{tabular}[t]{l}$j$\end{tabular}}}}%
    \put(0,0){\includegraphics[width=\unitlength,page=2]{Frobenius-proof-2.pdf}}%
    \put(0.9040647,1.90450147){\color[rgb]{0,0,0}\makebox(0,0)[lt]{\lineheight{1.25}\smash{\begin{tabular}[t]{l}$k$\end{tabular}}}}%
    \put(0,0){\includegraphics[width=\unitlength,page=3]{Frobenius-proof-2.pdf}}%
    \put(0.43177258,0.67753099){\color[rgb]{0,0,0}\makebox(0,0)[lt]{\lineheight{1.25}\smash{\begin{tabular}[t]{l}$j$\end{tabular}}}}%
    \put(0,0){\includegraphics[width=\unitlength,page=4]{Frobenius-proof-2.pdf}}%
  \end{picture}%
\endgroup%
}}}~~~=~~\vcenter{\hbox{{\def\svgscale{0.6}
%% Accompanies image file 'Frobenius-proof-3.pdf' (pdf, eps, ps)
\begingroup%
  \makeatletter%
  \providecommand\color[2][]{%
    \errmessage{(Inkscape) Color is used for the text in Inkscape, but the package 'color.sty' is not loaded}%
    \renewcommand\color[2][]{}%
  }%
  \providecommand\transparent[1]{%
    \errmessage{(Inkscape) Transparency is used (non-zero) for the text in Inkscape, but the package 'transparent.sty' is not loaded}%
    \renewcommand\transparent[1]{}%
  }%
  \providecommand\rotatebox[2]{#2}%
  \newcommand*\fsize{\dimexpr\f@size pt\relax}%
  \newcommand*\lineheight[1]{\fontsize{\fsize}{#1\fsize}\selectfont}%
  \ifx\svgwidth\undefined%
    \setlength{\unitlength}{74.83884594bp}%
    \ifx\svgscale\undefined%
      \relax%
    \else%
      \setlength{\unitlength}{\unitlength * \real{\svgscale}}%
    \fi%
  \else%
    \setlength{\unitlength}{\svgwidth}%
  \fi%
  \global\let\svgwidth\undefined%
  \global\let\svgscale\undefined%
  \makeatother%
  \begin{picture}(1,1.94563152)%
    \lineheight{1}%
    \setlength\tabcolsep{0pt}%
    \put(0,0){\includegraphics[width=\unitlength,page=1]{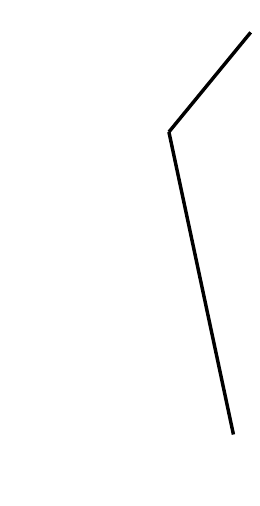}}%
    \put(0.78929754,0.02558381){\color[rgb]{0,0,0}\makebox(0,0)[lt]{\lineheight{1.25}\smash{\begin{tabular}[t]{l}$jk$\end{tabular}}}}%
    \put(0.17830443,0.01187455){\color[rgb]{0,0,0}\makebox(0,0)[lt]{\lineheight{1.25}\smash{\begin{tabular}[t]{l}$i$\end{tabular}}}}%
    \put(-0.00421624,1.87456727){\color[rgb]{0,0,0}\makebox(0,0)[lt]{\lineheight{1.25}\smash{\begin{tabular}[t]{l}$i$\end{tabular}}}}%
    \put(0.42585941,1.89315345){\color[rgb]{0,0,0}\makebox(0,0)[lt]{\lineheight{1.25}\smash{\begin{tabular}[t]{l}$j$\end{tabular}}}}%
    \put(0.9040647,1.90450147){\color[rgb]{0,0,0}\makebox(0,0)[lt]{\lineheight{1.25}\smash{\begin{tabular}[t]{l}$k$\end{tabular}}}}%
    \put(0,0){\includegraphics[width=\unitlength,page=2]{Frobenius-proof-3.pdf}}%
    \put(0.77679528,1.16460639){\color[rgb]{0,0,0}\makebox(0,0)[lt]{\lineheight{1.25}\smash{\begin{tabular}[t]{l}$jk$\end{tabular}}}}%
    \put(0,0){\includegraphics[width=\unitlength,page=3]{Frobenius-proof-3.pdf}}%
  \end{picture}%
\endgroup%
}}}\\[2.5ex]
\xlongequal{\eqref{eq44}}~~~&\vcenter{\hbox{{\def\svgscale{0.6}
%% Accompanies image file 'Frobenius-proof-4.pdf' (pdf, eps, ps)
\begingroup%
  \makeatletter%
  \providecommand\color[2][]{%
    \errmessage{(Inkscape) Color is used for the text in Inkscape, but the package 'color.sty' is not loaded}%
    \renewcommand\color[2][]{}%
  }%
  \providecommand\transparent[1]{%
    \errmessage{(Inkscape) Transparency is used (non-zero) for the text in Inkscape, but the package 'transparent.sty' is not loaded}%
    \renewcommand\transparent[1]{}%
  }%
  \providecommand\rotatebox[2]{#2}%
  \newcommand*\fsize{\dimexpr\f@size pt\relax}%
  \newcommand*\lineheight[1]{\fontsize{\fsize}{#1\fsize}\selectfont}%
  \ifx\svgwidth\undefined%
    \setlength{\unitlength}{78.82765693bp}%
    \ifx\svgscale\undefined%
      \relax%
    \else%
      \setlength{\unitlength}{\unitlength * \real{\svgscale}}%
    \fi%
  \else%
    \setlength{\unitlength}{\svgwidth}%
  \fi%
  \global\let\svgwidth\undefined%
  \global\let\svgscale\undefined%
  \makeatother%
  \begin{picture}(1,1.84717931)%
    \lineheight{1}%
    \setlength\tabcolsep{0pt}%
    \put(0,0){\includegraphics[width=\unitlength,page=1]{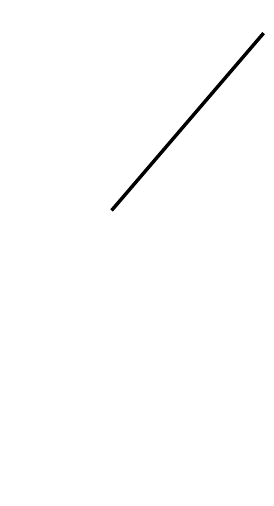}}%
    \put(0.79995943,0.02428923){\color[rgb]{0,0,0}\makebox(0,0)[lt]{\lineheight{1.25}\smash{\begin{tabular}[t]{l}$jk$\end{tabular}}}}%
    \put(0.2198836,0.01127368){\color[rgb]{0,0,0}\makebox(0,0)[lt]{\lineheight{1.25}\smash{\begin{tabular}[t]{l}$i$\end{tabular}}}}%
    \put(0.04659878,1.77971104){\color[rgb]{0,0,0}\makebox(0,0)[lt]{\lineheight{1.25}\smash{\begin{tabular}[t]{l}$i$\end{tabular}}}}%
    \put(0.45491188,1.79735673){\color[rgb]{0,0,0}\makebox(0,0)[lt]{\lineheight{1.25}\smash{\begin{tabular}[t]{l}$j$\end{tabular}}}}%
    \put(0.90891918,1.80813052){\color[rgb]{0,0,0}\makebox(0,0)[lt]{\lineheight{1.25}\smash{\begin{tabular}[t]{l}$k$\end{tabular}}}}%
    \put(0,0){\includegraphics[width=\unitlength,page=2]{Frobenius-proof-4.pdf}}%
    \put(0.60026606,1.1190914){\color[rgb]{0,0,0}\makebox(0,0)[lt]{\lineheight{1.25}\smash{\begin{tabular}[t]{l}$jk$\end{tabular}}}}%
    \put(0,0){\includegraphics[width=\unitlength,page=3]{Frobenius-proof-4.pdf}}%
    \put(-0.00516912,0.86721708){\color[rgb]{0,0,0}\makebox(0,0)[lt]{\lineheight{1.25}\smash{\begin{tabular}[t]{l}$ijk$\end{tabular}}}}%
  \end{picture}%
\endgroup%
}}}~~~\xlongequal{\eqref{eq48}}~~\vcenter{\hbox{{\def\svgscale{0.6}
%% Accompanies image file 'Frobenius-proof-5.pdf' (pdf, eps, ps)
\begingroup%
  \makeatletter%
  \providecommand\color[2][]{%
    \errmessage{(Inkscape) Color is used for the text in Inkscape, but the package 'color.sty' is not loaded}%
    \renewcommand\color[2][]{}%
  }%
  \providecommand\transparent[1]{%
    \errmessage{(Inkscape) Transparency is used (non-zero) for the text in Inkscape, but the package 'transparent.sty' is not loaded}%
    \renewcommand\transparent[1]{}%
  }%
  \providecommand\rotatebox[2]{#2}%
  \newcommand*\fsize{\dimexpr\f@size pt\relax}%
  \newcommand*\lineheight[1]{\fontsize{\fsize}{#1\fsize}\selectfont}%
  \ifx\svgwidth\undefined%
    \setlength{\unitlength}{77.6525504bp}%
    \ifx\svgscale\undefined%
      \relax%
    \else%
      \setlength{\unitlength}{\unitlength * \real{\svgscale}}%
    \fi%
  \else%
    \setlength{\unitlength}{\svgwidth}%
  \fi%
  \global\let\svgwidth\undefined%
  \global\let\svgscale\undefined%
  \makeatother%
  \begin{picture}(1,1.87513245)%
    \lineheight{1}%
    \setlength\tabcolsep{0pt}%
    \put(0,0){\includegraphics[width=\unitlength,page=1]{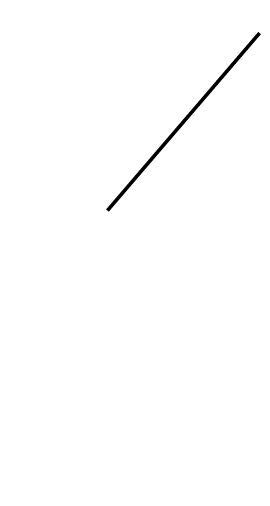}}%
    \put(0.79693224,0.0246568){\color[rgb]{0,0,0}\makebox(0,0)[lt]{\lineheight{1.25}\smash{\begin{tabular}[t]{l}$jk$\end{tabular}}}}%
    \put(0.20807819,0.01144428){\color[rgb]{0,0,0}\makebox(0,0)[lt]{\lineheight{1.25}\smash{\begin{tabular}[t]{l}$i$\end{tabular}}}}%
    \put(0.03217108,1.80664319){\color[rgb]{0,0,0}\makebox(0,0)[lt]{\lineheight{1.25}\smash{\begin{tabular}[t]{l}$i$\end{tabular}}}}%
    \put(0.44666313,1.82455591){\color[rgb]{0,0,0}\makebox(0,0)[lt]{\lineheight{1.25}\smash{\begin{tabular}[t]{l}$j$\end{tabular}}}}%
    \put(0.90754087,1.83549274){\color[rgb]{0,0,0}\makebox(0,0)[lt]{\lineheight{1.25}\smash{\begin{tabular}[t]{l}$k$\end{tabular}}}}%
    \put(0.48140731,0.9049043){\color[rgb]{0,0,0}\makebox(0,0)[lt]{\lineheight{1.25}\smash{\begin{tabular}[t]{l}$ijk$\end{tabular}}}}%
    \put(0,0){\includegraphics[width=\unitlength,page=2]{Frobenius-proof-5.pdf}}%
    \put(-0.00524734,1.18162466){\color[rgb]{0,0,0}\makebox(0,0)[lt]{\lineheight{1.25}\smash{\begin{tabular}[t]{l}$ij$\end{tabular}}}}%
    \put(0,0){\includegraphics[width=\unitlength,page=3]{Frobenius-proof-5.pdf}}%
  \end{picture}%
\endgroup%
}}}~~~=G(\mu_{i,j}\otimes\id_k).
\end{align*}
Therefore $F=d_A^{-1}F(\mu_{i,j}\otimes\id_k)(\mu_{i,j}\otimes\id_k)^*=d_A^{-1}G(\mu_{i,j}\otimes\id_k)(\mu_{i,j}\otimes\id_k)^*=G$, which proves the first equation. The second one is the adjoint of the first one.
\end{proof}
The above Frobenius relations are the decisive property that makes a tensor product theory unitary. They are indeed closely related to the locality axiom of the categorical extensions of conformal nets \cite{Gui20} where the \emph{adjoint commutativity} of left and right actions plays a central role. In subsequent works we will relate the $C^*$-tensor categories of conformal net extensions and unitary VOA extensions using Frobenius relations. \\

We close this section by showing that the $C^*$-tensor structure of $\BIMA$ is independent of the choice of unitary tensor products. Suppose that we have two systems of unitary tensor products: for any objects $W_i,W_j$ in $\BIMA$ we have unitary tensor products $(W_{i\times j},\mu_{i,j}),(W_{i\bullet j},\eta_{i,j})$ of $W_i,W_j$ over $A$, which define (strict) $C^*$-tensor categories $(\BIMA,\boxtimes_A),(\BIMA,\boxdot_A)$. Tensor products of morphisms are written as $\otimes_A,\odot_A$ respectively. By uniqueness up to unitaries, there exists a unique unitary $\Phi_{i,j}\in\Hom_A(W_{i\times j},W_{i\bullet j})$ such that
\begin{align}
\eta_{i,j}=\Phi_{i,j}\mu_{i,j}.\label{eq50}
\end{align}

\begin{pp}
$\Phi$ is functorial: for any unitary $A$-modules $W_i,W_{i'},W_j,W_{j'}$ and any $F\in\Hom_A(W_i,W_{i'}),G\in\Hom_A(W_j,W_{j'})$, 
\begin{align}
\Phi_{i',j'}(F\otimes_A G)=(F\odot_A G)\Phi_{i,j}.
\end{align}
\end{pp}
\begin{proof}
We compute
\begin{align*}
\Phi_{i',j'}(F\otimes_A G)\mu_{i,j}=\Phi_{i',j'}\cdot\mu_{i',j'}(F\otimes G)=\eta_{i',j'}(F\otimes G)=(F\odot_A G)\eta_{i,j}=(F\odot_A G)\Phi_{i,j}\cdot\mu_{i,j}.
\end{align*}
Thus the desired equation is proved by universal property.
\end{proof}

\begin{thm}\label{lb41}
$\Phi$ induces an equivalence of $C^*$-tensor categories $(\BIMA,\boxtimes_A)\simeq(\BIMA,\boxdot_A)$. More precisely, for any unitary $A$-bimodules $W_i,W_j,W_k$,
\begin{itemize}
\item The following diagram commutes. 
\begin{gather}
\begin{CD}
W_{i\times j\times k} @>\quad\id_i\otimes_A\Phi_{j,k}\quad>>  W_{i\times(j\bullet k)}\\
@V \Phi_{i,j}\otimes_A\id_k VV @V \Phi_{i,j\bullet k} VV\\
W_{(i\bullet j)\times k} @> ~~\quad\Phi_{i\bullet j,k} \quad~~>> W_{i\bullet j\bullet k}
\end{CD}.\label{eq51}
\end{gather}
\item The following two morphisms equal $\id_i$.
\begin{gather}
W_i=W_{a\times i}\xrightarrow{\Phi_{a,i}} W_{a\bullet i}=W_i,\\
W_i=W_{i\times a}\xrightarrow{\Phi_{i,a}} W_{i\bullet a}=W_i.
\end{gather}
\end{itemize}
\end{thm}

\begin{proof}
To prove the first condition, we calculate
\begin{align*}
&\Phi_{i,j\bullet k}(\id_i\otimes_A\Phi_{j,k})\cdot\mu_{i,j\times k}(\id_i\otimes\mu_{j,k})=\Phi_{i,j\bullet k}\cdot\mu_{i,j\bullet k}(\id_i\otimes\Phi_{j,k})(\id_i\otimes\mu_{j,k})\\
=&\Phi_{i,j\bullet k}\cdot\mu_{i,j\bullet k}(\id_i\otimes\Phi_{j,k}\cdot \mu_{j,k})\xlongequal{\eqref{eq50}} \eta_{i,j\bullet k}(\id_i\otimes\eta_{j,k})\xlongequal{\eqref{eq48}} \eta_{i\bullet j,k}(\eta_{i,j}\otimes \id_k),
\end{align*}
and also
\begin{align*}
&\Phi_{i\bullet j,k}(\Phi_{i,j}\otimes_A\id_k)\cdot\mu_{i,j\times k}(\id_i\otimes\mu_{j,k})\xlongequal{\eqref{eq48}} \Phi_{i\bullet j,k}(\Phi_{i,j}\otimes_A\id_k)\cdot\mu_{i\times j,k}(\mu_{i,j}\otimes\id_k)\\
=&\Phi_{i\bullet j,k}\cdot \mu_{i\bullet j,k}(\Phi_{i,j}\otimes \id_k)(\mu_{i,j}\otimes\id_k)=\Phi_{i\bullet j,k}\cdot \mu_{i\bullet j,k}(\Phi_{i,j}\mu_{i,j}\otimes\id_k) \xlongequal{\eqref{eq50}} \eta_{i\bullet j,k}(\eta_{i,j}\otimes \id_k).
\end{align*}
This proves \eqref{eq51} since $(W_{i\times j\times k},\mu_{i,j\times k}(\id_i\otimes\mu_{j,k}))$ is a unitary tensor product of $W_i,W_j,W_k$ over $A$ by proposition \ref{lb31}.

Let $\mu^i_L,\mu^i_R$ be the left and right actions of $W_i$. Then under the identifications $i=a\times i=i\times a=a\bullet i=i\bullet a$, we know by equations \eqref{eq52} that $\mu_{a,i},\eta_{a,i}$ both equal $\mu^i_L$, and $\mu_{i,a},\eta_{i,a}$ both equal $\mu^i_R$. Thus, by \eqref{eq50}, $\Phi_{a,i}=\id_i=\Phi_{i,a}$.
\end{proof}

\subsection{Dualizable unitary bimodules}\label{lb39}

Let $(W_i,\mu^i_L,\mu^i_R)$ be a unitary $A$-bimodule as usual. Recall that $W_i$ is called $\mc C$-dualizable if it is dualizable as an object in $\mc C$. The notion of $\BIMA$-dualizability is understood in a similar way. In \cite{NY16} section 6.2, it was shown that if $W_i$ is $\mc C$-dualizable, then it is $\BIMA$-dualizable. The converse is also true by proposition 6.13 of \cite{NY16}.\footnote{Note that although  $A$ is assumed in \cite{NY16} to be standard, the results there also apply to the non-standard case since any $C^*$-Frobenius algebra is isomorphic to a standard Q-system by \cite{NY18a} theorem 2.9.} In this section, we give a slightly different proof of this result; see theorem \ref{lb34}.

We first assume that $W_i$ is $\mc C$-dualizable. Our proof of the $\BIMA$-dualizability is motivated by \cite{KO02} lemma 1.16 and \cite{CKM17} proposition 2.77. Notice that $W_{a\boxtimes i\boxtimes a}=W_a\boxtimes W_i\boxtimes W_a$ is naturally a unitary $A$-bimodule with left action $\mu\otimes\id_i\otimes \id_a$ and right action $\id_a\otimes\id_j\otimes\mu$. Moreover, $\mu_{LR}^i:=\mu_R^i(\mu_L^i\otimes\id_a)=\mu_L^i(\id_a\otimes\mu_R^i):W_a\boxtimes W_i\boxtimes W_a\rightarrow W_i$ is an $A$-bimodule morphism, and $d_A^{-1}\mu_{LR}^i$ is a partial isometry with range $\id_i$. Therefore $W_i$ is a sub $A$-bimodule of $W_a\boxtimes W_i\boxtimes W_a$, and hence it suffices to show that $W_a\boxtimes W_i\boxtimes W_a$ is $\BIMA$-dualizable.

Let $W_{\ovl i}$ be a dual object of $W_i$, and choose $\ev_{i,\ovl i},\ev_{\ovl i,i}$ of $W_i,W_{\ovl i}$. Then a natural candidate of dual bimodule of $W_a\boxtimes W_i\boxtimes W_a$ is $W_a\boxtimes W_{\ovl i}\boxtimes W_a$. Let us first understand their unitary tensor product over $A$. For this purpose, we choose a general unitary $A$-bimodule $(W_j,\mu^j_L,\mu^j_R)$, and check easily using proposition \ref{lb27} that $(W_a\boxtimes W_i\boxtimes W_a\boxtimes W_j\boxtimes W_a,\id_a\otimes\id_i\otimes\mu\otimes\id_i\otimes\id_a)$ is a unitary tensor product of $W_a\boxtimes W_i\boxtimes W_a$ and $W_a\boxtimes W_j\boxtimes W_a$ over $A$. We thus define the unitary tensor product of $W_a\boxtimes W_i\boxtimes W_a$ and $W_a\boxtimes W_{\ovl i}\boxtimes W_a$ over $A$ in this way. Briefly, $(a\boxtimes i\boxtimes a)(a\boxtimes \ovl i\boxtimes a)=a\boxtimes i\boxtimes a\boxtimes \ovl i\boxtimes a$, and similarly, $(a\boxtimes \ovl i\boxtimes a)(a\boxtimes i\boxtimes a)=a\boxtimes \ovl i\boxtimes a\boxtimes i\boxtimes a$.

We now define $\ev^A_{a\boxtimes i\boxtimes a,a\boxtimes \ovl i\boxtimes a}\in\Hom_A(W_{a\boxtimes i\boxtimes a\boxtimes \ovl i\boxtimes a},W_a)$ and $\ev^A_{a\boxtimes \ovl i\boxtimes a,a\boxtimes i\boxtimes a}\in\Hom_A(W_{a\boxtimes \ovl i\boxtimes a\boxtimes i\boxtimes a},W_a)$ by
\begin{gather*}
\ev^A_{a\boxtimes i\boxtimes a,a\boxtimes \ovl i\boxtimes a}=\mu(\id_a\otimes(\ev_{i,\ovl i}(\id_i\otimes\iota^*\otimes\id_{\ovl i}))\otimes \id_a),\\
\ev^A_{a\boxtimes \ovl i\boxtimes a,a\boxtimes i\boxtimes a}=\mu(\id_a\otimes(\ev_{\ovl i,i}(\id_{\ovl i}\otimes\iota^*\otimes\id_i))\otimes \id_a).
\end{gather*}
Since we also have $(a\boxtimes \ovl i\boxtimes a)(a\boxtimes i\boxtimes a)(a\boxtimes \ovl i\boxtimes a)=a\boxtimes \ovl i\boxtimes a\boxtimes i\boxtimes a\boxtimes \ovl i\boxtimes a$, we check using \eqref{eq45} and the associativity of $A$ that
\begin{gather*}
\id_{a\boxtimes\ovl i\boxtimes a}\otimes_A\ev^A_{a\boxtimes i\boxtimes a,a\boxtimes \ovl i\boxtimes a}=\id_a\otimes\id_{\ovl i}\otimes\ev^A_{a\boxtimes i\boxtimes a,a\boxtimes \ovl i\boxtimes a},\\
\ev^A_{a\boxtimes i\boxtimes a,a\boxtimes \ovl i\boxtimes a}\otimes_A \id_{a\boxtimes i\boxtimes a}=\ev^A_{a\boxtimes i\boxtimes a,a\boxtimes \ovl i\boxtimes a}\otimes\id_i\otimes\id_a,
\end{gather*}
and that $\ev^A_{a\boxtimes \ovl i\boxtimes a,a\boxtimes i\boxtimes a}$ satisfies similar relations. Using these  equations it is straightforward to check that $\ev^A_{a\boxtimes i\boxtimes a,a\boxtimes \ovl i\boxtimes a}$ and $\ev^A_{a\boxtimes \ovl i\boxtimes a,a\boxtimes i\boxtimes a}$ are evaluations in $\BIMA$ of $W_a\boxtimes W_i\boxtimes W_a$ and $W_a\boxtimes W_{\ovl i}\boxtimes W_a$, which proves that $W_a\boxtimes W_i\boxtimes W_a$ and hence $W_i$ are $\BIMA$-dualizable.

To prove the inverse direction we need the following lemma. 

\begin{lm}\label{lb33}
Let  $W_{\bar i}$ be a unitary $A$-module, not yet known to be dual to $W_i$. Suppose that we have morphisms $\ev^A_{i,\bar i}\in\Hom_A(W_{i\bar i},W_a)$ \index{eviiA@$\ev^A_{i,\ovl i}$} and $\ev^A_{\bar i,i}\in\Hom_A(W_{\bar ii},W_a)$. Set
\begin{align}
\ev_{i,\bar i}=\iota^*\ev^A_{i,\bar i}\cdot\mu_{i,\bar i},\qquad  \ev_{\bar i,i}=\iota^*\ev^A_{\bar i,i}\cdot\mu_{\bar i,i},\label{eq53}
\end{align}
and also set $\coev^A_{i,\bar i}=(\ev^A_{i,\bar i})^*,\coev^A_{\bar i,i}=(\ev^A_{\bar i,i})^*,\coev_{i,\bar i}=(\ev_{i,\bar i})^*,\coev_{\bar i,i}=(\ev_{\bar i,i})^*$. \index{coevA@$\coev^A_{i,\ovl i}$} Then
\begin{gather}
(\id_{\bar i}\otimes\ev_{i,\bar i})(\coev_{\bar i,i}\otimes\id_{\bar i})=(\id_{\bar i}\otimes_A\ev^A_{i,\bar i})(\coev^A_{\bar i,i}\otimes_A\id_{\bar i}),\label{eq20}\\
(\ev_{i,\bar i}\otimes\id_i)(\id_i\otimes \coev_{\bar i,i})=(\ev^A_{i,\bar i}\otimes_A\id_i)(\id_i\otimes_A \coev^A_{\bar i,i}).\label{eq21}
\end{gather}
\end{lm}

We say that $\ev_{i,\ovl i}$ and $\ev^A_{i,\ovl i}$, $\ev_{\ovl i,i}$ and $\ev^A_{\ovl i,i}$ are \textbf{correlated} if they satisfy \eqref{eq53}.
\begin{proof}
The two equations can be proved in a similar way, so we only prove the first one. Let $L,R$ be respectively the left  and right hand sides of \eqref{eq20}. Then
\begin{align*}
L=~\vcenter{\hbox{{\def\svgscale{1.2}%% Creator: Inkscape inkscape 0.92.4, www.inkscape.org
%% Accompanies image file 'ev-proof.pdf' (pdf, eps, ps)
\begingroup%
  \makeatletter%
  \providecommand\color[2][]{%
    \errmessage{(Inkscape) Color is used for the text in Inkscape, but the package 'color.sty' is not loaded}%
    \renewcommand\color[2][]{}%
  }%
  \providecommand\transparent[1]{%
    \errmessage{(Inkscape) Transparency is used (non-zero) for the text in Inkscape, but the package 'transparent.sty' is not loaded}%
    \renewcommand\transparent[1]{}%
  }%
  \providecommand\rotatebox[2]{#2}%
  \newcommand*\fsize{\dimexpr\f@size pt\relax}%
  \newcommand*\lineheight[1]{\fontsize{\fsize}{#1\fsize}\selectfont}%
  \ifx\svgwidth\undefined%
    \setlength{\unitlength}{71.60731237bp}%
    \ifx\svgscale\undefined%
      \relax%
    \else%
      \setlength{\unitlength}{\unitlength * \real{\svgscale}}%
    \fi%
  \else%
    \setlength{\unitlength}{\svgwidth}%
  \fi%
  \global\let\svgwidth\undefined%
  \global\let\svgscale\undefined%
  \makeatother%
  \begin{picture}(1,1.52690089)%
    \lineheight{1}%
    \setlength\tabcolsep{0pt}%
    \put(0,0){\includegraphics[width=\unitlength,page=1]{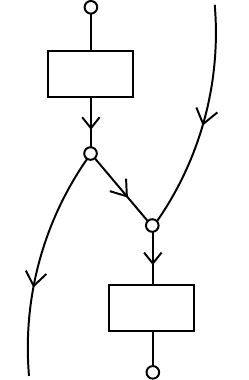}}%
    \put(0.91262333,1.22347217){\color[rgb]{0,0,0}\makebox(0,0)[lt]{\lineheight{1.25}\smash{\begin{tabular}[t]{l}$\bar i$\end{tabular}}}}%
    \put(0.41029936,0.62905134){\color[rgb]{0,0,0}\makebox(0,0)[lt]{\lineheight{1.25}\smash{\begin{tabular}[t]{l}$i$\end{tabular}}}}%
    \put(0.42372613,0.97218078){\color[rgb]{0,0,0}\makebox(0,0)[lt]{\lineheight{1.25}\smash{\begin{tabular}[t]{l}$\bar ii$\end{tabular}}}}%
    \put(0.67222616,0.47837266){\color[rgb]{0,0,0}\makebox(0,0)[lt]{\lineheight{1.25}\smash{\begin{tabular}[t]{l}$i\bar i$\end{tabular}}}}%
    \put(-0.00185257,0.21878725){\color[rgb]{0,0,0}\makebox(0,0)[lt]{\lineheight{1.25}\smash{\begin{tabular}[t]{l}$\bar i$\end{tabular}}}}%
    \put(0,0){\includegraphics[width=\unitlength,page=2]{ev-proof.pdf}}%
  \end{picture}%
\endgroup%
}}}\xlongequal{\eqref{eq49}}~\vcenter{\hbox{{\def\svgscale{1.2}%% Creator: Inkscape inkscape 0.92.4, www.inkscape.org
%% Accompanies image file 'ev-proof-2.pdf' (pdf, eps, ps)
\begingroup%
  \makeatletter%
  \providecommand\color[2][]{%
    \errmessage{(Inkscape) Color is used for the text in Inkscape, but the package 'color.sty' is not loaded}%
    \renewcommand\color[2][]{}%
  }%
  \providecommand\transparent[1]{%
    \errmessage{(Inkscape) Transparency is used (non-zero) for the text in Inkscape, but the package 'transparent.sty' is not loaded}%
    \renewcommand\transparent[1]{}%
  }%
  \providecommand\rotatebox[2]{#2}%
  \newcommand*\fsize{\dimexpr\f@size pt\relax}%
  \newcommand*\lineheight[1]{\fontsize{\fsize}{#1\fsize}\selectfont}%
  \ifx\svgwidth\undefined%
    \setlength{\unitlength}{52.43192508bp}%
    \ifx\svgscale\undefined%
      \relax%
    \else%
      \setlength{\unitlength}{\unitlength * \real{\svgscale}}%
    \fi%
  \else%
    \setlength{\unitlength}{\svgwidth}%
  \fi%
  \global\let\svgwidth\undefined%
  \global\let\svgscale\undefined%
  \makeatother%
  \begin{picture}(1,2.08531861)%
    \lineheight{1}%
    \setlength\tabcolsep{0pt}%
    \put(0,0){\includegraphics[width=\unitlength,page=1]{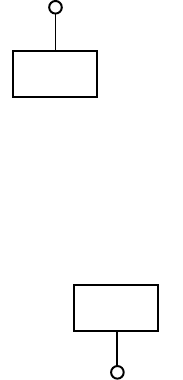}}%
    \put(0.88066796,1.73815628){\color[rgb]{0,0,0}\makebox(0,0)[lt]{\lineheight{1.25}\smash{\begin{tabular}[t]{l}$\bar i$\end{tabular}}}}%
    \put(0.17543894,1.31957553){\color[rgb]{0,0,0}\makebox(0,0)[lt]{\lineheight{1.25}\smash{\begin{tabular}[t]{l}$\bar ii$\end{tabular}}}}%
    \put(0.6153012,0.62683658){\color[rgb]{0,0,0}\makebox(0,0)[lt]{\lineheight{1.25}\smash{\begin{tabular}[t]{l}$i\bar i$\end{tabular}}}}%
    \put(0,0){\includegraphics[width=\unitlength,page=2]{ev-proof-2.pdf}}%
    \put(0.54036685,0.97116907){\color[rgb]{0,0,0}\makebox(0,0)[lt]{\lineheight{1.25}\smash{\begin{tabular}[t]{l}$\bar ii\bar i$\end{tabular}}}}%
    \put(-0.00176483,0.27401466){\color[rgb]{0,0,0}\makebox(0,0)[lt]{\lineheight{1.25}\smash{\begin{tabular}[t]{l}$\bar i$\end{tabular}}}}%
    \put(0,0){\includegraphics[width=\unitlength,page=3]{ev-proof-2.pdf}}%
  \end{picture}%
\endgroup%
}}}~\xlongequal{\eqref{eq45}}~\vcenter{\hbox{{\def\svgscale{1.2}%% Creator: Inkscape inkscape 0.92.4, www.inkscape.org
%% Accompanies image file 'ev-proof-3.pdf' (pdf, eps, ps)
\begingroup%
  \makeatletter%
  \providecommand\color[2][]{%
    \errmessage{(Inkscape) Color is used for the text in Inkscape, but the package 'color.sty' is not loaded}%
    \renewcommand\color[2][]{}%
  }%
  \providecommand\transparent[1]{%
    \errmessage{(Inkscape) Transparency is used (non-zero) for the text in Inkscape, but the package 'transparent.sty' is not loaded}%
    \renewcommand\transparent[1]{}%
  }%
  \providecommand\rotatebox[2]{#2}%
  \newcommand*\fsize{\dimexpr\f@size pt\relax}%
  \newcommand*\lineheight[1]{\fontsize{\fsize}{#1\fsize}\selectfont}%
  \ifx\svgwidth\undefined%
    \setlength{\unitlength}{53.17973255bp}%
    \ifx\svgscale\undefined%
      \relax%
    \else%
      \setlength{\unitlength}{\unitlength * \real{\svgscale}}%
    \fi%
  \else%
    \setlength{\unitlength}{\svgwidth}%
  \fi%
  \global\let\svgwidth\undefined%
  \global\let\svgscale\undefined%
  \makeatother%
  \begin{picture}(1,2.0556421)%
    \lineheight{1}%
    \setlength\tabcolsep{0pt}%
    \put(0.882346,1.72174952){\color[rgb]{0,0,0}\makebox(0,0)[lt]{\lineheight{1.25}\smash{\begin{tabular}[t]{l}$\bar i$\end{tabular}}}}%
    \put(0,0){\includegraphics[width=\unitlength,page=1]{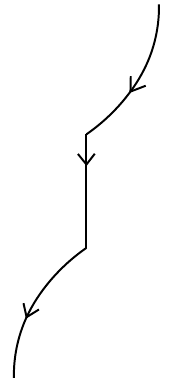}}%
    \put(-0.00174002,0.2781965){\color[rgb]{0,0,0}\makebox(0,0)[lt]{\lineheight{1.25}\smash{\begin{tabular}[t]{l}$\bar i$\end{tabular}}}}%
    \put(0,0){\includegraphics[width=\unitlength,page=2]{ev-proof-3.pdf}}%
    \put(0.52975548,1.13828168){\color[rgb]{0,0,0}\makebox(0,0)[lt]{\lineheight{1.25}\smash{\begin{tabular}[t]{l}$\bar i$\end{tabular}}}}%
    \put(0.5325965,0.71091078){\color[rgb]{0,0,0}\makebox(0,0)[lt]{\lineheight{1.25}\smash{\begin{tabular}[t]{l}$\bar i$\end{tabular}}}}%
  \end{picture}%
\endgroup%
}}}=R.
\end{align*}
\end{proof}

Now if $W_i$ is $\BIMA$-dualizable, then we can find a unitary $A$-bimodule $W_{\ovl i}$ dual to $W_i$, and evaluations $\ev^A_{i,\ovl i},\ev^A_{\ovl i,i}$ of $W_i,W_{\ovl i}$ in $\BIMA$. Define $\ev_{i,\ovl i},\ev_{\ovl i,i}$ by equations \eqref{eq53}. Then equations \eqref{eq20} and \eqref{eq21} imply that $\ev_{i,\ovl i},\ev_{\ovl i,i}$ are evaluations of $W_i,W_{\ovl i}$ in $\mc C$. Thus $W_i$ is $\mc C$-dualizable. This finishes the proof of the following theorem.

\begin{thm}\label{lb34}
If $W_i$ is a unitary $A$-bimodule, then $W_i$ is $\mc C$-dualizable if and only if $W_i$ is $\BIMA$-dualizable. Moreover, one can choose correlated evaluations $\ev,\ev^A$ in $\mc C$ and $\BIMA$.
\end{thm}

By the above theorem, we will no longer distinguish between $\mc C$- and $\BIMA$-dualizability.  Using the same argument as lemma \ref{lb33} one also proves that under correlated evaluations, the $\mc C$-transposes and $\BIMA$-transposes of a unitary $A$-bimodule morphism $F$ are equal. Therefore the symbols $^\vee F$ and $F^\vee$ are defined unambiguously. Compare \cite{NY16} lemma 6.10.
\begin{pp}
Let $W_i,W_j$ be dualizable unitary $A$-bimodules. Choose dual objects $W_{\ovl i},W_{\ovl j}$, and evaluations $\ev,\ev^A$ (with suitable subscripts) in $\mc C$ and $\BIMA$ respectively. Assume that $\ev,\ev^A$ are correlated. Then for any $F\in\Hom_A(W_i,W_j)$, its transposes in $\mc C$ are the same  as those in $\BIMA$. More precisely, we have
\begin{align}
&(\id_{\bar i}\otimes\ev_{j,\bar j})(\id_{\ovl i}\otimes F\otimes\id_{\ovl j})(\coev_{\bar i,i}\otimes\id_{\bar j})\nonumber\\
=&(\id_{\bar i}\otimes_A\ev^A_{j,\bar j})(\id_{\ovl i}\otimes_A F\otimes_A\id_{\ovl j})(\coev^A_{\bar i,i}\otimes_A\id_{\bar j}),\\[1ex]
&(\ev_{\ovl j,j}\otimes\id_{\ovl i})(\id_{\ovl j}\otimes F\otimes \id_{\ovl i})(\id_{\ovl j}\otimes\coev_{i,\ovl i})\nonumber\\
=&(\ev^A_{\ovl j,j}\otimes_A\id_{\ovl i})(\id_{\ovl j}\otimes_A F\otimes_A \id_{\ovl i})(\id_{\ovl j}\otimes_A\coev^A_{i,\ovl i}).
\end{align}
\end{pp}

Recall that for any $F\in\End(W_i)$, one can define scalars $\Tr_L(F),\Tr_R(F)$ such that $\ev_{i,\ovl i}(F\otimes\id_{\ovl i})\coev_{i,\ovl i}=\Tr_L(F)\id_0$ and $\ev_{\ovl i,i}(\id_{\ovl i}\otimes F)\coev_{\ovl i,i}=\Tr_R(F)\id_0$. If $A$ is \textbf{simple} in the sense that $\End_A(W_a)=\mathbb C\id_a$, and $F\in\End_A(W_i)$, one can similarly define scalars $\Tr^A_L(F),\Tr^A_R(F)$\index{TrA@$\Tr^A_L,\Tr^A_R$} such that $\ev^A_{i,\ovl i}(F\otimes_A\id_{\ovl i})\coev^A_{i,\ovl i}=\Tr^A_L(F)\id_a$ and $\ev^A_{\ovl i,i}(\id_{\ovl i}\otimes_A F)\coev^A_{\ovl i,i}=\Tr_R(F)\id_a$. In the case that $\ev^A$ and $\ev$ are correlated, these two traces satisfy very simple relations:

\begin{pp}\label{lb37}
If $A$ is a simple $Q$-system,  $W_i,W_{\ovl i}$ are mutually dual unitary $A$-bimodules, and the $\ev$ and $\ev^A$ for $W_i$ are correlated, then for any $F\in\End_A(W_i)$, we have
\begin{align}
\Tr_L(F)=D_A\Tr^A_L(F),\qquad\Tr_R(F)=D_A\Tr^A_R(F).\label{eq56}
\end{align}
As a consequence, $\ev^A$ are standard if the correlated  $\ev$ are so.
\end{pp}

\begin{proof}
We only prove the relation for left traces.
\begin{align*}
&\Tr_L(F)\id_0=\ev_{i,\ovl i}(F\otimes\id_{\ovl i})\coev_{i,\ovl i}\xlongequal{\eqref{eq53}} \iota^*\ev^A_{i,\ovl i}\cdot\mu_{i,\ovl i}(F\otimes\id_{\ovl i})(\mu_{i,\ovl i})^*\coev^A_{i,\ovl i}\iota\\
=&\iota^*\ev^A_{i,\ovl i}(F\otimes_A\id_{\ovl i})\cdot\mu_{i,\ovl i}(\mu_{i,\ovl i})^*\coev^A_{i,\ovl i}\iota=d_A\iota^*\ev^A_{i,\ovl i}(F\otimes_A\id_{\ovl i})\cdot\coev^A_{i,\ovl i}\iota\\
=&d_A\Tr^A_L(F)\iota^*\iota=D_A\Tr^A_L(F)\id_0.
\end{align*}
\end{proof}

Note that a simple $C^*$-Frobenius algebra is always a simple Q-system, since $\mu\mu^*$ is in $\End_A(W_a)$, which must be a scalar and hence proves the specialness. Examples of simple Q-systems include haploid $C^*$-Frobenius algebras, since in general we have $\dim\Hom(W_0,W_a)=\dim\End_{A,-}(W_a)\geq\dim\End_A(W_a)$ (cf. \cite{NY18a} remark 2.7-(1)).  Recall  that haploid $C^*$-Frobenius algebras are also standard ($D_A=d_a$) by proposition \ref{lb36}. As a consequence, the $\mc C$-algebra $A_U$ associated to a unitary VOA extension $U$ is haploid and hence a simple standard Q-system.

\subsubsection*{Construction of dual bimodules and correlated $\ev^A$}

In the remaining part of this section we assume that $A$ is standard. Then for  a dualizable unitary $A$-bimodule $(W_i,\mu^i_L,\mu^i_R)$ one can explicitly construct the dual bimodule  and $\ev^A$ following \cite{KO02} figures 9-11 or \cite{NY18b} section 4.1. This construction will be used in the next section to understand the ribbon structure of the unitary representation category of $A$.

Since $A$ is now standard, $\ev_{a,a}:=\iota^*\mu$ is a standard evaluation of $W_a$. Choose an object $W_{\ovl i}$ in $\mc C$ dual to $W_i$, and choose standard $\ev$ for $W_i,W_{\ovl i}$. Recall convention \ref{lb35}.  Motivated by corollary \ref{lb13}, we define
\begin{gather}
\mu^{\ovl i}_L=((\mu^i_R)^*)^\vee,\qquad \mu^{\ovl i}_R=((\mu^i_L)^*)^\vee.\label{eq54}
\end{gather}
Then using graphical calculus it is not hard to verify that $(W_{\ovl i},\mu^{\ovl i}_L,\mu^{\ovl i}_R)$ is a unitary $A$-bimodule. (Note that the standardness is used to verify the unitarity.) Moreover, using the above definition, and noting that $(\cdot)^\vee={^\vee(\cdot)}$, one checks that
\begin{gather}
\fk e_{i,\ovl i}:=~\vcenter{\hbox{{\def\svgscale{0.6}}%% Creator: Inkscape inkscape 0.92.4, www.inkscape.org
%% Accompanies image file 'ev-proof-4.pdf' (pdf, eps, ps)
\begingroup%
  \makeatletter%
  \providecommand\color[2][]{%
    \errmessage{(Inkscape) Color is used for the text in Inkscape, but the package 'color.sty' is not loaded}%
    \renewcommand\color[2][]{}%
  }%
  \providecommand\transparent[1]{%
    \errmessage{(Inkscape) Transparency is used (non-zero) for the text in Inkscape, but the package 'transparent.sty' is not loaded}%
    \renewcommand\transparent[1]{}%
  }%
  \providecommand\rotatebox[2]{#2}%
  \newcommand*\fsize{\dimexpr\f@size pt\relax}%
  \newcommand*\lineheight[1]{\fontsize{\fsize}{#1\fsize}\selectfont}%
  \ifx\svgwidth\undefined%
    \setlength{\unitlength}{50.64058246bp}%
    \ifx\svgscale\undefined%
      \relax%
    \else%
      \setlength{\unitlength}{\unitlength * \real{\svgscale}}%
    \fi%
  \else%
    \setlength{\unitlength}{\svgwidth}%
  \fi%
  \global\let\svgwidth\undefined%
  \global\let\svgscale\undefined%
  \makeatother%
  \begin{picture}(1,0.90122571)%
    \lineheight{1}%
    \setlength\tabcolsep{0pt}%
    \put(0,0){\includegraphics[width=\unitlength,page=1]{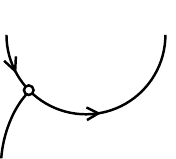}}%
    \put(0.00795785,0.81746264){\color[rgb]{0,0,0}\makebox(0,0)[lt]{\lineheight{1.25}\smash{\begin{tabular}[t]{l}$i$\end{tabular}}}}%
    \put(0.88321496,0.84044195){\color[rgb]{0,0,0}\makebox(0,0)[lt]{\lineheight{1.25}\smash{\begin{tabular}[t]{l}$i$\end{tabular}}}}%
  \end{picture}%
\endgroup%
}}~=~\vcenter{\hbox{{\def\svgscale{0.6}}%% Creator: Inkscape inkscape 0.92.4, www.inkscape.org
%% Accompanies image file 'ev-proof-5.pdf' (pdf, eps, ps)
\begingroup%
  \makeatletter%
  \providecommand\color[2][]{%
    \errmessage{(Inkscape) Color is used for the text in Inkscape, but the package 'color.sty' is not loaded}%
    \renewcommand\color[2][]{}%
  }%
  \providecommand\transparent[1]{%
    \errmessage{(Inkscape) Transparency is used (non-zero) for the text in Inkscape, but the package 'transparent.sty' is not loaded}%
    \renewcommand\transparent[1]{}%
  }%
  \providecommand\rotatebox[2]{#2}%
  \newcommand*\fsize{\dimexpr\f@size pt\relax}%
  \newcommand*\lineheight[1]{\fontsize{\fsize}{#1\fsize}\selectfont}%
  \ifx\svgwidth\undefined%
    \setlength{\unitlength}{49.93485797bp}%
    \ifx\svgscale\undefined%
      \relax%
    \else%
      \setlength{\unitlength}{\unitlength * \real{\svgscale}}%
    \fi%
  \else%
    \setlength{\unitlength}{\svgwidth}%
  \fi%
  \global\let\svgwidth\undefined%
  \global\let\svgscale\undefined%
  \makeatother%
  \begin{picture}(1,0.91396265)%
    \lineheight{1}%
    \setlength\tabcolsep{0pt}%
    \put(0,0){\includegraphics[width=\unitlength,page=1]{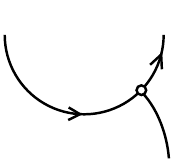}}%
    \put(-0.00606259,0.82901576){\color[rgb]{0,0,0}\makebox(0,0)[lt]{\lineheight{1.25}\smash{\begin{tabular}[t]{l}$i$\end{tabular}}}}%
    \put(0.88156445,0.85231984){\color[rgb]{0,0,0}\makebox(0,0)[lt]{\lineheight{1.25}\smash{\begin{tabular}[t]{l}$i$\end{tabular}}}}%
  \end{picture}%
\endgroup%
}}~~\in\Hom_A(W_i\boxtimes W_{\ovl i},W_a),\label{eq57}\\[2ex]
\fk e_{\ovl i,i}:=~\vcenter{\hbox{{\def\svgscale{0.6}}%% Creator: Inkscape inkscape 0.92.4, www.inkscape.org
%% Accompanies image file 'ev-proof-6.pdf' (pdf, eps, ps)
\begingroup%
  \makeatletter%
  \providecommand\color[2][]{%
    \errmessage{(Inkscape) Color is used for the text in Inkscape, but the package 'color.sty' is not loaded}%
    \renewcommand\color[2][]{}%
  }%
  \providecommand\transparent[1]{%
    \errmessage{(Inkscape) Transparency is used (non-zero) for the text in Inkscape, but the package 'transparent.sty' is not loaded}%
    \renewcommand\transparent[1]{}%
  }%
  \providecommand\rotatebox[2]{#2}%
  \newcommand*\fsize{\dimexpr\f@size pt\relax}%
  \newcommand*\lineheight[1]{\fontsize{\fsize}{#1\fsize}\selectfont}%
  \ifx\svgwidth\undefined%
    \setlength{\unitlength}{50.64058246bp}%
    \ifx\svgscale\undefined%
      \relax%
    \else%
      \setlength{\unitlength}{\unitlength * \real{\svgscale}}%
    \fi%
  \else%
    \setlength{\unitlength}{\svgwidth}%
  \fi%
  \global\let\svgwidth\undefined%
  \global\let\svgscale\undefined%
  \makeatother%
  \begin{picture}(1,0.90122571)%
    \lineheight{1}%
    \setlength\tabcolsep{0pt}%
    \put(0,0){\includegraphics[width=\unitlength,page=1]{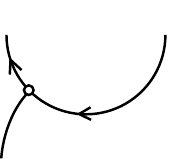}}%
    \put(0.00795785,0.81746264){\color[rgb]{0,0,0}\makebox(0,0)[lt]{\lineheight{1.25}\smash{\begin{tabular}[t]{l}$i$\end{tabular}}}}%
    \put(0.88321496,0.84044195){\color[rgb]{0,0,0}\makebox(0,0)[lt]{\lineheight{1.25}\smash{\begin{tabular}[t]{l}$i$\end{tabular}}}}%
  \end{picture}%
\endgroup%
}}~=~\vcenter{\hbox{{\def\svgscale{0.6}}%% Creator: Inkscape inkscape 0.92.4, www.inkscape.org
%% Accompanies image file 'ev-proof-7.pdf' (pdf, eps, ps)
\begingroup%
  \makeatletter%
  \providecommand\color[2][]{%
    \errmessage{(Inkscape) Color is used for the text in Inkscape, but the package 'color.sty' is not loaded}%
    \renewcommand\color[2][]{}%
  }%
  \providecommand\transparent[1]{%
    \errmessage{(Inkscape) Transparency is used (non-zero) for the text in Inkscape, but the package 'transparent.sty' is not loaded}%
    \renewcommand\transparent[1]{}%
  }%
  \providecommand\rotatebox[2]{#2}%
  \newcommand*\fsize{\dimexpr\f@size pt\relax}%
  \newcommand*\lineheight[1]{\fontsize{\fsize}{#1\fsize}\selectfont}%
  \ifx\svgwidth\undefined%
    \setlength{\unitlength}{49.93485797bp}%
    \ifx\svgscale\undefined%
      \relax%
    \else%
      \setlength{\unitlength}{\unitlength * \real{\svgscale}}%
    \fi%
  \else%
    \setlength{\unitlength}{\svgwidth}%
  \fi%
  \global\let\svgwidth\undefined%
  \global\let\svgscale\undefined%
  \makeatother%
  \begin{picture}(1,0.91396265)%
    \lineheight{1}%
    \setlength\tabcolsep{0pt}%
    \put(0,0){\includegraphics[width=\unitlength,page=1]{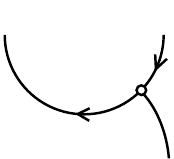}}%
    \put(-0.00606259,0.82901576){\color[rgb]{0,0,0}\makebox(0,0)[lt]{\lineheight{1.25}\smash{\begin{tabular}[t]{l}$i$\end{tabular}}}}%
    \put(0.88156445,0.85231984){\color[rgb]{0,0,0}\makebox(0,0)[lt]{\lineheight{1.25}\smash{\begin{tabular}[t]{l}$i$\end{tabular}}}}%
  \end{picture}%
\endgroup%
}}~~\in\Hom_A(W_{\ovl i}\boxtimes W_i,W_a),\label{eq58}
\end{gather}
and that $\fk e_{i,\ovl i}\Psi_{i,\ovl i}=0=\fk e_{\ovl i,i}\Psi_{\ovl i,i}$. Therefore there exist $\ev^A_{i,\ovl i}\in\Hom_A(W_{i,\ovl i},W_a),\ev^A_{\ovl i,i}\in\Hom_A(W_{\ovl i,i},W_a)$ satisfying
\begin{gather}
\fk e_{i,\ovl i}=\ev^A_{i,\ovl i}\mu_{i,\ovl i},\qquad \fk e_{\ovl i,i}=\ev^A_{\ovl i,i}\mu_{\ovl i,i}.\label{eq55}
\end{gather}
By unit property, $\ev^A_{i,\ovl i}$ and $\ev_{i,\ovl i}$, $\ev^A_{\ovl i,i}$ and $\ev_{\ovl i,i}$ are correlated. Therefore, by lemma \ref{lb33},  $\ev^A_{i,\ovl i}$ and $\ev^A_{\ovl i,i}$ are evaluations of $W_i,W_{\ovl i}$ in $\BIMA$. If, moreover, $A$ is simple, then $\ev^A$ are standard, and $\Tr_L(F)=\Tr_R(F)=d_a\Tr^A_L(F)=d_a\Tr^A_R(F)$ for any $F\in\End_A(W_i)$ by proposition \ref{lb37}. By the uniqueness up to unitaries of standard evaluations, the values of traces are independent of the choice of standard evaluations. Therefore we have (cf. \cite{KO02} theorem 1.18 and \cite{NY16} proposition 6.9.):

\begin{thm}
	If $A$ is a simple and standard Q-system, $W_i$ is a dualizable unitary $A$-bimodule, and $\Tr:=\Tr_L=\Tr_R$  and $\Tr^A:=\Tr^A_L=\Tr^A_R$ are defined using (not necessarily correlated) standard $\ev$ and standard $\ev^A$ respectively. Then $\Tr(F)=d_a\Tr^A(F)$, where $d_a$ is the ($\mc C$-)quantum dimension of $W_a$. In particular, the $\mc C$-quantum dimension of $W_i$ equals $d_a$ multiplied by the $\BIMA$-quantum dimension of $W_i$.
\end{thm}

%\begin{pp}
%Let $A$ be a standard $Q$-system, $(W_i,\mu^i_L,\mu^i_R)$  a dualizable unitary $A$-bimodule, and $W_{\ovl i}$ an object dual to $W_i$ in $\mc C$. Set $\ev_{a,a}=\iota^*\mu$, and choose standard evaluations $\ev_{i,\ovl i},\ev_{\ovl i,i}$ for $W_i,W_{\ovl i}$. Then $(W_{\ovl i},\mu^{\ovl i}_L,\mu^{\ovl i}_R)$ is a unitary $A$-bimodule dual to $(W_i,\mu^i_L,\mu^i_R)$, where $\mu^i_L,\mu^i_R$ are defined by \eqref{eq54}. Moreover, the $\ev^A_{i,\ovl i},\ev^A_{\ovl i,i}$ defined in \eqref{eq55} are evaluations in $\BIMA$ of $W_i,W_{\ovl i}$, and are correlated to $\ev_{i,\ovl i},\ev_{\ovl i,i}$.
%\end{pp}

\subsection{Braiding and ribbon structures}

In this section, $\mc C$ is a braided $C^*$-tensor category with (unitary) braiding $\ss$ \index{zz@$\ss=\ss_{i,j}$} and simple $W_0$, and $A$ is a commutative $Q$-system in $\mc C$. Let $\RepA$ \index{RepA@$\RepA$} be the $C^*$-category of single-valued unitary left $A$-modules. As discussed in section \ref{lb38}, single-valued unitary left $A$-modules admits a canonical unitary bimodule structure, $\RepA$ is a full $C^*$-subcategory of $\BIMA$, and $\Hom_{A,-},\Hom_{-,A},\Hom_A$ are the same for $\RepA$. If $W_i,W_j$ are in $\RepA$, $(W_k,\mu^k_L,\mu^k_R)$ is in $\BIMA$, and $\alpha\in\Hom_A(W_i\boxtimes W_j,W_k)$ satisfies $\alpha\Pij=0$, then one can show easily using graphical calculus that $\mu^k_L\ss_{k,a}(\alpha\otimes\id_a)=\mu^k_R(\alpha\otimes\id_a)$. Now we choose a unitary tensor product $(W_{ij},\mu_{i,j})$ of $W_i,W_j$ over $A$, where $W_{ij}$ is a unitary $A$-bimodule with left and right actions $\mu^{ij}_L,\mu^{ij}_R$. Set $W_k=W_{ij},\alpha=\mu_{i,j}$. Then we have $\mu^{ij}_L\ss_{ij,a}=\mu^{ij}_R$ since $(\alpha\otimes\id_a)(\alpha\otimes\id_a)^*=(\mu_{i,j}\otimes\id_a)(\mu_{i,j}\otimes\id_a)^*=d_A\id_{ij}\otimes\id_a$. Similar argument shows $\mu^{ij}_L\ss_{a,ij}^{-1}=\mu^{ij}_R$. Therefore $W_{ij}$ is  single-valued with left and right actions related by $\ss$. We conclude that $\RepA$ is closed under unitary tensor products. In other words, $\RepA$ is a full $C^*$-tensor subcategory of $\BIMA$.

\begin{pp}[cf. \cite{KO02} theorem 1.15]\label{lb42}
If $A$ is standard, and $W_i$ is an object in $\RepA$, then $W_i$ is $\RepA$-dualizable if and only if $W_i$ is $\BIMA$-dualizable (equivalently, $\mc C$-dualizable).
\end{pp}

\begin{proof}
We have seen in theorem \ref{lb34} that $\BIMA$-dualizability and $\mc C$-dualizability are the same. $\RepA$-dualizability clearly implies $\BIMA$-dualizability. Now assume that $W_i$ is $\mc C$-dualizable. In section \ref{lb39} we have constructed a unitary $A$-bimodule $W_{\ovl i}$ dual to $W_i$. It is easy to check that the left and right actions of $W_{\ovl i}$ defined by \eqref{eq54} are related by the braiding $\ss$ of $\mc C$. In particular, $W_{\ovl i}$ is an object in $\RepA$. Thus $W_i$ is $\RepA$-dualizable.
\end{proof}

\subsubsection*{Braiding}

We now define braiding for $\RepA$. Let $W_i,W_j$ be objects in $\RepA$, and let $(W_{ij},\mu_{i,j})$ and $(W_{ji},\mu_{j,i})$ be respectively the unitary tensor products over $A$ of $W_i,W_j$ and $W_j,W_i$ used to define the tensor structure of $\BIMA$. Since the braiding of $\mc C$ is unitary, using proposition \ref{lb27} one easily shows that $(W_{ij},\mu_{j,i}\ss_{i,j})$ is also a unitary tensor product of $W_j,W_i$ over $A$. Hence there exists a unique unitary $\ss^A_{i,j}\in\End_A(W_{ij},W_{ji})$ such that
\begin{align}
\mu_{j,i}\ss_{i,j}=\ss^A_{i,j}\mu_{i,j}.\label{eq60}
\end{align}

\begin{thm}
$(\RepA,\boxtimes_A,\ss^A)$ is a braided $C^*$-tensor category.
\end{thm}
\begin{proof}
The hexagon axioms
\begin{gather*}
(\ss^A_{i,k}\otimes_A\id_j)(\id_i\otimes_A\ss^A_{j,k})=\ss^A_{ij,k},\\
((\ss^A_{k,i})^{-1}\otimes_A\id_j)(\id_i\otimes_A(\ss^A_{k,j})^{-1})=(\ss^A_{k,ij})^{-1}
\end{gather*}
(for all $W_i,W_j,W_k$ in $\RepA$) can be proved in a similar way as pentagon axiom (proposition \ref{lb40}): one shows that both sides are equal when multiplied from the right by $\mu_{i,jk}(\id_i\otimes\mu_{j,k})=\mu_{ij,k}(\mu_{i,j}\otimes\id_k)$.
\end{proof}

Now assume that we have two systems of unitary tensor products $(W_{i\times j},\mu_{i,j}),(W_{i\bullet j},\eta_{i,j})$ which define two braided $C^*$-tensor categories $(\RepA,\boxtimes_A,\ss^A)$ and $(\RepA,\boxdot_A,\sigma^A)$. By theorem \ref{lb41}, the functorial unitary $\Phi$ defined by \eqref{eq50} induces an equivalence of the $C^*$-tensor categories. Indeed, it also preserves the braidings:

\begin{thm}
The functorial unitary $\Phi$ defined by \eqref{eq50} induces an equivalence of the braided $C^*$-tensor categories $(\RepA,\boxtimes_A,\ss^A)\simeq(\RepA,\boxdot_A,\sigma^A)$, which means that $\Phi$ satisfies the two conditions of theorem \ref{lb41}, together with the condition that for any objects $W_i,W_j$ in $\RepA$,
\begin{align}
\Phi_{j,i}\ss^A_{i,j}=\sigma^A_{i,j}\Phi_{i,j}.
\end{align}
\end{thm}

\begin{proof}
One verifies that $\Phi_{j,i}\ss^A_{i,j}\mu_{i,j}=\sigma^A_{i,j}\eta_{i,j}=\sigma^A_{i,j}\Phi_{i,j}\mu_{i,j}$.
\end{proof}

\subsubsection*{Ribbon structures}

Let us now assume that $\mc C$ is rigid, which means that any object of $\mc C$ is dualizable. By \cite{Mueg00} proposition 2.4, there is a canonical twist operator $\vartheta=\vartheta_i\in\End(W_i)$  for any object $W_i$ in $\mc C$: Choose $W_{\ovl i}$ dual to $W_i$, standard $\ev_{i,\ovl i},\ev_{\ovl i,i}$ and corresponding $\coev$ for $W_i,W_{\ovl i}$. Then by standardness of $\ev$ one can show \index{zz@$\vartheta=\vartheta_i$}
\begin{align}
&\vartheta_i:=(\ev_{\ovl i,i}\otimes\id_i)(\id_{\ovl i}\otimes \ss_{i,i})(\coev_{\ovl i,i}\otimes\id_i)\nonumber\\
=&(\id_i\otimes\ev_{i,\ovl i})(\ss_{i,i}\otimes\id_{\ovl i})(\id_i\otimes\coev_{i,\ovl i}).
\end{align}
(Note that by uniqueness up to unitaries of standard $\ev$, $\vartheta_i$ is independent of the choice of standard evaluations.) By this relation,  $\vartheta$ is unitary. Moreover, $\vartheta$ defines a ribbon structure on $\mc C$ (i.e., $\vartheta$ commutes with all morphisms, $\vartheta_{i\boxtimes j}=(\vartheta_i\otimes\vartheta_j)\ss_{j,i}\ss_{i,j}$, and $\vartheta_i^\vee=\vartheta_{\ovl i}$). Then $(\mc C,\boxtimes,\ss,\vartheta)$ is a rigid $C^*$-ribbon category. Using the definition of $\vartheta_i$, one easily shows
\begin{align}
\ev_{i,\ovl i}=\ev_{\ovl i,i}\ss_{i,\ovl i}(\vartheta_i\otimes\id_{\ovl i}),\label{eq59}
\end{align}
which completely determines the morphism $\vartheta_i$. In the case of $\RepV$, we have shown in \cite{Gui19b} section 7.3 (especially equation (7.30), which relies on \cite{Gui19a} formula (1.41)) that $e^{2\im\pi L_0}$ satisfies the above equation.  Thus the twist $\vartheta=e^{2\im\pi L_0}$ defined in the end of section \ref{lb2} is the canonical twist of the braided $C^*$-fusion category (unitary braided fusion category) $\RepV$.

Suppose now that $A$ is haploid. By the commutativity of $A$, haploidness is equivalent to simpleness since $\End_{A,-}(W_a)=\End_A(W_a)$. $A$ is also standard by proposition \ref{lb36}. Therefore, by theorem \ref{lb42}, $\RepA$ is rigid. Thus $\RepA$ also admits a canonical twist $\vartheta^A$ \index{zz@$\vartheta^A=\vartheta^A_i$} under which $\RepA$ becomes a $C^*$-ribbon category. The twist satisfies
\begin{align}
&\vartheta_i:=(\ev^A_{\ovl i,i}\otimes_A\id_i)(\id_{\ovl i}\otimes_A \ss^A_{i,i})(\coev^A_{\ovl i,i}\otimes_A\id_i)\nonumber\\
=&(\id_i\otimes_A\ev^A_{i,\ovl i})(\ss^A_{i,i}\otimes_A\id_{\ovl i})(\id_i\otimes_A\coev^A_{i,\ovl i}),
\end{align}
where $W_{\ovl i}$ is an object of $\RepA$ dual to $W_i$, and $\ev^A_{i,\ovl i},\ev^A_{\ovl i,i}$ are standard evaluations for $W_i,W_{\ovl i}$. We now show that the ribbon structures of $\mc C$ and $\RepA$ are compatible.

\begin{thm}\label{lb44}
Suppose that $\mc C$ is a rigid braided $C^*$-tensor category,  $A$ is a haploid commutative Q-system in $\mc C$, and $\vartheta$ and $\vartheta^A$ are the canonical unitary twists  of $\mc C$  and $\RepA$ respectively. Then $\vartheta_i=\vartheta^A_i$ for any object $W_i$ in $\RepA$. 
\end{thm}	

As an immediate consequence, $W_a$ has trivial ($\mc C$-)twist since $\vartheta^A_a=\id_a$.

\begin{proof}
Choose any $W_i$ in $\RepA$. Using the definition of twist one checks easily that $\vartheta_i\in\End_A(W_i)$. Let $W_{\ovl i}$ be a dual object in $\mc C$, equip $W_{\ovl i}$ with a unitary $A$-bimodule structure by \eqref{eq54}, and use equations \eqref{eq57}, \eqref{eq58}, and \eqref{eq55} to define standard $\RepA$-evaluations $\ev^A$ correlated to standard $\mc C$-evaluations $\ev$ for $W_i,W_{\ovl i}$.  We now prove $\vartheta_i=\vartheta^A_i$ by showing
\begin{align}
\ev^A_{i,\ovl i}=\ev^A_{\ovl i,i}\ss^A_{i,\ovl i}(\vartheta_i\otimes_A\id_{\ovl i}).
\end{align}
We compute
\begin{align*}
&\qquad\ev^A_{i,\ovl i}\cdot \mu_{i,\ovl i} \xlongequal{\eqref{eq55}} \fk e_{i,\ovl i}\xlongequal{\eqref{eq57}} ~\vcenter{\hbox{{\def\svgscale{0.6}}}}~\xlongequal{\eqref{eq59}}~\vcenter{\hbox{{\def\svgscale{0.6}}%% Creator: Inkscape inkscape 0.92.4, www.inkscape.org
%% Accompanies image file 'ev-proof-8.pdf' (pdf, eps, ps)
\begingroup%
  \makeatletter%
  \providecommand\color[2][]{%
    \errmessage{(Inkscape) Color is used for the text in Inkscape, but the package 'color.sty' is not loaded}%
    \renewcommand\color[2][]{}%
  }%
  \providecommand\transparent[1]{%
    \errmessage{(Inkscape) Transparency is used (non-zero) for the text in Inkscape, but the package 'transparent.sty' is not loaded}%
    \renewcommand\transparent[1]{}%
  }%
  \providecommand\rotatebox[2]{#2}%
  \newcommand*\fsize{\dimexpr\f@size pt\relax}%
  \newcommand*\lineheight[1]{\fontsize{\fsize}{#1\fsize}\selectfont}%
  \ifx\svgwidth\undefined%
    \setlength{\unitlength}{47.24581359bp}%
    \ifx\svgscale\undefined%
      \relax%
    \else%
      \setlength{\unitlength}{\unitlength * \real{\svgscale}}%
    \fi%
  \else%
    \setlength{\unitlength}{\svgwidth}%
  \fi%
  \global\let\svgwidth\undefined%
  \global\let\svgscale\undefined%
  \makeatother%
  \begin{picture}(1,1.00090167)%
    \lineheight{1}%
    \setlength\tabcolsep{0pt}%
    \put(-0.00640764,0.93389194){\color[rgb]{0,0,0}\makebox(0,0)[lt]{\lineheight{1.25}\smash{\begin{tabular}[t]{l}$i$\end{tabular}}}}%
    \put(0.87482356,0.93575039){\color[rgb]{0,0,0}\makebox(0,0)[lt]{\lineheight{1.25}\smash{\begin{tabular}[t]{l}$i$\end{tabular}}}}%
    \put(0,0){\includegraphics[width=\unitlength,page=1]{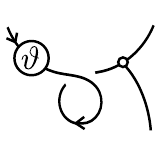}}%
    \put(0.42494064,0.01492361){\color[rgb]{0,0,0}\makebox(0,0)[lt]{\lineheight{1.25}\smash{\begin{tabular}[t]{l}$i$\end{tabular}}}}%
    \put(0,0){\includegraphics[width=\unitlength,page=2]{ev-proof-8.pdf}}%
  \end{picture}%
\endgroup%
}}~=~\vcenter{\hbox{{\def\svgscale{0.6}}%% Creator: Inkscape inkscape 0.92.4, www.inkscape.org
%% Accompanies image file 'ev-proof-9.pdf' (pdf, eps, ps)
\begingroup%
  \makeatletter%
  \providecommand\color[2][]{%
    \errmessage{(Inkscape) Color is used for the text in Inkscape, but the package 'color.sty' is not loaded}%
    \renewcommand\color[2][]{}%
  }%
  \providecommand\transparent[1]{%
    \errmessage{(Inkscape) Transparency is used (non-zero) for the text in Inkscape, but the package 'transparent.sty' is not loaded}%
    \renewcommand\transparent[1]{}%
  }%
  \providecommand\rotatebox[2]{#2}%
  \newcommand*\fsize{\dimexpr\f@size pt\relax}%
  \newcommand*\lineheight[1]{\fontsize{\fsize}{#1\fsize}\selectfont}%
  \ifx\svgwidth\undefined%
    \setlength{\unitlength}{47.24581359bp}%
    \ifx\svgscale\undefined%
      \relax%
    \else%
      \setlength{\unitlength}{\unitlength * \real{\svgscale}}%
    \fi%
  \else%
    \setlength{\unitlength}{\svgwidth}%
  \fi%
  \global\let\svgwidth\undefined%
  \global\let\svgscale\undefined%
  \makeatother%
  \begin{picture}(1,1.04046672)%
    \lineheight{1}%
    \setlength\tabcolsep{0pt}%
    \put(-0.00640764,0.973457){\color[rgb]{0,0,0}\makebox(0,0)[lt]{\lineheight{1.25}\smash{\begin{tabular}[t]{l}$i$\end{tabular}}}}%
    \put(0.87482356,0.97531545){\color[rgb]{0,0,0}\makebox(0,0)[lt]{\lineheight{1.25}\smash{\begin{tabular}[t]{l}$i$\end{tabular}}}}%
    \put(0,0){\includegraphics[width=\unitlength,page=1]{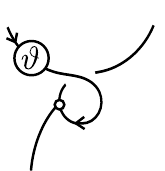}}%
    \put(0.64878054,0.34548027){\color[rgb]{0,0,0}\makebox(0,0)[lt]{\lineheight{1.25}\smash{\begin{tabular}[t]{l}$i$\end{tabular}}}}%
    \put(0,0){\includegraphics[width=\unitlength,page=2]{ev-proof-9.pdf}}%
  \end{picture}%
\endgroup%
}}\\
&\xlongequal{\eqref{eq58},~\eqref{eq55}}\ev^A_{\ovl i,i}\cdot\mu_{\ovl i,i}\ss_{i,\ovl i}(\vartheta_i\otimes\id_{\ovl i})\xlongequal{\eqref{eq60}}\ev^A_{\ovl i,i}\ss^A_{i,\ovl i}\cdot\mu_{i,\ovl i}(\vartheta_i\otimes\id_{\ovl i})\xlongequal{\eqref{eq45}}\ev^A_{\ovl i,i}\ss^A_{i,\ovl i}(\vartheta_i\otimes_A\id_{\ovl i})\cdot\mu_{i,\ovl i}.
\end{align*}
\end{proof}

A modular tensor category is called \emph{unitary} if it is a (rigid) braided $C^*$-tensor category, and if its twist is the canonical unitary twist associated to the rigid braided $C^*$-tensor structure.
\begin{co}
Let $(\mc C,\boxtimes,\ss,\vartheta)$ be a unitary modular tensor category. If $A$ is a haploid commutative Q-system in $\mc C$, then $(\RepA,\boxtimes_A,\ss^A,\vartheta)$ is also a unitary modular tensor category.
\end{co}

\begin{proof}
We have shown that $A$ has trivial $\mc C$-twist. Thus by \cite{KO02} theorem 4.5, $\RepA$ is a modular tensor category.\footnote{Our $\RepA$ is written as $\Rep^0(A)$ in \cite{KO02}.} By the above theorem, when restricted to $\RepA$, $\vartheta$ is the canonical unitary twist $\vartheta^A$ of $\RepA$. Thus $\RepA$ is a unitary modular tensor category. 
\end{proof}

\subsection{Complete unitarity of unitary VOA extensions}

Recall that $V$ is a CFT-type, regular, and completely unitary VOA. Let $U$ be a CFT-type unitary extension of $V$, and let $A_U=(W_a,\mu,\iota)$ be the corresponding standard commutative Q-system. (Note that the trivial twist condition for $A$ is now redundant by theorem \ref{lb44}.) Recall by theorem \ref{lb23} that any unitary $U$-module is naturally a single-valued unitary left $A_U$-module, and any single-valued unitary left $A_U$-module can be regarded as a unitary $U$-module. Thus  $\RepU$ is naturally equivalent to $\RepAU$ as $C^*$-categories. Note that $U$ is also regular (equivalently, rational and $C_2$-cofinite \cite{ABD04}) by the proof of \cite{McR20} theorem 4.13.\footnote{Although \cite{McR20} theorem 4.13 only discusses orbifold type extensions, the argument there is quite general and clearly applies to the general case. Note that the nonzeroness of quantum dimensions required in that theorem is obvious in the unitary case.} Thus, just as $V$, the tensor category of $U$-modules is (rigid and) modular. In the following, we shall show that $U$ is completely unitary, which implies that $\Rep^\uni(U)$ is a unitary modular tensor category. Moreover, we shall show that the unitary modular tensor categories $\RepU$ and $\RepAU$ are naturally equivalent.

Let $W_i,W_j,W_k$ be unitary $U$-modules, which can be regarded  respectively as unitary $A_U$-bimodules with left actions $\mu^i_L,\mu^j_L,\mu^k_L$ and  right actions $\mu^i_R,\mu^j_R,\mu^k_R$ related by $\ss$. Recall \eqref{eq61}. Then $\mc Y_{\mu^i_L},\mc Y_{\mu^j_L},\mc Y_{\mu^k_L}$ are the vertex operators of $U$ on $W_i,W_j,W_k$ respectively.  We let $\mc V_U{k\choose i~j}$ be the vector space of type $k\choose i~j$ intertwining operators of $U$. \index{Vijk@$\mc V_U{k\choose i~j}$} Again, $\mc V{k\choose i~j}$ denotes the vector space of  type $k\choose i~j$ intertwining operators of $V$. Since any intertwining operator of $U$ is also an intertwining operator of $V$, $\mc V_U{k\choose i~j}$ is a subspace of $\mc V{k\choose i~j}$. We give a categorical interpretation of $\mc V_U{k\choose i~j}$. Note that $W_i\boxtimes W_j$ is naturally a unitary $A_U$-bimodule, with left and actions defined by the left action of $W_i$ and the right action of $W_j$.

\begin{lm}\label{lb45}
Any $\alpha\in\Hom_{A_U}(W_i\boxtimes W_j,W_k)$ satisfies $\alpha\Pij=0$.
\end{lm}
\begin{proof}
This is easy to prove using graphical calculus and the fact that the left actions of $W_i,W_j,W_k$ are related by $\ss$ to the right ones.
\end{proof}

\begin{pp}[\cite{CKM17} section 3.4]\label{lb46}
The map $\mc Y:\Hom(W_i\boxtimes W_j,W_k)\xrightarrow{\simeq} \mc V{k\choose i~j},~\alpha\mapsto\mc Y_\alpha$ (see section \ref{lb2}) restricts to an isomorphism $\mc Y:\Hom_{A_U}(W_i\boxtimes W_j,W_k)\xrightarrow{\simeq} \mc V_U{k\choose i~j}$.
\end{pp}

\begin{proof}
We sketch the proof here; details can be found in the reference provided. Choose any $\alpha\in\Hom(W_i\boxtimes W_j,W_k)$. Then $\mc Y_\alpha$ being an intertwining operator of $U$ means precisely that $\mc Y_\alpha$ satisfies the Jacobi identity with the vertex operator of $U$. By contour integrals, the Jacobi identity is well known to be equivalent to the fusion relations
\begin{gather}
\mc Y_{\mu^k_L}(u,z)\mc Y_\alpha(w^{(i)},\zeta)=\mc Y_\alpha(\mc Y_{\mu^i_L}(u,z-\zeta)w^{(i)},\zeta),\label{eq71}\\
\mc Y_\alpha(w^{(i)},\zeta)\mc Y_{\mu^k_L}(u,z)=\mc Y_\alpha(\mc Y_{\mu^i_L}(u,z-\zeta)w^{(i)},\zeta)\label{eq70}
\end{gather}
for any $u\in U=W_a,w^{(i)}\in W_i$, where $0<|z-\zeta|<|\zeta|<|z|$ and $\arg(z-\zeta)=\arg\zeta=\arg z$ in the first equation, and $0<|\zeta-z|<|z|<|\zeta|$ and $\arg z=\arg \zeta=\pi+\arg(z-\zeta)$ in the second one. (See for instance \cite{Gui19a} proposition 2.13.) By \cite{Gui19a} proposition 2.9, \eqref{eq70} is equivalent to the fusion relation
\begin{align}
\mc Y_\alpha(w^{(i)},\zeta)\mc Y_{\mu^k_L}(u,z)=\mc Y_\alpha(\mc Y_{\mu^i_R}(w^{(i)},\zeta-z)u,z).\label{eq72}
\end{align}
The categorical interpretations of \eqref{eq71} and \eqref{eq72}  are respectively $\alpha\in\Hom_{A_U,-}(W_i\boxtimes W_j,W_k)$ and $\alpha\Pij=0$, which are clearly equivalent to that $\alpha\in\Hom_{A_U}(W_i\boxtimes W_j,W_k)$.
\end{proof}

Recall the definition of VOA modules in section \ref{lb2}. Recall by theorem \ref{lb53} that any irreducible $U$-module admits a unitary structure. We choose a representative $W_t$ for each equivalence class $[W_t]$ of irreducible unitary $U$-modules (equivalently, irreducible unitary single-valued left $A_U$-modules), and let all these $W_t$ form a set $\mc E_U$. \index{EU@$\mc E_U$} That $W_t\in\mc E_U$ is abbreviated to $t\in\mc E_U$. We also assume that the vacuum $U$-module $W_a$ is in $\mc E_U$. Then the tensor product of $U$-modules $W_i,W_j$ is
\begin{align}
W_{ij}\equiv W_i\boxtimes_U W_j=\bigoplus_{t\in\mc E_U}\mc V_U{t\choose i~j}^*\otimes W_t.\label{eq63}
\end{align}
We choose an inner product $\Lambda_U$ for any $\mc V_U{t\choose i~j}^*$, and assume that the above direct sum is orthogonal.  The vertex operator for $W_{ij}$ is $\bigoplus_t\id\otimes\mc Y_{\mu^t_L}$, where $\mu^t_L\in\Hom_{A_U}(W_a\boxtimes W_t,W_t)$ is the left action of the $A_U$-bimodule $W_t$.

Define a $U$-intertwining operator $\mc Y_{\mu_{i,j}}$ of type ${ij\choose i~j}={W_{ij}\choose W_iW_j}$, such that for any $w^{(i)}\in W_i,w^{(j)}\in W_j,t\in\mc E_U,\mc Y_\alpha\in\mc V_U{t\choose i~j}$, and $w^{(\ovl t)}\in W_{\ovl t}$ (the contragredient unitary $U$-module of $W_t$),	
\begin{align}
\bk{\mc Y_{\mu_{i,j}}(w^{(i)},z)w^{(j)},\mc Y_\alpha\otimes w^{(\ovl t)}}=\bk{\mc Y_\alpha(w^{(i)},z)w^{(j)},w^{(\ovl t)}}.	\label{eq67}
\end{align}
To write the above definition more explicitly, we choose a basis $\Upsilon^t_{i,j}$ of the vector space $\Hom_{A_U}(W_i\boxtimes W_j,W_t)$.\index{zz@$\Upsilon^t_{i,j}$} Then $\{\mc Y_\alpha:\alpha\in\Upsilon^k_{i,j}\}$ is a basis of $\mc V_U{t\choose i~j}$ whose dual basis is denoted by $\{ \widecheck {\mc Y}^\alpha:\alpha\in\Upsilon^k_{i,j}\}$. Then
\begin{align}
\mc Y_{\mu_{i,j}}(w^{(i)},z)w^{(j)}=\sum_{t\in\mc E_U}\sum_{\alpha\in\Upsilon^t_{i,j}}\widecheck {\mc Y}^\alpha\otimes\mc Y_\alpha(w^{(i)},z)w^{(j)}.
\end{align}
Note that $\mc \mu_{i,j}\in\Hom_{A_U}(W_i\boxtimes W_j,W_{ij})$. Then the above relation can also be written as
\begin{align}
\mu_{i,j}=\sum_{t\in\mc E_U}\sum_{\alpha\in\Upsilon^t_{i,j}}\widecheck {\mc Y}^\alpha\otimes\alpha.\label{eq64}
\end{align}

By lemma \ref{lb45} we have $\mu_{i,j}\Pij=0$ . We claim that $\mc Y_{\mu_{i,j}}$ satisfies the universal property that for any unitary $U$-module $W_k$ and any $\mc Y_\alpha$ in $\mc V_U{k\choose i~j}$ (equivalently, $\alpha\in\Hom_{A_U}(W_i\boxtimes W_j,W_k)$) there exists a unique $U$-module homomorphism $\wtd\alpha:W_{ij}\rightarrow W_k$ (equivalently, $\wtd\alpha\in\Hom_{A_U}(W_{ij},W_k)$) such that $\mc Y_\alpha=\wtd\alpha\mc Y_{\mu_{i,j}}$ (equivalently, $\alpha=\wtd\alpha\mu_{i,j}$ by \eqref{eq31}). Indeed, since  the vector space $\Hom_U(W_{ij},W_k)$  of $U$-module morphisms from $W_{ij}$ to $W_k$ is naurally identified with $\Hom_{A_U}(W_{ij},W_k)$, just as \eqref{eq61}, we have a natural isomorphism of vector spaces
\begin{align}
\wtd{\mc Y}:\Hom_U(W_{ij},W_k)=\Hom_{A_U}(W_{ij},W_k)\rightarrow\mc V_U{k\choose i~j},\qquad \wtd\alpha\mapsto\wtd{\mc Y}_{\wtd\alpha}.
\end{align}
It is easy check for any $\wtd\alpha\in\Hom_{A_U}(W_{ij},W_k)$ that
\begin{align}
\wtd{\mc Y}_{\wtd\alpha}=\wtd\alpha\mc Y_{\mu_{i,j}},\label{eq69}
\end{align}
which by \eqref{eq31} also equals $\mc Y_{\wtd\alpha\mu_{i,j}}$. Now, by proposition \ref{lb46}, for any $\alpha\in\Hom_{A_U}(W_i\boxtimes W_j,W_k)$, we can find an $\wtd\alpha$ satisfying
\begin{align}
\mc Y_\alpha=\wtd{\mc Y}_{\wtd\alpha}.\label{eq68}
\end{align}
Thus we have $\mc Y_\alpha=\wtd{\mc Y}_{\wtd\alpha}=\wtd\alpha\mc Y_{\mu_{i,j}}=\mc Y_{\wtd\alpha\mu_{i,j}}$, which shows $\alpha=\wtd\alpha\mu_{i,j}$. Recalling definition \ref{lb47}, we conclude that $(W_{ij},\mu_{i,j})$ is a tensor product of the $A_U$-bimodules $W_i,W_j$ over $A_U$. Indeed, under suitable choice of $\Lambda_U$ the tensor products become unitary:

\begin{thm}\label{lb50}
There exists for each $t\in\mc E_U$ a unique inner product $\Lambda_U$ on the vector space $\mc V_U{t\choose i~j}^*$ such that $(W_{ij},\mu_{i,j})$ becomes a unitary tensor product of the $A_U$-bimodules $W_i,W_j$ over $A_U$. Moreover, $\Lambda_U$ is the invariant sesquilinear form of $U$ (cf. section \ref{lb48}).
\end{thm}

\begin{proof}
Let $(W_{i\bullet j},\eta_{i,j})$ be a unitary tensor product of $W_i,W_j$ over $A_U$. By the first half of the proof of theorem \ref{lb28}, tensor products of unitary bimodules of a Q-system are unique up to multiplications by invertible morphisms. Thus there exists an invertible $K\in\Hom_{A_U}(W_{ij},W_{i\bullet j})$ such that $\eta_{i,j}=K\mu_{i,j}$. Therefore, the  decomposition of $W_{i\bullet j}$ into irreducible single-valued left $A_U$-modules is the same as that of $W_{ij}$, which takes the form  \eqref{eq63}. Now, using linear algebra, one can easily find an inner product $\Lambda_U$ on any $\mc V_U{t\choose i~j}^*$, such that $K$ becomes unitary. Then the tensor product $(W_{i\bullet j},\eta_{i,j})$ defined by such $\Lambda_U$ is clearly unitary. This proves the existence of $\Lambda_U$. The uniqueness of $\Lambda_U$ follows from the uniqueness up to unitaries of the unitary tensor products of $A_U$-bimodules (theorem \ref{lb28}).

Now assume that  $(W_{ij},\mu_{i,j})$ is a unitary tensor product. We show that $\Lambda_U$ is the invariant sesquilinear form. Assume that for each $t\in\mc E_U$,  $\Upsilon^t_{i,j}$ is chosen in  such a way that $\{ \widecheck {\mc Y}^\alpha:\alpha\in\Upsilon^k_{i,j}\}$ is an orthonormal basis of $\mc V_U{t\choose i~j}^*$ under $\Lambda_U$. Then  by $\Cij=\mu_{i,j}^*\mu_{i,j}$ and equation \eqref{eq64}, we have
\begin{align}
\vcenter{\hbox{{\def\svgscale{0.6}
%% Accompanies image file 'adjoint-7.pdf' (pdf, eps, ps)
\begingroup%
  \makeatletter%
  \providecommand\color[2][]{%
    \errmessage{(Inkscape) Color is used for the text in Inkscape, but the package 'color.sty' is not loaded}%
    \renewcommand\color[2][]{}%
  }%
  \providecommand\transparent[1]{%
    \errmessage{(Inkscape) Transparency is used (non-zero) for the text in Inkscape, but the package 'transparent.sty' is not loaded}%
    \renewcommand\transparent[1]{}%
  }%
  \providecommand\rotatebox[2]{#2}%
  \newcommand*\fsize{\dimexpr\f@size pt\relax}%
  \newcommand*\lineheight[1]{\fontsize{\fsize}{#1\fsize}\selectfont}%
  \ifx\svgwidth\undefined%
    \setlength{\unitlength}{58.40038919bp}%
    \ifx\svgscale\undefined%
      \relax%
    \else%
      \setlength{\unitlength}{\unitlength * \real{\svgscale}}%
    \fi%
  \else%
    \setlength{\unitlength}{\svgwidth}%
  \fi%
  \global\let\svgwidth\undefined%
  \global\let\svgscale\undefined%
  \makeatother%
  \begin{picture}(1,2.04920951)%
    \lineheight{1}%
    \setlength\tabcolsep{0pt}%
    \put(-0.00518377,1.99650224){\color[rgb]{0,0,0}\makebox(0,0)[lt]{\lineheight{1.25}\smash{\begin{tabular}[t]{l}$i$\end{tabular}}}}%
    \put(0.88216243,1.99613488){\color[rgb]{0,0,0}\makebox(0,0)[lt]{\lineheight{1.25}\smash{\begin{tabular}[t]{l}$j$\end{tabular}}}}%
    \put(0,0){\includegraphics[width=\unitlength,page=1]{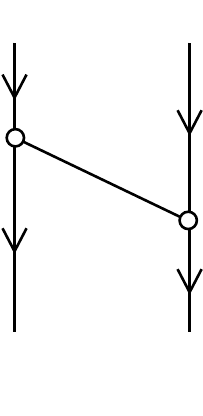}}%
    \put(0.01243959,0.01451404){\color[rgb]{0,0,0}\makebox(0,0)[lt]{\lineheight{1.25}\smash{\begin{tabular}[t]{l}$i$\end{tabular}}}}%
    \put(0.87132956,0.01414668){\color[rgb]{0,0,0}\makebox(0,0)[lt]{\lineheight{1.25}\smash{\begin{tabular}[t]{l}$j$\end{tabular}}}}%
  \end{picture}%
\endgroup%
}}}~~~=~~\sum_{t\in\mc E_U}\sum_{\alpha\in\Upsilon^t_{ij}}\vcenter{\hbox{{\def\svgscale{0.6}
			}}}~~,
\end{align}
and hence
\begin{align}
\vcenter{\hbox{{\def\svgscale{0.6}
%% Accompanies image file 'adjoint-8.pdf' (pdf, eps, ps)
\begingroup%
  \makeatletter%
  \providecommand\color[2][]{%
    \errmessage{(Inkscape) Color is used for the text in Inkscape, but the package 'color.sty' is not loaded}%
    \renewcommand\color[2][]{}%
  }%
  \providecommand\transparent[1]{%
    \errmessage{(Inkscape) Transparency is used (non-zero) for the text in Inkscape, but the package 'transparent.sty' is not loaded}%
    \renewcommand\transparent[1]{}%
  }%
  \providecommand\rotatebox[2]{#2}%
  \newcommand*\fsize{\dimexpr\f@size pt\relax}%
  \newcommand*\lineheight[1]{\fontsize{\fsize}{#1\fsize}\selectfont}%
  \ifx\svgwidth\undefined%
    \setlength{\unitlength}{59.64504442bp}%
    \ifx\svgscale\undefined%
      \relax%
    \else%
      \setlength{\unitlength}{\unitlength * \real{\svgscale}}%
    \fi%
  \else%
    \setlength{\unitlength}{\svgwidth}%
  \fi%
  \global\let\svgwidth\undefined%
  \global\let\svgscale\undefined%
  \makeatother%
  \begin{picture}(1,1.96857552)%
    \lineheight{1}%
    \setlength\tabcolsep{0pt}%
    \put(0,0){\includegraphics[width=\unitlength,page=1]{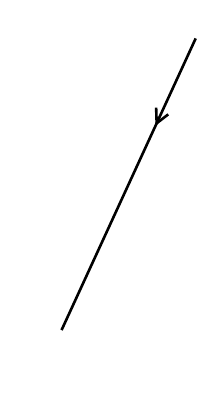}}%
    \put(-0.00161539,1.90458717){\color[rgb]{0,0,0}\makebox(0,0)[lt]{\lineheight{1.25}\smash{\begin{tabular}[t]{l}$i$\end{tabular}}}}%
    \put(0.41499815,1.91190204){\color[rgb]{0,0,0}\makebox(0,0)[lt]{\lineheight{1.25}\smash{\begin{tabular}[t]{l}$i$\end{tabular}}}}%
    \put(0.89626129,1.91696813){\color[rgb]{0,0,0}\makebox(0,0)[lt]{\lineheight{1.25}\smash{\begin{tabular}[t]{l}$j$\end{tabular}}}}%
    \put(0,0){\includegraphics[width=\unitlength,page=2]{adjoint-8.pdf}}%
    \put(0.240334,0.01385148){\color[rgb]{0,0,0}\makebox(0,0)[lt]{\lineheight{1.25}\smash{\begin{tabular}[t]{l}$j$\end{tabular}}}}%
  \end{picture}%
\endgroup%
}}}~~~=~~\sum_{t\in\mc E_U}\sum_{\alpha\in\Upsilon^t_{ij}}\vcenter{\hbox{{\def\svgscale{0.6}
%% Accompanies image file 'adjoint-9.pdf' (pdf, eps, ps)
\begingroup%
  \makeatletter%
  \providecommand\color[2][]{%
    \errmessage{(Inkscape) Color is used for the text in Inkscape, but the package 'color.sty' is not loaded}%
    \renewcommand\color[2][]{}%
  }%
  \providecommand\transparent[1]{%
    \errmessage{(Inkscape) Transparency is used (non-zero) for the text in Inkscape, but the package 'transparent.sty' is not loaded}%
    \renewcommand\transparent[1]{}%
  }%
  \providecommand\rotatebox[2]{#2}%
  \newcommand*\fsize{\dimexpr\f@size pt\relax}%
  \newcommand*\lineheight[1]{\fontsize{\fsize}{#1\fsize}\selectfont}%
  \ifx\svgwidth\undefined%
    \setlength{\unitlength}{64.2530111bp}%
    \ifx\svgscale\undefined%
      \relax%
    \else%
      \setlength{\unitlength}{\unitlength * \real{\svgscale}}%
    \fi%
  \else%
    \setlength{\unitlength}{\svgwidth}%
  \fi%
  \global\let\svgwidth\undefined%
  \global\let\svgscale\undefined%
  \makeatother%
  \begin{picture}(1,1.85761453)%
    \lineheight{1}%
    \setlength\tabcolsep{0pt}%
    \put(0,0){\includegraphics[width=\unitlength,page=1]{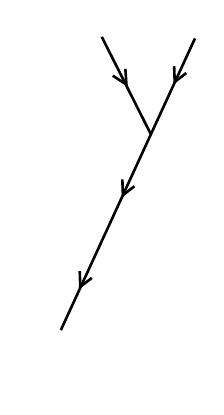}}%
    \put(-0.0047116,1.79821516){\color[rgb]{0,0,0}\makebox(0,0)[lt]{\lineheight{1.25}\smash{\begin{tabular}[t]{l}$i$\end{tabular}}}}%
    \put(0.3820241,1.80500544){\color[rgb]{0,0,0}\makebox(0,0)[lt]{\lineheight{1.25}\smash{\begin{tabular}[t]{l}$i$\end{tabular}}}}%
    \put(0.82877299,1.80970821){\color[rgb]{0,0,0}\makebox(0,0)[lt]{\lineheight{1.25}\smash{\begin{tabular}[t]{l}$j$\end{tabular}}}}%
    \put(0.40553729,1.05728956){\color[rgb]{0,0,0}\makebox(0,0)[lt]{\lineheight{1.25}\smash{\begin{tabular}[t]{l}$t$\end{tabular}}}}%
    \put(0.20332515,0.0128581){\color[rgb]{0,0,0}\makebox(0,0)[lt]{\lineheight{1.25}\smash{\begin{tabular}[t]{l}$j$\end{tabular}}}}%
    \put(0.72654491,1.17511938){\color[rgb]{0,0,0}\makebox(0,0)[lt]{\lineheight{1.25}\smash{\begin{tabular}[t]{l}$\alpha$\end{tabular}}}}%
    \put(0.50626289,0.63765951){\color[rgb]{0,0,0}\makebox(0,0)[lt]{\lineheight{1.25}\smash{\begin{tabular}[t]{l}$\alpha^*$\end{tabular}}}}%
    \put(0,0){\includegraphics[width=\unitlength,page=2]{adjoint-9.pdf}}%
  \end{picture}%
\endgroup%
}}}~~,
\end{align}
which by theorem \ref{lb6} implies for any $w_1^{(i)},w_2^{(i)}\in W_i$ the fusion relation
\begin{align}
\mc Y_{\mu^j_L}\big(\mc Y_{(\mu^i_R)^\dagger}(\ovl{w^{(i)}_2},z-\zeta)w^{(i)}_1,\zeta \big)=\sum_{t\in\mc E_U}\sum_{\alpha\in\Upsilon^t_{ij}}\mc Y_{\alpha^\dagger}(\ovl{w^{(i)}_2},z)\mc Y_\alpha(w^{(i)}_1,\zeta).\label{eq66}
\end{align}
Recall that $\mc Y_{\mu^j_L}$ is the vertex operator of $U$ on $W_j$. Since $\mu^i_R=\mu^i_L\ss_{a,i}$, we have $\mc Y_{\mu^i_R}=B_+\mc Y_{\mu^i_L}$ by \eqref{eq65}, which shows that $\mc Y_{\mu^i_R}\in\mc V_U{i\choose i~a}$ is the creation operator of the $U$-module $W_i$. (See section \eqref{lb2} for the definition of creation and annihilation operators). Thus by \eqref{eq9}, $\mc Y_{(\mu^i_R)^\dagger}\in\mc V_U{a\choose\ovl i~i}$ is the annihilation operator of the $U$-module $W_j$. Therefore, by definition \ref{lb49}, we see that \eqref{eq66} is the fusion relation that defines invariant sesquilinear forms for $U$. This shows that $\{ \widecheck {\mc Y}^\alpha:\alpha\in\Upsilon^k_{i,j}\}$ is also an orthonormal basis of $\mc V_U{t\choose i~j}^*$ under the invariant sesquilinear form, which proves that the later is positive definite and equals $\Lambda_U$.
\end{proof}

\begin{thm}\label{lb55}
Let $V$ be a CFT-type, regular, and completely unitary VOA, and let $U$ be a CFT-type unitary VOA extension of $V$. Then:
\begin{itemize}
	\item $U$ is also (regular and) completely unitary.
	\item Under the natural identification of $\RepU$ and $\RepAU$ as $C^*$-categories,  the monoidal, braiding, and ribbon structures of $\RepU$ agree with those of $(\RepAU,\boxtimes_{A_U},\ss^{A_U},\vartheta^{A_U})$ defined by the system of unitary tensor products $(W_{ij},\mu_{i,j})$ (for any $W_i,W_j$ in $\RepAU$) as constructed in \eqref{eq63}, \eqref{eq67} under the invariant inner product $\Lambda_U$.
\end{itemize}

\end{thm}

By ``natural identification", we mean that each unitary $U$-module $W_i$ is identified with the corresponding single-valued unitary left $A_U$-module $W_i$; a homomorphism $F:W_i\rightarrow W_j$ of unitary $U$-modules is identified with $F$, considered as a homomorphism of $A_U$-bimodules.

\begin{proof}
As mentioned at the beginning of this section, the regularity of $U$ is proved in \cite{McR20}. By theorems \ref{lb53} and \ref{lb50}, $U$ is completely unitary. Hence $\Rep^\uni(U)$ is a unitary modular tensor category. That $\RepU$ and $\RepAU$ share the  same tensor and braiding structures is proved in \cite{CKM17}. In order for this paper to be self-contained, we sketch the proof as follows.

For any objects $W_i,W_j,W_k$, let $\mc A_{i,j,k}:W_{(ij)k}=(W_i\boxtimes_U W_j)\boxtimes_U W_k\rightarrow W_{i(jk)}=W_i\boxtimes_U(W_j\boxtimes_U W_k)$ be the associativity isomorphism of $\RepU$. It is shown in \cite{Gui20} proposition 4.3 that under the identification of $W_{(ij)k}$ and $W_{i(jk)}$ via $\mc A_{i,j,k}$, one has the fusion relation
\begin{align}
\mc Y_{\mu_{i,jk}}(w^{(i)},z)\mc Y_{\mu_{j,k}}(w^{(j)},\zeta)=\mc Y_{\mu_{ij,k}}(\mc Y_{\mu_{i,j}}(w^{(i)},z-\zeta)w^{(j)},
\zeta)
\end{align}
for any $w^{(i)}\in W_i,w^{(j)}\in W_j$. This means that relation \eqref{eq48} holds under the identification via $\mc A_{i,j,k}$. But we know that due to equation \eqref{eq46}, the same relations also hold under the identification via $\fk A_{i,j,k}$, the associativity isomorphism of $\RepAU$. Thus $\mc A_{i,j,k}=\fk A_{i,j,k}$.

Next, we know that in $\RepU$, the identification $W_{ai}\simeq W_i$ is via the vertex operator $\mc Y_{\mu^i_L}$ of $U$ on $W_i$, and the identification $W_{ia}\simeq W_i$ is via the creation operator of the $U$-module $W_i$, which is $\mc Y_{\mu^i_R}$, as argued at the end of the proof of theorem \ref{lb50}. Define $\fk l_i\in\Hom_{A_U}(W_{ai},W_i)$ and $\fk r_i\in\Hom_{A_U}(W_{ia},W_i)$ using equations \eqref{eq52}. Then, from section \ref{lb52}, we know that $\fk l_i$ and $\fk r_i$ define respectively the equivalences $W_{ai}\simeq W_i$ and $W_{ia}\simeq W_i$ in $\RepAU$ (as a full $C^*$-tensor subcategory of $\BIM(A_U)$). On the other hand, by \eqref{eq68} we have $\mc Y_{\mu^i_L}=\wtd{\mc Y}_{\fk l_i}$. Therefore $\fk l_i:W_{ai}=W_a\boxtimes_U W_i\rightarrow W_i$ is the $U$-module homomorphism corresponding to the vertex operator of the $U$-module $W_i$. Thus, by the definition of the monoidal stuctures of VOA tensor categories (see section \ref{lb2}), $\fk l_i$ also defines the equivalence  $W_{ai}\simeq W_i$ in $\RepU$.  Similarly, $\fk r_i$ is the $U$-module morphism corresponding to the creation operator of $W_i$. Hence it defines the equivalence $W_{ia}\simeq W_i$ in $\RepU$. We have now proved that the ($C^*$-)monoidal structure of $\RepU$ agrees with that of $\RepAU$.

Let $\ss^U_{i,j}\in\Hom_U(W_{ij},W_{ji})$ be the braiding of $W_i\boxtimes_U W_j$ in $\Rep^U$. We want to show that $\ss^U$ equals the braiding $\ss^{A_U}$ of $\RepAU$. By \eqref{eq60} it suffices to check that for any object $W_k$ in $\RepAU$ and any $\wtd\alpha\in\Hom_{A_U}(W_{ij},W_k)$,
\begin{align}
\wtd\alpha\mu_{j,i}\ss_{i,j}=\wtd\alpha\ss^U_{i,j}\mu_{i,j},
\end{align}
where $\ss$ is the braiding of $\RepV$. Set $\alpha=\wtd\alpha\mu_{j,i}\in\Hom_{A_U}(W_i\boxtimes W_j,W_k)$. Then by \eqref{eq65}, $\mc Y_{\alpha\ss_{i,j}}=B_+\mc Y_\alpha$, and similarly $\wtd{\mc Y}_{\wtd\alpha\ss^U_{i,j}}=B_+\wtd{\mc Y}_{\wtd\alpha}$. Note that the braiding $B_+$ defined for $V$-intertwining operators and for $U$-intertwining operators are the same since $U$ and $V$ have the same Virasoro operators. We now compute
\begin{align*}
\mc Y_{\wtd\alpha\ss^U_{i,j}\mu_{i,j}}\xlongequal{\eqref{eq31}} \wtd\alpha\ss^U_{i,j} \mc Y_{\mu_{i,j}} \xlongequal{\eqref{eq69}} \wtd{\mc Y}_{\wtd\alpha\ss^U_{i,j}}=B_+\wtd{\mc Y}_{\wtd\alpha} \xlongequal{\eqref{eq68}} B_+\mc Y_\alpha=\mc Y_{\alpha\ss_{i,j}}=\mc Y_{\wtd\alpha\mu_{j,i}\ss_{i,j}}.
\end{align*}

Finally,  for both categories the twists are defined by the rigid braided $C^*$-tensor structures. Therefore the ribbon structures agree.
\end{proof}

\subsection{Applications}

To use theorem \ref{lb55} in its full power, we first prove the complete unitarity for another type of extensions (which do not preserve conformal vectors).

\begin{pp}\label{lb56}
Let $V$ and $\wtd V$ be CFT-type and regular VOAs. Then $V\otimes\wtd V$ is (regular and) completely unitary if and only if both $V$ and $\wtd V$ are completely unitary. If this is true then $\Rep^\uni(V\otimes \wtd V)$ is the tensor product of $\Rep^\uni(V)$ and $\Rep^\uni(\wtd V)$. 
\end{pp}

\begin{proof}
Clearly $V\otimes\wtd V$ is CFT-type. Note that  $V\otimes\wtd V$ is also regular by \cite{DLM97} proposition 3.3.  Assume first of all that $V$ and $\wtd V$ are completely unitary. By \cite{DL14} proposition 2.9, $V\otimes\wtd V$ is unitary.  By \cite{FHL93} theorem 4.7.4, any irreducible $V\otimes\wtd V$-module is the tensor product of an irreducible $V$-module and an irreducible $\wtd V$-module, which by the strong unitarity of $V$ and $\wtd V$ are unitarizable. Therefore the $V\otimes\wtd V$-module is also unitarizable, and hence $V\otimes\wtd V$ is strongly unitary.
	
We now show that $V\otimes \wtd V$ is completely unitary. Choose   unitary $V$-modules $W_i,W_j$ and unitary $\wtd V$-modules $W_{\wtd i},W_{\wtd j}$. Choose $W_t$ in $\mc E$. Let also $\wtd{\mc E}$ be a complete set of representatives of irreducible $\wtd V$-modules, and choose any $W_{\wtd t}$ in $\mc E$. Choose bases $\Xi^t_{i,j}$ of $\mc V{t\choose i~j}$ and $\wtd\Xi^{\wtd t}_{\wtd i,\wtd j}$ of $\wtd{\mc V}{\wtd t\choose \wtd i~\wtd j}$ (the vector space of type $\wtd t\choose \wtd i~\wtd j$ intertwining operators of $\wtd V$) so that their dual bases $\{\widecheck{\mc Y}^\alpha:\alpha\in \Xi^t_{i,j}\}$ and $\{\widecheck{\mc Y}^{\wtd \alpha}:\wtd \alpha\in \wtd\Xi^{\wtd t}_{\wtd i,\wtd j}\}$ are orthonormal under the invariant inner products in $\mc V{t\choose i~j}^*$ and $\wtd{\mc V}{\wtd t\choose \wtd i~\wtd j}^*$ respectively. Then equation \eqref{eq11} holds for $V$ with $\Lambda(\widecheck{\mc Y}^\alpha|\widecheck{\mc Y}^\beta)=\delta_{\alpha,\beta}$, and a similar relation holds for $\wtd V$. Let $Y_{j\otimes\wtd j}$ be the vertex operator of the $V\otimes\wtd V$-module $W_j\otimes W_{\wtd j}$, and let  $\mc Y_{\ev_{\ovl {i\otimes\wtd i},i\otimes\wtd i}}$ be the annihilation operator of $W_i\otimes W_{\wtd i}$. Then $Y_{j\otimes\wtd j}=Y_j\otimes Y_{\wtd j}$ and $\mc Y_{\ev_{\ovl {i\otimes\wtd i},i\otimes\wtd i}}=\mc Y_{\ev_{\ovl i,i}}\otimes\mc Y_{\ev_{\ovl {\wtd i},\wtd i}}$. Set $\mc Y_{\alpha\otimes\wtd\alpha}=\mc Y_\alpha\otimes\mc Y_{\wtd\alpha}$, which is a type $t\otimes\wtd t\choose i\otimes\wtd i~j\otimes\wtd j$ intertwining operator of $V\otimes \wtd V$ if $\alpha\in\Xi^t_{i,j},\wtd\alpha\in\wtd\Xi^{\wtd t}_{\wtd i,\wtd j}$. Then  for any $w_1^{(i)},w_2^{(i)}\in W_i,w_3^{(\wtd i)},w_4^{(\wtd i)}\in  W_{\wtd i}$ we have the fusion relation
	\begin{align}
	&Y_{j\otimes\wtd j}\Big(\mc Y_{\ev_{\ovl {i\otimes\wtd i},i\otimes\wtd i}}\big(\ovl{w_2^{(i)}}\otimes \ovl{w_4^{(\wtd i)}},z-\zeta\big)\big(w_1^{(i)}\otimes w_3^{(\wtd i)} \big),\zeta  \Big)\nonumber\\
	=&\sum_{t\in\mc E,\wtd t\in\wtd{\mc E}}\sum_{\alpha\in\Xi^t_{i,j},\wtd\alpha\in\wtd\Xi^{\wtd t}_{\wtd i,\wtd j}}\mc Y_{(\alpha\otimes\wtd\alpha)^\dagger}(\ovl{w_2^{(i)}}\otimes \ovl{w_4^{(\wtd i)}},z)\cdot \mc Y_{\alpha\otimes\wtd\alpha}(w_1^{(i)}\otimes w_3^{(\wtd i)},\zeta).
	\end{align}
	From this relations, we see  that the invariant sesquilinear form $\Lambda$ on $\mc V{t\otimes\wtd t\choose i\otimes\wtd i~j\otimes\wtd j}^*$ is positive. Moreover, by the non-degeneracy of this $\Lambda$,  $\Xi^t_{i,j}\times \wtd\Xi^{\wtd t}_{\wtd i,\wtd j}$ is a basis of $\mc V{t\otimes\wtd t\choose i\otimes\wtd i~j\otimes\wtd j}$ whose dual basis is therefore orthonormal in $\mc V{t\otimes\wtd t\choose i\otimes\wtd i~j\otimes\wtd j}^*$. Thus $\mc V{t\otimes\wtd t\choose i\otimes\wtd i~j\otimes\wtd j}=\mc V{t\choose i~j}\otimes \wtd{\mc V}{\wtd t\choose \wtd i~\wtd j}$, and the  $\Lambda$ on $\mc V{t\otimes\wtd t\choose i\otimes\wtd i~j\otimes\wtd j}^*$ equals $\Lambda\otimes \Lambda$ on  $\mc V{t\choose i~j}^*\otimes \wtd{\mc V}{\wtd t\choose \wtd i~\wtd j}^*$. That $\Rep^\uni(V\otimes \wtd V)=\Rep^\uni(V)\otimes \Rep^\uni(\wtd V)$ now follows easily. (It also follows from \cite{ADL05} theorem 2.10.)
	
We now prove the ``only if" part. Assume that $V\otimes\wtd V$ is completely unitary. We want to prove that $V$ (or similarly $\wtd V$) is completely unitary. Let $\wtd\Omega$ be the vacuum vector of $\wtd V$. Then $V$ can be regarded as a (non-conformal) vertex subalgebra of $V\otimes\wtd V$. The inner product on $V\otimes \wtd V$ restricts to one on $V$ which makes $V$ unitary. Now  let $W_i$ be an irreducible $V$-module. Write $\wtd V=W_{\wtd0}$ as the unitary vacuum $\wtd V$-module. Then the $V\otimes\wtd V$-module $W_i\otimes W_{\wtd 0}$ admits a unitary structure. The  restriction of the inner product of $W_i\otimes W_{\wtd 0}$ to $W_i\otimes\wtd\Omega$ produces a unitary structure on $W_i$ (cf. \cite{Ten18b} proposition 2.20).

Now choose unitary $V$-modules $W_i,W_j$, and choose $W_t$ in $\mc E$ again. Notice the natural isomorphism $\mc V{t\choose i~j}\xrightarrow{\simeq} \mc V{t\otimes\wtd 0\choose i\otimes\wtd 0~j\otimes\wtd 0}$ sending $\mc Y_\alpha\in\mc V{t\choose i~j}$ to $\mc Y_\alpha\otimes Y_{\wtd 0}$. Here $Y_{\wtd 0}$ is the vertex operator of $\wtd V$ (on the vacuum module $W_{\wtd 0}$). Such map is clearly injective. It is also surjective, since any intertwining operator in  $\mc V{t\otimes\wtd 0\choose i\otimes\wtd 0~j\otimes\wtd 0}$ can be restricted to the subspaces $W_i\otimes\wtd\Omega,W_j\otimes\wtd\Omega,W_t\otimes\wtd\Omega$ to produce the desired preimage. Note that the $\Lambda$ on $\mc V{t\otimes\wtd 0\choose i\otimes\wtd 0~j\otimes\wtd 0}^*$ is positive definite by the complete unitarity of $V\otimes\wtd V$. Choose a basis in $\mc V{t\otimes\wtd 0\choose i\otimes\wtd 0~j\otimes\wtd 0}$ whose dual basis is orthonormal under $\Lambda$. Using a suitable fusion relation, it is straightforward to check  that the corresponding basis in $\mc V{t\choose i~j}$ also has orthonormal dual basis in $\mc V{t\choose i~j}^*$ under $\Lambda$. In particular, $\Lambda$ is positive on $\mc V{t\choose i~j}^*$. This proves the complete unitarity of $V$.
\end{proof}

The above proposition implies a strategy of proving the completely unitarity of a non-conformal unitary extension $U$ of $V$. Let $V^c$ be the commutant of $V$ in $V^c$ (the coset subalgebra) which is unitary by \cite{CKLW18}. Then $U$ is a unitary (conformal) extension of $V\otimes V^c$ by \cite{Ten19a} proposition 2.21.  Now it suffices to show the regularity and complete unitarity of $V$ and $V^c$, the proof of which might require a similar trick applied to $V$ and $V^c$.

\begin{co}\label{lb57}
Let $V$ be a (finite) tensor product of $c<1$ unitary Virasoro VOAs, affine unitary VOAs, and even lattice VOAs. Let $U$ be a CFT-type unitary extension of $V$. Then $U$ is regular and completely unitary. Consequently, the category of unitary $U$-modules is a unitary modular tensor category. 
\end{co}

\begin{proof}
The affine unitary VOAs of these types are regular by \cite{DLM97} and completely unitary by \cite{Gui19b} theorem 8.4  , \cite{Gui19c} theorem 6.1, and \cite{Ten19b} theorem 5.5. The $c<1$ unitary Virasoro VOAs (resp. even lattice VOAs) are regular also by \cite{DLM97} and completely unitary by \cite{Gui19b} theorem 8.1 (resp. \cite{Gui20} theorem 5.8). Therefore, by  theorem \ref{lb55} and proposition \ref{lb56}, CFT-type unitary extensions of their tensor products are also regular and completely unitary.
\end{proof}

The above corollary is by no means  in the most general form. For example, we know that W-algebras in discrete series of type $A$ and $E$ are completely unitary by \cite{Ten19b} theorem 5.5. So one can definitely add these examples to the list in that corollary.

\begin{co}
Let $U$ be a CFT-type unitary VOA with central charge $c<1$. Then $U$ is completely unitary. Consequently, the category of unitary $U$-modules is a unitary modular tensor category. 
\end{co}

\begin{proof}
By \cite{DL14} theorem 5.1, $U$ is a unitary extension of the unitary Virasoro VOA $L(c,0)$.
\end{proof}

\printindex

\noindent {\small \sc Department of Mathematics, Rutgers University, USA.}

\noindent {\em E-mail}: bin.gui@rutgers.edu\qquad binguimath@gmail.com
\end{document}